\documentclass[11pt]{amsart}

\usepackage{amsfonts}
\usepackage{amssymb}
\usepackage{graphicx}
\usepackage{amsmath}
\usepackage{latexsym}
\usepackage{amscd}
\usepackage{xypic}
\usepackage{mathrsfs}
\usepackage{enumitem} 
\usepackage{braket}
\usepackage[margin=1.4in]{geometry}
\usepackage{bm}

\usepackage{imakeidx}
\makeindex[title=Index of Notation]

\usepackage{hyperref}

\allowdisplaybreaks 

\numberwithin{equation}{section}
\newtheorem{theorem}{Theorem}[section]
\newtheorem{lemma}[theorem]{Lemma}
\newtheorem{corollary}[theorem]{Corollary}

\newtheorem{proposition}[theorem]{Proposition}

\newtheorem*{theorem*}{Theorem}

\theoremstyle{definition}

\theoremstyle{remark}
\newtheorem*{remark}{Remark}

\theoremstyle{remark}
\newtheorem*{claim}{Claim}

\theoremstyle{definition}

\newcommand{\NN}{\mathbb{N}}

\newcommand{\RR}{\mathbb{R}}
\renewcommand{\SS}{\mathbb{S}}

\newcommand{\ZZ}{\mathbb{Z}}

\newcommand{\cB}{\mathcal B}
\newcommand{\cC}{\mathcal C}
\renewcommand{\cD}{\mathcal D}
\newcommand{\cE}{\mathcal E}
\newcommand{\cF}{\mathcal F}
\newcommand{\cG}{\mathcal G}
\renewcommand{\cH}{\mathcal H}
\newcommand{\cI}{\mathcal I}

\newcommand{\cK}{\mathcal K}

\newcommand{\cM}{\mathcal M}

\newcommand{\cO}{\mathcal O}

\renewcommand{\cR}{\mathcal R}
\newcommand{\cS}{\mathcal S}
\newcommand{\cT}{\mathcal T}
\newcommand{\cU}{\mathcal U}
\newcommand{\cV}{\mathcal V}

\newcommand{\bH}{\mathbf{H}}

\newcommand{\bx}{\mathbf{x}}
\newcommand{\by}{\mathbf{y}}
\newcommand{\bz}{\mathbf{z}}
\newcommand{\bOh}{\mathbf{0}}

\newcommand{\fF}{\mathfrak{F}}

\DeclareMathOperator{\tr}{tr}

\DeclareMathOperator{\Span}{span}

\DeclareMathOperator{\supp}{supp}

\newcommand{\bangle}[1]{\left\langle #1 \right\rangle}

\DeclareMathOperator{\Ric}{Ric}

\newcommand{\eps}{\varepsilon}

\DeclareMathOperator{\sing}{sing}
\DeclareMathOperator{\reg}{reg}
\DeclareMathOperator{\genus}{genus}

\DeclareMathOperator{\Graph}{graph}

\DeclareMathOperator{\Div}{div}

\title{Mean curvature flow with generic initial data}

\author{Otis Chodosh} 
\address{OC: Department of Mathematics, Bldg.\ 380, Stanford University, Stanford, CA 94305, USA}
\email{ochodosh@stanford.edu}

\author{Kyeongsu Choi}
\address{KC: School of Mathematics, Korea Institute for Advanced Study, 85 Hoegiro, Dongdaemun-gu, Seoul 02455, Republic of Korea}
\email{choiks@kias.re.kr}
\author{Christos Mantoulidis} 
\address{CM: Department of Mathematics, Rice University, Houston, TX 77005, USA}
\email{christos.mantoulidis@rice.edu}
\author{Felix Schulze}
\address{FS: Department of Mathematics, Zeeman Building, University of Warwick, Gibbet Hill Road, Coventry CV7 4AL,
UK}
\email{felix.schulze@warwick.ac.uk} 

\begin{document}

\begin{abstract}
We show that the mean curvature flow of generic closed surfaces in $\mathbb{R}^{3}$ avoids asymptotically conical and non-spherical compact singularities. We also show that the mean curvature flow of generic closed low-entropy hypersurfaces in $\RR^{4}$ is smooth until it disappears in a round point. The main technical ingredient is a long-time existence and uniqueness result for ancient mean curvature flows that lie on one side of asymptotically conical or compact shrinking solitons. 
\end{abstract}

\maketitle

\tableofcontents

\newpage

\section{Introduction}

\subsection{Overview of results} 

Mean curvature flow is the analog of the heat equation in extrinsic differential geometry. A family of surfaces $M(t) \subset \RR^{3}$ flows by mean curvature flow if
\begin{equation} \label{eq:mcf}
\left(\tfrac{\partial}{\partial t} \bx \right)^{\perp} = \bH_{M(t)}(\bx),
\end{equation}
where $\bH_{M(t)}(\bx)$ denotes the mean curvature vector of the surface $M(t)$ at $\bx$. Unlike the traditional heat equation, mean curvature flow is nonlinear. As a result, the mean curvature flow starting at a closed  surface $M\subset \RR^{3}$ is guaranteed to become singular in finite time. There are numerous possible singularities and, in general, they can lead to a breakdown of (partial) regularity and of well-posedness. A fundamental problem, then, is to understand singularities as they arise. 

A common theme in PDEs arising in geometry and physics is that a \emph{generic} solution exhibits better regularity or well-posedness behavior than the worst-case scenario. This aspect of the theory of mean curvature flow has been guided by the following well-known conjecture of Huisken \cite[\#8]{Ilmanen:Trieste}:
\begin{quote}
\emph{A generic mean curvature flow has only spherical and cylindrical singularities.}	
\end{quote}
The implications of this conjecture on the partial regularity and well-posedness of mean curvature flow is an important field of research in itself. See Section \ref{sec:intro.shrinkers} for the state of the art on the precise understanding of the effects of spherical and cylindrical singularities on the partial regularity and well-posedness of mean curvature flow.

The most decisive step toward Huisken's  conjecture was taken in the trailblazing work of Colding--Minicozzi \cite{ColdingMinicozzi:generic}, who proved that spheres and cylinders are the only linearly \emph{stable singularity models} for mean curvature flow. In particular, all remaining singularity models are linearly unstable and ought to occur only non-generically. See Section \ref{sec:intro.entropy.stable} for more discussion.
	
In this paper we introduce a new idea and take a second step toward the genericity conjecture and confirm that a large class of unstable singularity models are, in fact, avoidable by a slight perturbation of the initial data. Roughly stated, we prove:
\begin{quote}
\emph{The mean curvature flow of a generic closed embedded surface in $\RR^3$ encounters only spherical and cylindrical singularities until the first time it encounters a singularity (a) with multiplicity $\geq 2$, or (b) that has a cylindrical end but which is not globally a cylinder.}
\end{quote}
Cases (a) and (b) are conjectured to not occur (see the nonsqueezing conjecture and the no cylinder conjecture in \cite{Ilmanen:problems}). This would yield Huisken's conjecture in full.

Using a similar method, we also prove a related statement for hypersurfaces in $\RR^{4}$: 
\begin{quote}
\emph{The mean curvature flow starting from a generic hypersurface $M \subset \RR^4$ with low entropy remains smooth until it dissapears in a round point.}
\end{quote}
In particular, this gives a direct proof of the low-entropy Schoenflies conjecture (recently announced by Bernstein--Wang).

Our genericity results rely on keeping simultaneous track of flows coming out of a family of auxiliary initial surfaces on either side of $M$. The key ingredient is the following new classification result of ancient solutions to mean curvature flow that lie on one side of an asymptotically conical or compact singularity model:
\begin{quote}
\emph{For any smooth asymptotically conical or compact self-shrinker $\Sigma$, there is a unique ancient mean curvature flow lying on one side of $\sqrt{-t}\Sigma$ for all $t<0$. The flow exhibits only multiplicity-one spherical or cylindrical singularities.}
\end{quote}
See Section \ref{sec:intro.generic.statements} for more detailed statements of our results, and Section \ref{sec:intro.generic.method} for a discussion of the method and the technical ingredient.

\subsection{Singularities in mean curvature flow} \label{sec:intro.shrinkers}

Thanks to Huisken's monotonicity formula, if $X$ is a space-time singular point of a mean curvature flow $\cM$, it is possible to perform a parabolic rescaling around $X$ and take a subsequential (weak) limit to find a \emph{tangent flow} $\cM'$  \cite{Huisken:sing,Ilmanen:singularities}. A tangent flow is always self-similar in the sense that it only flows by homotheties. If the $t=-1$ slice of the flow is a smooth hypersurface $\Sigma$, then $\Sigma$ satisfies
\[
\bH + \tfrac{1}{2} \bx^{\perp} = 0,
\]
where $\bH$ is the mean curvature vector of $\Sigma$ and $\bx^{\perp}$ is the normal component of $\bx$. In this case, we call $\Sigma$ a \emph{self-shrinker}. The tangent-flow $\cM'$ at a time $t<0$ is then $\sqrt{-t}\,\Sigma$, though possibly with multiplicity. 

The simplest shrinkers are the generalized cylinders: $\RR^{n-k}\times \SS^{k}(\sqrt{2k})$, $k = 0, \ldots, n$. However, there are known to be many more examples: \cite{Angenent:torus,Nguyen:AC,Ketover:self-shrinkers,KKM:AC, BuzanoNguyenSchulz}. See also the earlier numerical work  \cite{AngenentIlmanenChopp,Chopp,Ilmanen:Trieste}. 

In general, non-cylindrical singularities (in the sense of generalized cylinders) can cause a breakdown in partial regularity or well-posedness of the flow (cf.\ \cite{AngenentIlmanenChopp,Ilmanen:Trieste,White:ICM}). It has thus been desirable to find situations where only cylindrical singularities arise and to use this information to analyze the partial regularity and well-posedness of the flow. To that end, Huisken  classified generalized cylinders as the only self-shrinkers with positive mean curvature \cite{Huisken:sing,Huisken:local-global} (and bounded curvature, cf.\ \cite{White:nature, ColdingMinicozzi:sing-generic}). This has led to a strong understanding of mean curvature flow in the mean convex case thanks to  Huisken--Sinestrari \cite{HuiskenSinnestrari:MCF-mean-convex,HuiskenSinestrari:convexity,HuiskenSinestrari:surgery}, White \cite{White:size,White:nature,White:subsequent}, Brendle and Brendle--Huisken \cite{Brendle:inscribed-sharp,BrendleHuisken:R3}, Haslhofer--Kleiner \cite{HaslhoferKleiner:estimates,HaslhoferKleiner:surgery}, Angenent--Daskalopoulos--\v{S}e\v{s}um \cite{ADS,ADS2}, and Brendle--Choi \cite{BrendleChoi:3d,BrendleChoi:nD}. 

The next level of difficulty is to understand flows of surfaces in $\RR^3$ that needn't be globally mean convex, but which happen to only experience multiplicity-one cylindrical singularities. There have been major recent advances on this topic. Colding--Minicozzi \cite{ColdingMinicozzi:sing-generic} proved (using their earlier work \cite{ColdingMinicozzi:uniqueness-tangent-flow}, cf.\ \cite{ColdingIlmanenMinicozzi}) that mean curvature flows in $\RR^{3}$ having only multiplicity-one cylindrical tangent flows are completely smooth at almost every time and any connected component of the singular set is contained in a time-slice. More recently, Choi--Haslhofer--Hershkovits showed \cite{ChoiHaslhoferHershkovits} (see also \cite{ChoiHaslhoferHershkovitsWhite}) that there is a (space-time) mean-convex neighborhood of any cylindrical singularity. In particular, combined with \cite{HershkovtisWhite}, this settles the well-posedness of a mean curvature flow in $\RR^{3}$ with only multiplicity-one cylindrical tangent flows. 

For flows of general surfaces in $\RR^{3}$, which may run into arbitrary singularities, our understanding of mean curvature flow near a singular point is quite limited at present. The most fundamental issue is the potential for higher multiplicity to arise when taking rescaled limits around a singular point. Nonetheless, some important information is available about the tangent flows at the first singular time due to important results of Brendle \cite{Brendle:genus0} classifying genus zero shrinkers in $\RR^{3}$ and of Wang \cite{Wang:ends-conical} showing that a smooth finite genus shrinker in $\RR^{3}$ has ends that are smoothly asymptotically conical or cylindrical. Besides the issue of multiplicity, another problem is the huge number of potential shrinkers that could occur as tangent flows, greatly complicating the analysis of the flow near such a singular point. (This issue presumably gets considerably worse for hypersurfaces in $\RR^{n+1}$.)

\subsection{Entropy and stability of shrinkers}  \label{sec:intro.entropy.stable} 
Huisken has conjectured \cite[\#8]{Ilmanen:Trieste} that cylinders and spheres are the only shrinkers that arise in a generic (embedded) mean curvature flow. This conjecture provides a promising way of avoiding the latter problem mentioned above.

Huisken's conjecture was reinforced by the numerical observation that non-cylindrical self-shrinkers are highly unstable. This instability was rigorously formulated and proven in the foundational work of Colding--Minicozzi \cite{ColdingMinicozzi:generic}. They defined the entropy 
\[
\lambda(M) := \sup_{\substack{\bx_0\in\RR^{3}\\t_0>0}}\int_{M} (4\pi t_0)^{-\frac n 2} e^{-\frac{1}{4t_0} |\bx - \bx_0|^{2}}  
\]
and observed that $t\mapsto\lambda(M_{t})$ is non-increasing along any mean curvature flow, by virtue of Huisken's monotonicity formula. Moreover, they proved that any smooth self-shrinker  with polynomial area growth, other than generalized cylinders (i.e., $\RR^{n-k} \times \SS^k(\sqrt{2k})$ with $k = 0, \ldots, n$), can be smoothly perturbed to have strictly smaller entropy. This result has been used fundamentally in \cite{ColdingMinicozziIlmanenWhite,BernsteinWang:1} (cf.  \cite{HershkovitsWhite:sharp-entropy}), though we will not need to make explicit use of it in this paper. 

There have been many important applications of Colding--Minicozzi's classification of entropy-stable shrinkers. First, they showed their result can be used to define a piecewise mean curvature flow that avoids non-spherical compact self-shrinkers. This idea has been used to classify low-entropy shrinkers, beginning with the work of Colding--Ilmanen--Minicozzi--White \cite{ColdingMinicozziIlmanenWhite} who showed that the round sphere $\SS^{n}\subset \RR^{n+1}$ has the least entropy among all non-planar self-shrinkers. Subsequently, Bernstein--Wang extended this to show that the round sphere has least entropy among all closed hypersurfaces \cite{BernsteinWang:1} (see also \cite{Zhu:entropy}) and that the cylinder $\RR\times \SS^{1}$ has second least entropy among non-planar self-shrinkers in $\RR^{3}$ \cite{BernsteinWang:TopologicalProperty}. Bernstein--Wang have recently used these classification results, along with a surgery procedure, to show that if $M^{3}\subset \RR^{4}$ has $\lambda(M) \leq \lambda(\SS^{2}\times \RR)$, then $M$ is diffeomorphic to $\SS^{3}$ \cite{BernsteinWang:topology-small-ent} (see also \cite{BernsteinWang:hausdorff-stability}). 

\subsection{Our perturbative statements} \label{sec:intro.generic.statements}
Let us describe our main perturbative results. First, we have a low-entropy result in $\RR^4$:
	
\begin{theorem}\label{theo:low-ent-generic-flow-4D}
Let $M^3\subset \RR^{4}$ be any closed connected hypersurface with $\lambda(M) \leq \lambda(\SS^{2}\times \RR)$. There exist arbitrarily small $C^{\infty}$ graphs $M'$ over $M$ so that the mean curvature flow starting from $M'$ is smooth until it disappears in a round point. 
\end{theorem}

We state and prove this ahead of our result for $\RR^3$ because its statement and proof are simpler. The low-entropy assumption allows us to perturb away \emph{all} unstable singularities (in the sense of Colding--Minicozzi) and thus obtain a fully regular nearby flow. In fact, Theorem \ref{theo:low-ent-generic-flow-4D} is a special case of Theorem \ref{theo:low-ent-generic-flow}, which applies in \emph{all} dimensions under suitable conditions. See also Theorem \ref{theo:low-ent-generic-flow-v2} and Corollary \ref{cor:low-ent-generic-flow-v3} for results showing that the above behavior is \emph{generic} in a precise sense. 

Theorem \ref{theo:low-ent-generic-flow-4D} immediately implies the following low-entropy Schoenflies theorem, recently announced by Bernstein--Wang (cf.\ \cite[p.\ 4]{bernsteinWang:top-uniqueness-expanders}).\footnote{We emphasize that our proof of Theorem \ref{theo:low-ent-generic-flow-4D} relies heavily on several of Bernstein--Wang's earlier works \cite{BernsteinWang:1,BernsteinWang:TopologicalProperty,BernsteinWang:topology-small-ent} and as such our proof here of Corollary \ref{cor:low-ent-schoenflies-R4} has several features in common with their announced strategy. The key point here, however, is that our study of generic flows in Theorem \ref{theo:low-ent-generic-flow-4D} allows us to completely avoid the need for any surgery procedure or the refined understanding of expanders obtained in \cite{BernsteinWang:SpaceOfExpanders,BernsteinWang:expander-compactness,BernsteinWang:degree-expander,bernsteinWang:top-uniqueness-expanders,BernsteinWang:relative-entropy}. }
\begin{corollary}[Bernstein--Wang {\cite{BernsteinWang:schoenflies}}]\label{cor:low-ent-schoenflies-R4}
If $M^3\subset \RR^4$ is a closed connected hypersurface with $\lambda(M) \leq \lambda(\SS^2\times \RR)$, then $M$ bounds a smoothly standard $4$-ball and is smoothly isotopic to a round $\SS^3$.\footnote{The isotopy from $M$ to the round $\SS^3$ follows from Theorem \ref{theo:low-ent-generic-flow-4D}, and the fact that $M$ bounds a smooth $4$-ball is then a consequence of the Isotopy Extension Theorem (cf.\ \cite[\S 8, Theorem 1.3]{Hirsch:diff-top}). }
\end{corollary}

For generic mean curvature flow of embedded surfaces in $\RR^{3}$, we show more:

\begin{theorem}\label{theo:mcf-R3}
	Let $M^2\subset \RR^{3}$ be a closed embedded surface. There exist arbitrarily small $C^{\infty}$ graphs $M'$ over $M$ so that:
	\begin{enumerate}
		\item the (weak) mean curvature flow of $M'$ has only multiplicity-one spherical and cylindrical tangent flows until it goes extinct, or 
		\item there is some $T>0$ so that the previous statement holds for times $t<T$ and at time $T$ there is a tangent flow of $M'$ that either 
		\begin{enumerate}
			\item has multiplicity $\geq 2$, or
			\item has a cylindrical end, but is not a cylinder.
		\end{enumerate}
	\end{enumerate}
\end{theorem}
Note two things:
\begin{itemize}
	\item In the $\RR^3$ theorem, unlike in the low-entropy higher dimensional theorems, we need to make use of a weak notion of mean curvature flow because we are placing no entropy assumptions and are thus interested in flowing \emph{through} spherical and cylindrical singularities. See Theorem \ref{theo:generic-R3} for the precise statement, which includes the notion of weak mean curvature flow that we make use of.
	\item \emph{Both} of the potential tangent flows in case (2) are conjectured to not exist (see the nonsqueezing conjecture and the no cylinder conjecture in \cite{Ilmanen:problems}). 
\end{itemize}

There are two features of our work that distinguish it from previous related work:
\begin{itemize}
	\item We only need to perturb the initial condition. See  \cite{ColdingMinicozzi:generic} for a piecewise flow construction that perturbs away compact singularity models (see also \cite{Sun:mult-gen}).
	\item We are able to perturb away (certain) non-compact singularity models.
\end{itemize}

\subsection{Our perturbative method: ancient one-sided flows} \label{sec:intro.generic.method} 

For a fixed hypersurface $M_{0}\subset \RR^{n+1}$, one has a weak mean curvature flow $t\mapsto M_{0}(t)$ starting at $M_{0}$. Suppose that $X = (\bx,T)$ is a singular point for $t\mapsto M_{0}(t)$. The usual method for analyzing the singularity structure at $X$ is to study the tangent flows of $t \mapsto M_0(t)$ at $X$, i.e., the (subsequential) limit of the flows
\[
t\mapsto \lambda (M_{0}(T + \lambda^{2}t)- \bx) =: M_{0}^{\lambda}(t)
\]
as $\lambda\to\infty$. As discussed above, by Huisken's monotonicity formula, for $t<0$, this will weakly (subsequentially) converge to a shrinking flow $t\mapsto M'(t)$ associated to a (weak) self-shrinker. 

Our new approach to generic mean curvature flow is to embed the flow $t\mapsto M_{0}(t)$ in a family of flows 
by first considering a local foliation $\{M_{s}\}_{s\in(-1,1)}$ and flowing the entire foliation, simultaneously, by mean curvature flow $t\mapsto M_{s}(t)$. The avoidance principle for mean curvature flow implies that $M_{s}(t)\cap M_{s'}(t) = \emptyset$ for $s\not=s'$. The entire foliation can be passed to the limit simultaneously, i.e., we can consider the flows
\[
t\mapsto \lambda (M_{s}(T + \lambda^{2}t)- \bx) : = M_{s}^{\lambda}(t)
\]
and send $\lambda\to\infty$. 

If we choose $s\searrow 0$ diligently as $\lambda\to\infty$, then after passing to a subsequence, $t\mapsto M_{s}^{\lambda}(t)$ will converge to a non-empty flow $t\mapsto \bar M(t)$ that stays on \emph{one side} of the original tangent flow $t\mapsto M'(t)$ and which is \emph{ancient}, i.e., it exists for all sufficiently negative $t$. If we can prove that the one-sided ancient flow $t\mapsto \bar M(t)$ has certain nice properties (i.e., only cylindrical singularities), then we can exploit this to find a choice of $s$ small so that $t\mapsto M_{s}(t)$ is well behaved. 

We proceed to give more details as to how we exploit this ancient one-sided flow, $t \mapsto \bar M(t)$. Assume that the tangent flow to $M_{0}(t)$ at $X$ is smooth and has multiplicity one, so $M'(t) = \sqrt{-t}\, \Sigma$ for $t<0$. Then, considering the rescaled flow $\tau \mapsto e^{\frac\tau 2} \bar M(-e^{\tau})$, we note that $e^{\frac\tau 2} \bar M(-e^{\tau})$ lies strictly on one side of $\Sigma$ and 
\begin{equation}\label{eq:intro-lim-rescaled-flow-mult1}
\Sigma = \lim_{\tau\to-\infty}e^{\frac\tau 2} \bar M(-e^{\tau}) 
\end{equation}
(a priori, this could occur with multiplicity, but in practice one can rule this out by upper semi-continuity of density). In the current work, we will deal with all $\Sigma$ that are: (i) compact but not spheres, or (ii) non-compact with asymptotically (smoothly) conical structure. These tangent flows encompass all the necessary ones for our aforementioned theorem statements, by virtue of L. Wang's \cite{Wang:ends-conical} characterization of the asymptotic structure of non-compact singularity models.

Our definitive rigidity theorem of ancient one-sided flows is:

\begin{theorem}\label{theo:exist-unique-intro}
Let $\Sigma^{n}\subset \RR^{n+1}$ be a smooth self-shrinker that is either compact or asymptotically (smoothly) conical. Up to parabolic dilation around $(\bOh,0)\in\RR^{n+1}\times \RR$, there exists a unique\footnote{For technical reasons, the long-time aspect of the existence statement currently requires $2 \leq n \leq 6$. If one only cares about sufficiently negative times, existence and uniqueness hold true for all dimensions.} ancient solution to mean curvature flow $t\mapsto \bar M(t)$ so that $\bar M(t)$ is disjoint from $\sqrt{-t}\Sigma$ and has entropy $ < 2 F(\Sigma)$. 
\end{theorem}

\begin{remark}
There has  recently been an outburst of activity regarding the rigidity of ancient solutions to geometric flows. We mention here \cite{BrendleHuiskenSinestrari,XJWang:ancient,DHS:RF-surf,DaskalopolousHamiltonSesum:CSF,HuiskenSinestrari:ancient,HaslhoferHerskovits:ancient,DdPS:type2Yamabe,BrendleKapouleas,BrendleChoi:3d,BrendleChoi:nD,Brendle:RF-ancient,ABDS,BDS}.  In the setting at hand, Theorem \ref{theo:exist-unique-intro} was motivated from the recent work in \cite{ChoiMantoulidis} on the classification of compact ancient solutions of gradient flows of elliptic functionals in Riemannian manifolds. However, this is the first time that the \emph{one-sidedness} condition has been exploited so crucially, and geometrically, in the setting of ancient geometric flows. In the elliptic setting, there have been interesting exploitations of one-sided foliations by minimal surfaces; see, e.g., Hardt--Simon  \cite{HardtSimon:foliation}, Ilmanen--White \cite{IlmanenWhite:sharp.entropy}, and Smale \cite{Smale}. Our current parabolic setting, however, presents a number of complications that come from the fact that the shrinkers $\Sigma$ we are interested in are primarily noncompact, and thus the flows cannot be written as global perturbations of the self-similarly shrinking solution.
\end{remark}

\begin{remark}
	Neither of the hypothesis in Theorem \ref{theo:exist-unique-intro} can be removed. There can be many ancient flows that intersect $\sqrt{-t}\, \Sigma$ and converge to $\Sigma$ as $t\to-\infty$ after rescaling; see Theorem \ref{theo:contraction}. Also, for $a\geq0$, the grim reaper in the slab $\RR\times (a,a+\pi)$ is a nontrivial example of an ancient flow that is disjoint from its tangent flow at $-\infty$, $2[\RR\times \{0\}]$. 
\end{remark}

Next, we show that $\bar M(t)$ encounters only generic singularities for as long as it exists. We establish many properties of $\bar M(t)$ in  Theorem \ref{theo:one.sided.construction}, and some  the important ones are summarized here. 

\begin{theorem}\label{theo:prop-ancient-intro}
	Let $t \mapsto \bar M(t)$, $\Sigma^n \subset \RR^{n+1}$ be as in Theorem \ref{theo:exist-unique-intro} and $2 \leq n \leq 6$. Then:
	\begin{itemize}
		\item The flow $t \mapsto \bar M(t)$ only has multiplicity-one, generalized cylindrical singularities: $\RR^{n-k} \times \SS^k(\sqrt{2k})$, $k = 1, \ldots, n$.
		\item At $t=0$, $\bar M(0)$ is smooth and star-shaped.
		\item If $\Sigma$ is noncompact, then $t\mapsto \bar M(t)$ exists for all $t \in \RR$ and
		\[ \lim_{t \to \infty} \tfrac{1}{\sqrt{t}} \bar M(t) \]
		is an outermost expander associated to the asymptotic cone of $\Sigma$. 
	\end{itemize} 
\end{theorem}

To prove Theorem \ref{theo:prop-ancient-intro}, we show that the one-sided ancient flow $t \mapsto \bar M(t)$ must be \emph{shrinker mean convex}; geometrically, this means that the rescaled flow moves in one direction. This is where the one-sided property is crucially used. Recalling the spectral instability of shrinkers discovered in \cite{ColdingMinicozzi:generic}, and that only the \emph{first} eigenfunction of the linearization of Gaussian area along $\Sigma$ has a sign, we show that the evolution of a one-sided flow is dominated by the first eigenfunction, which in turn yields shrinker mean convexity. Shrinker mean convexity is preserved under the flow and can be used analogously to mean convexity to establish regularity of the flow (cf.\ \cite{Smoczyk,White:size,White:nature,lin:star-shaped,HershkovitsWhite:sharp-entropy}). We emphasize that our analysis of the flow $\bar M(t)$ in Theorem \ref{theo:prop-ancient-intro} is influenced by the work of Bernstein--Wang \cite{BernsteinWang:TopologicalProperty} where they studied a (nearly ancient) flow on one side of a asymptotically conical shrinker of low-entropy. Because we do not assume that the flow has low-entropy (besides assuming the limit at $-\infty$ has multiplicity one), we must allow for singularities (while in \cite{BernsteinWang:TopologicalProperty}, the flow is \emph{a posteriori} smooth). In particular, this complicates the analysis of the flow near $t=0$ significantly. 

Finally, we explain how Theorems \ref{theo:exist-unique-intro} and \ref{theo:prop-ancient-intro} can be used to prove the main results of the paper, Theorems  \ref{theo:low-ent-generic-flow-4D} and \ref{theo:mcf-R3}. We begin by considering the setting of  Theorem \ref{theo:low-ent-generic-flow-4D}, namely $M^3\subset \RR^4$ with $\lambda(M) \leq \lambda(\SS^2\times \RR)$. Up to performing an initial perturbation using \cite{ColdingMinicozzi:generic}, we can assume this inequality is strict $\lambda(M) < \lambda(\SS^2\times \RR)$. As described above, we embed $M$ in a local foliation $\{M_s\}_{s\in (-1,1)}$ in space with $\sup_s \lambda(M_s) < \lambda(\SS^2\times \RR)$. We now flow the entire foliation simultaneously, obtaining flows $\{ M_s(t) \}$.  Suppose that $M_0(t)$ encounters a singularity at $(\bx,T)$. Our low-entropy assumption and work of Bernstein--Wang \cite{BernsteinWang:topology-small-ent} implies that any tangent flow to $M_0(t)$ at $(\bx, T)$ is associated to some compact or asymptotically (smoothly) conical self-shrinker $\Sigma$. If $\Sigma$ is a round sphere, we are done. Otherwise, we may combine $\lambda(M_s) < \lambda(\SS^2\times \RR)$ with Theorems \ref{theo:exist-unique-intro} and \ref{theo:prop-ancient-intro} to find that the $M_s(t)$ are free of singularities at points that are captured by the ancient flow on one side of $\Sigma$. 

The major issue is that points $\by \in M_s(t)$ close to $(\bx,T)$ but with $t-T \gg |\by-\bx|^2$ will \emph{not} be captured by a one-sided flow. In this case, if we rescale $M_s(t)$ around $(\bx,T)$ so that $(\by,t)$ is moved to a point of unit distance, and pass to limits, we instead obtain a flow that agrees with the shrinking $\Sigma$ for $t<0$ and is some unknown flow flowing out of the cone at infinity of $\Sigma$ for $t>0$. The insight is that, by parabolic cone-splitting, these flows will have strictly lower Gaussian density than $\Sigma$. As such, the flow $M_s(t)$ improves as compared to $M_0(t)$ in that it has a lower-density maximal density singular point. We iterate this finitely many times to prove Theorem \ref{theo:low-ent-generic-flow-4D}. 

We now describe the necessary modifications to prove Theorem \ref{theo:mcf-R3}. Consider $M^2\subset \RR^3$ (without any entropy bounds). Arguing as above, we can flow a foliation $\{M_s\}_{s\in (-1,1)}$ and use Theorems \ref{theo:exist-unique-intro} and \ref{theo:prop-ancient-intro} to show that $M_s(t)$ is well-approximated by the ancient-one sided flow in some neighborhood of a compact or asymptotically (smoothly) conical self-shrinking singularity of $M_0(t)$. At this point we do not use the density drop argument described above, but instead must rely on a genus monotonicity argument. To do this, we note that such a self-shrinking singularity must have genus $>0$ by a result of Brendle \cite{Brendle:genus0}. On the other hand, the ancient one-sided flow is star-shaped at time $0$ by Theorem \ref{theo:prop-ancient-intro}. Using this, we find that the one-sided flow strictly loses genus when bypassing the singularity. Thus, after finitely many perturbations there cannot be any singularities that are not round spheres or cylinders. This proves Theorem \ref{theo:mcf-R3}.

\begin{remark}
There have been several significant results related to this paper that appeared between the time the paper first appeared and this version. On one hand, the results of this paper were extended to shrinkers with asymptotically cylindrical ends in \cite{CCS2}. On the other hand, the density drop argument was generalized into a standalone tool to prove generic regularity results for low-entropy flows in \cite{CCMS:low1,CMS:low2}. This density drop argument was also used to generalize the Hardt--Simon generic regularity result \cite{HardtSimon:foliation} for area-minimizing hypersurfaces from $8$ to $9$ and $10$ ambient dimensions \cite{CMS:HS}. In terms of its use in the current paper, the first- and fourth-named authors, along with Daniels-Holgate recently proved \cite{CDHS} that the outermost level set flows are completely smooth for a short time after the occurrence of a singularity modeled on an asymptotically conical self-shrinker. In particular, this result would allow us to avoid the density drop argument used here altogether. Finally, we mention the recent major breakthrough by Bamler--Kleiner who proved \cite{BK:mult} the multiplicity-one conjecture in $\RR^3$.
\end{remark}

\subsection{Other results} We list several other new results we've obtained in this work that might be of independent interest:
\begin{itemize}
\item For any smooth compact or asymptotically conical shrinker $\Sigma$, we construct an $I$ parameter family of smooth ancient mean curvature flows (where $I$ is the index of $\Sigma$ as a critical point of Gaussian area, as defined in \eqref{eq:linearized.equation.eigenvalues})  that---after rescaling---limit to $\Sigma$ as $t\to-\infty$; see Theorem \ref{theo:contraction}.
\item We show that the outermost flows of the level set flow of a regular cone are smooth self-similarly expanding solutions. We also construct associated expander mean convex flows that converge to the given expander after rescaling; see Theorem \ref{theo:outermost-flows}.
\item We include a proof of a localized version of the avoidance principle for weak set flows due to Ilmanen; see Theorem \ref{theo:ilmanen-avoidance}. This implies a strong version of the Frankel property for shrinkers; see Corollary \ref{coro:Frankel}.
\item We improve known results concerning the connectivity of the regular set of a unit-regular Brakke flow with sufficiently small singular set. See Corollary \ref{cor:connected-reg-part}.
\item We localize the topological monotonicity of White \cite{White:topology-weak}. In particular, our results should be relevant in the context of the strict genus reduction conjecture of Ilmanen \cite[\#13]{Ilmanen:problems}. See Appendix \ref{app:loc-top-monotonicity} and the proof of Proposition \ref{prop:key-perturb-result-3d}. 
\end{itemize}

\subsection{Organization of the paper}  In Section \ref{sec:prelim} we recall some conventions and definitions used in the paper. 

The main technical work of the paper is contained in Sections \ref{sec:linearized.equation}--\ref{sec:exist.unique.ancient.Brakke}, which establish the existence and uniqueness, together with regularity of ancient one-sided flows. The geometric applications of this existence and uniqueness result are then given in Sections \ref{sec:generic-schoenflies} and \ref{sec:generic.R3}. As such, the reader less interested in the technicalities in proving existence/uniqueness of the one-sided ancient flow may want to jump straight to Section  \ref{sec:generic-schoenflies}.

More precisely, in Section \ref{sec:linearized.equation} we analyze the linearized graphical mean curvature flow equation over an asymptotically conical shrinker. We use this to study the nonlinear problem in Section \ref{sec:dynamics}. These results are applied in Section \ref{sec:one.sided.flows} to prove our main analytic input, Corollary \ref{coro:one.sided.decay.uniqueness}, the uniqueness of ancient one-sided graphical flows. 

Section \ref{sec:family.flows} contains a construction of the full $I$-parameter family of ancient flows. This is not used elsewhere, since we construct the one-sided flows by GMT methods allowing us to flow through singularities. We begin this GMT construction in Section \ref{sec:gmt.existence} where we construct an ancient one-sided Brakke/weak-set flow pair. In Section \ref{sec:long.time.reg} we establish optimal regularity of the ancient one-sided flow. We put everything together in Section \ref{sec:exist.unique.ancient.Brakke} and give the full existence and uniqueness statement for the ancient one-sided flows. 

We apply this construction to the study of the mean curvature flow of generic low entropy hypersurfaces in Section \ref{sec:generic-schoenflies} and to the study of the first non-generic time of the mean curvature flow of a generic surface in $\RR^3$ in Section \ref{sec:generic.R3}. 

In Appendix \ref{sec:shrinker.geometry} we improve some decay estimates for asymptotically conical ends of shrinkers. In Appendix \ref{sec:schauder} we recall Knerr's non-standard parabolic Schauder estimates. In Appendix \ref{app:uniqueness-BF} we prove that mean curvature flows with bounded curvature and controlled area ratios are unique in the class of Brakke flows. We prove Ilmanen's localized avoidance principle in Appendix \ref{sec:Ilmanen.avoidance}. Appendix \ref{sec:Ecker-Huisken maximum principle} recalls the non-compact Ecker-Huisken maximum principle. In Appendix \ref{app:weak-set-flows} we study weak set flows coming out of cones. We show that Brakke flows with sufficiently small singular set have connected regular part in Appendix \ref{app:reg.set.conn}. Finally, in Appendix \ref{app:loc-top-monotonicity} we localize certain topological monotonicity results.

\subsection{Acknowledgements} O.C.~was partially supported by a Terman Fellowship, a Sloan Fellowship, and NSF grants  DMS-1811059, DMS-2016403, and DMS-2304432. He would also like to acknowledge the MATRIX Institute for its hospitality during the time which some of this article was completed. K.C.~was supported by the KIAS Individual Grant MG078902, a POSCO Science Fellowship, an Asian Young Scientist Fellowship, and the National Research Foundation of Korea (NRF) grant funded by the Korea government (MSIT) (RS-2023-00219980). C.M.~was supported by the NSF grant DMS-1905165. F.S.~was supported by a Leverhulme Trust Research Project Grant RPG-2016-174.  We are grateful to Richard Bamler, Costante Bellettini, Robert Haslhofer, Or Hershkovits, Daren Cheng, Ciprian Manolescu, Leon Simon, and Brian White for useful conversations related to this paper as well as to the anonymous referees for their careful reading and many helpful suggestions.


\section{Preliminaries} \label{sec:prelim}

In this section we collect some useful definitions, conventions, and useful ways to recast mean curvature flow, which we will make use of in the sequel.

\subsection{Spacetime} We will often consider the spacetime of our mean curvature flows, $\RR^{n+1}\times \RR$, with its natural time-projection map $\mathfrak{t} : \RR^{n+1} \times \RR \to \RR$:
\[
\mathfrak{t}(\bx,t) := t.
\]
For any subset $E\subset \RR^{n+1}\times \RR$ we will denote
\[
E(t) := \{ \mathbf{x} \in \RR^{n+1} : (\mathbf{x}, t) \in E \}.
\]

\subsection{The spacetime track of a classical flow}

Let us fix a compact $n$-manifold $M$, possibly with boundary. Suppose that $f:M\times[a,b]\to\RR^{n+1}$ is a continuous map that is smooth on $M^{\circ} \times (a,b]$ (where $M^{\circ}=M\setminus \partial M$) and injective on each $M\times \{t\}$ for $t\in[a,b]$. Assume that $t\mapsto f(M^{\circ},t)$ is flowing by mean curvature flow. Then, we call 
\[
\cM : = \{ f(M,t) \times \{t\} : t\in[a,b]\} \subset \RR^{n+1}\times \RR
\]
a \emph{classical mean curvature flow} and define the \emph{heat boundary} of $\cM$ by
\[
\partial\cM : = f(M,a) \cup f(\partial M,[a,b]). 
\]
By the maximum principle, classical flows that intersect must intersect in a point that belongs to either one of their heat boundaries (cf.\ \cite[Lemma 3.1]{White:topology-weak}). 

\subsection{Weak set flows and level set flows} \label{sec:prelim.weak.flow.level.set.flow}

If $\Gamma \subset \RR^{n+1}\times \RR^{+}$ (where $\RR^{+}=[0,\infty)$ could be shifted as necessary) is a closed subset of spacetime, then $\cM \subset \RR^{n+1}\times \RR$ is a \emph{weak set flow} (generated by $\Gamma$) if:
\begin{enumerate}
\item $\cM$ and $\Gamma$ coincide at $t=0$ and 
\item if $\cM'$ is a classical flow with $\partial\cM'$ disjoint from $\cM$ and $\cM'$ disjoint from $\Gamma$, then $\cM'$ is disjoint from $\cM$. 
\end{enumerate} 
We will often consider the analogous definition with $\RR^{+}$ replaced by $\RR$ in which case one should omit requirement (1).

There may be more than one weak set flow generated by a given $\Gamma$. See \cite{White:topology-weak}. However, there is one weak set flow that contains all other weak set flows generated by $\Gamma$. It is called the \emph{level set flow} (or biggest flow). For $\Gamma \subset \RR^{n+1}\times \RR^{+}$ as above, we define it inductively as follows. Set
\[
W_{0} := \{ (\bx,0) : (\bx,0) \not \in \Gamma\}
\]
and then let $W_{k+1}$ be the union of all classical flows $\cM'$ with $\cM'$ disjoint from $\Gamma$ and $\partial \cM'\subset W_{k}$. We define the level set flow (or biggest flow) generated by $\Gamma$ as:
\[ \cM := (\RR^{n+1} \times \RR^+) \setminus (\cup_{k=0}^\infty W_k) \subset \RR^{n+1} \times \RR^+. \]
See \cite{EvansSpruck1,Ilmanen:elliptic,White:subsequent} for further references for weak set flows and level set flow.

We will sometimes engage in a slight abuse of notation, referring to a weak set flow (or a level set flow) generated by a closed subset $\Gamma_0 \subset \RR^{n+1}$, when we really mean that it is  generated by $\Gamma_0 \times \{0\}$ (or a suitable time-translate) in the sense defined above.

\subsection{Integral Brakke flows} \label{sec:prelim.brakke.flow}

Another important notion of weak mean curvature flow is a Brakke flow (cf.\ \cite{Brakke,Ilmanen:elliptic}). We follow here the conventions used in \cite{White:MCFboundary}. 

An ($n$-dimensional) \emph{integral Brakke flow} in $\RR^{n+1}$ is a $1$-parameter family of Radon measures $(\mu(t))_{t \in I}$ over an interval $I \subset \RR$ so that:
\begin{enumerate}
\item For almost every $t \in I$, there exists an integral $n$-dimensional varifold $V(t)$ with $\mu(t) = \mu_{V(t)}$ so that $V(t)$ has locally bounded first variation and has mean curvature $\bH$ orthogonal to $\textrm{Tan}(V(t),\cdot)$ almost everywhere.
\item For a bounded interval $[t_1,t_2] \subset I$ and any compact set $K\subset \RR^{n+1}$, 
\[
\int_{t_1}^{t_2}\int_K (1+|\bH|^2) d\mu(t) dt < \infty.
\]
\item If $[t_1,t_2] \subset I$ and $f \in C^{1}_{c}(\RR^{n+1}\times [t_1,t_2])$ has $f\geq 0$ then 
\[
\int f(\cdot,t_{2}) \, d\mu(t_{2}) - \int f(\cdot,t_{1}) \, d\mu(t_{1}) \leq \int_{t_{1}}^{t_{2}} \int\left( - |\bH|^{2} f + \bH \cdot \nabla f + \tfrac{\partial }{\partial t} f \right) \, d\mu(t) \, dt.
\]
\end{enumerate}
We will often write $\cM$ for a Brakke flow $(\mu(t))_{t \in I}$, with the understanding that we're referring to the family $I \ni t\mapsto \mu(t)$ of measures satisfying Brakke's inequality.

A key fact that relates Brakke flows to weak set flows, which we will use implicitly throughout the paper, is that the support of the spacetime track of a Brakke flow is a weak set flow \cite[10.5]{Ilmanen:elliptic}.\footnote{The definition of Brakke flow used in \cite{Ilmanen:elliptic} is slightly different than the one given here, but it is easy to see that the proof of \cite[10.5]{Ilmanen:elliptic} applies to our definition as well.}

\subsection{Density and Huisken's monotonicity} \label{sec:prelim.density.huisken.monotonicity}

For $X_{0} : = (\bx_{0},t_{0})\in \RR^{n+1}\times \RR$ consider the (backward) heat kernel based at $(\bx_{0},t_{0})$:
\begin{equation} \label{eq:backwards.heat.kernel}
\rho_{X_{0}}(\bx,t) := (4\pi(t_{0}-t))^{-\frac n2} \exp \left( -\frac{|\bx-\bx_{0}|^{2}}{4(t_{0}-t)} \right),
\end{equation}
for $\bx \in \RR^{n+1}$, $t< t_{0}$. For a Brakke flow $\cM$ and $r>0$ we set
\begin{equation} \label{eq:density.ratio.r}
\Theta_{\cM}(X_{0},r) : = \int_{\bx\in\RR^{n+1}} \rho_{X_{0}}(\bx,t_{0}-r^{2}) \, d\mu(t_{0}-r^{2}).
\end{equation}
This is the density ratio at $X_0$ at a fixed scale $r > 0$. Huisken's monotonicity formula \cite{Huisken:sing} (cf.\ \cite{Ilmanen:singularities}) implies that
\[
\tfrac{d}{dt} \int \rho_{X_{0}}(\bx,t) \, d\mu(t) \leq - \int \left| \bH - \frac{(\bx-\bx_{0})^{\perp}}{2(t-t_{0})} \right|^2 \rho_{X_{0}}(\bx,t) \, d\mu(t)
\]
so in particular, we can define the density of $\cM$ at $X_{0}$ by
\begin{equation} \label{eq:density}
\Theta_{\cM}(X_{0}) : = \lim_{r\searrow 0}\Theta_{\cM}(X_{0},r). 
\end{equation}

\subsection{Unit-regular and cyclic Brakke flows} \label{sec:prelim.unit.regular.cyclic}

An integral Brakke flow $\cM = (\mu(t))_{t \in I}$ is said to be 
\begin{itemize}
	\item \emph{unit-regular} if $\cM$ is smooth in some space-time neighborhood of any spacetime point $X$ for which $\Theta_{\cM}(X) = 1$;
	\item \emph{cyclic} if, for a.e.\ $t \in I$, $\mu(t) = \mu_{V(t)}$ for an integral varifold $V(t)$ whose unique associated rectifiable mod-2 flat chain $[V(t)]$ has $\partial[V(t)]=0$ (see \cite{White:cyclic}). 
\end{itemize} 
Integral Brakke flows constructed by Ilmanen's elliptic regularization approach \cite{Ilmanen:elliptic} (see also \cite[Theorem 22]{White:MCFboundary}) are \emph{unit-regular} and \emph{cyclic}. More generally, if $\cM_{i}$ are unit-regular (resp. cyclic) integral Brakke flows with $\cM_{i}\rightharpoonup \cM$, then $\cM$ is also unit-regular (resp. cyclic) by \cite{White:Brakke} (cf.\ \cite[Theorem 4.2]{SchulzeWhite}; resp. \cite[Theorem 4.2]{White:cyclic}). Recall that a sequence of integral Brakke flows $\cM_{i}$ \emph{converges} to an integral Brakke flow $\cM$, denoted $\cM_i \rightharpoonup \cM$, if
\begin{enumerate}
	\item $\mu_{i}(t) \rightharpoonup \mu(t)$ for each $t$, and
	\item for a.e.\ $t$, we can pass to a subsequence depending on $t$ so that $V_{i}(t)\rightharpoonup V(t)$ as varifolds. 
\end{enumerate}
The motivation for this definition of convergence is that these are the conditions that follow (after passing to a subsequence) if we have local mass bounds for $\cM_{i}$ and seek to prove a compactness theorem (cf.\ \cite[\S7]{Ilmanen:elliptic}).

\subsection{Shrinkers} \label{sec:prelim.shrinkers}

A smooth hypersurface $\Sigma\subset \RR^{n+1}$ is a \emph{self-shrinker} if 
\begin{equation}\label{eq:defn-shrinker}
\bH_{\Sigma} + \tfrac{1}{2} \bx^{\perp} = 0
\end{equation}
where $\bH_{\Sigma}$ is the mean curvature vector of $\Sigma$ and $\bx^{\perp}$ is the normal component of $\bx$. We will always assume that $\Sigma$ has empty boundary, unless specified otherwise. One can easily check that \eqref{eq:defn-shrinker} is equivalent to any of the following properties:
\begin{itemize}
\item $t\mapsto \sqrt{-t}\, \Sigma$ is a mean curvature flow for $t<0$,
\item $\Sigma$ is a minimal hypersurface for the metric $e^{-\frac{1}{2n} |\bx|^{2}}g_{\RR^{n+1}}$, or
\item $\Sigma$ is a critical point of the \emph{$F$-functional}
\[ 
F(\Sigma) := (4\pi)^{-\frac n2} \int_{\Sigma} e^{-\frac{1}{4} |\bx|^{2}}
\]
among compactly supported deformations, as well as translations and dilations.
\end{itemize}
See \cite[\S3]{ColdingMinicozzi:generic}.

We will say that $\Sigma$ is \emph{asymptotically conical} there is a regular cone $\cC$ (i.e., the cone over a smooth submanifold of $\SS^{n}$) so that $\lambda\Sigma \to \cC$ in $C^\infty_\textrm{loc}(\RR^{n+1}\setminus\{0\})$ as $\lambda\searrow 0$.

\begin{remark}
By considering the $t\nearrow 0$ limit (in the Brakke flow sense) of the flow $t\mapsto \sqrt{-t}\,\Sigma$, we see that $\lim_{\lambda \searrow 0} \lambda\Sigma$ is unique in the Hausdorff sense, so the asymptotic cone of $\Sigma$ must be unique. Moreover, because we have assumed that the convergence is in $C^\infty_\textrm{loc}$, there is no potential higher multiplicity in the limit (see, e.g., \cite[\S5]{Wang:ends-conical}).
\end{remark}

\subsection{Curvature conventions} \label{sec:prelim.curvature}

Consider $\Omega\subset \RR^{n+1}$ open with $\partial\Omega=\Sigma$ a self-shrinker. Write $\nu_\Sigma$ for the unit normal vector field to $\Sigma$ that points into $\Omega$. We define the \emph{second fundamental form} 2-tensor $A_\Sigma$ at each $p \in \Sigma$ to equal
\[
A_\Sigma: T_p \Sigma \times T_p \Sigma \to \RR,\qquad A_\Sigma(\xi,\zeta) = - D_\xi \nu_\Sigma \cdot \zeta.
\]
Recall that dual to $A_\Sigma$ is the \emph{shape operator} or \emph{Weingarten map}, defined at each $p \in \Sigma$ to be the tangent space endomorphism given by
\[
S_\Sigma : T_p \Sigma \to T_p \Sigma, \qquad S_\Sigma(\xi) = - D_\xi \nu_\Sigma.
\]
We fix the sign of the \emph{scalar mean curvature} $H_\Sigma$ as follows
\[ \bH_\Sigma = H_\Sigma \, \nu_\Sigma.\]
Thus, $H_\Sigma = \tr_\Sigma A_\Sigma$, with the principal curvatures of $\Sigma$ being the eigenvalues of $A_\Sigma$. With these conventions, the shrinker mean curvature from \eqref{eq:defn-shrinker} can be written as
\[
\bH_\Sigma + \tfrac12 \bx^\perp = \left(H_\Sigma + \tfrac 12  \bx \cdot \nu_\Sigma \right) \nu_\Sigma.
\]
For example, the sphere bounding a unit ball has normal vector pointing to the inside, positive mean curvature, and positive principal curvatures. Conversely, the same sphere bounding the complement of a closed unit ball has normal vector pointing to the outside, negative mean curvature, and negative principal curvatures.

\subsection{Entropy} Following \cite{ColdingMinicozzi:generic}, one uses the backward heat kernel $\rho_{(\mathbf{x}_0,t_0)}$ from \eqref{eq:backwards.heat.kernel} to define the \emph{entropy} of a Radon measure $\mu$ on $\RR^{n+1}$ by
\begin{equation} \label{eq:entropy.t}
\lambda(\mu) := \sup_{\substack{\bx_{0} \in \RR^{n+1}\\t_{0}>0}} \int \rho_{(\bx_{0},t_{0})}(\bx,0) \, d\mu.
\end{equation}
Then, one can define the entropy of an arbitrary Brakke flow $\cM = (\mu(t))_{t\in I}$ by:
\begin{equation} \label{eq:entropy}
\lambda(\cM) := \sup_{t\in I} \lambda(\mu(t)). 
\end{equation}
Huisken's monotonicity formula implies that $t\mapsto \lambda(\mu(t))$ is non-increasing.


\section{Linearized rescaled flow equation} \label{sec:linearized.equation}

Let $\Sigma^n \subset \RR^{n+1}$ be a smooth properly immersed asymptotically conical shrinker.\footnote{The analysis here also holds in the much simpler case of compact $\Sigma$.}

\subsection{Spectral theory in Gaussian $L^2$ space}  \label{sec:linearized.equation.L2.notation}

We consider the following operator on $\Sigma$:
\begin{equation} \label{eq:linearized.equation.linear.operator}
Lu := \Delta_\Sigma u - \tfrac12 \mathbf{x} \cdot \nabla_\Sigma u + \tfrac12 u + |A_\Sigma|^2 u.
\end{equation}
This is the ``stability'' operator for the $F$-functional in Section \ref{sec:prelim.shrinkers} in the sense that 
\[
\tfrac{d^{2}}{ds^{2}}\Big|_{s=0} F(\Graph_{\Sigma}(su)) = \int_{\Sigma} - u (L u) \, \rho,
\]
for any compactly supported function $u : \Sigma \to \RR$, where $\rho$ is the Gaussian weight
\begin{equation} \label{eq:linearized.equation.weighted.density}
\rho(\mathbf{x}) := (4\pi)^{-\frac{n}{2}} e^{-\frac{1}{4} |\mathbf{x}|^2},
\end{equation}
i.e., $\rho := \rho_{(\mathbf{0}, 0)}(\cdot, -1)$ in the notation of \eqref{eq:backwards.heat.kernel}. See \cite[Theorem 4.1]{ColdingMinicozzi:generic}. This stability operator, \eqref{eq:linearized.equation.linear.operator}, is only self-adjoint if we we work on  Sobolev spaces weighted by $\rho$. We thus define a weighted $L^2$ dot product for measurable functions $u$, $v : \Sigma \to \RR$:
\begin{equation} \label{eq:linearized.equation.weighted.L2.dot}
	\langle u, v \rangle_{W} := \int_\Sigma \langle u, v \rangle \rho \, d\cH^n.
\end{equation}
This induces a metric $\Vert \cdot \Vert_{W}$ and a Hilbert space
\begin{equation} \label{eq:linearized.equation.weighted.L2.space}
	L^2_W(\Sigma) := \{ u : \Sigma \to \RR : \Vert u \Vert_{W} < \infty \}.
\end{equation}
Likewise, we define the higher order weighted Sobolev spaces
\begin{equation} \label{eq:linearized.equation.weighted.sobolev.space}
	H^k_W(\Sigma) := \{ u : \Sigma \to \RR : \Vert u \Vert_{W} + \Vert \nabla_\Sigma u \Vert_{W} + \ldots + \Vert \nabla_\Sigma^k u \Vert_{W} < \infty \}.
\end{equation}
They are Hilbert spaces for the dot product
\begin{equation} \label{eq:linearized.equation.weighted.sobolev.dot}
	\langle u, v \rangle_{W,k} := \langle u, v \rangle_{W} + \langle \nabla_\Sigma u, \nabla_\Sigma v \rangle_{W} + \ldots + \langle \nabla_\Sigma^k u, \nabla_\Sigma^k v \rangle_{W},
\end{equation}
whose induced norm is denoted $\Vert \cdot \Vert_{W,k}$. It is with respect to these weighted measures spaces that $L$ is self-adjoint, i.e.,
\begin{equation} \label{eq:linearized.equation.weighted.L2.self.adjoint}
	\langle Lu, v \rangle_{W} = \langle u, Lv \rangle_{W}, \; \forall u, v \in H^2_W(\Sigma).
\end{equation}
We have:

\begin{lemma} \label{lemma:linearized.equation.spectrum}
	There exist real numbers $\lambda_1 \leq \lambda_2 \leq \ldots$ and a corresponding complete $L^2_W$-orthonormal set $\varphi_1, \varphi_2, \ldots : \Sigma \to \RR$ such that $L\varphi_i = -\lambda_i \varphi_i$ and $\lim_i \lambda_i = \infty$.
\end{lemma}
\begin{proof}
	This follows from the standard min-max construction of eigenvalues and eigenfunctions and the compactness of the inclusion $H^1_W(\Sigma) \subset L^2(\Sigma)$, in the spirit of the Rellich--Kondrachov theorem, proven in \cite[Proposition B.2]{BernsteinWang:TopologicalProperty}.
\end{proof}

Since $\lambda_j \to \infty$ as $j \to \infty$, there exist $I$, $K \in \NN$ such that
\begin{equation} \label{eq:linearized.equation.eigenvalues}
	\lambda_1 \leq \ldots \leq \lambda_I < 0 \leq \lambda_{I+1} = 0 = \ldots = \lambda_{I+K} < \lambda_{I+K+1} \leq \ldots
\end{equation}
For notational convenience, for any binary relation $\sim \; \in \{ =, \neq, <, >, \leq, \geq \}$ we define the spectral projector $\Pi_{\sim \mu} : L^2_W(\Sigma) \to L^2_W(\Sigma)$ given by:
\begin{equation} \label{eq:linearized.equation.projections}
	\Pi_{\sim \mu} : f \mapsto \sum_{j : \lambda_j \sim \mu} \langle f, \varphi_j \rangle_{W} \varphi_j.
\end{equation}

We wish to study solutions of the inhomogeneous linear PDE
\begin{equation} \label{eq:linearized.equation.L2.pde}
(\tfrac{\partial}{\partial \tau} - L)u = h \text{ on } \Sigma \times \RR_-,
\end{equation}
where $\RR_- = (-\infty, 0]$ in all that follows. (Of course, in practice, $h$ may depend on $u$.)

At a formal level, if $u(\cdot, \tau) \in H^2_W(\cdot, \tau)$ and $h(\cdot, \tau) \in L^2_W(\Sigma)$ for $\tau \in \RR_-$, then we can use Lemma \ref{lemma:linearized.equation.spectrum} and Hilbert space theory to decompose
\begin{equation} \label{eq:linearized.equation.L2.pde.u.h.decomposition}
u(\cdot, \tau) =: \sum_{j=1}^\infty u_j(\tau) \varphi_j, \; h(\cdot, \tau) =: \sum_{j=1}^\infty h_j(\tau) \varphi_j,
\end{equation}
where the $u_j$, $h_j : \RR_- \to \RR$ are expected (by virtue of  \eqref{eq:linearized.equation.L2.pde}) to be solutions of
\begin{equation} \label{eq:linearized.equation.L2.pde.coefficient.ode}
u_j'(\tau) = -\lambda_j u_j(\tau) + h_j(\tau).
\end{equation}
Turning this formal argument into a rigorous one is standard:

\begin{lemma}[Weighted $L^2$ estimate]  \label{lemma:linearized.equation.L2.pde.estimate}
	Fix $\delta > 0$, $0 < \delta' < \min\{ \delta, -\lambda_I \}$. Suppose that
	\begin{equation} \label{eq:linearized.equation.L2.pde.assumption}
	\int_{-\infty}^0 \Big| e^{-\delta \tau} \Vert h(\cdot, \tau) \Vert_{W} \Big|^2 \, d\tau < \infty.
	\end{equation}
	There exists a unique solution $u$ (``strong in $L^2$'') of \eqref{eq:linearized.equation.L2.pde} such that
	\begin{equation} \label{eq:linearized.equation.L2.pde.conclusion.1}
	\Pi_{<0}(u(\cdot, 0)) = 0,
	\end{equation}
	\begin{equation} \label{eq:linearized.equation.L2.pde.conclusion.2}
	\int_{-\infty}^0 \Big| e^{-\delta' \tau} ( \Vert u(\cdot, \tau) \Vert_{W,2} + \Vert \tfrac{\partial}{\partial \tau} u(\cdot, \tau) \Vert_{W}) \Big|^2 \, d\tau < \infty.
	\end{equation}
	It is given by the series representation in \eqref{eq:linearized.equation.L2.pde.u.h.decomposition} with coefficients:
	\begin{equation} \label{eq:linearized.equation.L2.pde.coeff.negative}
	u_j(\tau) := - \int_\tau^0 e^{\lambda_j(\sigma-\tau)} h_j(\sigma) \, d\sigma, \; j = 1, \ldots, I,
	\end{equation}
	\begin{equation} \label{eq:linearized.equation.L2.pde.coeff.nonnegative}
	u_j(\tau) := \int_{-\infty}^\tau e^{\lambda_j(\sigma-\tau)} h_j(\sigma) \, d\sigma, \; j = I+1, I+2, \ldots
	\end{equation}
	Moreover, for every $\tau \in \RR_-$,
	\begin{equation} \label{eq:linearized.equation.L2.pde.conclusion.3}
	e^{-\delta' \tau} \Vert u(\cdot, \tau) \Vert_{W} \leq C \Big[ \int_{-\infty}^0 \Big| e^{-\delta \sigma} \Vert h(\cdot, \sigma) \Vert_{W} \Big|^2 \, d\sigma \Big]^{\frac12},
	\end{equation}
	where $C = C(\delta, \delta', \lambda_1, \ldots, \lambda_I)$. 
\end{lemma}
\begin{proof}
	The proof is a straightforward computation and adaptation of Galerkin's method from linear parabolic PDE. One starts with ``weak $L^2$'' solutions (\cite[\S 7.1.2, Theorems 3, 4]{Evans:PDE}) and upgrades them to strong ones (\cite[\S 7.1.3, Theorem 5]{Evans:PDE}).
\end{proof}

\subsection{Weighted H\"older space notation} \label{sec:linearized.equation.schauder.notation}

Let $\Omega \subset \Sigma$. We assume that the injectivity radius of $\Sigma$ at points in $\Omega$ is at least $i_0>0$. For $k \in \NN$, $\alpha \in (0, 1)$, we will use the following notation for the standard $C^k$ norm, $C^\alpha$ seminorm, and $C^{k,\alpha}$ norm:
\begin{equation} \label{eq:linearized.equation.schauder.ck.norm}
	\Vert f \Vert_{k;\Omega} := \sum_{i=0}^k \sup_\Omega |\nabla^i_\Sigma f|,
\end{equation}
\begin{equation} \label{eq:linearized.equation.schauder.calpha.seminorm}
	[\nabla^k_\Sigma f]_{\alpha;\Omega} := \sup_{\substack{\mathbf{x} \neq \mathbf{y} \in \Omega\\ d_\Sigma(\bx,\by) < i_0}}\frac{|\nabla_\Sigma^k f(\mathbf{x}) - P_{\by\to\bx}\nabla_\Sigma^k f(\mathbf{y})|}{d_\Sigma(\mathbf{x},\mathbf{y})^\alpha}
\end{equation}
for $P_{\by\to\bx}$ parallel transport defined along the unique minimizing geodesic from $\by$ to $\bx$,
\begin{equation} \label{eq:linearized.equation.schauder.ckalpha.norm}
	\Vert f \Vert_{k,\alpha;\Omega} := \Vert f \Vert_{k;\Omega} + [\nabla^k_\Sigma f]_{\alpha;\Omega}.
\end{equation}
Now let $d \in \RR$. We define the weighted counterparts of the quantities above:
\begin{equation} \label{eq:linearized.equation.schauder.ckd.norm}
	\Vert f \Vert_{k;\Omega}^{(d)} := \sum_{i=0}^k \sup_{\mathbf{x} \in \Omega} \tilde r(\mathbf{x})^{-d+i} |\nabla^i_\Sigma f(\mathbf{x})|,
\end{equation}
\begin{equation} \label{eq:linearized.equation.schauder.calphad.seminorm}
	[\nabla^k_\Sigma f]_{\alpha;\Omega}^{(d)} := \sup_{\substack{\mathbf{x} \neq \mathbf{y} \in \Omega\\ d_\Sigma(\bx,\by) < i_0}} \frac{1}{\tilde r(\mathbf{x})^{d-\alpha} + \tilde r(\mathbf{y})^{d-\alpha}} \frac{|\nabla^k_\Sigma f(\mathbf{x}) - P_{\by\to\bx} \nabla^k_\Sigma f(\mathbf{y})|}{d_\Sigma(\mathbf{x}, \mathbf{y})^\alpha},
\end{equation}
\begin{equation} \label{eq:linearized.equation.schauder.ckalphad.norm}
	\Vert f \Vert_{k,\alpha;\Omega}^{(d)} := \Vert f \Vert_{k;\Omega}^{(d)} + [\nabla^k_\Sigma f]_{\alpha;\Omega}^{(d-k)}.
\end{equation}

Above, $\tilde r$ is as in \cite{ChodoshSchulze}, so we briefly remind the reader what it is. Recall from  \cite[Section 2]{ChodoshSchulze} that \cite[Lemma 2.3]{ChodoshSchulze} gives a diffeomorphism $\cC \setminus B_{R}(0) \simeq \Gamma \times [R,\infty)$ on the non-compact part of $\Sigma$, where $\Gamma$ is the link of the asymptotic cone $\cC$. We will thus parametrize points of $\Sigma$ by $(\omega, r) \in \Gamma \times [R,\infty)$. We emphasize that the coordinate $r$ along $\Sigma$ is \emph{not} exactly equal to $d_{\RR^{n+1}}(\cdot,0)$ (like it is along the cone).  Then $r$ is extended to $\tilde r$ defined on all of $\Sigma$ so that $\tilde r\geq 1$ on $\Sigma$ and $\tilde r=r$ outside of $B_{R}$ for $R\geq 1$ as above. 

In any of the above estimates, if we don't indicate the domain $\Omega$ over which the norm is taken, then it must be understood to be $\Omega = \Sigma$.

\subsection{Pointwise estimates} \label{sec:linearized.equation.schauder.estimates}

We fix $\delta_0 \in (0, -\lambda_I)$ and $\alpha \in (0, 1)$ throughout the section. 

We revisit the inhomogeneous linear PDE
\begin{equation} \label{eq:linearized.equation.schauder.pde}
	(\tfrac{\partial}{\partial \tau} - L)u = h \text{ on } \Sigma \times \RR_-.
\end{equation} 
We will treat classical solutions of the PDE, i.e., ones that satisfy it pointwise. We use implicitly throughout the fact that regularity on $h$ yields improved regularity on $u$ by standard (local) parabolic Schauder theory.

\begin{lemma}[Interior $C^{2,\alpha}$ estimate] \label{lemma:linearized.equation.interior.c2alpha}
	Suppose $u$, $h$ satisfy \eqref{eq:linearized.equation.schauder.pde}, 
	\begin{equation} \label{eq:linearized.equation.schauder.interior.c2alpha.apriori.assumption.h}
		\sup_{\tau \in \RR_-} e^{-2\delta_0 \tau} \Vert h(\cdot, \tau) \Vert_{0,\alpha}^{(-1)} < \infty,
	\end{equation}
	\begin{equation} \label{eq:linearized.equation.schauder.interior.c2alpha.projection.assumption}
	\Pi_{<0}(u(\cdot, 0)) \equiv 0,
	\end{equation}
	and
	\begin{equation} \label{eq:linearized.equation.schauder.interior.c2alpha.apriori.assumption.u}
	\int_{-\infty}^0 \Big| e^{-\delta_0 \tau} (\Vert u(\cdot, \tau) \Vert_{W,2} + \Vert \tfrac{\partial}{\partial t} u(\cdot, \tau) \Vert_{W}) \Big|^2 \, d\tau  < \infty.
	\end{equation}
	Then, for every $\tau \in \RR_-$ and compact $K \subset \RR^{n+1}$, 
	\begin{equation} \label{eq:linearized.equation.interior.c2alpha}
	e^{-\delta_0 \tau} \Vert u(\cdot, \tau) \Vert_{2,\alpha; \Sigma \cap  K} \leq C \sup_{\sigma \in \RR_-} e^{-2\delta_0 \sigma} \Vert h(\cdot, \sigma) \Vert_{0,\alpha}^{(-1)},
	\end{equation}
	with $C = C(\Sigma, \alpha, \delta_0, K)$.
\end{lemma}
\begin{proof}
	Lemma \ref{lemma:linearized.equation.L2.pde.estimate} applies with 
	$\delta \in (\delta_0,2\delta_0)$
	and 
	$\delta' =\delta_0 < \min\{ \delta, -\lambda_I\}$
	 by virtue of \eqref{eq:linearized.equation.schauder.interior.c2alpha.apriori.assumption.h},  \eqref{eq:linearized.equation.schauder.interior.c2alpha.projection.assumption}, \eqref{eq:linearized.equation.schauder.interior.c2alpha.apriori.assumption.u}, and gives
	\[ e^{-\delta_0 \tau} \Vert u(\cdot, \tau) \Vert_W \leq c \Big[ \int_{-\infty}^0 \Big|e^{-\delta \sigma} \Vert h(\cdot, \sigma) \Vert_W\Big|^2 \, d\sigma \Big]^{\frac12}. \]
	Apply the non-standard Schauder estimate in Corollary \ref{coro:schauder.estimate.L1} of Appendix \ref{sec:schauder} on $\Sigma \cap K'$, where $K'$ is a compact set containing $K$ in its interior. It shows that, for $\tau \leq 0$:
	\begin{align*}
	& \Vert u(\cdot, \tau) \Vert_{C^{2,\alpha}(\Sigma \cap K)} \\
	& \qquad \leq C(\Sigma, \alpha, K) \Big( \Vert u \Vert_{L^1((\Sigma \cap K') \times [\tau-1,\tau])} + \sup_{\sigma \in [\tau-1, \tau]} \Vert h(\cdot, \sigma) \Vert_{0,\alpha;\Sigma \cap K'} \Big) \\
	& \qquad \leq C(\Sigma, \alpha, \delta_0, K) \Big( e^{\delta_0 \tau} \Big[ \int_{-\infty}^0 \Big| e^{-\delta \sigma} \Vert h(\cdot, \sigma) \Vert_{W} \Big|^2 \, d\sigma \Big]^{\frac12} + \sup_{\sigma \in [\tau-1, \tau]} \Vert h(\cdot, \sigma) \Vert_{0,\alpha;\Sigma \cap K'} \Big) \\
	& \qquad \leq C(\Sigma, \alpha, \delta_0, K) \Big( e^{\delta_0 \tau} \Big[ \int_{-\infty}^0 (e^{-\delta \sigma} \Vert h(\cdot, \sigma) \Vert_{0}^{(-1)})^2 \, d\sigma \Big]^{\frac12} + \sup_{\sigma \in [\tau-1, \tau]} \Vert h(\cdot, \sigma) \Vert_{0,\alpha}^{(-1)} \Big) \\
	& \qquad \leq C(\Sigma, \alpha, \delta_0, K) \Big( e^{\delta_0 \tau} \cdot \sup_{\sigma \in \RR_-} e^{-2 \delta_0 \sigma} \Vert h(\cdot, \sigma) \Vert_{0}^{(-1)} + e^{2 \delta_0 \tau} \cdot \sup_{\sigma \in \RR_-} e^{-2 \delta_0 \sigma} \Vert h(\cdot, \sigma) \Vert_{0,\alpha}^{(-1)} \Big).
	\end{align*}
	This gives \eqref{eq:linearized.equation.interior.c2alpha}.
\end{proof}

\begin{lemma}[Global $C^0$ estimate]  \label{lemma:linearized.equation.global.c0}
	Suppose $u$, $h$ satisfy \eqref{eq:linearized.equation.schauder.pde}. If
	\begin{equation} \label{eq:linearized.equation.global.c0.assumption.l2w.decay}
		\lim_{\tau \to -\infty} \Vert u(\cdot, \tau) \Vert_{W} = 0,
	\end{equation}
	and
	\begin{equation} \label{eq:linearized.equation.global.c0.assumption.l2w.bound}
		\int_{\tau-1}^\tau (\Vert u(\cdot, \sigma) \Vert_{W,2}^2 + \Vert \tfrac{\partial}{\partial t} u(\cdot, \sigma) \Vert_{W}^2) \, d\sigma < \infty,
	\end{equation}
	for all $\tau \in \RR_{-}$, 
	then for all $\tau \in \RR_-$ and $R \geq R_0(\Sigma)$:
	\begin{equation} \label{eq:linearized.equation.global.c0}
		\Vert u(\cdot, \tau) \Vert_{0}^{(1)}
		\leq C \sup_{\sigma \leq \tau} \Big[ \Vert u(\cdot, \sigma) \Vert_{0;\Sigma \cap B_R(\mathbf{0})}^{(1)} + \Vert h(\cdot, \sigma) \Vert_{0;\Sigma \setminus B_R(\mathbf{0})}^{(0)} \Big],
	\end{equation}
	for $C = C(\Sigma)$. 
\end{lemma}
\begin{proof}
	Fix $\tau_0 \in \RR_-$. Following \cite[Lemma 3.15]{ChodoshSchulze}, we consider $\varphi := \alpha |\mathbf{x}| - \beta$ with
	\[ \alpha := 2 \sup_{(\Sigma \setminus B_R(\mathbf{0})) \times (-\infty, \tau_0]} |h| + 2 R^{-1}  \sup_{(\Sigma \cap \partial B_R(\mathbf{0})) \times (-\infty, \tau_0]} |u|, \]
	\[ \beta := 4 \sup_{(\Sigma \setminus B_R(\mathbf{0})) \times (-\infty, \tau_0]} |h| + R^{-1} \sup_{(\Sigma \cap \partial B_R(\mathbf{0})) \times (-\infty, \tau_0]} |u|. \]
	Note that, by definition, 
	\begin{equation} \label{eq:linearized.equation.global.c0.phi.static}
		\tfrac{\partial}{\partial \tau} \varphi \equiv 0,
	\end{equation}
	and, if $R \geq 2$,
	\begin{equation} \label{eq:linearized.equation.global.c0.phi.positive}
		\varphi > 0 \text{ on } \Sigma \setminus B_R(\mathbf{0}).
	\end{equation}
	Consider the function
	\[ f := u - \varphi \text{ on } (\Sigma \setminus B_R(\mathbf{0})) \times (-\infty, \tau_0]. \]
	As in \cite[Lemma 3.15]{ChodoshSchulze}, by construction:
	\begin{equation} \label{eq:linearized.equation.global.c0.barrier.diff.boundary.ineq}
		f \leq 0 \text{ on } (\Sigma \cap \partial B_R(\mathbf{0})) \times (-\infty, \tau_0],
	\end{equation}
	\begin{equation} \label{eq:linearized.equation.global.c0.barrier.diff.pde.ineq}
		(\tfrac{\partial}{\partial \tau} - L) f \leq 0 \text{ on } (\Sigma \setminus B_R(\mathbf{0})) \times (-\infty, \tau_0]
	\end{equation}
	for $R$ sufficiently large. Multiply \eqref{eq:linearized.equation.global.c0.barrier.diff.pde.ineq} by $f_+ \rho$, where $f_+ := \max \{ f, 0 \}$ and $\rho$ is as in \eqref{eq:linearized.equation.weighted.density}, and integrate over $\Sigma \setminus B_R(\mathbf{0})$. Using \eqref{eq:linearized.equation.weighted.L2.self.adjoint}, and differentiating under the integral sign using \eqref{eq:linearized.equation.global.c0.assumption.l2w.bound},  \eqref{eq:linearized.equation.global.c0.phi.static}, and \cite[\S 5.9.2, Theorem 3]{Evans:PDE} we have, for a.e. $\tau \leq \tau_0$:
	\begin{align*}
		\int_{\Sigma \setminus B_R(\mathbf{0})} |\nabla_\Sigma f_+|^2 \rho \, d\cH^n
			& \leq \int_{\Sigma \setminus B_R(\mathbf{0})} (\tfrac12 f_+^2 + |A_\Sigma|^2 f_+^2) \rho \, d\cH^n - \tfrac12 \int_{\Sigma \setminus B_R(\mathbf{0})} \big( \tfrac{\partial}{\partial \tau} f_+^2 \big) \rho \, \cH^n \\
			& \leq (\tfrac12 + O(R^{-2})) \int_{\Sigma \setminus B_R(\mathbf{0})} f_+^2 \rho \, d\cH^n - \tfrac12 \tfrac{\partial}{\partial \tau} \int_{\Sigma \setminus B_R(\mathbf{0})} f_+^2 \rho \, d\cH^n.
	\end{align*}
	Plugging into Ecker's Sobolev inequality \cite{Ecker:Sobolev} (cf.\ \cite[Proposition 3.9]{ChodoshSchulze}) we get:
	\begin{equation} \label{eq:ecker.spacetime}
		(R^2 - 4n - 8 + O(R^{-2})) \int_{\Sigma \setminus B_R(\mathbf{0})} f_+^2 \rho \, d\cH^n \leq -8 \tfrac{\partial}{\partial \tau} \int_{\Sigma \setminus B_R(\mathbf{0})} f_+^2 \rho \, d\cH^n.
	\end{equation}
	We take $R_0 \geq 2$ large enough so that the above computation holds and $R^2 - 4n - 8 + O(R^{-2}) > 0$ whenever $R \geq R_0$. By \eqref{eq:linearized.equation.global.c0.assumption.l2w.decay}, \eqref{eq:linearized.equation.global.c0.phi.positive},
	\[ 
	\lim_{\tau\to-\infty} \int_{\Sigma \setminus B_R(\mathbf{0})} f_+^2 \rho \, d\cH^n = 0.
	\]
	Because \cite[\S 5.9.2, Theorem 3]{Evans:PDE} shows absolute continuity of $\int_{\Sigma\setminus B_R(\bOh)} f_+^2 \rho \, d\cH^n$ with respect to $\tau \leq \tau_0$, integrating \eqref{eq:ecker.spacetime} over $(-\infty, \tau_0]$, we find that $f_+ \equiv 0$ on $\Sigma \setminus B_R(\mathbf{0}) \times (-\infty, \tau_0]$. Therefore, $f \leq 0$ on $\Sigma \times (-\infty, \tau_0]$. Thus, on $\Sigma \setminus B_R(\mathbf{0})$,
	\[ \tilde{r}^{-1} u(\cdot, \tau_0) \leq 2 \sup_{(\Sigma \setminus B_R(\mathbf{0})) \times (-\infty, \tau_0]} |h| + 2 R^{-1} \sup_{(\Sigma \cap \partial B_R(\mathbf{0})) \times (-\infty, \tau_0]} |u|. \]
	Redoing this with $-f$ in place of $f$ implies \eqref{eq:linearized.equation.global.c0}.
\end{proof}

\begin{lemma}[Global $C^{2,\alpha}$ estimate]  \label{lemma:linearized.equation.global.c2alpha}
	Suppose $u$, $h$ satisfy \eqref{eq:linearized.equation.schauder.pde}. Then, 
	\begin{equation} \label{eq:linearized.equation.global.c2alpha}
		\Vert u(\cdot, \tau) \Vert_{2,\alpha}^{(1)} \leq C \sup_{\sigma \leq \tau} \Big[ \Vert u(\cdot, \sigma) \Vert_{0}^{(1)} + \Vert h(\cdot, \sigma) \Vert_{0,\alpha}^{(-1)} \Big],
	\end{equation}
	for all $\tau \in \RR_-$, with $C = C(\Sigma, \alpha)$. 
\end{lemma}
\begin{proof}
It suffices to prove 
	\begin{equation*}
		\Vert u(\cdot, 0) \Vert_{2,\alpha}^{(1)} \leq C \sup_{\tau\in\RR_{-}} \Big[ \Vert u(\cdot, \tau) \Vert_{0}^{(1)} + \Vert h(\cdot, \tau) \Vert_{0,\alpha}^{(-1)} \Big],
	\end{equation*}
since the general claim will follow by translation in time. 
	
Define the function $\Psi : \Sigma \times (-\infty, 0) \to \RR^{n+1}$ so that $t \mapsto \Psi(\mathbf{x}, t)$ tracks the normal movement in $\RR^{n+1}$ of $\mathbf{x} \in \Sigma$ by mean curvature:
	\[ \tfrac{\partial}{\partial t} \Psi(\mathbf{x}, t) = \mathbf{H}_{\Sigma_t}(\Psi(\mathbf{x}, t)), \]
	where $\Sigma_t := \sqrt{-t}\, \Sigma$, and $\Psi(\cdot, -1) \equiv \operatorname{Id}$. Note that the map 
	$$\Phi_t(\bx):\Sigma  \to \Sigma, (\bx, t) \mapsto \frac{1}{\sqrt{-t}} \Psi(\cdot, t)$$
	satisfies
\begin{equation}\label{eq:evol_Phi}
 \frac{\partial}{\partial t} \Phi_t = \frac{1}{\sqrt{-t}}\left( \frac{\Psi(\cdot, t)}{(-2 t)} + \mathbf{H}_{\Sigma_t}(\Psi(\cdot, t))\right) = \frac{1}{\sqrt{-t}}\frac{\Psi^T(\cdot, t)}{(-2 t)} = \frac{1}{(-2 t)} \Phi_t^T\, .
\end{equation}

	Let $\tau_0 \in \RR_{-}$ be arbitrary. Set:
	\[ \widehat{u}(\cdot, \tau) := u(\cdot, \tau + \tau_0), \; \widehat{h}(\cdot, \tau) := h(\cdot, \tau + \tau_0), \]
	\[ \widehat{v}(\cdot, t) := \sqrt{-t} \widehat{u} \big(\Phi_t(\cdot), -\log(-t) \big). \]
	Note that $\widehat{v}(\cdot, t)$ makes sense for $t \in (-\infty, -e^{\tau_0}]\supset (-\infty,-1]$. Noting that \eqref{eq:evol_Phi} is the same evolution equation as considered in \cite[Definition 3.7]{ChodoshSchulze}, we find that, as in \cite[(3.3)]{ChodoshSchulze}, the transformed function $\hat v$ satisfies 
	\begin{equation} \label{eq:linearized.equation.global.c2alpha.transform}
		\tfrac{\partial}{\partial t} \widehat{v}(\mathbf{x}, t) = \Delta_{\Sigma_t} \widehat{v}(\mathbf{x}, t) + |A_{\Sigma_t}|^2 \widehat{v}(\mathbf{x}, t) + \tfrac{1}{\sqrt{-t}}  \widehat{h}\big( \Phi_t(\mathbf{x}), -\log(-t)\big)
	\end{equation}
	as a function on $\Sigma_{t}$. 
	
	We have, for $R = R(\Sigma)$ sufficiently large and $t \in [-e,-e^{\tau_{0}}]$,
	\begin{equation} \label{eq:linearized.equation.global.c2alpha.transform.c0}
		\Vert \widehat{v}(\cdot, t) \Vert_{0; \Sigma_{t} \cap (B_{4R}(\bOh) \setminus B_{R}(\bOh))} \leq C \Vert \widehat{u}(\cdot, -\log(-t)) \Vert_{0}^{(1)}.
	\end{equation}
	Likewise,
	\begin{equation} \label{eq:linearized.equation.global.c2alpha.transform.calpha}
		\Vert \tfrac{1}{\sqrt{-t}}  \widehat{h}\big( \Phi_t(\cdot), -\log(-t)\big) \Vert_{0,\alpha; \Sigma_{t} \cap(B_{4R}(\bOh) \setminus B_R(\mathbf{0}))} \leq C \Vert \widehat{h}(\cdot, -\log(-t)) \Vert_{0,\alpha}^{(-1)},
	\end{equation}
	with $C = C(\Sigma, \alpha, R)$.

	By Knerr's Schauder estimates (Theorem \ref{theo:schauder.estimate.knerr} in Appendix \ref{sec:schauder}) applied to sufficiently small balls, and \eqref{eq:linearized.equation.global.c2alpha.transform.c0}, \eqref{eq:linearized.equation.global.c2alpha.transform.calpha}, we find that for $R = R(\Sigma)$ sufficiently large
	\begin{align*}
		& \Vert \widehat{v}(\cdot, -e^{\tau_{0}}) \Vert_{2,\alpha; \Sigma_{-e^{\tau_{0}}} \cap (B_{3R}(\bOh) \setminus B_{2R}(\mathbf{0}))} \\
		& \qquad \leq C \sup_{t \in [-1, -e^{\tau_{0}}]}  \Big[ \Vert \widehat{v}(\cdot, t) \Vert_{0;\Sigma_{t}\cap (B_{4R}(\bOh) \setminus B_R(\mathbf{0}))} \\
		& \qquad \qquad +  \Vert \tfrac{1}{\sqrt{-t}} \widehat{h}\big(  \Phi_t(\cdot), -\log(-t) \big) \Vert_{0,\alpha;\Sigma_{t}\cap(B_{4R}(\bOh) \setminus B_R(\mathbf{0}))} \Big] \\
		& \qquad \leq C \sup_{t \in [-1, -e^{\tau_{0}}]} \Big[ \Vert \widehat{u}(\cdot, -\log(-t)) \Vert_{0}^{(1)} + \Vert \widehat{h}(\cdot, -\log(-t)) \Vert_{0,\alpha}^{(-1)} \Big]\\
		& \qquad \leq C \sup_{\tau\in\RR_{-}} \Big[ \Vert u(\cdot, \tau) \Vert_{0}^{(1)} + \Vert h(\cdot, \tau) \Vert_{0,\alpha}^{(-1)} \Big]
	\end{align*}
	Undoing the renormalization for $\hat v$, we thus find 
	\[
	\Vert u(\cdot,0)\Vert_{2,\alpha;\Sigma \cap \left(B_{3Re^{-\tau_{0}/2}}(\bOh)\setminus B_{2Re^{-\tau_{0}/2}}(\bOh)\right)}^{(1)} \leq C \sup_{\tau\in\RR_{-}} \Big[ \Vert u(\cdot, \tau) \Vert_{0}^{(1)} + \Vert h(\cdot, \tau) \Vert_{0,\alpha}^{(-1)} \Big]
	\]
	Taking the supremum over all $\tau_{0}\in\RR_{-}$, we have 
	\[
	\Vert u(\cdot,0)\Vert_{2,\alpha;\Sigma \setminus B_{2R}(\bOh)}^{(1)} \leq C \sup_{\tau\in\RR_{-}} \Big[ \Vert u(\cdot, \tau) \Vert_{0}^{(1)} + \Vert h(\cdot, \tau) \Vert_{0,\alpha}^{(-1)} \Big].
	\]
	Along with standard interior parabolic Schauder estimates, this yields \eqref{eq:linearized.equation.global.c2alpha}. 
\end{proof}

\subsection{Nonlinear error term} \label{sec:linearized.equation.error} 

We work in graphical coordinates over $\Sigma$. On $\Sigma$ itself, we denote the position vector by $\mathbf{x}_\Sigma$, and we fix a unit normal $\nu_\Sigma$ so that following our conventions from Section \ref{sec:prelim.curvature} we can form the mean curvature scalar $H_\Sigma = \mathbf{H}_\Sigma \cdot \nu_\Sigma$. For graphical surfaces $S = \operatorname{graph}_\Sigma u$, with unit normal $\nu$ (so that $\nu \cdot \nu_\Sigma > 0$) and mean curvature $H$ the rescaled mean curvature flow is:
\begin{equation} \label{eq:linearized.equation.graphical.pde}
	\tfrac{\partial}{\partial \tau} u = v \big[   H + \tfrac12 \mathbf{x} \cdot \nu \big],
\end{equation}
where $v$ is the geometric function
\begin{equation} \label{eq:geometric.identities.v}
v := (1 + |(\operatorname{Id}  + uA_\Sigma)^{-1}(\nabla_\Sigma u)|^2)^{\frac12} = (\nu \cdot \nu_\Sigma)^{-1}.
\end{equation}
We can rewrite \eqref{eq:linearized.equation.graphical.pde} as
\begin{equation} \label{eq:linearized.equation.graphical.pde.linearized}
	(\tfrac{\partial}{\partial \tau} - L)u = E(u) \text{ on } \Sigma \times \RR_-,
\end{equation}
where we take $L$ to be precisely the operator from \eqref{eq:linearized.equation.linear.operator} and 
\begin{equation} \label{eq:linearized.equation.error}
	E(u) := v \big[ H + \tfrac12 \mathbf{x} \cdot \nu \big] - \big[ H_\Sigma + \tfrac12 \mathbf{x}_\Sigma \cdot \nu_\Sigma \big] - Lu.
\end{equation}
Note that the second term in parentheses vanishes, since $\Sigma$ satisfies the shrinker equation, but it is helpful to keep this vanishing term in mind in terms of estimating the error. The nonlinear error term can be estimated as follows:

\begin{lemma} \label{lemma:linearized.equation.error}
	There exists $\eta = \eta(\Sigma)$ such that for $u : \Sigma \to \RR$ with $\Vert u \Vert_2^{(1)} \leq \eta$, the nonlinear error term $E(u)$ from \eqref{eq:linearized.equation.error} decomposes as 
	\begin{multline} \label{eq:linearized.equation.error.decomposition}
		E(u)(\mathbf{x}) = u(\mathbf{x}) E_1(\mathbf{x}, u(\mathbf{x}), \nabla_\Sigma u(\mathbf{x}), \nabla_\Sigma^2 u(\mathbf{x})) \\
		+  \nabla_\Sigma u(\mathbf{x}) \cdot \mathbf{E}_2(\mathbf{x}, u(\mathbf{x}), \nabla_\Sigma u(\mathbf{x}), \nabla^2_\Sigma u(\mathbf{x})),
	\end{multline}
	where $E_1$, $\mathbf{E}_2$ are smooth functions on the following domains:
	\[ E_1(\mathbf{x}, \cdot, \cdot, \cdot) : \RR \times T_{\mathbf{x}} \Sigma \times \operatorname{Sym}(T_{\mathbf{x}} \Sigma \otimes T_{\mathbf{x}} \Sigma) \to \RR, \]
	\[ \mathbf{E}_2(\mathbf{x}, \cdot, \cdot, \cdot) : \RR \times T_{\mathbf{x}} \Sigma \times \operatorname{Sym}(T_{\mathbf{x}} \Sigma \otimes T_{\mathbf{x}} \Sigma) \to T_{\mathbf{x}} \Sigma. \]
	Moreover, we can estimate:
	\begin{align} 
		& \tilde r(\mathbf{x}) ^{2+j-\ell} |\nabla_{\mathbf{x}}^i \nabla_z^j \nabla_{\mathbf{q}}^k \nabla_{\mathbf{A}}^\ell E_1(\mathbf{x}, z, \mathbf{q}, \mathbf{A})| \nonumber \\
		& \qquad \qquad \qquad \leq C(\tilde r(\mathbf{x})^{-1} |z| + |\mathbf{q}| + r(\mathbf{x}) |\mathbf{A}|)^{ \max\{0, 1-j-k-\ell\}}, \label{eq:linearized.equation.error.term.1} \\ 
		& \tilde r(\mathbf{x}) ^{1+j-\ell} |\nabla_{\mathbf{x}}^i \nabla_z^j \nabla_{\mathbf{q}}^k \nabla_{\mathbf{A}}^\ell \mathbf{E}_2(\mathbf{x}, z, \mathbf{q}, \mathbf{A})| \nonumber \\
		& \qquad \qquad \qquad \leq C(\tilde r(\mathbf{x})^{-1} |z| + |\mathbf{q}| + r(\mathbf{x}) |\mathbf{A}|)^{ \max\{0, 1-j-k-\ell\}} \label{eq:linearized.equation.error.term.2}.
	\end{align}
	In the above, $C = C(\Sigma)$, $\tilde r$ is as in Section \ref{sec:linearized.equation.schauder.notation}, and $i$, $j$, $k$, $\ell \geq 0$. 
\end{lemma}
\begin{proof}
	It will be convenient to rewrite \eqref{eq:linearized.equation.error} as
	\[ E(u) = \underbrace{\big[ vH - H_\Sigma - \Delta_\Sigma u - |A_\Sigma|^2 u \big]}_{=: \; E^H(u)} + \tfrac12 \underbrace{\big[ v (\mathbf{x} \cdot \nu) - \mathbf{x}_\Sigma \cdot \nu_\Sigma + \mathbf{x}_\Sigma \cdot \nabla_\Sigma u - u \big]}_{=: \; E^{\mathbf{x} \cdot \nu}(u)}. \]
	By linearity, it suffices to check \eqref{eq:linearized.equation.error.decomposition}, \eqref{eq:linearized.equation.error.term.1}, \eqref{eq:linearized.equation.error.term.2} separately for $E^H(u)$, $E^{\mathbf{x} \cdot \nu}(u)$.
	
	Using \cite[(C.4)]{ChodoshSchulze}, $E^H(u)$ readily decomposes as
	\[ E^H(u) = u E^H_1(\cdot, u, \nabla_\Sigma u, \nabla^2_\Sigma u) + \nabla_\Sigma u \cdot \mathbf{E}^H_2(\cdot, u, \nabla_\Sigma u, \nabla^2_\Sigma u). \]
	Estimates \eqref{eq:linearized.equation.error.term.1}, \eqref{eq:linearized.equation.error.term.2} for $E^H_1$, $E^H_2$ are a simple consequence of scaling; indeed, they are the scale-invariant manifestation of the quadratic error nature of the linearization of $H$ on an asymptotically conical manifold where, crucially, $|A_\Sigma| + \tilde r |\nabla_\Sigma A_\Sigma| \leq C \tilde r^{-1}$.
	
	Using \cite[(C.2)]{ChodoshSchulze}, the term $E^{\mathbf{x} \cdot \nu}(u)$ can in fact be written as
	\begin{equation} \label{eq:linearized.equation.error.term.xnu}
		E^{\mathbf{x} \cdot \nu}(u) = \sum_{k=1}^\infty u^k A_\Sigma^k(\mathbf{x}_\Sigma, \nabla_\Sigma u),
	\end{equation}
	where $A_\Sigma^k$  is the $2$-tensor corresponding to the $k$-times composition of the shape operator (the dual to $A_\Sigma$). Note that this is also of the required form, \eqref{eq:linearized.equation.error.decomposition}, and in fact it can be viewed as both $u E^{\mathbf{x}\cdot \nu}_1$ or $\nabla_\Sigma u \cdot \mathbf{E}^{\mathbf{x}\cdot \nu}_2$. The power series in \eqref{eq:linearized.equation.error.term.xnu} is absolutely convergent by \cite[Lemma 2.7]{ChodoshSchulze}. By the sharp derivative estimate in Corollary \ref{coro:shrinker.geometry.sff.radial} of Appendix \ref{sec:shrinker.geometry}, the series in \eqref{eq:linearized.equation.error.term.xnu} can also be differentiated and estimated termwise to yield \eqref{eq:linearized.equation.error.term.1} if we view it as $u E^{\mathbf{x}\cdot\nu}_1$, or \eqref{eq:linearized.equation.error.term.2} if we view it as $\nabla_\Sigma u \cdot \mathbf{E}^{\mathbf{x} \cdot \nu}_2$.
\end{proof}

\begin{corollary} \label{coro:linearized.equation.error}
	There exists $\eta = \eta(\Sigma)$ such that for $u : \Sigma \to \RR$ with $\Vert u \Vert_{2}^{(1)} \leq \eta$:
	\begin{equation} \label{eq:linearized.equation.error.c0}
		\tilde{r} |E(u)| \leq C(\tilde{r}^{-1} |u| + |\nabla_\Sigma u|)(\tilde{r}^{-1} |u| + |\nabla_\Sigma u| + \tilde{r} |\nabla^2_\Sigma u|), 
	\end{equation}
	\begin{equation} \label{eq:linearized.equation.error.calpha}
		\Vert E(u) \Vert_{0,\alpha}^{(-1)} \leq C \Vert u \Vert_{1,\alpha}^{(1)} \Vert u \Vert_{2,\alpha}^{(1)},
	\end{equation}
	and for $\bar u : \Sigma \to \RR$ also with $\Vert \bar u \Vert_{2}^{(1)} \leq \eta$:
	\begin{multline} \label{eq:linearized.equation.error.diff.c0}
			\tilde{r} |E(\bar u) - E(u)| \leq C(\tilde{r}^{-1} |u| + |\nabla_\Sigma u| + \tilde{r} |\nabla_\Sigma^2 u| + \tilde{r}^{-1} |\bar u| + |\nabla_\Sigma \bar u| + \tilde{r} |\nabla_\Sigma^2 \bar u|) \\
			\cdot (\tilde{r}^{-1} |\bar u-u| + |\nabla_\Sigma (\bar u-u)| + \tilde{r} |\nabla^2_\Sigma (\bar u-u)|),
	\end{multline}
	\begin{equation} \label{eq:linearized.equation.error.diff.calpha}
		\Vert E(\bar u) - E(u) \Vert_{0,\alpha}^{(-1)} \leq C (\Vert u \Vert_{2,\alpha}^{(1)} + \Vert \bar u \Vert_{2,\alpha}^{(1)}) \Vert \bar u-u \Vert_{2,\alpha}^{(1)}.
	\end{equation}
	Above, $C = C(\Sigma)$, and \eqref{eq:linearized.equation.error.c0},   \eqref{eq:linearized.equation.error.diff.c0} are pointwise estimates on $\Sigma$.
\end{corollary}
\begin{proof}
	Estimates \eqref{eq:linearized.equation.error.c0}, \eqref{eq:linearized.equation.error.calpha} follow by applying \eqref{eq:linearized.equation.error.decomposition} to decompose $E(u)$ and \eqref{eq:linearized.equation.error.term.1}, \eqref{eq:linearized.equation.error.term.2} with $i = j = k = \ell = 0$ to estimate the two terms in the decomposition. 
	
	Estimates \eqref{eq:linearized.equation.error.diff.c0}, \eqref{eq:linearized.equation.error.diff.calpha} follow by applying \eqref{eq:linearized.equation.error.decomposition} to decompose $E(u)$, $E(\bar u)$, using the fundamental theorem of calculus to expand 
	\[ E_1(\cdot, \bar u, \nabla_\Sigma \bar u, \nabla_\Sigma^2 \bar u) - E_1(\cdot, u, \nabla_\Sigma u, \nabla_\Sigma^2 u), \; \mathbf{E}_2(\cdot, \bar u, \nabla_\Sigma \bar u, \nabla_\Sigma^2 \bar u) - \mathbf{E}_2(\cdot, u, \nabla_\Sigma u, \nabla_\Sigma^2 u), \]
	and then using \eqref{eq:linearized.equation.error.term.1}, \eqref{eq:linearized.equation.error.term.2} with $i = 0$, $j + k + \ell = 1$ to estimate the Taylor expansion coming from the fundamental theorem of calculus.
\end{proof}


\section{Dynamics of smooth ancient rescaled flows} \label{sec:dynamics}

In what follows, we make extensive use of the $L^2$ projection notation from \eqref{eq:linearized.equation.projections}.

\begin{lemma} \label{lemma:dynamics}
	Suppose $u$, $h$ are such that
	\begin{equation} \label{eq:dynamics.pde}
		(\tfrac{\partial}{\partial \tau} - L)u = h,
	\end{equation}
	and that for some $\mu \in \{ \lambda_1, \ldots, \lambda_I \} \cup \{0\}$:
	\begin{equation} \label{eq:dynamics.mu.stable.decay}
		\lim_{\tau \to -\infty} e^{\mu \tau} \Vert \Pi_{>\mu} u(\cdot, \tau) \Vert_{W} = 0.
	\end{equation}
	Suppose that $h$ satisfies, respectively for each binary relation $>$, $=$, $<$, that
	\begin{equation} \label{eq:dynamics.l2.bound}
		|\langle h(\cdot, \tau), \Pi_{\gtreqqless \mu} u(\cdot, \tau) \rangle_{W}| \leq \delta(\tau) \Vert u(\cdot, \tau) \Vert_{W} \Vert \Pi_{\gtreqqless \mu} u(\cdot, \tau) \Vert_{W} 
	\end{equation}
	for some non-decreasing $\delta : \RR_- \to [0, \delta_0]$. If $\delta_0$ is sufficiently small depending on $\Sigma$, then
	\begin{equation} \label{eq:dynamics.mu.stable.dominated}
		\Vert \Pi_{>\mu} u(\cdot, \tau) \Vert_{W} \leq C \delta(\tau) \Vert \Pi_{\leq \mu} u(\cdot, \tau) \Vert_{W}, \; \forall \tau \in \RR_-,
	\end{equation}
	and either
	\begin{equation} \label{eq:dynamics.mu.unstable.dominant}
		\Vert \Pi_{=\mu} u(\cdot, \tau) \Vert_{W} \leq C \delta(\tau) \Vert \Pi_{<\mu} u(\cdot, \tau) \Vert_{W}, \; \forall \tau \in \RR_-,
	\end{equation}
	or there exists a non-decreasing $\tau_0 : \RR_- \to \RR_-$ such that $\tau_0(\tau) \leq \tau$ for all $\tau \in \RR_-$ and 
	\begin{equation} \label{eq:dynamics.mu.neutral.dominant}
		\Vert \Pi_{<\mu} u(\cdot, \tau) \Vert_{W} \leq C \delta(\bar \tau) \Vert \Pi_{=\mu} u(\cdot, \tau) \Vert_{W}, \; \forall \bar \tau \in \RR_-, \; \tau \leq \tau_0(\bar \tau).
	\end{equation}
	Here, $C = C(\Sigma)$. 
\end{lemma}
\begin{proof}
	Let $\underline{\mu}$ (resp.\ $\overline{\mu}$) be the largest (resp.\ smallest) eigenvalue of $L$ strictly below (resp.\ strictly above) $\mu$---if $\mu=\lambda_1$, the choice of $\underline\mu$ is irrelevant. Taking the dot product of \eqref{eq:dynamics.pde} with eigenfunctions of $L$ we find, by \eqref{eq:dynamics.l2.bound}, that:
	\begin{equation*} 
		\tfrac{d}{d\tau} \Vert \Pi_{<\mu} u(\cdot, \tau) \Vert_{W} + \underline{\mu} \Vert \Pi_{<\mu} u(\cdot, \tau) \Vert_{W} \geq - C \delta(\tau) \Vert u(\cdot, \tau) \Vert_{W},
	\end{equation*}
	\begin{equation*} 
		|\tfrac{d}{d\tau} \Vert \Pi_{=\mu} u(\cdot, \tau) \Vert_{W} + \mu \Vert \Pi_{=\mu} u(\cdot, \tau) \Vert_{W}| \leq C \delta(\tau) \Vert u(\cdot, \tau) \Vert_{W},
	\end{equation*}
	\begin{equation*} 
		\tfrac{d}{d\tau} \Vert \Pi_{>\mu} u(\cdot, \tau) \Vert_{W} + \overline{\mu} \Vert \Pi_{>\mu} u(\cdot, \tau) \Vert_{W} \leq C \delta(\tau) \Vert u(\cdot, \tau) \Vert_{W},
	\end{equation*}
	for $C = C(\Sigma)$. Note that we may multiply through with $e^{\mu \tau}$ and rewrite these as:
	\begin{equation} \label{eq:dynamics.mu.ode.unstable}
		\tfrac{d}{d\tau} (e^{\mu \tau} \Vert \Pi_{<\mu} u(\cdot, \tau) \Vert_{W}) + (\underline{\mu} - \mu) (e^{\mu \tau} \Vert \Pi_{<\mu} u(\cdot, \tau) \Vert_{W}) \geq - C \delta(\tau) (e^{\mu \tau} \Vert u(\cdot, \tau) \Vert_{W}),
	\end{equation}
	\begin{equation} \label{eq:dynamics.mu.ode.neutral}
		|\tfrac{d}{d\tau} (e^{\mu \tau} \Vert \Pi_{=\mu} u(\cdot, \tau) \Vert_{W})| \leq C \delta(\tau) (e^{\mu \tau} \Vert u(\cdot, \tau) \Vert_{W}),
	\end{equation}
	\begin{equation} \label{eq:dynamics.mu.ode.stable}
		\tfrac{d}{d\tau} (e^{\mu \tau} \Vert \Pi_{>\mu} u(\cdot, \tau) \Vert_{W}) + (\overline{\mu} - \mu) (e^{\mu \tau} \Vert \Pi_{>\mu} u(\cdot, \tau) \Vert_{W}) \leq C \delta(\tau) (e^{\mu \tau} \Vert u(\cdot, \tau) \Vert_{W});
	\end{equation}	
	By the Merle--Zaag ODE lemma (see \cite[Lemma B.1]{ChoiMantoulidis}), applied to \eqref{eq:dynamics.mu.ode.unstable}, \eqref{eq:dynamics.mu.ode.neutral}, \eqref{eq:dynamics.mu.ode.stable}, together with the a priori assumption \eqref{eq:dynamics.mu.stable.decay}, it follows that if $\delta_0$ is sufficiently small, then 
	\begin{equation*}
		e^{\mu \tau} \Vert \Pi_{>\mu} u(\cdot, \tau) \Vert_{W} \leq C \delta(\tau) (e^{\mu \tau} \Vert \Pi_{\leq \mu} u(\cdot, \tau) \Vert_{W}), \; \forall \tau \in \RR_-,
	\end{equation*}
	and that either\footnote{The Merle--Zaag ODE lemma is for a fixed coefficient $\delta$, rather than a variable coefficient $\delta(\cdot)$, on the right hand sides of the differential inequalities. Per the lemma, for any fixed value $\delta(\tau)$, we have a dichotomy: looking backwards in time, either the first alternative (``unstable dominates neutral'') holds \textit{immediately}, or the second alternative (``neutral dominates unstable'') holds \textit{eventually}. Note that if either alternative holds for one coefficient, then it must hold for all smaller coefficients (in view of the alternative) and all previous times (by the monotonicity of $\delta(\cdot)$). The function $\tau_0(\cdot)$ succinctly arranges for the unavoidable fact that the second alternative may go into effect at different times for different values of $\delta(\cdot)$.}
	\begin{equation*} 
		e^{\mu \tau} \Vert \Pi_{=\mu} u(\cdot, \tau) \Vert_{W} \leq C \delta(\tau) (e^{\mu \tau} \Vert \Pi_{<\mu} u(\cdot, \tau) \Vert_{W}), \; \forall \tau \in \RR_-,
	\end{equation*}
	or there exists a non-decreasing $\tau_0 : \RR_- \to \RR_-$ such that $\tau_0(\tau) \leq \tau$ for all $\tau \in \RR_-$ and
	\begin{equation*} 
		e^{\mu \tau} \Vert \Pi_{<\mu} u(\cdot, \tau) \Vert_{W} \leq C \delta(\bar \tau) (e^{\mu \tau} \Vert \Pi_{=\mu} u(\cdot, \tau) \Vert_{W}), \; \forall \bar \tau \in \RR_-, \; \tau \leq \tau_0(\bar \tau).
	\end{equation*}
	This is the required result after canceling out $e^{\mu \tau}$ from all sides.
\end{proof}

\begin{corollary} \label{coro:dynamics}
	Suppose $u$, $h$ are such that \eqref{eq:dynamics.pde}, \eqref{eq:dynamics.l2.bound} hold for all $\mu \in \{ \lambda_1, \ldots, \lambda_I \} \cup \{0\}$. If
	\begin{equation} \label{eq:dynamics.delta.bounded.by.u}
		\delta(\tau) \leq C_0 \sup_{\sigma \leq \tau} \Vert u(\cdot, \sigma) \Vert_{W}, \; \forall \tau \in \RR_-,
	\end{equation}
	and
	\begin{equation} \label{eq:dynamics.u.decay}
		\lim_{\tau \to -\infty} \Vert u(\cdot, \tau) \Vert_{W} = 0,
	\end{equation}
	then either $u \equiv 0$ or there exists $\mu \in \{ \lambda_1, \ldots, \lambda_I \} \cup \{ 0 \}$ and a non-decreasing $\tau_0 : \RR_- \to \RR_-$ with $\tau_0(\tau) \leq \tau$ for all $\tau \in \RR_-$ such that
	\begin{equation} \label{eq:dynamics.dominant.mode}
		\Vert \Pi_{\neq \mu} u(\cdot, \tau) \Vert_{W} \leq C \delta(\bar \tau) \Vert \Pi_{=\mu} u(\cdot, \tau) \Vert_{W}, \; \forall \bar \tau \in \RR_-, \; \tau \leq \tau_0(\bar \tau),
	\end{equation}
	for $C = C(\Sigma, C_0)$, and, if $\mu < 0$, then
	\begin{equation} \label{eq:dynamics.u.decay.sharp}
		0 < \liminf_{\tau \to -\infty} e^{\mu \tau} \Vert u(\cdot, \tau) \Vert_{W} \leq \limsup_{\tau \to -\infty} e^{\mu \tau} \Vert u(\cdot, \tau) \Vert_{W} < \infty.
	\end{equation}
	If $K = 0$ (recall, $K = \dim \ker L$ in $L^2_W(\Sigma)$), then $\mu \neq 0$.
\end{corollary}
\begin{proof}
	Let $\mu \in \{ \lambda_1, \ldots, \lambda_I \} \cup \{0\}$ be the smallest possible choice for which \eqref{eq:dynamics.mu.stable.decay} holds true; note that this statement isn't vacuous, since \eqref{eq:dynamics.u.decay} guarantees \eqref{eq:dynamics.mu.stable.decay} at least for $\mu = 0$. 
	
	\begin{claim}
		\eqref{eq:dynamics.mu.unstable.dominant} cannot hold.
	\end{claim}
	\begin{proof}[Proof of claim]
		Note that if \eqref{eq:dynamics.mu.unstable.dominant} held, then $\mu\not = \lambda_1$. If $\underline{\mu}$ is the largest eigenvalue smaller than $\mu$, by \eqref{eq:dynamics.pde}, \eqref{eq:dynamics.l2.bound}, \eqref{eq:dynamics.mu.stable.dominated}, and \eqref{eq:dynamics.mu.unstable.dominant}, if it did hold, we would have that
		\[ \tfrac{d}{d\tau} \Vert \Pi_{<\mu} u(\cdot, \tau) \Vert_{W} + \underline{\mu} \Vert \Pi_{<\mu} u(\cdot, \tau) \Vert_{W} \geq - C \delta(\tau) \Vert \Pi_{<\mu} u(\cdot, \tau) \Vert_{W}. \]
		Arguing as in \cite[Claim 4.5]{ChoiMantoulidis}, which requires the knowledge that $\delta(\tau)$ is bounded per \eqref{eq:dynamics.delta.bounded.by.u}, it would follow that
		\[ \Vert \Pi_{<\mu} u(\cdot, \tau) \Vert_{W} \leq C' e^{-\underline{\mu} \tau}, \]
		at which point \eqref{eq:dynamics.mu.stable.dominated}, \eqref{eq:dynamics.mu.unstable.dominant} guarantee that
		\[ \Vert \Pi_{> \underline{\mu}} u(\cdot, \tau) \Vert_{W} = \Vert \Pi_{\geq \mu} u(\cdot, \tau) \Vert_{W} \leq C \delta(\tau) \Vert \Pi_{<\mu} u(\cdot, \tau) \Vert_{W} \leq C e^{-2\underline{\mu} \tau}, \]
		violating the minimal nature of $\mu$. Thus, \eqref{eq:dynamics.mu.unstable.dominant} cannot hold.
	\end{proof}

	So, \eqref{eq:dynamics.mu.neutral.dominant} must hold. Together, \eqref{eq:dynamics.mu.stable.dominated}, \eqref{eq:dynamics.mu.neutral.dominant} give \eqref{eq:dynamics.dominant.mode}. If $\mu = 0$, there is nothing left to prove; the result follows. Otherwise, we simply note that \eqref{eq:dynamics.pde}, \eqref{eq:dynamics.l2.bound}, \eqref{eq:dynamics.mu.stable.dominated}, \eqref{eq:dynamics.mu.neutral.dominant} give:
	\begin{equation} \label{eq:dynamics.dominant.ode}
		|\tfrac{d}{d\tau} \Vert \Pi_{=\mu} u(\cdot, \tau) \Vert_{W} + \mu \Vert \Pi_{=\mu} u(\cdot, \tau) \Vert_{W}| \leq C \delta(\tau) \Vert \Pi_{=\mu} u(\cdot, \tau) \Vert_{W}.
	\end{equation}
	Arguing as in \cite[Claim 4.5]{ChoiMantoulidis} again, with $\mu$ in place of $\lambda_I$, gives the rightmost inequality of \eqref{eq:dynamics.u.decay.sharp}, and the leftmost inequality is obtained by instead using the two-sided nature of the bound in \eqref{eq:dynamics.dominant.ode}.
\end{proof}

The following lemma verifies that assumptions \eqref{eq:dynamics.l2.bound}, \eqref{eq:dynamics.delta.bounded.by.u} are met for ancient rescaled mean curvature flows that stay sufficiently close to $\Sigma$ in the suitable scale-invariant sense:

\begin{lemma} \label{lemma:dynamics.condition.check}
	If $u : \Sigma \times \RR_- \to \RR$ is such that \eqref{eq:linearized.equation.graphical.pde.linearized} and
	\begin{equation} \label{eq:dynamics.condition.c3.o1}
		\lim_{\tau \to -\infty} \Vert u(\cdot, \tau) \Vert_3^{(1)} = 0,
	\end{equation}
	then the choice 
	\begin{equation} \label{eq:dynamics.condition.delta}
		\delta(\tau) := \sup_{\sigma \leq \tau} \Vert u(\cdot, \sigma) \Vert_{2,\alpha}^{(1)}
	\end{equation}
	satisfies \eqref{eq:dynamics.l2.bound} with $h = E(u)$, and \eqref{eq:dynamics.delta.bounded.by.u}.
\end{lemma}
\begin{proof}	
	First let's show that $\delta(\tau)$ satisfies \eqref{eq:dynamics.l2.bound} with $h = E(u)$. We use Lemma \ref{lemma:linearized.equation.error}'s decomposition, \eqref{eq:linearized.equation.error.decomposition}. By virtue of \eqref{eq:linearized.equation.error.term.1} and \eqref{eq:dynamics.condition.delta}, we only need to check that
	\begin{equation} \label{eq:dynamics.condition.check.E2}
		\langle \nabla_\Sigma u(\cdot, \tau) \cdot \mathbf{E}_2(\cdot, u, \nabla_\Sigma u, \nabla_\Sigma^2 u), \Pi_{\gtreqqless \mu} u(\cdot, \tau) \rangle_W \leq C \delta(\tau) \Vert u(\cdot, \tau) \Vert_W \Vert \Pi_{\gtreqqless \mu} u(\cdot, \tau) \Vert_W.
	\end{equation}
	We deal with the cases $<$, $=$ differently than $>$. 
	
	We can deal with $<$ and $=$ at the same time, and we use the symbol $\leqq$ to denote either of these binary relations. Since there are only finitely many eigenvalues $\leq \mu$ by \eqref{eq:linearized.equation.eigenvalues}, one easily sees that:
	\begin{equation} \label{eq:dynamics.condition.check.low.modes}
		\Vert \nabla \Pi_{\leqq \mu} u(\cdot, \tau) \Vert_W \leq C \Vert \Pi_{\leqq \mu} u(\cdot, \tau) \Vert_W,
	\end{equation}
	where $C$ depends on $\Sigma$, $\mu$. In particular, \eqref{eq:dynamics.condition.check.low.modes} implies \eqref{eq:dynamics.condition.check.E2} for $\leqq$ after integrating by parts and using \eqref{eq:linearized.equation.error.term.2} with $i + j + k + \ell \leq 1$.
	
	We now deal with the binary relation $>$. Since
	\[ u(\cdot, \tau) = \Pi_{> \mu} u(\cdot, \tau) + \Pi_{=\mu} u(\cdot, \tau) + \Pi_{< \mu} u(\cdot, \tau), \]
	we can rewrite the left hand side of \eqref{eq:dynamics.condition.check.E2} as
	\begin{align*}
		& \langle \nabla_\Sigma u(\cdot, \tau) \cdot \mathbf{E}_2(\cdot, u, \nabla_\Sigma u, \nabla_\Sigma^2 u), \Pi_{> \mu} u(\cdot, \tau) \rangle_W  \\
		& \qquad = \langle \nabla_\Sigma \Pi_{> \mu} u(\cdot, \tau) \cdot \mathbf{E}_2(\cdot, u, \nabla_\Sigma u, \nabla_\Sigma^2 u), \Pi_{> \mu} u(\cdot, \tau) \rangle_W \\
		& \qquad \qquad + \langle \nabla_\Sigma \Pi_{= \mu} u(\cdot, \tau) \cdot \mathbf{E}_2(\cdot, u, \nabla_\Sigma u, \nabla_\Sigma^2 u), \Pi_{> \mu} u(\cdot, \tau) \rangle_W \\
		& \qquad \qquad + \langle \nabla_\Sigma \Pi_{< \mu} u(\cdot, \tau) \cdot \mathbf{E}_2(\cdot, u, \nabla_\Sigma u, \nabla_\Sigma^2 u), \Pi_{> \mu} u(\cdot, \tau) \rangle_W \\
		& \qquad = \tfrac12 \langle \mathbf{E}_2(\cdot, u, \nabla_\Sigma u, \nabla_\Sigma^2 u), \nabla_\Sigma (\Pi_{> \mu} u(\cdot, \tau))^2 \rangle_W \\
		& \qquad \qquad + \langle \nabla_\Sigma \Pi_{= \mu} u(\cdot, \tau) \cdot \mathbf{E}_2(\cdot, u, \nabla_\Sigma u, \nabla_\Sigma^2 u), \Pi_{> \mu} u(\cdot, \tau) \rangle_W \\
		& \qquad \qquad + \langle \nabla_\Sigma \Pi_{< \mu} u(\cdot, \tau) \cdot \mathbf{E}_2(\cdot, u, \nabla_\Sigma u, \nabla_\Sigma^2 u), \Pi_{> \mu} u(\cdot, \tau) \rangle_W.
	\end{align*}
	The second and third terms we estimate via \eqref{eq:dynamics.condition.check.low.modes} and then $\Vert \Pi_{\leqq \mu} u(\cdot, \tau) \Vert_W \leq \Vert u(\cdot, \tau) \Vert_W$ and \eqref{eq:linearized.equation.error.term.2} with $i+j+k+\ell=0$. The first term we estimate by integrating by parts and then using $\Vert \Pi_{>\mu} u(\cdot, \tau) \Vert_W \leq \Vert u(\cdot, \tau) \Vert_W$ and \eqref{eq:linearized.equation.error.term.2} with $i+j+k+\ell=1$. This completes our proof of \eqref{eq:dynamics.condition.check.E2} and thus \eqref{eq:dynamics.l2.bound} with $h = E(u)$.

	Now we check that $\delta(\tau)$ satisfies \eqref{eq:dynamics.delta.bounded.by.u}. Fix $R > 0$. By Lemma \ref{lemma:linearized.equation.global.c2alpha}, then Lemma \ref{lemma:linearized.equation.global.c0}, and then Corollary  \ref{coro:linearized.equation.error}:
	\begin{align*}
		\Vert u(\cdot, \tau) \Vert_{2,\alpha}^{(1)} 
			& \leq C \sup_{\sigma \leq \tau} \Big[ \Vert u(\cdot, \sigma) \Vert_{0}^{(1)} + \Vert E(u)(\cdot, \sigma) \Vert_{0,\alpha}^{(-1)} \Big] \\
			& \leq C \sup_{\sigma \leq \tau} \Big[ \Vert u(\cdot, \sigma) \Vert_{0;\Sigma \cap B_R(\mathbf{0})} + \Vert E(u)(\cdot, \sigma) \Vert_{0,\alpha}^{(-1)} \Big] \\
			& \leq C \sup_{\sigma \leq \tau} \Big[ \Vert u(\cdot, \sigma) \Vert_{0;\Sigma \cap B_R(\mathbf{0})} + \delta(\sigma) \Vert u(\cdot, \sigma) \Vert_{2,\alpha}^{(1)} \Big].
	\end{align*}
	In particular, since $\delta(\cdot) = o(1)$ by \eqref{eq:dynamics.condition.c3.o1}, we deduce
	\[ \delta(\tau) = \sup_{\sigma \leq \tau} \Vert u(\cdot, \sigma) \Vert_{2,\alpha}^{(1)} \leq C \sup_{\sigma \leq \tau} \Vert u(\cdot, \sigma) \Vert_{0;\Sigma \cap B_R(\mathbf{0})}^{(1)}. \]
	In the compact set $\Sigma \cap B_R(\mathbf{0})$, we can thus control the $C^0$ norm of $u(\cdot,\sigma)$ by the $L^2(\Sigma \cap B_{2R}(\mathbf{0})\times [\sigma-1,\sigma])$ norm of $u$, which is dominated by the $L^2_W(\Sigma)$ norm. Thus, $\delta(\tau)$ satisfies \eqref{eq:dynamics.delta.bounded.by.u}, completing the proof.
\end{proof}


\section{Uniqueness of smooth one-sided ancient rescaled flows} \label{sec:one.sided.flows}

In this section, we characterize smooth ancient flows lying on one side of an asymptotically conical shrinker $\Sigma$, with Gaussian density no larger than twice that of the entropy of $\Sigma$.

\begin{lemma}[One-sided decay] \label{lemma:one.sided.decay}
	Let $(S(\tau))_{\tau \leq 0}$ be an ancient rescaled mean curvature flow lying on one side of $\Sigma$ and such that, for $\tau \leq 0$, we can write $S(\tau) := \operatorname{graph}_\Sigma u(\cdot, \tau)$, $u \geq 0$, with
	\begin{equation} \label{eq:one.sided.decay.assumption}
		\lim_{\tau \to -\infty} \Vert u(\cdot, \tau) \Vert_{3}^{(1)} = 0.
	\end{equation}
	Then, either $u \equiv 0$, or there exists a nonzero constant $\alpha_1 \in \RR$ such that:
	\begin{equation} \label{eq:one.sided.decay.conclusion.1}
		\lim_{\tau \to -\infty} e^{\lambda_1 \tau} \Pi_{=\lambda_1} u(\cdot, \tau) = \alpha_1 \varphi_1,
	\end{equation}
	\begin{equation} \label{eq:one.sided.decay.conclusion.2}
		\limsup_{\tau \to -\infty} e^{2\lambda_1 \tau} \Vert \Pi_{=\lambda_1} u(\cdot, \tau) - \alpha_1 e^{-\lambda_1 \tau} \varphi_1 \Vert_{W} < \infty.
	\end{equation}
	\begin{equation} \label{eq:one.sided.decay.conclusion.3}
		\limsup_{\tau \to -\infty} e^{2\lambda_1 \tau} \Vert u(\cdot, \tau) - \Pi_{=\lambda_1} u(\cdot, \tau) \Vert_{W} < \infty,
	\end{equation}
\end{lemma}
\begin{proof}
	Lemma \ref{lemma:dynamics.condition.check} and \eqref{eq:one.sided.decay.assumption} imply that Lemma \ref{lemma:dynamics}, Corollary \ref{coro:dynamics} are applicable with
	\[ \delta(\tau) := \sup_{\sigma \leq \tau} \Vert u(\cdot, \sigma) \Vert_{2,\alpha}^{(1)}. \]
	Invoke Corollary \ref{coro:dynamics}. If $u \equiv 0$, there is nothing left to prove. Let us suppose $u \not \equiv 0$.
	
	\begin{claim} 
		$\mu = \lambda_1$.
	\end{claim}
	\begin{proof}[Proof of claim]
		Note that
		\[ 0 \leq u(\cdot, \tau) = \Pi_{=\mu} u(\cdot, \tau) + \Pi_{\neq \mu} u(\cdot, \tau) \implies (\Pi_{=\mu} u(\cdot, \tau))_- \leq |\Pi_{\neq \mu} u(\cdot, \tau)|. \]
		By \eqref{eq:dynamics.dominant.mode},
		\begin{align} \label{eq:one.sided.neg.part.bound}
			\Vert (\Pi_{=\mu} u(\cdot, \tau))_- \Vert_{W} 
				& \leq \Vert \Pi_{\neq \mu} u(\cdot, \tau) \Vert_{W} \nonumber \\
				& \leq C \delta(\bar \tau) \Vert \Pi_{=\mu} u(\cdot, \tau) \Vert_{W}, \; \forall \,  \bar \tau \in \RR_-, \; \tau \leq \tau_0(\bar \tau).
		\end{align}
		Denote $h^{(\tau)} := \Vert \Pi_{=\mu} u(\cdot, \tau) \Vert_{W}^{-1} \Pi_{=\mu} u(\cdot, \tau)$. Since $\lambda_1 \leq \mu \leq 0$, it follows from the Rellich--Kondrachov theorem on $L^2_W(\Sigma)$  that $h^{(\tau)}$ converges after passing to a subsequence to some $\mu$-eigenfunction $h^{(-\infty)}$ with $\Vert h^{(-\infty)} \Vert_{W} = 1$. By \eqref{eq:one.sided.neg.part.bound} and the fact that $\lim_{\tau \to -\infty} \delta(\tau) = 0$, it follows that $h^{(-\infty)} \geq 0$, and the claim follows from elementary elliptic theory.
	\end{proof}
	
	In view of $\mu = \lambda_1$,  \eqref{eq:dynamics.u.decay.sharp} implies 
	\begin{equation} \label{eq:one.sided.decay.delta.bound}
		\limsup_{\tau \to -\infty} e^{\lambda_1 \tau} \delta(\tau) < \infty.
	\end{equation}
	Thus,
	\[\Vert \tfrac{d}{d\tau} \Pi_{=\lambda_1} u(\cdot, \tau) + \lambda_1  \Pi_{=\lambda_1} u(\cdot, \tau) \Vert_{W} \leq C \delta(\tau) \Vert \Pi_{=\lambda_1} u(\cdot, \tau) \Vert_{W} \]
	can be integrated to yield the existence of a limit $\lim_{\tau \to -\infty} e^{\lambda_1 \tau} \Pi_{=\lambda_1} u(\cdot, \tau)$, i.e., \eqref{eq:one.sided.decay.conclusion.1}, and by \eqref{eq:one.sided.decay.delta.bound} also gives \eqref{eq:one.sided.decay.conclusion.2}. Finally, we note that Lemma \ref{lemma:dynamics} is applicable with $\mu = \lambda_1$. Indeed,  \eqref{eq:dynamics.l2.bound} always holds by Lemma \ref{lemma:dynamics.condition.check}, while \eqref{eq:dynamics.mu.stable.decay} holds by \eqref{eq:dynamics.dominant.mode},  \eqref{eq:one.sided.decay.delta.bound}. Therefore, conclusion \eqref{eq:dynamics.mu.stable.dominated} of Lemma \ref{lemma:dynamics} implies
	\begin{equation} \label{eq:one.sided.decay.diff.lambda1}
		\Vert u(\cdot, \tau) - \Pi_{=\lambda_1} u(\cdot, \tau) \Vert_{W} = \Vert \Pi_{>\lambda_1} u(\cdot, \tau) \Vert_{W} \leq C \delta(\tau) \Vert \Pi_{=\lambda_1} u(\cdot, \tau) \Vert_{W} \leq C e^{-2\lambda_1 \tau},
	\end{equation}
	which implies \eqref{eq:one.sided.decay.conclusion.3}.
\end{proof}

\begin{corollary}[One-sided uniqueness for graphical flows]  \label{coro:one.sided.decay.uniqueness}
	Up to time translation, there is at most one non-steady ancient rescaled mean curvature flow $(S(\tau))_{\tau \leq 0}$ on one side of $\Sigma$ satisfying \eqref{eq:one.sided.decay.assumption}.
\end{corollary}
\begin{proof}
	We assume that $u$, $\bar u \not \equiv 0$ are two such solutions. It follows from Lemma \ref{lemma:one.sided.decay} that we can translate either $u$ or $\bar u$ in time so that
	\begin{equation} \label{eq:one.sided.decay.uniqueness.neg.infty}
		\lim_{\tau \to -\infty} e^{\lambda_1 \tau} \Vert (\bar u - u)(\cdot, \tau) \Vert_W = 0.
	\end{equation}
	It will also be convenient to write $\delta(\tau)$, $\bar \delta(\tau)$ for the quantities corresponding to \eqref{eq:dynamics.condition.delta} for $u$, $\bar u$. By Lemmas \ref{lemma:dynamics.condition.check} and \ref{lemma:one.sided.decay}, 
	\begin{equation} \label{eq:one.sided.decay.uniqueness.delta}
		\delta(\tau) + \bar \delta(\tau) \leq C_1 e^{-\lambda_1 \tau}, \; \tau \in \RR_-
	\end{equation}
	for a fixed $C_1$. Finally, we introduce the notation
	\[ w := \bar u - u, \; E^w := E(\bar u) - E(u), \]
	so that
	\begin{equation} \label{eq:one.sided.decay.uniqueness.w.pde}
		(\tfrac{\partial}{\partial \tau} - L) w = E^w.
	\end{equation}
	Using \eqref{eq:linearized.equation.error.decomposition} and the fundamental theorem of calculus,
	\begin{align} \label{eq:one.sided.decay.uniqueness.w.error}
		E^w 
			& = \bar u E_1(\cdot, \bar u, \nabla_\Sigma \bar u, \nabla_\Sigma^2 \bar u) + \nabla_\Sigma \bar u \cdot \mathbf{E}_2(\cdot, \bar u, \nabla_\Sigma \bar u, \nabla_\Sigma^2 \bar u) \nonumber \\
			& \qquad - u E_1(\cdot, u, \nabla_\Sigma u, \nabla_\Sigma^2 u) - \nabla_\Sigma u \cdot \mathbf{E}_2(\cdot, u, \nabla_\Sigma u, \nabla_\Sigma^2 u) \nonumber \\
			& = w E_1(\cdot, u, \nabla_\Sigma u, \nabla_\Sigma^2 u) \nonumber \\
			& \qquad + \nabla_\Sigma w \cdot \mathbf{E}_2(\cdot, u, \nabla_\Sigma u, \nabla_\Sigma^2 u) \nonumber \\
			& \qquad + \bar u \big[ E_1(\cdot, \bar u, \nabla_\Sigma \bar u, \nabla_\Sigma^2 \bar u) - E_1(\cdot, u, \nabla_\Sigma u, \nabla_\Sigma^2 u) \big] \nonumber \\
			& \qquad + \nabla_\Sigma \bar u \cdot \big[ \mathbf{E}_2(\cdot, \bar u, \nabla_\Sigma \bar u, \nabla_\Sigma^2 \bar u) - \mathbf{E}_2(\cdot, u, \nabla_\Sigma u, \nabla_\Sigma^2 u) \big] \nonumber \\
			& = w E_1(\cdot, u, \nabla_\Sigma u, \nabla_\Sigma^2 u) \nonumber \\
			& \qquad + \nabla_\Sigma w \cdot \mathbf{E}_2(\cdot, u, \nabla_\Sigma u, \nabla_\Sigma^2 u) \nonumber \\
			& \qquad + \Big[ \bar u \int_0^1 D_z E_1(\cdots) \, dt \Big] w \nonumber \\
			& \qquad + \Big[ \bar u \int_0^1 D_{\mathbf{q}} E_1(\cdots) \, dt \Big] \cdot \nabla_\Sigma w \nonumber \\
			& \qquad + \Big[ \bar u \int_0^1 D_{\mathbf{A}} E_1(\cdots) \, dt \Big] \cdot \nabla^2_\Sigma w \nonumber \\
			& \qquad + \Big[ \nabla_\Sigma \bar u \cdot \int_0^1 D_z \mathbf{E}_2(\cdots) \, dt \Big] w \nonumber \\
			& \qquad + \Big[ \nabla_\Sigma \bar u \cdot \int_0^1 D_{\mathbf{q}} \mathbf{E}_2(\cdots) \, dt \Big] \cdot \nabla_\Sigma w \nonumber \\
			& \qquad + \Big[ \nabla_\Sigma \bar u \cdot \int_0^1 D_{\mathbf{A}} \mathbf{E}_2(\cdots) \, dt \Big] \cdot \nabla^2_\Sigma w,
	\end{align}
	where, in all six instances, $\cdots$ stands for $(\cdot, u + tw, \nabla_\Sigma u + t \nabla_\Sigma w, \nabla_\Sigma^2 u + t \nabla^2_\Sigma w)$. 
We note that we can formally write
$$ E^w= w F + \nabla_\Sigma w \cdot \mathbf{F} + \nabla^2_\Sigma w \cdot \mathcal{F}$$
with 
$$|F| + |\mathbf{F}| + |\mathcal{F}| \leq C (\delta(\tau) + \bar\delta(\tau))\, .$$	
	We take the $L^2_W$ dot product of \eqref{eq:one.sided.decay.uniqueness.w.error} with $w$ and integrate the $ w \nabla^2_\Sigma w \cdot \mathcal{F} $ terms by parts so that, in every term, we have at least two instances of $w$ and $\nabla_\Sigma w$. In particular, we will pick up derivatives of $D_{\mathbf{A}} E_1$ and $D_{\mathbf{A}} \mathbf{E}_2$. Furthermore, when integrating by parts, we pick up terms of the form
	$$ \int_\Sigma w\, (\bx^T \otimes \nabla_\Sigma w) \cdot \mathcal{F} \, \rho\, d\mathcal{H}^n \, .$$
	Recall that Ecker's Sobolev inequality \cite{Ecker:Sobolev} (cf.\ \cite[Proposition 3.9]{ChodoshSchulze}) implies that
$$ \| |\bx| f  \|^2_W \leq 4n \| f\|_{W,1}^2\, ,$$
and we can thus estimate
$$ \left| \int_\Sigma w\, (\bx^T \otimes \nabla_\Sigma w) \cdot \mathcal{F} \, \rho\, d\mathcal{H}^n \right| \leq C (\delta(\tau) + \bar\delta(\tau)) (\| |\bx|w\|_W \|\nabla_\Sigma w\|_W) \leq C (\delta(\tau) + \bar\delta(\tau)) \| w\|_{W,1}^2 \, .$$
	 Using Lemma \ref{lemma:linearized.equation.error}, \eqref{eq:one.sided.decay.assumption}, and  \eqref{eq:one.sided.decay.uniqueness.delta}, we find
	\begin{equation} \label{eq:one.sided.decay.uniqueness.w.dot.error}
		|\langle w(\cdot, \tau), E^w(\cdot, \tau) \rangle_W| \leq C_2 e^{-\lambda_1 \tau} \Vert w(\cdot, \tau) \Vert^2_{W,1}, \; \tau \in \RR_-,
	\end{equation}
	for a fixed $C_2$. Here, $\Vert \cdot \Vert_{W,1}$ is the norm induced from \eqref{eq:linearized.equation.weighted.sobolev.dot} with $k=1$.
	
	We use \eqref{eq:one.sided.decay.uniqueness.w.dot.error} to derive two estimates on the evolution of $\Vert w \Vert_W^2$. First, together with \eqref{eq:one.sided.decay.uniqueness.w.pde} and \eqref{eq:linearized.equation.eigenvalues}, it implies
	\begin{align*}
		\tfrac12 \tfrac{d}{d\tau} \Vert w(\cdot, \tau) \Vert_W^2
			& = \langle w(\cdot, \tau), Lw(\cdot, \tau) + E^w(\cdot, \tau) \rangle_W \\
			& \leq - \lambda_1 \Vert w(\cdot, \tau) \Vert^2_W + C_2 e^{-\lambda_1 \tau} \Vert w(\cdot, \tau) \Vert^2_{W,1}, \; \tau \in \RR_-,
	\end{align*}
	which in turn implies
	\begin{equation} \label{eq:one.sided.decay.uninqueness.ddt.e2lambda.w.sq}
		\tfrac{d}{d\tau} (e^{2\lambda_1 \tau} \Vert w(\cdot, \tau) \Vert^2_W) \leq C_2 e^{\lambda_1 \tau} \Vert w(\cdot, \tau) \Vert^2_{W,1}, \; \tau \in \RR_-.
	\end{equation}
	Second, recalling the definition of $L$ in \eqref{eq:linearized.equation.linear.operator}, integrating by parts, and using \eqref{eq:one.sided.decay.uniqueness.w.dot.error}, it follows that there exists a sufficiently negative $\tau_0$ such that:
	\begin{align}
		\tfrac12 \tfrac{d}{d\tau} \Vert w \Vert_W^2
			& = 	- \Vert \nabla_\Sigma w \Vert_W^2 + \langle w, (\tfrac12 + |A_\Sigma|^2)w + E^w \rangle_W \nonumber \\
			& \leq - \tfrac12 \Vert \nabla_\Sigma w \Vert_W^2 + C_3 \Vert w \Vert_W^2, \; \tau \leq \tau_0, \label{eq:one.sided.decay.uninqueness.ddt.w.sq}
	\end{align}
	with a fixed $C_3$. 

	We next compute the evolution of $\Vert \nabla_\Sigma w \Vert_W^2$. To that end, we need a couple of preliminary computations. By the Gauss equation,
	\begin{equation} \label{eq:one.sided.decay.uniqueness.gauss}
		\Ric_\Sigma(\nabla_\Sigma w, \nabla_\Sigma w) = H_\Sigma A_\Sigma(\nabla_\Sigma w, \nabla_\Sigma w) - A_\Sigma^2(\nabla_\Sigma w, \nabla_\Sigma w),
	\end{equation}
	where $A_\Sigma^2$ is the $2$-tensor corresponding to the self-composition of the shape operator (the dual to $A_\Sigma$). From the definition of the second fundamental form and the shrinker equation \eqref{eq:defn-shrinker},  $H_\Sigma + \tfrac12 \mathbf{x} \cdot \nu_\Sigma = 0$, we have
	\begin{equation} \label{eq:one.sided.decay.uniqueness.xgradw}
		\nabla_\Sigma (\mathbf{x} \cdot \nabla_\Sigma w) \cdot \nabla_\Sigma w = |\nabla_\Sigma w|^2 - 2H_\Sigma A_\Sigma(\nabla_\Sigma w, \nabla_\Sigma w) + \mathbf{x} \cdot \nabla^2_\Sigma w(\nabla_\Sigma w, \cdot).
	\end{equation}
	In what follows, we recall the Gaussian density $\rho$, defined in \eqref{eq:linearized.equation.weighted.density}, which satisfies $\nabla \rho = - \tfrac12 \rho \mathbf{x}$. An integration by parts, followed by the Bochner formula $\Delta_\Sigma \nabla_\Sigma w = \nabla_\Sigma \Delta_\Sigma w + \Ric_\Sigma(\nabla_\Sigma w, \cdot)$, \eqref{eq:one.sided.decay.uniqueness.gauss}, \eqref{eq:one.sided.decay.uniqueness.xgradw}, implies:
	\begin{align} \label{eq:one.sided.decay.uniqueness.laplacian.trick}
		& \int_\Sigma (\Delta_\Sigma w - \tfrac12 \mathbf{x} \cdot \nabla_\Sigma w)^2 \rho \, d\cH^n \nonumber \\
		& \qquad = \int_\Sigma (\Delta_\Sigma w - \tfrac12 \mathbf{x} \cdot \nabla_\Sigma w) \operatorname{div}_\Sigma(\rho \nabla_\Sigma w) \, d\cH^n \nonumber \\
		& \qquad = - \int_\Sigma \nabla_\Sigma (\Delta_\Sigma w - \tfrac12 \mathbf{x} \cdot \nabla_\Sigma w) \cdot \nabla_\Sigma w \, \rho \, d\cH^n \nonumber \\
		& \qquad = - \int_\Sigma (\Delta_\Sigma \nabla_\Sigma w - \Ric_\Sigma(\nabla_\Sigma w, \cdot) - \tfrac12 \nabla_\Sigma (\mathbf{x} \cdot \nabla_\Sigma w)) \cdot \nabla_\Sigma w \, \rho \, d\cH^n \nonumber \\
		& \qquad = - \int_\Sigma (\Delta_\Sigma \nabla_\Sigma w - \tfrac12 \mathbf{x} \cdot \nabla^2_\Sigma w + A_\Sigma^2(\nabla_\Sigma w, \cdot) - \tfrac12 \nabla_\Sigma w) \cdot \nabla_\Sigma w \, \rho \, d\cH^n \nonumber \\
		& \qquad = \int_\Sigma \big[ - \operatorname{div}_\Sigma (\rho \nabla_\Sigma^2 w) + \big( - A_\Sigma^2(\nabla_\Sigma w, \cdot) + \tfrac12 \nabla_\Sigma w \big) \rho \big] \cdot \nabla_\Sigma w \, d\cH^n \nonumber \\
		& \qquad = \int_\Sigma (|\nabla_\Sigma^2 w|^2 - A_\Sigma^2(\nabla_\Sigma w, \nabla_\Sigma w) + \tfrac12 |\nabla_\Sigma w|^2) \, \rho \, d\cH^n.
	\end{align}
	We can now estimate the evolution of $\Vert \nabla_\Sigma w \Vert_\Sigma^2$. Using \eqref{eq:one.sided.decay.uniqueness.w.pde} and the definition of $L$ in \eqref{eq:linearized.equation.linear.operator}:
	\begin{align*} 
		\tfrac12 \tfrac{d}{d\tau} \Vert \nabla_\Sigma w \Vert_W^2
			& = \langle \nabla_\Sigma w, \nabla_\Sigma \tfrac{\partial}{\partial \tau} w \rangle_W \\
			& = -\langle \Delta_\Sigma w - \tfrac12 \mathbf{x} \cdot \nabla_\Sigma w, \tfrac{\partial}{\partial \tau} w \rangle_W \\
			& = -\langle \Delta_\Sigma w - \tfrac12 \mathbf{x} \cdot \nabla_\Sigma w, \Delta_\Sigma w - \tfrac12 \mathbf{x} \cdot \nabla_\Sigma w + |A_\Sigma|^2 w + \tfrac12 w + E^w  \rangle_W \\
			& = - \Vert \Delta_\Sigma w - \tfrac12 \mathbf{x} \cdot \nabla_\Sigma w \Vert_W^2 + \langle \nabla_\Sigma w, \nabla_\Sigma(|A_\Sigma|^2 w + \tfrac12 w) \rangle_W \\
			& \qquad - \langle \Delta_\Sigma w - \tfrac12 \mathbf{x} \cdot \nabla_\Sigma w, E^w \rangle_W \\
			& = - \Vert \Delta_\Sigma w - \tfrac12 \mathbf{x} \cdot \nabla_\Sigma w \Vert_W^2 + \Vert (\tfrac12 + |A_\Sigma|^2)^{\frac12} \nabla_\Sigma w \Vert_W^2 + \langle \nabla_\Sigma w, w \nabla_\Sigma |A_\Sigma|^2 \rangle_W \\
			& \qquad - \langle \Delta_\Sigma w - \tfrac12 \mathbf{x} \cdot \nabla_\Sigma w, E^w \rangle_W.
	\end{align*}
	We claim that this implies:
	\begin{equation} \label{eq:one.sided.decay.uninqueness.ddt.gradw.sq}
		\tfrac12 \tfrac{d}{d\tau} \Vert \nabla_\Sigma w \Vert_W^2 \leq C_4 \Vert w \Vert_{W,1}^2, \; \tau \leq \tau_0,
	\end{equation}
	with fixed $C_4$, after possibly choosing a more negative $\tau_0$. Indeed, in the immediately preceding expression, we use Cauchy--Schwarz on the last term, which together with the first term yield
	\[ - \Vert \Delta_\Sigma w - \tfrac12 \mathbf{x} \cdot \nabla_\Sigma w \Vert_W^2 - \langle \Delta_\Sigma w - \tfrac12 \mathbf{x} \cdot \nabla_\Sigma w, E^w \rangle_W \leq - \tfrac12 \Vert \Delta_\Sigma w - \tfrac12 \mathbf{x} \cdot \nabla_\Sigma w \Vert_W^2 + \tfrac12 \Vert E^w \Vert_W^2. \]
	In the right hand side, the $- \tfrac12 \Vert \Delta_\Sigma w - \tfrac12 \mathbf{x} \cdot \nabla_\Sigma w \Vert_W^2$ term is used, via \eqref{eq:one.sided.decay.uniqueness.laplacian.trick}, to dominate all $\nabla_\Sigma^2 w$ terms in $E^w$, which we computed in  \eqref{eq:one.sided.decay.uniqueness.w.error}; note that these terms have small coefficients for sufficiently negative $\tau$ by virtue of \eqref{eq:one.sided.decay.uniqueness.delta}. This yields \eqref{eq:one.sided.decay.uninqueness.ddt.gradw.sq}.
	
	Together, \eqref{eq:one.sided.decay.uninqueness.ddt.w.sq}, \eqref{eq:one.sided.decay.uninqueness.ddt.gradw.sq} imply that there exist $C_5 \geq 1$, $C_6$ such that
	\begin{equation} \label{eq:one.sided.decay.uniqueness.ddt.w.gradw.sq}
		\tfrac{d}{d\tau} (\Vert \nabla_\Sigma w \Vert_W^2 + C_5 \Vert w \Vert_W^2) \leq C_6 \Vert w \Vert_W^2, \; \tau \leq \tau_0.
	\end{equation}
	Integrating \eqref{eq:one.sided.decay.uniqueness.ddt.w.gradw.sq} from $-\infty$ to $\tau$ and using the decay of $w$, we deduce:
	\begin{equation} \label{eq:one.sided.decay.uniqueness.w.gradw.sq}
		\Vert w(\cdot, \tau) \Vert_{W,1}^2 \leq \Vert \nabla_\Sigma w(\cdot, \tau) \Vert_W^2 + C_5 \Vert w(\cdot, \tau) \Vert_W^2 \leq C_6 \int_{-\infty}^\tau \Vert w(\cdot, s) \Vert_W^2 \, ds, \; \tau \leq \tau_0.
	\end{equation}
	By \eqref{eq:one.sided.decay.uniqueness.neg.infty}, we may take $\tau_0$ more negative yet so that
	\begin{equation} \label{eq:one.sided.decay.uniqueness.w.small.1}
		\Vert w(\cdot, \tau) \Vert_W^2 \leq e^{-2\lambda_1 \tau}, \; \tau \leq \tau_0.
	\end{equation}
	Thus, by evaluating the integral in \eqref{eq:one.sided.decay.uniqueness.w.gradw.sq} using the crude estimate in \eqref{eq:one.sided.decay.uniqueness.w.small.1}, we find
	\begin{equation} \label{eq:one.sided.decay.uniqueness.w.small.2}
		\Vert w(\cdot, \tau) \Vert_{W,1}^2 \leq \frac{C_6}{2|\lambda_1|} e^{-2\lambda_1 \tau}, \; \tau \leq \tau_0,
	\end{equation}
	with the same $\tau_0$. Integrating \eqref{eq:one.sided.decay.uninqueness.ddt.e2lambda.w.sq} from $-\infty$ to $\tau$, and using \eqref{eq:one.sided.decay.uniqueness.neg.infty} at $-\infty$ and \eqref{eq:one.sided.decay.uniqueness.w.small.2}, we get the following improvement over \eqref{eq:one.sided.decay.uniqueness.w.small.1}:
	\begin{equation} \label{eq:one.sided.decay.uniqueness.w.small.3}
		\Vert w(\cdot, \tau) \Vert_W^2 \leq \frac{C_2 C_6}{2 |\lambda_1|^2} e^{-3\lambda_1 \tau}, \; \tau \leq \tau_0,
	\end{equation}
	with the same $\tau_0$. Now we iterate. Using \eqref{eq:one.sided.decay.uniqueness.w.gradw.sq} again, with \eqref{eq:one.sided.decay.uniqueness.w.small.3} in place of \eqref{eq:one.sided.decay.uniqueness.w.small.1}:
	\begin{equation} \label{eq:one.sided.decay.uniqueness.w.small.4}
		\Vert w(\cdot, \tau) \Vert_{W,1}^2 \leq \frac{C_2 C_6^2}{3! |\lambda_1|^3} e^{-3\lambda_1 \tau}, \; \tau \leq \tau_0,
	\end{equation}
	with the same $\tau_0$. Integrating \eqref{eq:one.sided.decay.uninqueness.ddt.e2lambda.w.sq} from $-\infty$ to $\tau$, and using \eqref{eq:one.sided.decay.uniqueness.w.small.4} rather than \eqref{eq:one.sided.decay.uniqueness.w.small.2}, we get the following improvement over \eqref{eq:one.sided.decay.uniqueness.w.small.3}:
	\begin{equation} \label{eq:one.sided.decay.uniqueness.w.small.5}
		\Vert w(\cdot, \tau) \Vert_W^2 \leq \frac{C_2^2 C_6^2}{2 \cdot 3! |\lambda_1|^4} e^{-4\lambda_1 \tau}, \; \tau \leq \tau_0,
	\end{equation}
	with the same $\tau_0$. Repeating this $k \in \NN$ times altogether (we showed steps $k = 1$, $2$), we find
	\begin{equation} \label{eq:one.sided.decay.uniqueness.w.small.6}
		\Vert w(\cdot, \tau) \Vert_W^2 \leq \frac{C_2^k C_6^k}{k! (k+1)! |\lambda_1|^{2k}} e^{-(2+k)\lambda_1 \tau}, \; \tau \leq \tau_0,
	\end{equation}
	with the same $\tau_0$. Fixing $\tau \leq \tau_0$ and sending $k \to \infty$, \eqref{eq:one.sided.decay.uniqueness.w.small.6} gives $w(\cdot, \tau) \equiv 0$. 
\end{proof}


\section{A family of smooth ancient rescaled flows} \label{sec:family.flows}

In this section we construct an $I$-dimensional family (recall, $I$ is as in \eqref{eq:linearized.equation.eigenvalues}) of smooth ancient rescaled mean curvature flows that flow out of the fixed asymptotically conical shrinker $\Sigma^n\subset \mathbb{R}^{n+1}$ as $\tau \to -\infty$. Using the tools at our disposal, this is a straightforward adaptation of \cite[Section 3]{ChoiMantoulidis}. For the convenience of the reader, we emphasize that this section is not used elsewhere in the paper and may be skipped on first read. It is purely of independent interest.

\begin{remark} \label{rema:existence.one.sided}
When $\Sigma$ is asymptotically conical, it seems nontrivial to verify that the construction in this section proves the existence of a \emph{one-sided} flow without performing further error-term analysis near infinity. (If $\Sigma$ is compact this follows easily.) We find it easier to instead prove this existence of one-sided flows in Section \ref{sec:gmt.existence} using geometric measure theory, which we also use to show that the one-sided flow can be continued \emph{through} singularities, which is crucial for subsequent applications. We emphasize that the uniqueness of one-sided flows was established in Section \ref{sec:one.sided.flows}. 
\end{remark}

\subsection{The nonlinear contraction} \label{sec:family.flows.contraction}

We continue to fix $\delta_0 \in (0, -\lambda_I)$, $\alpha \in (0, 1)$. It will be convenient to also consider the operator 
\begin{equation} \label{eq:contraction.iota.minus}
\iota_- : \bm{a} = (a_1, \ldots, a_I) \in \RR^I \mapsto \sum_{j=1}^I a_j e^{-\lambda_j \tau} \varphi_j.
\end{equation}

\begin{theorem} \label{theo:contraction}
	There exists $\mu_{0} = \mu_{0}(\Sigma, \alpha, \delta_0)$ such that, for every $\mu \geq \mu_0$,  there exists a corresponding $\eps = \eps(\Sigma, \alpha, \delta_0, \mu)$ with the following property: 
	
	For any $\bm{a} \in B_\eps(\mathbf{0}) \subset \RR^I$ there exists a unique $\mathscr{S}(\bm{a}) : \Sigma \times \RR_- \to \RR$ so that the hypersurfaces $S(\tau) := \operatorname{graph}_\Sigma \mathscr{S}(\bm{a})(\cdot, \tau)$ satisfy the rescaled mean curvature flow
	\begin{equation} \label{eq:linearized.equation.rescaled.mcf.pde}
	\tfrac{\partial}{\partial \tau} \mathbf{x} = \mathbf{H}_{S(\tau)}(\mathbf{x}) + \tfrac12 \mathbf{x}^\perp, \; \forall \mathbf{x} \in S(\tau),
	\end{equation}
	with the a priori decay 
	\begin{equation} \label{eq:linearized.equation.quadratic.decay}
	\sup_{\tau \in \RR_-} e^{-\delta_0 \tau} \Vert (\mathscr{S}(\bm{a}) - \iota_-(\bm{a}))(\cdot, \tau) \Vert_{2,\alpha}^{(1)} \leq \mu |\bm{a}|^2
	\end{equation}
	and the terminal condition $\Pi_{<0}(\mathscr{S}(\bm{a}))(\cdot, 0) = \iota_-(\bm{a})(\cdot, 0)$.
\end{theorem}
\begin{proof}
	The geometric PDE \eqref{eq:linearized.equation.rescaled.mcf.pde} is equivalent to  \eqref{eq:linearized.equation.graphical.pde.linearized}.
	Consider the affine space
	\[ \mathscr{C}[\bm{a}] := \{ u : \Sigma \times \RR_- \to \RR : \Pi_{<0}(u(\cdot, 0)) = \iota_-(\bm{a})(\cdot, 0), \; \Vert u \Vert_* < \infty \} \]
	where 
	\[ \Vert u \Vert_* := \sup_{\tau \in \RR_-} e^{-\delta_0 \tau} (\Vert u(\cdot, \tau) \Vert_{2,\alpha}^{(1)} + \Vert \tfrac{\partial}{\partial \tau} u(\cdot, \tau) \Vert_{0,\alpha}^{(-1)}). \]
	It is complete with respect to $d_*(\bar u, u) := \Vert \bar u - u \Vert_*$. Note that Lemmas \ref{lemma:linearized.equation.global.c0} and \ref{lemma:linearized.equation.global.c2alpha} imply that $\Vert \iota_-(\bm{a})\Vert_* \leq C|\bm{a}|$.
	
	Let $\eta > 0$ be as in Corollary  \ref{coro:linearized.equation.error}. For $u \in \mathscr{C}[\bm{a}]$, $\Vert u \Vert_* \leq \eta$, let $\mathscr{S}(u; \bm{a})$ be a solution of
	\begin{equation} \label{eq:linearized.equation.iterative.pde}
	(\tfrac{\partial}{\partial \tau} - L) \mathscr{S}(u; \bm{a}) = E(u)
	\end{equation}
	with $\mathscr{S}(u; \bm{a})(\cdot, 0) = \iota_-(\bm{a})(\cdot, 0)$. Equivalently, we are solving
	\[ (\tfrac{\partial}{\partial \tau} - L)(\mathscr{S}(u; \bm{a}) - \iota_-(\bm{a})) = E(u), \; \Pi_{<0}(\mathscr{S}(u; \bm{a}) - \iota_-(\bm{a}))(\cdot, 0) = 0. \]
	Existence is guaranteed by Lemma \ref{lemma:linearized.equation.L2.pde.estimate}, since the a priori decay of $u$ implies quadratic decay of $E(u)$ by \eqref{eq:linearized.equation.error.calpha}. Now Lemma \ref{lemma:linearized.equation.interior.c2alpha} and \eqref{eq:linearized.equation.error.calpha} imply:
	\begin{equation} \label{eq:linearized.equation.iterative.pde.interior.c2alpha}
	\sup_{\tau \in \RR_-} e^{-\delta_0 \tau} \Vert (\mathscr{S}(u; \bm{a}) - \iota_-(\bm{a}))(\cdot, \tau) \Vert_{2,\alpha;\Sigma \cap B_R(\mathbf{0})} \leq C \Vert u \Vert_*^2.
	\end{equation}
	Then, \eqref{eq:linearized.equation.error.c0},   \eqref{eq:linearized.equation.iterative.pde.interior.c2alpha}, and Lemma \ref{lemma:linearized.equation.global.c0} imply, for $\tau \in \RR_-$:
	\begin{multline} \label{eq:linearized.equation.iterative.pde.global.c0}
	e^{-\delta_0 \tau} \Vert (\mathscr{S}(u; \bm{a}) - \iota_-(\bm{a}))(\cdot, \tau) \Vert_{0}^{(1)} \\
	\leq C e^{-\delta_0 \tau} \Big[ \sup_{(\Sigma \setminus B_R(\mathbf{0})) \times (-\infty, \tau]} |E(u)| + \sup_{\partial B_R(\mathbf{0}) \times (-\infty, \tau]} |\mathscr{S}(u; \bm{a}) - \iota_-(\bm{a})| \Big] \leq C \Vert u \Vert_*^2.
	\end{multline}
	Finally, \eqref{eq:linearized.equation.error.calpha},  \eqref{eq:linearized.equation.iterative.pde.global.c0} and Lemma \ref{lemma:linearized.equation.global.c2alpha} imply, for $\tau \in \RR_-$:
	\begin{multline*} 
	e^{-\delta_0 \tau} \Vert (\mathscr{S}(u; \bm{a}) - \iota_-(\bm{a}))(\cdot, \tau) \Vert_{2,\alpha}^{(1)}  \\
	\leq C e^{-\delta_0 \tau} \sup_{\sigma \leq \tau} \Big[ \Vert (\mathscr{S}(u; \bm{a}) - \iota_-(\bm{a}))(\cdot, \sigma) \Vert_{0}^{(1)} + \Vert E(u)(\cdot, \sigma) \Vert_{0,\alpha}^{(-1)} \Big] \leq C \Vert u \Vert_*^2.
	\end{multline*}
	Recalling also Knerr's parabolic Schauder estimates (see Theorem \ref{theo:schauder.estimate.knerr} in Appendix \ref{sec:schauder}), this implies:
	\begin{equation} \label{eq:linearized.equation.iterative.pde.global.c2alpha}
	\Vert \mathscr{S}(u; \bm{a}) - \iota_-(\bm{a}) \Vert_* \leq C \Vert u \Vert_*^2.
	\end{equation}
	Therefore, $\mathscr{S}(u; \bm{a}) \in \mathscr{C}[\bm{a}]$. Note that solutions of \eqref{eq:linearized.equation.iterative.pde} are uniquely determined within $\mathscr{C}[\bm{a}]$ (e.g., due to Lemma \ref{lemma:linearized.equation.L2.pde.estimate}). Thus, $\mathscr{S}(\cdot, \bm{a})$ is a well-defined map of small 
	elements of $\mathscr{C}[\bm{a}]$ into $\mathscr{C}[\bm{a}]$.
	
	Likewise, for $\bar u \in \mathscr{C}[\bm{a}]$, 
	$\Vert \bar u \Vert_* \leq \eta$, we have
	\[ (\tfrac{\partial}{\partial \tau} - L)(\mathscr{S}(\bar u; \bm{a}) - \mathscr{S}(u; \bm{a})) = E(u) - E(\bar u), \; \Pi_{<0}(\mathscr{S}(\bar u; \bm{a}) - \mathscr{S}(u; \bm{a}))(\cdot, 0) = 0. \]
	Therefore the discussion above applies with $\bar u - u$ in place of $u - \iota_-(\bm{a})$ and Corollary  \ref{coro:linearized.equation.error}'s \eqref{eq:linearized.equation.error.diff.c0}, \eqref{eq:linearized.equation.error.diff.calpha} instead of \eqref{eq:linearized.equation.error.c0}, \eqref{eq:linearized.equation.error.calpha}, and gives:
	\begin{multline*} \label{eq:linearized.equation.iterative.pde.global.c2alpha.diff}
	e^{-\delta_0 \tau} \Vert (\mathscr{S}(\bar u; \bm{a}) - \mathscr{S}(u; \bm{a}))(\cdot, \tau) \Vert_{2,\alpha}^{(1)} \\
	\leq C \sup_{\sigma \in \RR_-} e^{-\delta_0 \sigma} \Big[ \Vert u(\cdot, \sigma) \Vert_{2,\alpha}^{(1)} + \Vert \bar u(\cdot, \sigma) \Vert_{2,\alpha}^{(1)} \Big] \cdot \sup_{\sigma \in \RR_-} e^{-\delta_0 \sigma} \Vert \bar u - u \Vert_{2,\alpha}^{(1)}
	\end{multline*}
	i.e.,
	\begin{equation} \label{eq:linearized.equation.iterative.pde.global.c2alpha.diff}
	\Vert \mathscr{S}(\bar u; \bm{a}) - \mathscr{S}(u; \bm{a}) \Vert_* \leq C (\Vert u \Vert_* + \Vert \bar u \Vert_*) \Vert \bar u - u \Vert_*.
	\end{equation}
	Consider the subset $X := \{ u \in \mathscr{C}[\bm{a}] : \Vert u - \iota_-(\bm{a}) \Vert_* \leq \mu |\bm{a}|^2 \}$. There exists $\mu_0 = \mu_0(\Sigma, \alpha, \delta_0)$ such that, for all $\mu \geq \mu_0$, there exists $\eps = \eps(\Sigma, \alpha, \delta_0, \mu)$ such that $\bm{a} \in B_\eps(\mathbf{0}) \subset \RR^I$ and $u \in X$ imply $\mathscr{S}(u; \bm{a}) \in X$, by the triangle inequality and \eqref{eq:linearized.equation.iterative.pde.global.c2alpha}. Thus, $\mathscr{S}(\cdot; \bm{a})$ maps $X$ into itself. By \eqref{eq:linearized.equation.iterative.pde.global.c2alpha.diff}, it is a contraction mapping. By the completeness of $X$, there exists a unique fixed point of $\mathscr{S}(\cdot; \bm{a})$ in $X$, which we denote $\mathscr{S}(\bm{a})$. Note that, by construction, it satisfies \eqref{eq:linearized.equation.quadratic.decay} and $\Pi_{<0}(\mathscr{S}(\bm{a})(\cdot, 0)) = \iota_-(\bm{a})(\cdot, 0)$. 
	It remains to check that $\mathscr{S}(\bm{a})$ satisfies \eqref{eq:linearized.equation.graphical.pde.linearized} smoothly.  Indeed, $E(\mathscr{S}(\bm{a}))$ is H\"older continuous in spacetime by Corollary \ref{coro:linearized.equation.error} and Theorem \ref{theo:schauder.estimate.knerr} in Appendix \ref{sec:schauder}, and the result follows by bootstrapping standard parabolic Schauder estimates to get smoothness on $\mathscr{S}(\bm{a})$.
\end{proof}

\begin{remark} \label{rema:bourni.langford.mramor}
	Bourni--Langford--Mramor \cite{BLM:angenent-torus} recently constructed, using different methods, ancient one-sided flows coming out the Angenent torus and its higher dimensional analog.  
	Our work can be used to construct one-sided flows coming out of any compact and any asymptotically conical shrinker.
\end{remark}


\section{Existence of a smooth ancient shrinker mean convex flow} \label{sec:gmt.existence}

In this section, we construct a smooth ancient shrinker mean convex flow on one side of an asymptotically conical shrinker $\Sigma^n\subset \mathbb{R}^{n+1}$. It would be possible to prove this more in the spirit of the previous section, but thanks to the uniqueness statement from Corollary \ref{coro:one.sided.decay.uniqueness}, we can construct such a flow by any method that is convenient. As such, we use methods that will also apply to construct a (generalized) eternal flow which is smooth for very negative times. We will do so by modifying techniques used in \cite{BernsteinWang:TopologicalProperty} to the present setting. 

We fix a component $\Omega$ of $\RR^{n+1}\setminus \Sigma$ and assume that the unit normal to $\Sigma$ points into $\Omega$. Note that by Colding--Minicozzi's classification of entropy stable shrinkers, \cite[Theorems 0.17 and 9.36]{ColdingMinicozzi:sing-generic}, asymptotically conical shrinkers are entropy unstable. This (and more) is encoded in the following result:

\begin{lemma}[{\cite[Propositions 4.1 and 4.2] {BernsteinWang:TopologicalProperty}}]
The first eigenvalue $\mu: = \lambda_1$ of the $L$ operator (see Lemma \ref{lemma:linearized.equation.spectrum}) satisfies $\lambda_1 < -1$. The corresponding eigenfunction $\varphi_1$ can be taken to be positive. For any $\beta>0$, it satisfies
\begin{align*}
(1+|\mathbf{x}|^2)^{\frac 12 + \mu - \beta} \lesssim \varphi_1(\mathbf{x}) & \lesssim (1+|\mathbf{x}|^2)^{\frac 12 + \mu + \beta}\\
|\nabla^m_\Sigma \varphi_1(\mathbf{x})| & \lesssim (1+|\mathbf{x}|^2)^{\frac 12 + \mu + \beta - \frac m2}.
\end{align*}
Moreover, there is $\eps_0 = \eps_0(\Sigma)>0$ so that for $\eps \in (0,\eps_0)$, the normal graph of $\eps\varphi_1$ is a smooth surface $\Sigma_\eps \subset \Omega$. Denote $\Omega_\eps \subset \Omega$ by the open set with $\partial\Omega_\eps = \Sigma_\eps$. The surface $\Sigma_\eps$ is strictly shrinker mean convex to the interior of $\Omega_\eps$ in the sense that
\[
2H_{\Sigma_\eps} + \mathbf{x}\cdot \nu_{\Sigma_\eps} \geq C \eps (1+ |\mathbf{x}|^2)^{\mu}
\]
for $C=C(\Sigma)$. 
\end{lemma}

The following lemma is essentially \cite[Proposition 4.4]{BernsteinWang:TopologicalProperty}. Note that because $\Sigma_\eps$ has uniformly bounded curvature (along with derivatives) the time interval for which \cite{EckerHuisken:interior} guarantees short-time existence is independent of $\eps\to 0$. 
\begin{lemma}\label{lemm:short-time-exist-Sigma-eps}
There is $\delta = \delta(\Sigma) \in (0,1)$ so that there is a smooth mean curvature flow $\Sigma_\eps(t)$ for $t \in [-1,-1+\delta]$ with $\Sigma_\eps(-1) = \Sigma_\eps$. The flow remains strictly shrinker mean convex with the bound
\[
2t H_{\Sigma_\eps(t)} + \mathbf{x}\cdot\nu_{\Sigma_\eps(t)} \geq C \eps (1+|\mathbf{x}|^2 + 2n(t+1))^{\mu}.
\]
\end{lemma}

We now begin the construction of an eternal weak flow that we will later prove to have the desired properties. Fix $R>0$ so that  for all $\eps \in (0,\eps_0)$ and $\rho \geq 1$, $\Sigma$ and $\Sigma_\eps$ intersect $\partial B_{\rho R}$ transversely. 

\begin{proposition}\label{prop:exist-compact-approx-flows}
There is a smooth hypersurface $\Sigma_{\eps,\rho}$ that formed by smoothing the corners of $(\Sigma_\eps \cap B_{\rho R}) \cup (\partial B_{\rho R} \cap \Omega_\eps)$ and then perturbing slightly so that:
\begin{itemize}
\item as $\rho\to \infty$, $\Sigma_{\eps,\rho}$ converges smoothly on compact sets to $\Sigma_\eps$,
\item the level set flow of $\Sigma_{\eps,\rho}$ is non-fattening, and
\item letting $\cK_{\eps,\rho}$ denote the level set flow of the compact region bounded by $\Sigma_{\eps,\rho}$, there is a unit-regular integral Brakke flow $\cM_{\eps,\rho}$ with initial condition $\cM_{\eps,\rho}(-1) = \cH^n\lfloor \Sigma_{\eps,\rho}$ and so that $\supp\cM_{\eps,\rho} \cap\mathfrak{t}^{-1}((-1,\infty)) = \partial \cK_{\eps,\rho} \cap\mathfrak{t}^{-1}((-1,\infty)) $. 
\end{itemize}
\end{proposition} 
\begin{proof}
Let $\{\Sigma_{\eps,\rho}^a\}_{a \in (-1,1)}$ denote a foliation of smooth surfaces close to $(\Sigma_\eps \cap B_{\rho R}) \cup (\partial B_{\rho R} \cap \Omega_\eps)$ chosen so that as $\rho\to\infty$, each $\Sigma_{\eps,\rho}^a$ converges smoothly on compact sets to $\Sigma_\eps$. For all but countably many $a$, the level set flow of $\Sigma_{\eps,\rho}^a$ does not develop a space-time interior (i.e., does not fatten); see \cite[11.3]{Ilmanen:elliptic}. Write $\Gamma_{\eps,\rho}^a(t) : = \{\bx : u(\bx,t) = a\}$ for the corresponding level set flow. We can arrange (after re-labeling $a$ and changing $u$ if necessary) that the level set flow of the pre-compact open set bounded by $\Sigma_{\eps,\rho}^a$ is $\{x : u(x,t) > a\}$. On the other hand, for a.e.\ $a \in (-1,1)$, \cite[12.11]{Ilmanen:elliptic} guarantees that\footnote{Note that in \cite{Ilmanen:elliptic}, there is a typo in the definition of $(\cdot)_+$; it is clear that the proof of \cite[12.11]{Ilmanen:elliptic} only considers points $(t,x)$ for $t$ strictly greater than the initial time.} 
\begin{equation}\label{eq:reduced-bdry-interior-bdry}
\{u=a\}_+ = (\overline{\partial^*\{u>a\}})_+,
\end{equation} 
where $Z_+ = Z \cap \mathfrak{t}^{-1}((-1,\infty))$. Assume that $a=a(\eps,\rho) \in (-1,1)$ is chosen so that \eqref{eq:reduced-bdry-interior-bdry} holds and the level set flow does not fatten. 

Non-fattening guarantees that $t\mapsto \cH^n\lfloor \partial^*\{x : u(x,t) > a\}$ is a unit-regular integral Brakke flow $\cM_{\eps,\rho}$ by \cite[11.4]{Ilmanen:elliptic} (cf.\ \cite[Theorem 3.10]{BernsteinWang:1}). It remains to check the condition concerning the support of both flows. Note that
\begin{align*}
(\supp\cM_{\eps,\rho})_+& = \left(\overline{\bigcup_{t\geq -1} \partial^*\{\bx : u(\bx,t) > a\} \times \{t\}}\right)_+\\
& = (\overline{\partial^*\{u>a\}})_+\\
& = \{u=a\}_+\\
& = (\partial \{u\geq a\})_+
\end{align*}
The second equality is proven as in \cite[11.6(iii)]{Ilmanen:elliptic}, the third is \eqref{eq:reduced-bdry-interior-bdry} and the final equality follows from non-fattening of $\Gamma_{\eps,\rho}^a$. This completes the proof. 
\end{proof}

Note that we could have used the work of Evans--Spruck \cite{EvansSpruck:levelset4} instead of Ilmanen's approach \cite{Ilmanen:elliptic} in the previous proof. 

\begin{lemma}\label{lemm:approximators-conicality}
There is $r_0 = r_0(\Sigma) > 0$ so that for $r > \tfrac{r_0}{2}$, we can take $\rho$ sufficiently large depending on $r$ to conclude that in the space-time region
\[
(B_r\setminus \bar B_{r_0/2})\times [-1,2],
\]
we have that $\partial\cK_{\eps,\rho}$ and $\cM_{\eps,\rho}$ agree with the set flow and Brakke flow associated to the same smooth mean curvature flow of hypersurfaces. Moreover, there is $C=C(\Sigma)>0$ independent of $r$ so that this flow has second fundamental form bounds 
\[
|\bx| |A| + |\bx|^2|\nabla A| + |\bx|^3 |\nabla^2 A| \leq C. 
\]
\end{lemma}
\begin{proof}
This follows from pseudolocality (cf.\ \cite[Theorem 1.5]{IlmanenNevesSchulze}) and local curvature estimates (cf.\ \cite[Proposition 3.21 and 3.22]{Ecker:book}) applied on large balls far out along $\Sigma_\eps$. See also \cite[Proposition 4.4]{BernsteinWang:TopologicalProperty}.
\end{proof}

We can now pass to a subsequential limit\footnote{We always use Kuratowski convergence to consider limits of sets. Recall that $Z_n\to Z$ if $Z=\{x : \limsup_n d(x,Z_n)=0\}=\{x : \liminf_n d(x,Z_n)\}$. Because $\RR^{n+1}\times \RR$ is separable, subsequential limits in this sense always exist. See \cite[\S 9]{HerskovitsWhite:avoidance-set-theoretic} for further discussion.} $\rho_i\to\infty$ to find a Brakke flow $\cM_\eps$ (resp.\ weak set flow $\cK_\eps$) with initial conditions $\cH^n\lfloor \Sigma_\eps$ (resp.\ $K_\eps$, the closed region above $\Sigma_\eps$; in other words, $K_{\eps}$ is the unique closed set with $K_\eps\subset \Omega$ and $\partial K_\eps = \Sigma_\eps$). 

\begin{lemma}\label{lemm:K-eps-vs-M-eps-vs-K}
We have $\partial\cK_\eps \setminus \mathfrak{t}^{-1}(-1) \subset\supp\cM_\eps \subset \cK_\eps$. 
\end{lemma}
\begin{proof}
For $X \in \partial\cK_\eps \setminus\mathfrak{t}^{-1}(-1)$, there is $X_i \in \partial\cK_{\eps,\rho_i} \setminus\mathfrak{t}^{-1}(-1) = \supp \cM_{\eps,\rho_i} \setminus \mathfrak{t}^{-1}(-1)$ with $X_i\to X$. The monotonicity formula thus guarantees that $X \in \supp \cM_\eps$. The other claim follows directly from the fact that $\cK_\eps$ is closed. 
\end{proof}

\begin{lemma}\label{lemm:conical-ends-smooth-flow}
Take $r_0=r_{0}(\Sigma)$ in Lemma \ref{lemm:approximators-conicality} larger if necessary. Then in the space-time region
\[
(\RR^{n+1}\setminus \bar{B}_{r_0}) \times [-1,1],
\]
both $\partial\cK_\eps$ and $\cM_\eps$ are the same smooth mean curvature flow which we denote by $\Sigma_{\eps}(t)$, and satisfy
\[
|\bx||A| + |\bx|^2 |\nabla A| + |\bx|^3|\nabla^2A| \leq C.
\]
Finally, $\Sigma_{\eps}(t)$ intersects the spheres $\partial B_{r}$ transversely, for all $r > r_{0}$. 
\end{lemma}
Note that the smooth flows from Lemmas \ref{lemm:short-time-exist-Sigma-eps} and \ref{lemm:conical-ends-smooth-flow} agree when they are both defined, so naming this flow $\Sigma_{\eps}(t)$ is not a serious abuse of notation. 
\begin{proof}
The smoothness and curvature estimates follow by passing the curvature estimates in Lemma \ref{lemm:approximators-conicality} to the limit along a diagonal sequence $r\to\infty$. Since $\partial \cK_\eps \subset \supp \cM_\eps \subset  \cK_\eps$, we see that the smooth flows must agree. Finally, transverse intersection follows from \cite[Theorem 2.1]{EckerHuisken:interior} applied to balls far out along $\Sigma_{\eps} = \Sigma_{\eps}(-1)$. 
\end{proof}

\begin{lemma}\label{lemm:Keps-smooth-start}
There is $\delta=\delta(\Sigma)>0$ so that in the space-time region
\[
\mathfrak{t}^{-1}([-1,-1+\delta]),
\]
both $\partial \cK_\eps$ and $\cM_\eps$ agree with the smooth mean curvature flow $\Sigma_\eps(t)$ from Lemma \ref{lemm:short-time-exist-Sigma-eps}. 
\end{lemma}
\begin{proof}
Because $\Sigma_{\eps,\rho_i}$ are converging smoothly to $\Sigma_\eps$ on compact sets, pseudolocality and interior estimates guarantee that for any $r>0$, there is a uniform $\delta>0$ so that taking $i$ sufficiently large, one component of 
\[
\partial \cK_{\eps,\rho_i}\cap ({B_r}\times [-1,-1+\delta])
\]
is a smooth mean curvature flow with uniformly bounded curvature (and similarly for $\cM_{\eps,\rho_i}$) for $t \in [-1,-1+\delta]$. 

Small spherical barriers show that for $i$ large, no other component of 
\[
\partial \cK_{\eps,\rho_i}\cap (B_r \times [-1,-1+\delta] )
\]
can intersect $B_{r/2}\times[-1,-1+\delta]$. As such, sending $i\to\infty$, we can pass the curvature estimates to the limit to find that $\partial \cK_\eps \cap \mathfrak{t}^{-1}([-1,-1+\delta])$ (and similarly for $\cM_\eps$) are both smooth mean curvature flows with uniformly bounded curvature that agree with $\Sigma_\eps$ at $t=-1$. The assertion thus follows from $\partial\cK_\eps\subset \supp\cM_\eps\subset \cK_\eps$ as before, or alternatively from the uniqueness of smooth solutions to mean curvature flow with bounded curvature, \cite[Theorem 1.1]{ChenYin}.
\end{proof}

We define the parabolic dilation map
\[ \cF_\lambda : \RR^{n+1} \times \RR\to  \RR^{n+1} \times \RR, \qquad \cF_\lambda : (\bx,t) \mapsto (\lambda \bx,\lambda^2 t). \]
The following result is a consequence of Lemma \ref{lemm:short-time-exist-Sigma-eps} and relates the analytic property of shrinker mean convexity to the behavior of the flow under parabolic dilation. It is convenient to define
\begin{equation}\label{eq:defn-lambda_0-rescaling-param}
\lambda_0 : = \bigg(\frac{1-\frac{\delta}{2}}{1-\delta}\bigg)^{\frac{1}{2}} > 1
\end{equation}
where $\delta$ is as in Lemma \ref{lemm:Keps-smooth-start}. Observe that $\cF_\lambda(\Sigma_\eps(t)\times \{t\}) = \lambda \Sigma_\eps(t) \times \{\lambda^2t\}$, so $\lambda \Sigma_\eps(-\lambda^{-2})$ is the $t=-1$ slice of the parabolic rescaling (by $\lambda$) of the space-time track of the flow $t\mapsto \Sigma_\eps(t)$  and the maximal smooth existence time $T>-1+\delta/2$.

\begin{corollary}\label{coro:sep-smooth-flow}
For $\lambda \in (1,\lambda_0)$ the surface $\lambda \Sigma_\eps(-\lambda^{-2})$ is contained in the interior of $\Omega_\eps$. Moreover, for any $r>0$ large, there is $c=c(r,\Sigma) > 0$ so that 
\[
d(\Sigma_\eps \cap B_r, \lambda \Sigma_\eps(-\lambda^{-2}) \cap B_r) \geq c \eps (\lambda-1). 
\]
for all $\lambda \in (1,\lambda_0)$. 
\end{corollary} 
\begin{proof}
By Lemma \ref{lemm:short-time-exist-Sigma-eps}, the family of hypersurfaces defined by $\lambda \mapsto \lambda \Sigma_\eps(-\lambda^{-2})$ has normal speed given by
\[
 2(- \lambda^{-2}) H_{\Sigma_\eps(-\lambda^{-2})} + \mathbf{x} \cdot \nu_{\Sigma_\eps(-\lambda^{-2})} \geq C \eps (1+|\mathbf{x}|^2 + 2n(1-\lambda^{-2}))^{\mu}.
\]
This is strictly positive, which proves the first statement. Moreover, the speed is strictly bounded below on $B_r$, which proves the second statement. 
\end{proof}

For $\lambda\geq1$, we define $\cK_{\eps}^\lambda : = \cF_\lambda(\cK_{\eps})\cap \mathfrak{t}^{-1}([-1,1])$ and similarly for $\cM_{\eps}^\lambda$. Recall that $\lambda_0$ has been defined in \eqref{eq:defn-lambda_0-rescaling-param}. Note that $\partial \cK_\eps^\lambda \cap \mathfrak{t}^{-1}(-1)= \lambda\Sigma_\eps(-\lambda^{-2})$ for $\lambda \in [1,\lambda_0)$. Below, we will write $\cK_\eps^1$ (and similarly $\cM_\eps^1$) (as opposed to. $\cK_\eps$ and $\cM_\eps$), the difference being that the time parameter has been restricted to $-1\leq t\leq 1$.

\begin{lemma}\label{lemm:rescaling-end-graphical-over-other-end}
There is $r_{1}=r_{1}(\Sigma)>r_0$ so that for any $\lambda \in (1,\lambda_{0})$, $\lambda \Sigma_{\eps}(\lambda^{-2}t) \setminus \bar{B}_{r_{1}}$ can be written as the normal graph of a function $f_{t}$ defined on the end of $\Sigma_{\eps}(t)$ for all $t \in [-1,1]$. The function $f_{t}$ satisfies
\begin{align*}
|\mathbf{x}||f_{t}| + |\mathbf{x}|^{2}|\nabla f_{t}| + |\mathbf{x}|^{3}|\nabla^{2}f_{t}|\leq C,
\end{align*}
where $C=C(\Sigma)$. Moreover, 
\[
\left( \tfrac{\partial}{\partial t} - \Delta_{\Sigma_\eps(t)} \right) f_t =\mathbf{a} \cdot \nabla_{\Sigma_\eps(t)} f_t + b f_t
\]
where $|\mathbf{a}| + |b| \leq C=C(\Sigma)$. 
\end{lemma}
\begin{proof}
This follows from the argument in \cite[Proposition 4.4]{BernsteinWang:TopologicalProperty}. Indeed, we first observe that by taking $r_{1}$ sufficiently large, $\lambda \Sigma_{\eps}(\lambda^{-2}t)$ and $ \Sigma_{\eps}(t)$ are locally graphs of some functions $u,u^\lambda$ over
\[
B_{\eta |\mathbf{z}|}(\bz) \subset T_{\mathbf{z}}\cC
\]
for $\eta=\eta(\Sigma)>0$ and $|\mathbf{z}| > r_{1}$ sufficiently large. Differentiating the mean curvature flow equation as in \cite[Lemma 2.2]{Wang:uniqueness}  yields curvature estimates that prove that $f_t$ exists and satisfies the asserted estimates. Finally, the fact that $f_t$ satisfies the given equation follows by considering the quadratic error terms when linearizing the mean curvature flow equation; a similar argument can be found in \cite[Lemma 2.5]{Sesum:rateMCF}. 
\end{proof}

\begin{proposition}\label{prop:brakke-eps-disjoint-Klambda}
The support of the Brakke flow $\supp\cM_\eps^1$ is disjoint from the scaled weak set flow $\cK_\eps^\lambda$, for all $\lambda \in (1,\lambda_0)$. 
\end{proposition} 
\begin{proof}
We follow the proof of \cite[Proposition 4.4]{BernsteinWang:TopologicalProperty}, but use Ilmanen's localized avoidance principle in the compact part, due to the possible presence of singularities. Fix $\lambda \in (1,\lambda_{0})$ and let $T \in [-1,1]$ denote the first time the claim fails:
\[
T = \sup\{ \tau : \supp\cM^{1}_{\eps}\cap \cK^{\lambda}_{\eps} \cap \mathfrak{t}^{-1}((-1,\tau)) = \emptyset\}.
\]
By Lemmas \ref{lemm:short-time-exist-Sigma-eps} and \ref{lemm:Keps-smooth-start}, $T> -1+\delta$. Assume that $T<1$. 

Using Theorem \ref{theo:ilmanen-avoidance} (recall that, by \cite[10.5]{Ilmanen:elliptic} the support of a Brakke flow is a weak set flow), we find that 
\[
\supp \cM_\eps^1(T) \cap \cK_\eps^\lambda(T) \cap B_{5\lambda_0r_1} = \emptyset.
\]
Because $\cM_\eps^1$ and $\partial \cK_\eps$ agree with the smooth flow $\Sigma_\eps(t)$ outside of $B_{r_0}$ by Lemma \ref{lemm:conical-ends-smooth-flow}, there is $\eta>0$ so that\footnote{This is just the claim that two smooth flows that are initially disjoint remain disjoint for a short time; this holds for flows with boundary moving in any arbitrary manner. } 
\[
\supp \cM_\eps^1(t) \cap \cK_\eps^\lambda(t) \cap (B_{4\lambda_0r_1} \setminus B_{2\lambda_0r_1}) = \emptyset.
\]
for $t \in [-1,T+\eta]$. Now, observe that $\Sigma_\eps(t)$ and $\lambda \Sigma_\eps(\lambda^{-2}t)$ are smooth flows with the curvature estimates from Lemma \ref{lemm:conical-ends-smooth-flow} and so that the second is graphical over the first by Lemma \ref{lemm:rescaling-end-graphical-over-other-end} (with appropriate curvature estimates). Moreover, at $t=-1$, the two surfaces are disjoint, so the graphical function is initially positive.  

Now, the Ecker--Huisken maximum principle \cite{EckerHuisken:graphs}, specifically the version in Theorem \ref{theo:ecker-huisken} (which applies because the graphical function satisfies the PDE given in Lemma \ref{lemm:rescaling-end-graphical-over-other-end}), to conclude that the graphical function remains non-negative for $t \in [-1,T+\eta]$ (over the flow $\Sigma_\eps(t) \cap (\RR^{n+1}\setminus B_{3\lambda_0 r_1}))$. Now, the strong maximum principle implies that the graphical function is strictly positive in $\Sigma_\eps(t) \cap (\RR^{n+1}\setminus \bar{B}_{3\lambda_0 r_1})$ for $t\in[-1,T+\eta]$. Applying Theorem \ref{theo:ilmanen-avoidance} again, we conclude that 
\[
\supp \cM_\eps^1(t) \cap \cK_\eps^\lambda(t) = \emptyset.
\]
for $t \in [-1,T+\eta]$. This contradicts the choice of $T$.

Finally, we can repeat the same argument to show that the flows cannot make contact at $t=1$. This completes the proof. 
\end{proof}

\begin{corollary}\label{coro:partialKeps-lambda-disjoint}
For $\lambda \in (1,\lambda_{0})$, $\partial\cK_{\eps}^{1}\setminus\mathfrak{t}^{-1}(-1)$ is disjoint from $\cK_{\eps}^{\lambda}$. 
\end{corollary}
\begin{proof}
This follows from combining Proposition \ref{prop:brakke-eps-disjoint-Klambda} with Lemma \ref{lemm:K-eps-vs-M-eps-vs-K}. 
\end{proof}

Intuitively, this corollary proves that $\cK_\eps^\lambda$ lies inside of $\cK_\eps^1$ (since it has moved away from its boundary). We make this intuition precise below. Write $B^{\circ}$ for the interior of a set $B$ and $B^c$ for its complement. 
\begin{lemma}\label{lemm:bdry-cuts-set-in-two}
If $A,B$ are closed subsets of a topological space with $A$ connected and $\partial B \cap A = \emptyset$, then either $A \subset B^{\circ}$ or $A\cap B = \emptyset$. 
\end{lemma}
\begin{proof}
We have $A = (A \cap B^{\circ}) \cup (A \cap B^{c}) \cup (A \cap \partial B)$ for any sets $A,B$.
\end{proof}
\begin{lemma}\label{lemm:cK_epslambda-connected}
For each $\lambda \in (1,\lambda_{0})$, $\cK_{\eps}^{\lambda}\setminus\mathfrak{t}^{-1}(\{\pm1\})$ is connected. 
\end{lemma}

\begin{proof}
We will prove that $\cK_{\eps}\cap \mathfrak{t}^{-1}((-\lambda^{-2},\lambda^{-2}))$ is connected for any $\lambda \in (1,\lambda_{0})$. By Lemma \ref{lemm:conical-ends-smooth-flow}, we have that
\[
\cK_{\eps}\cap (\mathfrak{t}^{-1}((-\lambda^{-2},\lambda^{-2})) \setminus (\RR\times B_{r_{0}}))
\]
is the space-time track of the region above $\Sigma_{\eps}(t)$. Hence, if $\cK_{\eps}\cap \mathfrak{t}^{-1}((-\lambda^{-2},\lambda^{-2}))$ is disconnected, then there is a connected component 
\[
\mathscr{R} \subset B_{r_{0}}\times(-\lambda^{-2},\lambda^{-2}).
\] 
Note that $\mathscr{R}\cap \mathfrak{t}^{-1}((-\lambda^{-2},-\lambda_{0}^{-2})) = \emptyset$ by Lemma \ref{lemm:Keps-smooth-start}. 

The component $\mathscr{R}$ ``appears from nowhere,'' which easily leads to a contradiction. Indeed, we have shown that there is a point $(\bx,t) \in \mathscr{R}$ with minimal $t$-coordinate and because $\mathscr{R}$ is a closed connected component of $\cK_\eps$, there is $r>0$ so that $B_{2r}(\mathbf{x})\times \{t-r^2\}$ is disjoint from $\cK_\eps$. This contradicts the avoidance property of $\cK_\eps$. 
\end{proof}

\begin{corollary}\label{coro:Kepslambda-subset-Keps1}
For all $\lambda \in (1,\lambda_{0})$, $\cK_{\eps}^{\lambda}\setminus \mathfrak{t}^{-1}(\{\pm 1\}) \subset (\cK_{\eps}^{1})^{\circ}$. 
\end{corollary}
\begin{proof}
This follows by combining Corollary \ref{coro:partialKeps-lambda-disjoint} with Lemmas \ref{lemm:bdry-cuts-set-in-two} and \ref{lemm:cK_epslambda-connected}. 
\end{proof}

\begin{lemma}\label{lemm:orig-not-in-K-eps}
We have $(\bOh,0) \not \in \cK_{\eps}$ and for each $t \in [-1,0)$,
\[
\supp\cM_\eps(t) \subset \sqrt{-t} \Omega,
\]
where $\Omega$ is the open set lying above $\Sigma$. 
\end{lemma} 
\begin{proof}
This follows adapting of the argument \cite[Proposition 4.4]{BernsteinWang:TopologicalProperty} to the present setting (using Theorem \ref{theo:ilmanen-avoidance}); as we have already given similar arguments in the proof of Proposition \ref{prop:brakke-eps-disjoint-Klambda}, we omit the details. 
\end{proof}

We now rescale the flow as $\eps\to 0$ to obtain an ancient solution. We consider $\cF_\lambda(\cK_\eps)$ for $\eps$ small and $\lambda$ large (the precise relationship to be quantified in \eqref{eq:choice-lambda-i-constr} below) and consider this a weak set flow with initial condition $ \lambda \Sigma_\eps\times \{-\lambda^2\}$. 

\begin{lemma}\label{lemm:dist-to-origin-eps-lambda}
For $\eps>0$ fixed, the space-time distance satisfies
\[
\lim_{\lambda \to \infty} d((\bOh,0),\cF_\lambda(\cK_\eps)) = \infty.
\]
On the other hand, for $\lambda\geq 1$ fixed,
\[
\lim_{\eps \to 0} d((\bOh,0),\cF_\lambda(\cK_\eps)) = 0.
\]
\end{lemma}
\begin{proof}
The first claim follows immediately from Lemma \ref{lemm:orig-not-in-K-eps}. To prove the second claim, it suffices (by Lemma \ref{lemm:K-eps-vs-M-eps-vs-K}) to show that 
\[
\lim_{\eps\to0} d((\bOh,0),\supp\cM_\eps) = 0.
\]
Choose a subsequential limit $\tilde \cM$ of the flows $\cM_\eps$ as $\eps\to 0$. Note that $\tilde\cM(-1) = \cH^n\lfloor \Sigma$, since $\Sigma_\eps$ converges locally smoothly to $\Sigma$. Using unit-regularity and uniqueness of smooth mean curvature flows with bounded curvature (cf.\ \cite{ChenYin,EckerHuisken:graphs}) we conclude via Proposition \ref{prop:unique-cont-brakke} that
\[
\tilde \cM(t) = \cH^n\lfloor \sqrt{-t}\, \Sigma
\]
for $t<0$. This proves the claim. 
\end{proof}

Now, choose $\eps_i\to 0$. It is clear that $\lambda\mapsto d((\bOh,0),\cF_\lambda(\cK_{\eps_i}))$ is continuous. Thus, for $i$ sufficiently large, we can choose $\lambda_i$ so that 
\begin{equation}\label{eq:choice-lambda-i-constr}
d((\bOh,0),\cF_{\lambda_i}(\cK_{\eps_i})) = 1. 
\end{equation}
Taking a subsequential limit as $i\to\infty$, we find a weak set flow $\cK$ and Brakke flow $\cM$. Note that since $\eps_i\to 0$ we can ensure that $\lambda_i \to \infty$ and thus the flow $(\cM,\cK)$ is ancient. We summarize the basic properties of $(\cM,\cK)$ in the following theorem. 

\begin{theorem}\label{theo:basic-prop-cK-cM}
The flows $\cM$ and $\cK$ have the following properties: 
\begin{enumerate}
\item we have $d((\bOh,0), \cK) = 1$,
\item the Brakke flow $\cM$ has entropy $\lambda(\cM) \leq F(\Sigma)$,
\item we have $\partial\cK \subset \supp \cM \subset \cK$, 
\item for $\lambda > 1$ we have $\cF_\lambda(\cK)\subset \cK^{\circ}$ and $\supp\cM \cap \cF_\lambda(\cK) = \emptyset$,
\item there is $T>0$ large so that for $t<-T$, $\cM(t)$ and $\partial\cK(t)$ are the same smooth flow which we denote $\Sigma(t)$, 
\item the flow $\Sigma(t)$ lies in $\sqrt{-t}\Omega$ for all $t<-T$,
\item the flow $\Sigma(t)$ is strictly shrinker mean convex for all $t<-T$, 
\item $\tfrac{1}{\sqrt{-t}}\Sigma(t)$ converges smoothly on compact sets to $\Sigma$ as $t\to -\infty$, and
\item there is a continuous function $R(t)$ so that, for any $t \in \RR$,
\[ \cM(t) \lfloor (\RR^{n+1}\setminus B_{R(t)}) \text{ and } \partial \cK \cap (\mathfrak{t}^{-1}(t) \setminus B_{R(t)}) \]
are the same smooth, multiplicity-one, strictly shrinker mean convex  flow, which we will denote by $\Sigma(t)$; moreover, there is $C>0$ so that the curvature of $\Sigma$ satisfies $|\bx| |A_{\Sigma(t)}| \leq C$.
\end{enumerate}
\end{theorem}
\begin{proof}
Claim (1) follows by construction. Claim (2) follows from the fact\footnote{Note that the simpler statement $\lambda(\Sigma_\eps) \leq F(\Sigma) + o(1)$ as $\eps\to0$ would suffice here.} that $\lambda(\Sigma_\eps) \leq F(\Sigma)$ proven in \cite[Appendix C]{BernsteinWang:TopologicalProperty}. The claim (3) follows as in Lemma \ref{lemm:K-eps-vs-M-eps-vs-K}. We prove (4) below, but for now, we note that Corollary \ref{coro:Kepslambda-subset-Keps1} immediately implies that $\cF_{\lambda}(\cK) \subset \cK$ for $\lambda \in (1,\lambda_0)$. We will refer to this weaker property as (4'). 

We now turn to (5). Consider $\cM_{-\infty}$, any tangent flow to $\cM$ at $t=-\infty$. We know that $\cM_{-\infty}$ exists and is the shrinking Brakke flow associated to an $F$-stationary varifold $V_{-\infty}$ thanks to the monotonicity formula and the entropy bound $\lambda(\cM) \leq F(\Sigma)$. Lemma \ref{lemm:orig-not-in-K-eps} implies that $\supp V_{-\infty} \subset \overline\Omega$. By the Frankel property for self-shrinkers (cf.\ Theorem \ref{coro:Frankel}), it must hold that $\Sigma \cap \supp V_{-\infty} \not = \emptyset$. By the strong maximum principle for stationary varifolds \cite{SolomonWhite,Ilmanen:maximum} (either result applies here because $\Sigma$ is smooth), there must exist a component of $\supp V_{-\infty}$ which is equal to $\Sigma$. By the constancy theorem (and Frankel property again) we find that $V_{-\infty} = k \cH^n\lfloor \Sigma$, for some integer $k\geq 1$. By the entropy bound in (2), $k=1$. Thus, by Brakke's theorem (c.f.\ \cite{White:Brakke}) and Lemma \ref{lemm:entire-graph-from-cpt}, there is $T>0$ large so that $\cM(t)$ is the multiplicity one Brakke flow associated to a smooth flow $\Sigma(t)$ (and $\tfrac{1}{\sqrt{-t}}\Sigma(t)$ converges smoothly on compact sets to $\Sigma$ as $t\to-\infty$). Since $\partial\cK \subset \supp\cM$, we see that $\partial \cK(t) = \Sigma(t)$ as well. This completes the proof of (5); note that we have proven (8) as well.

By Lemma \ref{lemm:orig-not-in-K-eps}, $\Sigma(t) \subset \sqrt{-t}\, \bar{\Omega}$. Since $\sqrt{-t}\, \Sigma$ and $\Sigma(t)$ are both smooth (for $t<-T$), they cannot touch unless $\Sigma(t) = \sqrt{-t}\, \Sigma$ for all $t < -T$. This cannot happen by an argument along the lines of Lemma \ref{lemm:dist-to-origin-eps-lambda}. This proves (6).

Now, we note that (4') implies that $\Sigma(t)$ is weakly shrinker mean convex. By the strong maximum principle (see \cite[Proposition 4]{Smoczyk} for the evolution equation for the shrinker mean curvature), $\Sigma(t)$ is either a shrinker for all $t<-T$ or strictly shrinker mean convex. The first case cannot occur (by the argument used for (6)), proving (7). 

By Lemma \ref{lemm:entire-graph-from-cpt} proven below, we know that for $t$ sufficiently negative, $\frac{1}{\sqrt{-t}}\Sigma(t)$ is an entire graph over $\Sigma$ of a function with small $\Vert\cdot\Vert_{3}^{(1)}$ norm. From this, we can use pseudolocality to prove (9) exactly as in \cite[Proposition 4.4(1)]{BernsteinWang:TopologicalProperty} (the exterior flow $\cM(t)\lfloor (\RR^{n+1}\setminus B_{R(t)}) = \Sigma(t)$ is weakly shrinker mean convex by (4') and thus strictly so by the strong maximum principle). 

Finally, we prove (4). Strict shrinker mean convexity of the exterior flow guarantees that for $\lambda>1$, $\supp\cM$ and $\cF_{\lambda}(\cK)$ are disjoint outside of a set $D$ in space-time which has $D \cap \mathfrak{t}^{-1}([a,b])$ compact for any $a<b$. Thus, we may apply Ilmanen's localized avoidance principle, Theorem \ref{theo:ilmanen-avoidance}, to show that $\supp \cM$ and $\cF_{\lambda}(\cK)$ are indeed disjoint. Using (3) and (4'), this completes the proof of (4). 
\end{proof}

The following lemma was used above, and we will also use it again when proving uniqueness of ancient one-sided flows. 
\begin{lemma}\label{lemm:entire-graph-from-cpt}
Suppose that $(S(\tau))_{\tau \leq 0}$ is an an ancient rescaled mean curvature flow so that $S(\tau)$ converges to $\Sigma$ smoothly with multiplicity one on compact sets as $\tau\to-\infty$. Then, for $\tau$ sufficiently negative, there is a function $u(\cdot,\tau)$ on $\Sigma$ so that $S(\tau)$ is the normal graph of $u(\cdot,\tau)$ over $\Sigma$ and so that 
\[
\lim_{\tau \to -\infty} \Vert u(\cdot, \tau) \Vert_{3}^{(1)} = 0.
\]
\end{lemma}
\begin{proof}
This follows from an simplified version of the argument used in \cite[Lemma 9.1]{ChodoshSchulze}. Indeed, 
\[
\hat S_{\tau_0}(t) : = \sqrt{-t}\, S(\tau_0-\log(-t))
\]
is an ancient mean curvature flow for $t \leq -e^{\tau_0}$. Moreover, as $\tau_0\to-\infty$ the flows $(\hat S_{\tau_0}(t) )_{t  \leq -e^{\tau_0}}$ converge smoothly on any compact subset of $((-\infty,0]\times \RR^{n+1}) \setminus \{(0,0)\}$ to the shrinking flow $\{\sqrt{-t}\Sigma\}_{t \leq 0}$ (cf.\ the proof of Lemma \ref{lemm:orig-not-in-K-eps}). This implies that there is $\tau_{1}$ sufficiently negative so that for $\tau_{0}<\tau_{1}$ and $t \in [-1,-e^{\tau_0}]$,
\[
\hat S_{\tau_0}(t) \cap (B_{2r}\setminus B_r)
\]
is the graph of some smooth function $\hat u_{\tau_0}$ defined on a subset of $\sqrt{-t}\Sigma$ and that
\begin{equation}\label{eq:hat-u-tau-0-C3-estimate}
\sup_{t\in[-1,-e^{\tau_0}]} \Vert \hat u_{\tau_0}(\cdot,t)\Vert_{C^3} \to 0
\end{equation}
as $\tau_0\to -\infty$. Below, we will always assume that $\tau_{0}<\tau_{1}$.

We can rescale the above observation back to $S(\tau)$ to find that
\[
S(\tau) \cap \left\{ \bx : r e^{\frac{\tau-\tau_0}{2}} \leq |\bx| < 2 r e^{\frac{\tau-\tau_0}{2}} \right\}
\]
is the graph of some function $u(\cdot,\tau)$ defined on some $\Omega_{r,\tau_{0}}(\tau) \subset \Sigma$, as long as $\tau_{0}\in(-\infty,\tau]$ and $\tau \leq \tau_1$.  For such $\tau$, by varying $\tau_0 \in (-\infty,\tau]$, we find that $S(\tau) \setminus B_r$ is the graph of some function $u(\cdot, \tau)$ over the domain
\[
\Omega_{r}(\tau) : = \bigcup_{\tau_{0} \in (-\infty,\tau]} \Omega_{r,\tau_{0}}(\tau) \subset \Sigma.
\]
Shrinking $\tau_1$ if necessary, we can assume that 
\[
\Sigma \setminus B_{2r} \subset \Omega_{r}(\tau). 
\]
for $\tau < \tau_1$ by the smooth convergence of $S(\tau)$ to $\Sigma$ on compact sets. 

Finally, the $C^3$ estimate \eqref{eq:hat-u-tau-0-C3-estimate} rescales as follows. We have
\[
\Vert u(\cdot,\tau)\Vert_{3;\Omega_{r,\tau_{0}}(\tau)}^{(1)} \lesssim \Vert \hat u_{\tau_0}(\cdot,-e^{\tau_0-\tau}) \Vert_{C^3} 
\]
for $\tau_{0}\in (-\infty,\tau]$. In particular, we find 
\begin{align*}
\Vert u(\cdot,\tau)\Vert_{3;\Omega_{r}(\tau)}^{(1)} & = \sup_{\tau_{0}\in(-\infty,\tau]} \Vert u(\cdot,\tau)\Vert_{3;\Omega_{r,\tau_{0}}(\tau)}^{(1)}\\
& \lesssim \sup_{\tau_{0}\in(-\infty,\tau]} \Vert \hat u_{\tau_0}(\cdot,-e^{\tau_0-\tau}) \Vert_{C^3} \\
& \leq \sup_{\tau_{0}\in(-\infty,\tau]} \sup_{t\in[-1,-e^{\tau_{0}}]} \Vert \hat u_{\tau_0}(\cdot,t) \Vert_{C^3} .
\end{align*}
For $\tau$ sufficiently negative, $u(\cdot,\tau)$ extends across the compact part of $\Sigma$ with $C^{3}$-norm tending to $0$, so combined with the previous inequality and \eqref{eq:hat-u-tau-0-C3-estimate}, the result follows. 
\end{proof}


\section{Long-time regularity of the flow}\label{sec:long.time.reg}

In this section, we analyze further the flow $(\cM,\cK)$ from Theorem \ref{theo:basic-prop-cK-cM}. We must separate our analysis into three time scales, $t<0$, $t = 0$, $t>0$. 

\subsection{Regularity for $t<0$} Here, we show that White's regularity theory \cite{White:size,White:nature} for mean-convex flows applies to the flow $(\cM,\cK)$ for $t<0$.

\begin{remark}
Because it plays a fundamental role in our analysis, we briefly recall White's strategy. The basic setup is to prove that the Brakke flow and level-set flow are compatible in the sense that the Brakke flow is the $n$-dimensional Hausdorff measure restricted to the boundary of the level-set flow. This lets White combine Brakke flow (density/monotonicity) arguments with the fact that the level set flow is moving to one side. In particular, the one-sidedness implies that the level set flow is minimizing to the outside. The key to White's regularity is then to rule out multiplicity two planes as static/quasi-static tangent flows (higher multiplicity cannot occur by the minimizing property and other tangent flows are less common and thus less troublesome thanks to stratification). Using the fact that the flow is moving to one-side and a ``no holes'' argument, White then proves that such a tangent flow will locally separate into two sheets. Then, linear analysis can be used to rule out such a situation.
\end{remark}

We define the rescaled flow $\tilde\cK$ (and analogously for $\tilde \cM$) by 
\[
\tilde \cK := \bigcup_{\tau \in (-\infty,\infty) }e^{\frac \tau 2} \cK(-e^{-\tau})
\]
for $\tau \in (-\infty,\infty)$. It is easy to see that $\tilde \cK$ is still a closed subset of space-time. Indeed, it is the image of a closed set under the diffeomorphism 
\[
\cR: \RR^{n+1}\times (-\infty,0) \to \RR^{n+1}\times \RR, \qquad \cR : (\bx,t) \mapsto ((-t)^{-\frac 12} \bx, -\log(-t)).
\]

\begin{remark} The rescaled flows will be seen to be moving to one side (like a mean convex mean curvature flow). Because the rescaling is ``lower order'' and in particular disappears after taking a blow-up, White's theory will apply to the rescaled flows as well. Here, a serious issue will be that we do not \emph{a priori} know compatibility of the level set flow and the Brakke flow (we only know ``partial' compatibility, cf.\ (3) in Theorem \ref{theo:basic-prop-cK-cM}). As such, we will combine White's theory with a continuity argument to work up until the first time the theory breaks down (cf.\ \eqref{eq:white-reg-cont}). A crucial observation is that White can rule out static/quasi-static multiplicity tangent flows at some time $\bar \tau$ using knowledge of the flow \emph{only} for prior times $\tau< \bar \tau$ (of course, this is simply a manifestation of the parabolic nature of the flow). 
\end{remark}

Let 
\begin{equation}\label{eq:defn-time-translation}
\cT_{h}:\RR^{n+1}\times\RR \to \RR^{n+1}\times \RR, \qquad\cT_{h}:(\bx,t) \mapsto (\bx,t-h)
\end{equation} 
denote the time-translation map. 
\begin{lemma}\label{lemm:rescaled-avoidance-time-trans}
For $h > 0$, we have
\[ \cT_{h}(\tilde\cK) \subset \tilde\cK^{\circ} \text{ and } \supp\tilde\cM \cap \cT_{h}(\tilde\cK) = \emptyset. \]
\end{lemma}
\begin{proof}
Note that $\cT_{h}(\tilde \cK) = \cR(\cF_{e^{\frac h 2}}(\cK))$. Thus,  (4) in Theorem \ref{theo:basic-prop-cK-cM} implies both claims.  
\end{proof}

\begin{proposition}\label{prop:rescaled-weak-set-equals-brakke}
We have $\partial\tilde\cK = \supp\tilde\cM$. 
\end{proposition}
\begin{proof}
Suppose that $(\bx,\tau) \in \supp\tilde\cM\setminus \partial\tilde\cK \subset \tilde \cK^{\circ}$. Choose $r>0$ so that $B_{r}(\bx) \subset \tilde \cK(\tau)$. Lemma \ref{lemm:rescaled-avoidance-time-trans} implies that $B_{r}(x)$ is disjoint from $\supp \tilde \cM (\tau - h)$ for all $h>0$ small. For $h$ sufficiently small, the rescaled level set flow $\cB$ generated by $B_{r}(\bx) \times \{\tau -h\}$ has $(\bx,\tau) \in \cB^{\circ}$. On the other hand, $\supp \cM\cap\cB = \emptyset$ by the avoidance principle. In particular, $(\bx,\tau) \not \in \supp\cM$. This is a contradiction. 
\end{proof}

Proposition \ref{prop:rescaled-weak-set-equals-brakke} and \cite[10.5]{Ilmanen:elliptic} imply that $\partial\tilde \cK$ is a (rescaled) weak set flow.

\begin{corollary}\label{coro:reg-pt-rescaled-flows-agree}
If $(\bx,\tau_0) \in\reg\tilde\cM$ then there is $r>0$ so that 
\[
\tilde \cM(\tau) \lfloor B_r(\bx) = \cH^n \lfloor ( \partial \tilde\cK(\tau) \cap B_r(\bx))
\]
for $\tau \in (\tau_0-r^2,\tau_0+r^2)$, and $\tilde \cK^{\circ}\cap (B_r(\bx)\times (\tau_0-r^2,\tau_0+r^2)) \neq \emptyset$.
\end{corollary}
\begin{proof}
By definition, there is $r>0$ sufficiently small so that $\tilde \cM(\tau) \lfloor B_r(\bx) = \cH^n\lfloor \tilde M(\tau)$ for $\tilde M(\tau)$ a smooth rescaled mean curvature flow in $B_r(\bx)$. Thus, 
\[
\partial\tilde \cK \cap (B_r(\bx) \times (\tau_0-r^2,\tau_0+r^2)) = \supp\tilde\cM \cap (B_r(\bx) \times (\tau_0-r^2,\tau_0+r^2)) = \bigcup_{|\tau -\tau_0| < r^2} \tilde M(\tau) \times \{\tau\}.
\]
This proves the first statement. The second statement follows from Lemma \ref{lemm:rescaled-avoidance-time-trans} and Proposition \ref{prop:rescaled-weak-set-equals-brakke}. 
\end{proof}

\begin{corollary}
For $\tau_{0} \in \RR$, we have $(\partial\tilde \cK) \cap \{\tau = \tau_{0}\} = \partial (\tilde\cK \cap \{\tau = \tau_{0}\})$. 
\end{corollary}
As such, we can (and will) unambiguously write $\partial\tilde\cK(\tau_{0})$ for either of these sets.
\begin{proof}
It is clear that 
\[
(\partial \tilde \cK) \cap \{\tau = \tau_{0}\} \supset \partial (\tilde\cK \cap \{\tau = \tau_{0}\}).
\]
and that 
\[
(\partial \tilde\cK) \cap \{\tau = \tau_{0}\} \subset (\tilde\cK \cap \{\tau = \tau_{0}\}).
\]
Consider now 
\[
x \in (\partial \tilde\cK) \cap \{\tau = \tau_{0}\} \cap (\tilde\cK \cap \{\tau = \tau_{0}\})^{\circ}.
\]
Considering a small shrinking ball from a slightly earlier time, as in the proof of Proposition \ref{prop:rescaled-weak-set-equals-brakke}, we see that $(x,\tau) \in \tilde\cK^{\circ}$, a contradiction.
\end{proof}

\begin{lemma}
The sets $\{\partial\tilde \cK(\tau)\}_{\tau \in \RR}$ form a singular foliation of $\Omega$. 
\end{lemma}
\begin{proof}
Note that the sets $\{\partial\tilde \cK(\tau)\}_{\tau \in \RR}$ are disjoint by Lemma \ref{lemm:rescaled-avoidance-time-trans}. Now, note that 
\[
\lim_{h\to-\infty}\cT_{h}(\tilde\cK) = \overline\Omega, \qquad \lim_{h\to\infty} \cT_{h}(\tilde\cK) = \emptyset,
\]
by Theorem \ref{theo:basic-prop-cK-cM}. As such, for $x \in\Omega$, we can choose the maximal $T \in \RR$ so that $(x,T) \in \tilde \cK$. Assume that $(x,T) \in \tilde\cK^{\circ}$. By considering a small shrinking ball barrier as in the proof of Proposition \ref{prop:rescaled-weak-set-equals-brakke}, we can contradict the choice of $T$.
\end{proof}

Recall that the $F$-area of a measure $\mu$ (with $\mu(B_{r})\lesssim r^{k}$ for some $k>0$) is
\[
F(\mu) := (4\pi)^{-\frac n2} \int e^{-\frac{1}{4} |\bx|^{2}} d\mu(x).
\]
(cf. Section \ref{sec:prelim.shrinkers}.) Set also $F(A) : = F(\cH^{n}\lfloor A)$ when it is defined. We have the following proposition, which is a straightforward modification of the corresponding result in the mean-convex case.
\begin{proposition}[cf.\ {\cite[Theorems 3.5, 3.8, and 3.9]{White:size}}]\label{prop:F-area-min-rescaled}
Suppose that $V$ is a locally $F$-area minimizing hypersurface (integral current) contained in $\Omega$ with boundary in $\tilde \cK(\tau)$. Then $V\subset \tilde\cK(\tau)$. In particular, $\partial\tilde\cK(\tau)$ has locally finite $\cH^{n}$-measure and for any $B_{r}(\bx)\subset \Omega$,
\[
F(\partial\tilde \cK(\tau) \cap B_{r}(\bx)) \leq F(\partial B_{r}(\bx)).
\]
Finally, for $B_{r}(\bx) \subset \Omega$, if $S$ is a slab of thickness $2\eps r$ passing through $x$ and $\partial\tilde\cK(\tau) \cap B_{r}(\bx) \subset S$ then $\tilde\cK(\tau)\cap (B_{r}(\bx)\setminus S)$ consists of $k = 0,1$, or $2$ of the connected components of $B_{r}(x)\setminus S$ and
\[
F(\partial\cK(\tau) \cap B_{r}(\bx)) \leq (2-k+2n\eps + e(r)) \omega_{n}r^{n}
\]
where $e(r) = o(1)$ as $r\to0$.\footnote{We emphasize that this last statement does not hold uniformly for all $\bx$.}
\end{proposition}

At this point, we have no guarantee that the Brakke flow $\tilde \cM$ has $\tilde \cM(\tau) = \cH^{n}\lfloor \partial \tilde\cK(\tau)$ as in \cite[\S5]{White:size}. As such, we cannot immediately deduce regularity following \cite{White:size,White:nature}. Instead, we must use a continuity argument: consider the set in space-time 
\begin{equation}\label{eq:white-reg-cont}
\mathfrak{D} : = \{ X \in \RR^{n+1}\times \RR : \Theta_{\tilde \cM}(X) \geq 2\}.
\end{equation}
By upper semi-continuity of density, it is clear that $\mathfrak{D}$ is closed. Moreover, by (5) in Theorem \ref{theo:basic-prop-cK-cM}, it is clear that the projection of $\mathfrak{D}$ onto the $\tau$-axis is bounded from below, and the projection on $\RR^{n+1}$-factor is bounded. As such, if $\mathfrak{D}$ is non-empty, we can choose an element $\bar X = (\bar \bx, \bar \tau) \in \mathfrak{D}$ with smallest possible $\tau$-coordinate. 

\begin{lemma}[cf.\ {\cite[Theorem 5.5]{White:size}}]
If $(\tilde \cM_{i},\tilde\cK_{i})$ is a blow-up sequence limiting to $(\tilde\cM',\tilde\cK')$, around points $(x_{i},\tau_{i})$ with $\limsup_{i\to\infty}\tau_{i} < \bar \tau$ then $\supp\tilde\cM' = \partial\tilde \cK'$ and $\partial\tilde\cK'_{i}\to\partial\tilde\cK'$. 
\end{lemma}
\begin{proof}
As usual, we can show that $\partial\tilde\cK'\subset \supp\tilde\cM'\subset \tilde\cK'$. On the other hand, by \cite[\S9]{White:stratification}, almost every $X \in \supp \tilde\cM'$ has a tangent flow that is a static or quasi-static plane. By definition of $\bar \tau$, these must be static and multiplicity-one (by unit regularity). Thus, Corollary \ref{coro:reg-pt-rescaled-flows-agree} implies that there must be points in the complement of $\tilde\cK'$ that are arbitrarily close to $X$, since $(\tilde\cM_i,\tilde\cK_i)$ converges smoothly near $X$. 
This implies that a dense subset of $\supp\tilde\cM'$ is contained in $\partial \tilde\cK'$. This completes the proof. 
\end{proof}

\begin{lemma}[{\cite[Theorem 7.2]{White:size}}]\label{lemm:stat-quas-stat-cont-arg-rescaled-flow}
If $(\tilde\cM',\tilde\cK')$ is a static or quasi-static limit flow at $(x,\tau)$ with $\tau < \bar \tau$, then $\tilde M'$ is a stable minimal hypersurface whose singular set has Hausdorff dimension at most $n-7$. In particular, a non-flat static or quasi-static limit flow cannot exist when $n < 7$. 
\end{lemma}

From now on, we assume that $n<7$.
\begin{corollary}\label{coro:sing-set-rescaled-flow-small}
We have $(\sing \tilde \cM) \cap \{\tau < \bar \tau\}$ is of parabolic Hausdorff dimension $\leq n-1$. Moreover, for each $\tau < \bar \tau$, at time $\tau$, the singular set $\sing \tilde\cM(\tau)$ has spatial Hausdorff dimension at most $n-1$. 
\end{corollary}
\begin{proof}
This follows from Lemma \ref{lemm:stat-quas-stat-cont-arg-rescaled-flow} and \cite[\S9]{White:stratification}. See also \cite[Theorem 1.3]{White:size}.
\end{proof}

\begin{corollary}\label{coro:rescaled-Brakke-agrees-weakset}
For $\tau < \bar \tau$, $\tilde \cM(\tau) = \cH^{n}\lfloor \partial\tilde\cK(\tau)$. 
\end{corollary}
\begin{proof}
Corollary \ref{coro:sing-set-rescaled-flow-small} implies that $\cH^{n}(\supp\tilde\cM(\tau) \setminus \reg \tilde\cM(\tau)) = 0$. Because $\tilde\cM$ has bounded entropy, we have that $\tilde\cM(\tau)(B_{r}(\bx)) \lesssim r^{n}$ which implies that 
\[
\tilde\cM(\tau)(\supp\tilde\cM(\tau) \setminus \reg \tilde\cM(\tau)) = 0. 
\]
Combined with $\supp\tilde\cM(\tau) = \partial\tilde\cK(\tau)$, the assertion follows. 
\end{proof}

\begin{proposition}
The set $\mathfrak{D}$ is empty. Moreover, for any limit flow $(\cM',\cK')$, we have that $\supp \cM' =  \partial \cK'$ and there is $T \leq \infty$ so that 
\begin{enumerate}
\item $\cK'(t)$ is weakly convex for all $t$,
\item $\cK'(t)$ has interior points if and only if $t < T$, 
\item $\partial \cK'(t)$ are smooth for $t<T$,
\item $\cM'(t)$ is smooth and multiplicity one for $t<T$,
\item $\cK'(t)$ is empty for $t>T$. 
\end{enumerate}
If $(\cM',\cK')$ is a tangent flow, then it is a multiplicity one generalized cylinder $\SS^{n-k}\times \RR^{k}$. 
\end{proposition}
\begin{proof}
We first prove that $\mathfrak{D}$ is empty by arguing that we can apply the regularity theory of \cite{White:size,White:nature} at $\bar \tau$. Observe that Lemma \ref{lemm:rescaled-avoidance-time-trans}, Proposition \ref{prop:F-area-min-rescaled}, and Corollary \ref{coro:rescaled-Brakke-agrees-weakset} allow us to apply all of the arguments in \cite{White:size} that do not consider any points from $\{\tau \geq \bar \tau\}$ (after this time we do not know how to relate $\tilde\cM$ and $\tilde\cK$). 

Assuming $\mathfrak{D}\neq \emptyset$, we can fix $(\bar \bx,\bar\tau)  \in \mathfrak{D}$. Let $(\cM',\cK')$ denote a tangent flow pair to $(\tilde \cM,\tilde\cK)$. Arguing as in \cite[Theorem 5.5]{White:size} we find that $\cM'$ is compatible with the associated weak set flow $\cK'$ for times $t<0$ and the rescalings of $\partial\tilde\cK$ around $(\bar \bx,\bar \tau)$ converges to $\partial \cK'$ on $\{t<0\}$ as sets. In particular, this allows us to apply the arguments in \cite[\S 9]{White:size} to conclude that $(\cM',\cK')$ cannot be a multiplicity-two hyperplane for $t<0$ (either static or quasistatic).

\begin{remark}
Note that \cite[\S 9]{White:size} considers multiplicity-two quasistatic planes (and static planes were already ruled out in \cite[Corollary 8.5]{White:size}). We cannot appeal to \cite[Corollary 8.5]{White:size} in this setting, since the argument would need information about the flow for times $\tau > \bar \tau$. However, one may carefully check that the argument in \cite[\S 9]{White:size} makes no reference to any time $\tau > \bar \tau$ nor does it need \cite[Corollary 8.5]{White:size}, just the sheeting theorem \cite[Theorem 8.2]{White:size} which is applied to the blow-up sequence on compact subsets of $\RR^{n+1}\times \{t<0\}$. (In particular, we are making the observation that the argument in \cite[\S 9]{White:size} can be used to rule out static (or quasi-static) multiplicity-two planes while only considering times before the singular time.)
\end{remark}

Now that we have seen that $(\cM',\cK')$ cannot be a multiplicity-two hyperplane for $t<0$, we claim that $\sing\cM'\cap\{t<0\} = \emptyset$. If not, there is some $X \in \sing\cM'\cap\{t<0\}$. An iterated tangent flow $(\cM'',\cK'')$ at $X$ will be static (since $t<0$) and is a limit flow of $(\tilde \cM,\tilde \cK)$ (and will only see points at $<\bar\tau$), cf.\ \cite[Theorem 5.2(1)]{White:size}. Thus, we can repeat the argument in \cite[Theorem 12.3]{White:size} to show that $(\cM'',\cK'')$ cannot be a multiplicity-two plane. (Because $(\cM',\cK')$ is not a multiplicity-two plane, the rescaling chosen in the proof \cite[Theorem 12.3]{White:size} will still yield a tangent flow to $(\cM',\cK')$ and will thus not see any points with $\tau > \bar\tau$.) Iterating this argument and ruling out a non-trivial union of half-planes using Proposition \ref{prop:F-area-min-rescaled} as in \cite[Theorem 7.2]{White:size}, we can conclude (using the assumption $n<7$) that $(\cM'',\cK'')$ is a multiplicity-one hyperplane, contradicting $X\in\sing\cM'$. Since $(\cM',\cK')$ is regular for $t<0$, it must be a generalized cylinder by (mean) convexity, \cite[Theorem 10]{White:nature} (cf.\ \cite[Theorem 10.1]{ColdingMinicozzi:generic}). In particular, this implies that $\Theta_{\tilde\cM}(\bar\bx,\bar\tau) < 2$, contradicting the definition of $\mathfrak{D}$. 

Now that $\mathfrak{D} =\emptyset$, we can use Lemma \ref{lemm:rescaled-avoidance-time-trans}, Proposition \ref{prop:F-area-min-rescaled}, and Corollary \ref{coro:rescaled-Brakke-agrees-weakset} to see that White's regularity theory \cite{White:size,White:nature} applies to $(\tilde\cM,\tilde\cK)$ for all time. This completes the proof.
\end{proof}

We will say that $(\bx,t) \in \sing \cM$ has a \emph{mean convex neighborhood} if\footnote{This definition is slightly simpler than the one used in \cite{HershkovtisWhite,ChoiHaslhoferHershkovits} since all singularities in our setting have the ``same orientation'' (since the shrinker mean convexity rescales to mean convexity in the blow-up limit).} there's $\eps >0$ so that $t+\eps^2 < 0$ and if $t-\eps^2 < t_1 < t_2 < t+\eps^2$ then
\[
\cK(t_2) \cap B_\eps(\bx) \subset \cK(t_1) \cap B_\eps(\bx) \setminus \partial \cK(t_1).
\]
With this definition, we can now summarize the above conclusions for the non-rescaled flow.  
\begin{corollary}\label{coro:summary-tleq0-non-rescaled}
The non-rescaled flows $(\cM,\cK)$ have the following properties for $t<0$
\begin{enumerate}
\item $\cM(t) = \cH^{n}\lfloor \partial\cK(t)$,
\item $\sing\cM \cap\{t<0\}$ has parabolic Hausdorff dimension $\leq n-1$ and for $t<0$, $\sing\cM(t)$ has spatial Hausdorff dimension $\leq n-1$,
\item any limit flow at $X=(\bx,t)$ with $t<0$ is weakly convex on the regular part and all tangent flows are multiplicity one generalized cylinders, and
\item any singular point has a (strict) mean-convex neighborhood. 
\end{enumerate}
\end{corollary}
\begin{proof}
Everything but the last claim is proven above (in the rescaled setting). The last claim follows from the fact that all limit flows are convex so \cite{HershkovtisWhite} applies. 
\end{proof}

\subsection{Regularity at $t = 0$} We now turn to regularity near time $t=0$.

For $A,B\subset \RR^{n+1}\times\RR$, subsets of space-time, we write 
\[
d_{E}(A,B) = \inf_{(\mathbf{x}_{a},t_{a})\in A,(\mathbf{x}_{b},t_{b})\in B} \sqrt{|\mathbf{x}_{a}-\mathbf{x}_{b}|^{2} + (t_{a}-t_{b})^{2}}
\]
for the \emph{Euclidean} distance between the two sets. We emphasize that this differs from the usual \emph{parabolic} distance between the sets. Note that the parabolic dilation map $\cF_{\lambda} : \RR^{n+1}\times\RR\to\RR^{n+1}\times \RR$ generates the vector field
\[
V : = \frac{d}{d\lambda}\Big|_{\lambda=1}\cF_{\lambda} = (\mathbf{x},2t) \in T_{x} \RR^{n+1}\oplus T_{t} \RR.
\]
We now consider the geometry of hypersurfaces in space-time swept out by a mean-curvature flow.
\begin{lemma}\label{lemm:shrinker-mean-cuvature-vs-par-star-shaped}
Consider a family of smooth hypersurfaces $(a,b)\mapsto M(t) \subset \RR^{n+1}$ flowing by mean curvature flow. Set 
\[
\mathfrak{M} : = \bigcup_{t\in(a,b)} M(t)\times\{t\}. 
\]
Then, $\mathfrak{M}$ is a smooth hypersurface in spacetime $\RR^{n+1}\times \RR$ with unit normal\footnote{We emphasize that the unit normal is taken with respect to the Euclidean inner product on spacetime $\RR^{n+1}\times \RR\simeq \RR^{n+2}$.} at $(\mathbf{x},t)$ given by
\[
\nu_{\mathfrak{M}} = \frac{\nu_{M(t)} + H_{M(t)}(\mathbf{x}) \partial_{t}}{\sqrt{1+H_{M(t)}(\mathbf{x})^{2}}}.
\]
Moreover, the normal speed of $\lambda \mapsto \cF_{\lambda}(\mathfrak{M})$ at $\lambda =1$ is
\begin{equation}\label{eq:normal-speed-par-dilation}
\frac{2t H_{M(t)} + \mathbf{x}\cdot \nu_{M(t)}}{\sqrt{1+H_{M(t)}(\mathbf{x})^{2}}} 
\end{equation}
\end{lemma}
\begin{proof}
The given unit vector is orthogonal to 
\[
T_{(x,t)}\mathfrak{M}=T_{x}M(t) \oplus \Span_{\RR} (\partial_{t} +\mathbf{H}_{M(t)}(x)).
\]
This implies the expression for $\nu_{\mathfrak{M}}$. To prove \eqref{eq:normal-speed-par-dilation}, we may compute
\[
V\cdot\nu_{\mathfrak{M}} = \frac{(\mathbf{x}+2t\partial_{t})\cdot(\nu_{M(t)} + H_{M(t)}(x) \partial_{t})}{\sqrt{1+H_{M(t)}(x)^{2}}} = \frac{2t H_{M(t)} + \mathbf{x}\cdot\nu_{M(t)}}{\sqrt{1+H_{M(t)}(x)^{2}}} .
\]
This completes the proof. 
\end{proof}

Now, recall that by Theorem \ref{theo:basic-prop-cK-cM}, there is a smooth flow $\Sigma(t)$ so that $\partial\cK(t)$ and $\cM(t)$ agree with $\Sigma(t)$ outside of $B_{R(t)}$ and on $\RR^{n+1}\times (-\infty,-T)$. Choose $R_{0}$ sufficiently large so that $R_{0} \geq R(t)$ for $t \in [-4T,0]$ (we will take $R_{0}$ larger in \eqref{eq:defn-R0-t0-part} in Proposition \ref{prop:euc-space-time-sep-t=0} below). Then, define
\[
\mathfrak{S} : = \left( \bigcup_{-4T\leq t \leq-2T} (\Sigma(t)\cap \bar{B}_{3R_{0}}) \times \{t\} \right) \cup \left( \bigcup_{-2T\leq t \leq 1} \Sigma(t) \cap (\bar{B}_{3R_{0}}\setminus B_{R_{0}}) \times \{t\}\right).
\]

\begin{lemma}\label{lemm:smooth-part-euc-dist-par-dilate}
There is $c=c(R_{0},\Sigma),C(R_{0},\Sigma)> 0$ and $\lambda_{1}=\lambda_{1}(R_{0},\Sigma)>1$ so that 
\[
c(\lambda -1) \leq d_{E}(\mathfrak{S},\cF_{\lambda}(\mathfrak{S}))\leq C(\lambda-1)
\]
for $\lambda \in (1,\lambda_{1})$. 
\end{lemma}
\begin{proof}
It suffices to show that 
\[
\frac{d}{d\lambda}\Big|_{\lambda = 1} d_{E}(\mathfrak{S},\cF_{\lambda}(\mathfrak{S})) \in (0,\infty).
\]
This follows from \eqref{eq:normal-speed-par-dilation} (and the compactness of $\mathfrak{S}$) since positivity of the shrinker mean curvature of $\Sigma(t)$ was established as (7) and (9) in Theorem \ref{theo:basic-prop-cK-cM}. 
\end{proof}

\begin{proposition}\label{prop:euc-space-time-sep-t=0}
For $r>0$ sufficiently large, there is $c'=c'(r,\Sigma)>0$ and $\lambda_{1}'=\lambda_{1}'(r,\Sigma)$ so that 
\[
d_{E}(\partial\cK \cap (\bar{B}_{r}\times [-1,0]), \cF_{\lambda}(\partial\cK)\cap (\bar{B}_{r}\times [-1,0])) \geq c'(\lambda-1)
\]
for $\lambda \in (1,\lambda_{1}')$. 
\end{proposition}
\begin{proof}
Given $r>0$ large, we fix $R_{0}$ by requiring that
\begin{equation}\label{eq:defn-R0-t0-part}
4R_{0}^{2} + 6nT \geq 4 r^{2}
\end{equation}
and that $R_{0} \geq R(t)$ for $t \in [-3T,0]$ (where $R(t)$ is defined in Theorem \ref{theo:basic-prop-cK-cM}). This choice of $R_{0}$ will allow us to use Theorem \ref{theo:ilmanen-avoidance} below. We fix $c=c(R_{0},\Sigma)$ as in Lemma \ref{lemm:smooth-part-euc-dist-par-dilate} and will choose $c'\ll c$ below.

For $\lambda-1>0$ sufficiently small, assume that 
\begin{equation}\label{eq:dist-par-dil-small-contradiction}
d_{E}(\partial\cK \cap (\bar{B}_{r}\times [-1,0]) , \cF_{\lambda}(\partial\cK)\cap (\bar{B}_{r}\times [-1,0]) ) < \frac{c}{2}(\lambda-1)
\end{equation}
(otherwise the assertion follows) and that the distance is achieved at 
\[
(\bx,t) \in \partial\cK \cap (\bar{B}_{r}\times [-1,0]),\qquad (\bx+\bz,t+s) \in\cF_{\lambda}(\partial\cK)\cap (\bar{B}_{r}\times [-1,0]). 
\]
In particular, $|s| \leq \frac c2 (\lambda-1)$. 

Recalling the translation map $\cT_{s}$ defined in \eqref{eq:defn-time-translation}, observe that, Lemma \ref{lemm:smooth-part-euc-dist-par-dilate} and \eqref{eq:dist-par-dil-small-contradiction} imply that
\begin{align*}
d_{E}(\cT_{s}(\cF_{\lambda}(\mathfrak{S})), \mathfrak{S}) & \geq d_{E}(\cF_{\lambda}(\mathfrak{S}), \mathfrak{S}) - d_{E}(\cT_{s}(\cF_{\lambda}(\mathfrak{S})), \cF_{\lambda}(\mathfrak{S})) \\
& \geq d_{E}(\cF_{\lambda}(\mathfrak{S}), \mathfrak{S}) - |s|\\
& \geq \frac c 2(\lambda -1). 
\end{align*}
Consider the weak set flows $\cT_{s}(\cF_{\lambda}(\partial\cK))$ and $\partial\cK$. From the previous estimate and Theorem \ref{theo:ilmanen-avoidance} with $a = t_0= -3T$, $b=t$, $R = 2R_{0}$, $\bx_0 = \bOh$, and $\gamma$ small we see that $\cT_{s}(\cF_{\lambda}(\partial\cK))$ and $\partial\cK$ are disjoint for $t \in [-3T, 0]$. Recall that in Theorem \ref{theo:ilmanen-avoidance} the distance $d_{t}$ (see \eqref{eq:ilmanen-distance}) is defined with the choice $u=(R^2 - |\bx-\bx_0|^2 - 2n(t-t_0))_+$. \eqref{eq:defn-R0-t0-part} implies that $u$ is uniformly bounded from below away from zero on $\bar{B}_{r}\times [-3T,0]$, so Theorem \ref{theo:ilmanen-avoidance} allows further to conclude that (here and below, the implied constant in $\gtrsim,\lesssim$ depend on $r, R_0, \Sigma$ but not on $\lambda$ and $t$)
\[
|\bz| \gtrsim d_{t}(\cT_{s}(\cF_{\lambda}(\partial \cK)),\partial\cK) \geq d_{-3T}(\cT_{s}(\cF_{\lambda}(\partial \cK)),\partial\cK)\, .
\]
However, by Lemma \ref{lemm:smooth-part-euc-dist-par-dilate}, and since $|u| \leq 4R_0^2$, we see that
\[
d_{-3T}(\cT_{s}(\cF_{\lambda}(\partial \cK)),\partial\cK)\gtrsim c(\lambda-1) - |s|\, .
\]
Putting these inequalities together, we find that
\[
c(\lambda-1) \lesssim |\bz| + |s| = d_{E}(\partial\cK \cap (\bar{B}_{r}\times [-1,0]) , \cF_{\lambda}(\partial\cK)\cap (\bar{B}_{r}\times [-1,0]) ). 
\]
This completes the proof. 
\end{proof}

\begin{corollary}\label{coro:bound-shrinker-mean-curvature-below-tsim0}
For $r>0$, there is $s = s(r,\Sigma)>0$ with the following property. Choose $(\mathbf{x},t)\in\reg\cM \cap B_{r}\times [-1,0]$ and fix a space-time neighborhood $U$ of $(\bx,t)$ so that in $U$, $\cM$ agrees with $t\mapsto \cH^n\lfloor M(t)$,  for a smooth mean curvature flow $M(t)$. Then,
\[
2 t H_{M(t)}(\mathbf{x}) + {\mathbf{x}\cdot\nu_{M(t)}(\mathbf{x})} \geq s \sqrt{1+H_{M(t)}(\mathbf{x})^{2}}
\]
\end{corollary} 
\begin{proof}
Proposition \ref{prop:euc-space-time-sep-t=0} implies that the speed of $\lambda \mapsto \lambda M(\lambda^{-2}t)$ at $\lambda=1$ has a uniformly positive lower bound. Thus, the conclusion follows from \eqref{eq:normal-speed-par-dilation}. 
\end{proof}

\begin{corollary}\label{coro:Hbd-t-sim-0}
There is $C=C(\Sigma)>0$ and $\delta = \delta(\Sigma) \in (-1,0)$ so that $\partial\cK(t)$ is smooth with $|H_{\partial\cK(t)}| \leq C$ for $t\in(\delta,0)$. 
\end{corollary}
\begin{proof}
By (9) in Theorem \ref{theo:basic-prop-cK-cM} it suffices to prove this for points in $B_{r}$ for some $r>0$ sufficiently large. Fixing such an $r$, Corollary \ref{coro:bound-shrinker-mean-curvature-below-tsim0} implies that there is $s>0$ so that
\[
2 t H_{M(t)}(\mathbf{x}) + {\mathbf{x}\cdot\nu_{M(t)}(\mathbf{x})} \geq s \sqrt{1+H_{M(t)}(\mathbf{x})^{2}}
\]
for $(\mathbf{x},t) \in \reg\cM\cap B_{r}\times [-1,0]$. Solving for $H$, we find that $|H|\leq C$ on $\reg \cM\cap (B_{r}\times(-2\delta,0)$ for some $\delta\in(-1,0)$ sufficiently small. 

However, by (3) in Corollary \ref{coro:summary-tleq0-non-rescaled}, any $X \in \sing\cM \cap\{t<0\}$ has a multiplicity-one generalized cylinder as a tangent flow. In particular, there are points $X_{i}\in \reg \cM\cap\{t<0\}$ with $X_i \to X$ and $H(X_{i}) \to\infty$. This contradicts the mean curvature bound, completing the proof. 
\end{proof}

\begin{corollary}
We have that $\sing \cM(0) = \emptyset$, $\cM(0) = \cH^n\lfloor \partial\cK(0)$ and ${\mathbf{x}\cdot\nu_{\partial\cK(0)}} > 0$, 
\end{corollary} 
\begin{proof}
By Corollary \ref{coro:Hbd-t-sim-0}, we know that for $t \in (\delta,0)$ and $\mathbf{x} \in \partial\cK(t)$, $|H_{\partial\cK(t)}(\mathbf{x})| \leq C$. Thus, by Corollary \ref{coro:bound-shrinker-mean-curvature-below-tsim0}, we conclude that for $r$ chosen as in the proof of Corollary \ref{coro:Hbd-t-sim-0}, taking $\delta$ smaller if necessary, for $t \in (\delta,0)$ we find that $\partial\cK(t)$ is strictly star-shaped in $B_r$, i.e., there is $c>0$ so that
\[
{\mathbf{x}\cdot\nu_{\partial\cK(t)}} \geq c 
\]
for $\mathbf{x}\in \partial\cK(t)\cap B_r$. In particular, this implies that $\partial\cK(t)$ is locally uniformly graphical. Interior estimates \cite[Theorem 3.1]{EckerHuisken:interior} then imply that the flow $\partial\cK(t)$ remains smooth and strictly star-shaped up to $t=0$ (outside of $B_r$, the flow is automatically smooth and strictly star-shaped by (7) and (9) in Theorem \ref{theo:basic-prop-cK-cM}). 
\end{proof}

\subsection{Regularity for $t>0$} \label{sec:pos-time}

Using $\sing\cM(0) =\emptyset$ and (9) from Theorem \ref{theo:basic-prop-cK-cM}, there is some $\hat \delta>0$ so that $\cM(t)=\cH^n\lfloor \partial\cK(t)$ is smooth for $t\in[0,\hat\delta)$. We can now consider the rescaled flow 
\[
\hat \cK := \bigcup_{\tau \in (-\infty,\infty)} e^{-\frac \tau 2} \cK(e^\tau),
\]
and $\hat \cM$ similarly defined, exactly as in the $t<0$ situation. The only difference is that the flow is moving outwards rather than inwards:
\[
\cT_h(\hat \cK) \subset \hat \cK^\circ
\]
for $h < 0$ (cf.\ Lemma \ref{lemm:rescaled-avoidance-time-trans}).  This does not seriously affect the arguments used above, and we find that Corollary \ref{coro:summary-tleq0-non-rescaled} holds for $t>0$ as well. 

\subsection{Long time asymptotics} We continue to use our notation from the $t > 0$ regularity section. Moreover, we denote with $C_\Sigma$ the asymptotic cone of the asymptotically conical shrinker. We will also need to consider the integral unit-regular Brakke flows $t \in [0,\infty) \mapsto \mu^\pm(t)$ constructed in Theorem \ref{thm:app.2} whose support agrees with the inner and outer flow $M^\pm(t)$ of $C_\Sigma$. They can be used to prove:

\begin{lemma}\label{lem:disjoint_levelsetflow} For all $t\geq 0$,  $\partial \cK(t)$ is disjoint from the level set flow of $C_\Sigma$.
\end{lemma}

\begin{proof} 
	Note that $(\mu^\pm(t))_{t\geq 0}$ is smooth with unit multiplicity outside of $B_{\sqrt{t}R_0}(0)$ for some $R_0 >0$. Moreover, $\mu^\pm(0)$ is disjoint from $C$ at $t=0$. Thus, we can argue as in Proposition \ref{prop:brakke-eps-disjoint-Klambda}: we may couple the  Ecker--Huisken Maximum Principle (Theorem \ref{theo:ecker-huisken}), with Ilmanen's localized avoidance principle (Theorem \ref{theo:ilmanen-avoidance}) to show that  $\partial \cK(t)$ is disjoint from  $M^\pm(t)$ for all $t \geq 0$. This implies the claim. 
\end{proof}

This allows to characterize the convergence of the rescaled flow for $\tau \rightarrow \infty$. We assume that $M(t)$ lies outside the outer flow $M^+(t)$ of the level set flow of $C_\Sigma$. 

\begin{theorem} \label{theo:convergence-pos-times} The rescaled flow $\hat{\cM}(\tau)$ converges smoothly as $\tau \rightarrow \infty$ to an expander $E$, which is smoothly asymptotic to $C_\Sigma$ and  minimizes the expander functional 
\begin{equation}\label{eq:expander-energy}
\cE(S) = \int_S e^{\frac{1}{4} |\mathbf{x}|^2} \, d\cH^n
\end{equation}
from the outside (relative to compact perturbations) and is thus smooth. 
Furthermore, 
$$M^+(t) = \sqrt{t}\, E$$
for $t>0$.
\end{theorem}

\begin{proof} Since $\tau \in (0,\infty) \mapsto \hat{\cM}(\tau)$ is expander mean convex, and is smooth with uniform control on all derivatives outside of $B_{R_0}(0)$, it follows from the arguments in \cite[\S11]{White:size}, that  $\hat{\cM}(\tau)$ converges smoothly to an outward minimizing minimal surface $E$ in the expander metric $g = e^{\frac{1}{2n} |\mathbf{x}|^2} g_{\RR^{n+1}}$. This yields the claimed regularity and the smoothness of the convergence. Note that any blow down of the flow $t \in [0,\infty) \mapsto \cM(t)$ lies inside the level set flow of $C_\Sigma$, so $E$ has to be smoothly asymptotic to $C_\Sigma$. By Lemma \ref{lem:disjoint_levelsetflow} the flow $t \mapsto \sqrt{t} E$ has to agree with the outer flow of $C_\Sigma$. 
\end{proof}

\subsection{The outermost flows of general hypercones} We consider, for $n < 7$, a general embedded, smooth hypersurface $\Gamma \subset \SS^n$ and the regular hypercone $C(\Gamma) \subset \RR^{n+1}$. We show in this subsection that the previous arguments can be generalized to characterize the outer and inner flows of the level set flow of $C(\Gamma)$ as in Theorem \ref{theo:convergence-pos-times}.

Note that $\Gamma$ divides $\SS^n$ into two open sets $S^\pm$. We can construct smooth hypersurfaces $M^\pm$ which are smooth radial graphs over $S^\pm$, smoothly asymptotic to $C(\Gamma)$ with sufficiently fast decay such that $\mathbf{x} \cdot \nu_{M^\pm} $ (with $\nu_{M^\pm}$ the upwards unit normal) decays to zero at infinity along $M^\pm$. Let $(M^\pm(t))_{t \in [0,T^\pm)}$ be the maximal smooth evolution of $M^\pm$. Note that by the maximum principle of Ecker--Huisken \cite{EckerHuisken:graphs} together with the strong maximum principle  we have that 
\[
2t H_{M^\pm(t)} +  \mathbf{x} \cdot \nu_{M^\pm(t)}  > 0
\]
along $(M^\pm(t))_{t \in (0,T^\pm)}$. We can thus repeat the arguments in Section \ref{sec:pos-time} to construct expander mean convex flows $(\cM^\pm(t))_{t>0}$ such that the corresponding rescaled flows converge to expanders, smoothly asymptotic to $C(\Gamma)$. This implies 

\begin{theorem} \label{theo:outermost-flows} The outermost flows of $C(\Gamma)$ are given by expanding solutions  $t \in (0, \infty) \mapsto \sqrt{t}  E^\pm$ smoothly asymptotic to $C_\Sigma$. The expanders $E^\pm$ minimize the expander energy \eqref{eq:expander-energy} from the outside (relative to compact perturbations) and are smooth.
\end{theorem}

See also the notes of Ilmanen \cite{Ilmanen:Trieste} for the proof of smoothness in case $n=2$. Furthermore by an argument of Ilmanen--White \cite{Ilmanen:Trieste} any such outermost expander has genus zero.


\section{Uniqueness and regularity of one-sided ancient Brakke flows} \label{sec:exist.unique.ancient.Brakke}

We now combine the three regimes considered above with Theorem \ref{theo:basic-prop-cK-cM} to conclude the following existence and regularity for the flow $(\cM,\cK)$. 

\begin{theorem}[One-sided existence]\label{theo:one.sided.construction}
For $n\leq 6$ and $\Sigma^{n}$ a smooth asymptotically conical self-shrinker, choose $\Omega$ a fixed component of $\RR^{n+1}\setminus \Sigma$. Then, there exists an ancient unit-regular integral Brakke flow $\cM$ and weak set flow $\cK$ with the following properties: 
\begin{enumerate}
\item $\cM(t) = \cH^{n}\lfloor \partial\cK(t)$,
\item $\partial\cK(t) \subset \sqrt{-t}\Omega$ for all $t<0$,
\item there is $T>0$ so that for $t<-T$, $\cM(t)$ is a smooth multiplicity one flow $\Sigma(t)$ with $\Sigma(t)$ is strictly shrinker mean convex,
\item $\frac{1}{\sqrt{-t}}\Sigma(t)$ converges smoothly on compact sets to $\Sigma$ as $t\to-\infty$,
\item there is a continuous function $R(t)$ so that for any $t \in \RR$, $\cM(t) \lfloor (\RR^{n+1}\setminus B_{R(t)})$ is a smooth strictly shrinker mean convex multiplicity one flow $\Sigma(t)$,
\item the Brakke flow $\cM$ has entropy $\lambda(\cM) \leq F(\Sigma)$,
\item $\sing\cM$ has parabolic Hausdorff dimension $\leq n-1$ and for any $t \in \RR$, $\sing\cM(t)$ has spatial Hausdorff dimension $\leq n-1$
\item any limit flow is weakly convex on the regular part and all tangent flows are multiplicity one generalized cylinders,
\item any singular point has a strictly mean-convex neighborhood, 
\item there is $\delta>0$ so that $\partial\cK(t)$ is completely smooth for $t \in (-\delta,\delta)$ and $\partial\cK(0)$ is strictly star-shaped, and
\item $\frac{1}{\sqrt{t}} \partial\cK(t)$ converges smoothly on compact sets to an outermost expander coming out of the cone at infinity of $\Sigma$, as $t\to\infty$.
\end{enumerate}
\end{theorem}

Now, we will combine Theorem \ref{theo:one.sided.construction} with Corollary \ref{coro:one.sided.decay.uniqueness} to prove uniqueness of the flow constructed above. 

\begin{theorem}[One-sided uniqueness] \label{theo:one.sided.uniqueness} For $n\leq 6$, fix $\Sigma^n$ a smooth asymptotically conical self-shrinker as in Theorem \ref{theo:one.sided.construction}. Let $(\mu_t)_{-\infty < t < \infty}$ be a unit-regular integral Brakke flow such that
	\begin{equation} \label{eq:one.sided.assumption}
		\supp \mu_t \text{ is strictly on one side of } \sqrt{-t} \Sigma \text{ for every } t \in (-\infty, 0),
	\end{equation}
	and
	\begin{equation} \label{eq:one.sided.density.assumption}
		0 < \Theta((\mu_t), -\infty) < 2 \lambda(\Sigma).
	\end{equation}
	After a time translation, $\mu_t$ coincides with the Brakke flow from Theorem \ref{theo:one.sided.construction}. 
\end{theorem}
\begin{proof}
As in the proof of (5) in Theorem \ref{theo:basic-prop-cK-cM}, the Gaussian density bound guarantees that the tangent flow to $\mu_{t}$ at $-\infty$ is the multiplicity one shrinker associated to $\Sigma$. As such, Lemma \ref{lemm:entire-graph-from-cpt} and Corollary \ref{coro:one.sided.decay.uniqueness} imply that after a time-translation there is $T>0$ so that for $t \leq -T$, $\mu_{t} = \cH^{n}\lfloor \Sigma(t)$, where $\Sigma(t)$ is the smooth flow from Theorem \ref{theo:one.sided.construction} (4). 

As in Proposition \ref{prop:brakke-eps-disjoint-Klambda}, Ilmanen's localized avoidance principle (Theorem \ref{theo:ilmanen-avoidance}) combined with Ecker--Huisken's maximum principle at infinity (Theorem \ref{theo:ecker-huisken}), we see that $\supp \mu_{t}$ is disjoint from $\cF_{\lambda}(\supp \cM)$ for $\lambda \not = 1$. This implies that $\supp \mu_{t} \subset \supp\cM$. 

Finally, since $\reg\cM$ is connected by (9) in Theorem \ref{theo:one.sided.construction} and Corollary \ref{cor:connected-reg-part}, we see that $\mu_{t} = \cM(t)$ in $(\sing \cM)^{c}$ (using the unit-regularity of $\cM$ and $\mu_{t}$). This completes the proof. 
\end{proof}

\begin{remark}
Both Theorems \ref{theo:one.sided.construction} and \ref{theo:one.sided.uniqueness} clearly hold (with simpler proofs) in the case that $\Sigma$ is a smooth compact shrinker. 
\end{remark}

\begin{remark}
We expect that the dimensional restriction in Theorems \ref{theo:one.sided.construction} and \ref{theo:one.sided.uniqueness} can be removed (cf.\ \cite{White:subsequent,HaslhoferHershkovits,EHIJ:free-boundary}). We note that when $\Sigma$ has sufficiently small $F$-area, Theorems \ref{theo:one.sided.construction} and \ref{theo:one.sided.uniqueness} hold in all dimensions. See \S \ref{sec:generic-schoenflies} for a precise statement. 
\end{remark}


\section{Generic mean curvature flow of low entropy hypersurfaces} \label{sec:generic-schoenflies}
We recall the following notions from \cite{BernsteinWang:topology-small-ent}. Denote by $\cS_{n}$ the set of smooth self-shrinkers in $\RR^{n+1}$ and $\cS_{n}^{*}$ the non-flat elements. Let
\[
\cS_{n}(\Lambda) : = \{ \Sigma \in \cS_{n} : \lambda(\Sigma) < \Lambda\}
\]
and similarly for $\cS_{n}^{*}(\Lambda)$. Let $\cR\cM\cC_{n}$ denote the set of regular minimal cones in $\RR^{n+1}$ and define $\cR\cM\cC_{n}^{*},\cR\cM\cC_{n}(\Lambda),\cR\cM\cC_{n}^{*}(\Lambda)$ analogously. We now recall the following two ``low-entropy'' conditions from \cite{BernsteinWang:topology-small-ent}:
\[
\tag{$\star_{n,\Lambda}$}  \cR\cM\cC_{k}^{*}(\Lambda) = \emptyset \quad \textrm{for all} \quad 3\leq k\leq n
\]
and 
\[
\tag{$\star\star_{n,\Lambda}$}  \cS_{n-1}^{*}(\Lambda) = \emptyset. 
\]
It's convenient to set $\lambda_{k}=\lambda(\RR^{n-k}\times \SS^{k})$.

Given these definitions, we can state the following result. 
\begin{theorem}\label{theo:low-ent-generic-flow}
 For $n\geq 3$ and $\Lambda \in (\lambda_{n},\lambda_{n-1}]$, assume that $(\star_{n,\Lambda})$ and $(\star\star_{n,\Lambda})$ hold. Then if $M\subset \RR^{n+1}$ is a closed hypersurface with $\lambda(M) \leq \Lambda$ there exist arbitrarily small $C^{\infty}$ graphs $M'$ over $M$ and corresponding unit-regular integral Brakke flows $\cM'$ with $\cM'(0) = \cH^n\lfloor M'$, so that $\cM'$ is completely regular until it disappears in a round point. That is, there is $X \in \RR^{n+1}\times \RR$ so that $\sing \cM' = \{X\}$ and so that any tangent flow at $X$ is a round shrinking $\SS^{n}$. 
\end{theorem}

We will prove this below. Note that $(\star_{3,\lambda_{2}})$ holds by \cite[Theorem B]{MarquesNeves} and $(\star\star_{3,\lambda_{2}})$ holds by \cite[Corollary 1.2]{BernsteinWang:TopologicalProperty}, so Theorem \ref{theo:low-ent-generic-flow} implies Theorem \ref{theo:low-ent-generic-flow-4D}.

We also note that Theorems \ref{theo:one.sided.construction} and \ref{theo:one.sided.uniqueness} hold in all dimensions with the assumption that $(\star_{n,\Lambda})$ holds and $F(\Sigma) \leq \Lambda$. Indeed, the dimension restriction in Theorems \ref{theo:one.sided.construction} and \ref{theo:one.sided.uniqueness} arises due to the use of \cite{White:nature}, where it is used to rule out static cones as limit flows to a mean-convex flow (cf.\ \cite[Theorem 4]{White:nature}). However, in the low-entropy setting static cones cannot occur as limit flows, by assumption $(\star\star_{n,\Lambda})$  (cf.\ \cite[Lemma 3.1]{BernsteinWang:topology-small-ent}) even without assuming mean-convexity.

\begin{lemma}[{\cite[{Proposition 3.3}]{BernsteinWang:topology-small-ent}}]\label{lemm:BW-small-ent-char-tangent-flow}
 For $n\geq 3$ and $\Lambda \in (\lambda_{n},\lambda_{n-1}]$, assume that $(\star_{n,\Lambda})$ and $(\star\star_{n,\Lambda})$ hold. If $\cM$ is a unit-regular integral Brakke flow with $\lambda(\cM) \leq \Lambda$ then any tangent flow to $\cM$ is the multiplicity one shrinker associated to a smooth shrinker that is either (i) compact and diffeomorphic to $\SS^{n}$ or (ii) smoothly asymptotically conical. 
\end{lemma}

\begin{lemma}[{\cite[Proposition 3.5]{BernsteinWang:topology-small-ent}}]\label{lemm:BW-small-ent-compact-AC}
Fix $n\geq 3$ and $\Lambda \leq \lambda_{n-1}$ and $\eps>0$. Assume that $(\star_{n,\Lambda})$ and $(\star\star_{n,\Lambda})$ hold. Then, the space of compact or non-flat smoothly asymptotically conical shrinkers $\Sigma\subset\RR^{n+1}$ with entropy $\lambda(\Sigma) \leq \Lambda-\eps$ is compact\footnote{We emphasize that because $\Lambda < 2$, any limit of such shrinkers has multiplicity one. Note that the proof of \cite[Proposition 3.5]{BernsteinWang:topology-small-ent} directly allows to include compact shrinkers with $\lambda(\Sigma) \leq \Lambda-\eps$. Furthermore, \cite[Proposition 3.7]{BernsteinWang:topology-small-ent} gives an extrinsic diameter bound for such compact shrinkers, so a sequence of these cannot converge to a non-compact shrinker.} in $C^{\infty}_{\textnormal{loc}}$.
\end{lemma}

From now on, we fix $n\geq 3$, $\Lambda \in (\lambda_{n},\lambda_{n-1}]$ satisfying $(\star_{n,\Lambda})$ and $(\star\star_{n,\Lambda})$.

\begin{lemma}\label{lemm:schoenflies-entropy-drop}
There is $\delta = \delta(n,\Lambda,\eps) > 0$ so that if $ \cM$ is a unit-regular integral Brakke flow with $ \cM(t) = \cH^{n}\lfloor \sqrt{-t}\Sigma$ for $t<0$, where $\Sigma$ is a compact or non-flat smooth asymptotically conical shrinker with $F(\Sigma) \leq \Lambda-\eps$, then for any $X \in \RR^{n+1}\times\RR$ with $|X| = 1$, we have that $\Theta_{\cM}(X) \leq F(\Sigma) - \delta$. 
\end{lemma}
\begin{proof} Note that for compact shrinkers one has $\Theta_{\cM}(X) \leq 1$ for all $X \in \RR^{n+1}\times\RR$ with $|X| = 1$ so White's local regularity theorem \cite{White:Brakke} yields the statement.

For the noncompact case, assume there is $\cM_{j}$ (and the associated smooth asymptotically conical shrinkers $\Sigma_{j}$) and $X_{j}$ with $|X_{j}| =1$ so that 
\[
\Theta_{\cM_{j}}(X_{j}) \geq F(\Sigma_{j}) - \frac 1 j.
\]
Up to a subsequence, we can use Lemma \ref{lemm:BW-small-ent-compact-AC} to find a Brakke flow $\cM_\infty$ so that $\cM_\infty = \cH^{n}\lfloor \sqrt{-t}\Sigma_\infty$ for $t<0$, where $\Sigma_\infty$ is a non-flat smooth asymptotically conical shrinker, and $ X$ with $| X| = 1$ and $\Theta_{\cM_\infty}( X) \geq F(\Sigma_\infty)$. Parabolic cone-splitting (cf.\ \cite{White:stratification} and \cite[p.\ 840-1]{CheegerHaslhoferNaber}) implies that either $\Sigma_\infty$ splits off a line or it is static or quasi-static. This is a contradiction, completing the proof. 
\end{proof}

\begin{lemma}\label{lemm:stab-spherical-sing}
For integral unit regular Brakke flows $\cM_i,\cM$, suppose that $X_i\in\sing\cM_i$ has $\cM_i\rightharpoonup\cM$ and $X_i\to X \in \sing \cM$. Suppose that some tangent flow to $\cM$ at $X$ is a round shrinking sphere with multiplicity one, $t\mapsto \cH^n\lfloor \SS^{n}\sqrt{-2nt}$. Then, for $i$ sufficiently large, any tangent flow to $\cM_i$ at  $X_i$ is a round shrinking sphere. 
\end{lemma}
\begin{proof}
Assume that $X=(\bOh,0)$. For any $r>0$, there is $\eta>0$, so that 
\[
\cM \lfloor (B_{2r\sqrt{\eta}}(\bOh) \times (-4\eta,-\eta))
\]
is a smooth, strictly convex mean curvature flow (without spatial boundary). Thus, for $i$ sufficiently large, 
\[
\cM_i \lfloor (B_{r\sqrt{\eta}}(\bOh) \times (-3\eta,-2\eta))
\]
is a smooth, strictly convex mean curvature flow. Taking $r$ sufficiently large, this completes the proof (using e.g., \cite{Huisken:convex}). 
\end{proof}

For any integral unit regular Brakke flow $\cM$ with $\lambda(\cM) < \lambda_{n-1} = \lambda(\SS^{n-1}\times \RR)$, denote 
\[
\sing_{\textrm{gen}}\cM 
\]
for the set of singular points for which all tangent flows given by multiplicity-one round shrinking spheres. The previous lemma proves stability of these sets. 

Assume for now that $M^n\subset \RR^{n+1}$ has $\lambda(M) < \Lambda$,  and consider $\varphi \in C^\infty(M), \varphi >0$. Fix $s_0$ small enough so that for $s \in (-s_0,s_0)$, the graph of $s\varphi$, denoted $M_s$, has $\lambda(M_s) < \Lambda-\eps$, for $\eps>0$ fixed. For any $s \in (-s_0,s_0)$, let $\fF(s)$ denote the set of integral unit regular Brakke flows $\cM$ with $\cM(0) = \cH^n\lfloor M_s$. Note that $\fF(s) \not = \emptyset$ (e.g., choose $s_i \to s$ with the level set flow of $M_{s_i}$ non-fattening and pass Brakke flows starting from $M_{s_i}$---these exist by \cite[Theorem 11.4]{Ilmanen:elliptic}---to the limit).

For $s \in (-s_0,s_0)$, set 
\[
\cD(s) : = \sup\{ \Theta_\cM(X) : \cM \in \fF(s), X \in \sing\cM \setminus \sing_\textrm{gen}\cM \}.
\]
We recall that $\sup \emptyset = -\infty$ and note that by compactness of integral unit-regular Brakke flows and upper-semicontinuity of density, $\cD(s)$ is always attained. 

\begin{proposition}\label{prop:schoenflies-blow-up-argument}
We continue to assume that $\lambda(M) \leq \Lambda-\eps$. For $s_0$ as above, suppose that $s_i \nearrow s \in (-s_0,s_0)$. Then 
\[
\limsup_{i\to\infty} \cD(s_i) \leq \cD(s) - \delta,
\]
where $\delta=\delta(n,\Lambda,\eps)>0$ is fixed in Lemma \ref{lemm:schoenflies-entropy-drop}. 
\end{proposition}
\begin{proof}
It suffices to assume that $\limsup_{i\to\infty} \cD(s_i) > - \infty$. Choose integral unit-regular Brakke flows $\cM_i\in\cF(s_i)$ and space-time points $X_i \in \sing \cM_i \setminus \sing_\textrm{gen} \cM_i$ with
\[
\Theta_{\cM_i}(X_i) = \cD(s_i).
\]
Passing to a subsequence, we can assume that $\cM_i\rightharpoonup\cM \in \fF(s)$ and
\[
X_i \to X \in \sing \cM \setminus \sing_\textrm{gen} \cM.
\]
The fact that $X \not \in \sing_\textrm{gen}\cM$ follows from Lemma \ref{lemm:stab-spherical-sing}. We now rescale around $X$ so that we can apply Theorem \ref{theo:one.sided.uniqueness}. Note that $\supp\cM_i,\supp\cM$ are all pairwise disjoint, since their initial conditions are compact pairwise disjoint hypersurfaces.

We will repeatedly pass to subsequences without relabeling in the following. Rescale $\cM_i$ around $X$ by $|X_i-X| \not = 0$ to $\hat\cM_i$ and assume that $\hat \cM_i\rightharpoonup \hat\cM$. Similarly, rescale $\cM$ around $X$ by $|X_i-X| \not = 0$ to $\widetilde\cM_i$ and assume that $\widetilde\cM_i\rightharpoonup \widetilde\cM$. Since $\widetilde\cM$ is a tangent flow to $\cM$ at $X \not \in \sing_\textrm{gen}\cM$, by Lemma \ref{lemm:BW-small-ent-char-tangent-flow}, there is a smooth shrinker $\Sigma^n\subset\RR^{n+1}$ that is either compact or asymptotically conical so that $\widetilde\cM(t) = \sqrt{-t} \Sigma$ for $t<0$. Finally, assume that after rescaling $X_i$ around $X$ by $|X_i-X|$ to $\hat X_i$, $\hat X_i\to \hat X$. 

We claim that $\lambda(\hat\cM) \leq \Theta_\cM(X) = F(\Sigma)$. Indeed, choose $\overline X_i\to X$ and $r_i\to 0$ so that
\[
\Theta_{\cM_i}(\overline X_i,r_i) = \lambda(\hat\cM) + o(1). 
\]
On the other hand, 
\[
\Theta_\cM(X,r) = \Theta_\cM(X) + o(1)
\]
as $r\to0$. Hence, 
\[
\Theta_\cM(X,r) = \lim_{i\to\infty} \Theta_{\cM_i}(\overline X_i,r) \geq \lim_{i\to\infty} \Theta_{\cM_i}(\overline X_i,r_i) = \lambda(\hat \cM).
\]
Sending $r\to0$ shows that $\lambda(\hat\cM) \leq \Theta_\cM(X)$. 

Consider a tangent flow to $\hat \cM$ at $-\infty$. Since $\lambda(\hat \cM) \leq \Lambda$, Lemma \ref{lemm:BW-small-ent-char-tangent-flow} implies that any such tangent flow is the shrinking flow associated to a smooth shrinker $\hat \Sigma$. We claim that $\hat\Sigma = \Sigma$ and that $\hat \cM$ lies (weakly) on one side of the shrinking flow associated to $\Sigma$. Indeed, by the Frankel property for self-shrinkers (Corollary \ref{coro:Frankel}), there is $\bx \in \hat \Sigma \cap \Sigma$. Because $\Sigma,\hat\Sigma$ have multiplicity one, we can find regions in $\hat\cM_i,\tilde \cM_i$ that are (after a common rescaling) smooth graphs over connected regions in $\hat \Sigma$ and $\Sigma$ containing $\bx$. Because $\supp\cM_i$ and $\supp\cM$ are disjoint, it must hold that $\hat\Sigma = \Sigma$. Applying Lemma \ref{lemm:entire-graph-from-cpt} (and the maximum principle), we can find a sequence of times $t_i\to-\infty$ so that either $\hat\cM(t_i) = \cH^n\lfloor\sqrt{-t_i} \Sigma$, or $\hat\cM(t_i)$ is a smooth graph over $\sqrt{-t_i}\Sigma$ of a nowhere vanishing function. In the first case, we see that $\hat\cM(t) = \cH^n\lfloor\sqrt{-t}\Sigma$ for all $t<0$ by Proposition \ref{prop:unique-cont-brakke} (cf.\ the proof of Lemma \ref{lemm:dist-to-origin-eps-lambda}), while in the second case, we see that $\supp\hat\cM(t)$ is disjoint from $\sqrt{-t}\Sigma$ for all $t<0$ (by Ilmanen's localized avoidance and the Ecker--Huisken maximum principle, as in Lemma \ref{lemm:orig-not-in-K-eps}). 

We claim that the second case cannot occur. Indeed, Theorems \ref{theo:one.sided.construction} and \ref{theo:one.sided.uniqueness} (and $\Lambda \leq \lambda_{n-1}$) imply (since $\lambda(\hat\cM) \leq F(\Sigma)$) that $\sing\hat\cM = \sing_\textrm{gen}\hat\cM$, so Lemma \ref{lemm:stab-spherical-sing} implies that $\hat X_i\in \sing_\textrm{gen}\hat\cM_i$ for $i$ sufficiently large. This is a contradiction, so the first case (i.e., $\hat \cM(t) = \sqrt{-t}\Sigma$ for $t<0$) must hold. 

Now, we can apply Lemma \ref{lemm:schoenflies-entropy-drop} to $\hat \cM$ and $\hat X$ (the limit of the rescaled points $\hat X_i$; note that $|\hat X| = 1$) to conclude that
\[
\limsup_{i\to\infty} \cD(s_i) = \limsup_{i\to\infty} \Theta_{\cM_i}(X_i) \leq \Theta_{\hat \cM}(\hat X) \leq F(\Sigma) - \delta \leq \cD(s) - \delta.
\]
This completes the proof. 
\end{proof}
Using this, we can prove the existence of generic flows. 
\begin{theorem}\label{theo:low-ent-generic-flow-v2}
For $n\geq 3$ and $\Lambda \in (\lambda_{n},\lambda_{n-1}]$, assume that $(\star_{n,\Lambda})$ and $(\star\star_{n,\Lambda})$ hold. Then if $M\subset \RR^{n+1}$ is a closed hypersurface with $\lambda(M)  < \Lambda$ and $\varphi>0$ is a smooth positive function on $M$, let $M_{s}$ denote the normal graph of $s\varphi$ over $M$. Then, there is $s_0> 0$ and a closed, countable set $\cB \subset (-s_0,s_0)$ so that for $s \in (-s_0,s_0)\setminus \cB$, any unit-regular integral Brakke flow with initial condition $M_{s}$ is completely regular until it disappears in a round point.  
\end{theorem}
\begin{proof}
Note that 
\[
\cB = \{s \in (-s_0,s_0) : \cD(s) > -\infty \}
\]
is closed by upper semicontinuity of density and Lemma \ref{lemm:stab-spherical-sing}. Thus, it suffices to prove that $\cB$ is countable. Define
\[
\cB_j : = \{ s \in (-s_0,s_0) : \cD(s) \in [\Lambda - j \delta,\Lambda - (j+1)\delta)\},
\]
so $\cB = \cup_{j=0}^J \cB_j$, for $J > \frac{\Lambda}{\delta}$. By Proposition \ref{prop:schoenflies-blow-up-argument}, if $s \in \cB_j$, there is an open interval $I$ so that $I\cap \cB_j = \{s\}$. Hence, $\cB_j$ is countable. This completes the proof. 
\end{proof}

\begin{proof}[Proof of Theorem \ref{theo:low-ent-generic-flow}]
By \cite[Theorem 4.30]{ColdingMinicozzi:generic}, if $M$ is not a round sphere, then after replacing $M$ by a nearby $C^{\infty}$-close hypersurface, we can assume that $\lambda(M) \leq \Lambda - \eps$ for some $\eps>0$. The statement then follows from Theorem \ref{theo:low-ent-generic-flow-v2}.
\end{proof}

\begin{corollary}\label{cor:low-ent-generic-flow-v3}
For $n\geq 3$ and $\Lambda \in (\lambda_{n},\lambda_{n-1}]$, assume that $(\star_{n,\Lambda})$ and $(\star\star_{n,\Lambda})$ hold. Consider the set
\[
\textnormal{Emb}_{<\Lambda}(\RR^{n+1}) : = \{ M \subset \RR^{n+1}, \lambda(M) < \Lambda\} 
\]
with the $C^{\infty}$ topology.\footnote{For example, we can say that $M_{j}\to M$ if for $j$ large, $M_{j}=\Graph_{M}u_{j}$ with $u_{j}\to0$ in $C^{\infty}(M)$.} Define a subset $\cG$ by the set of $M \in \textnormal{Emb}_{<\Lambda}(\RR^{n+1})$ so that any unit-regular integral Brakke flow starting from $M$ is regular until it disappears in a round point. Then $\cG$ is open and dense.
\end{corollary}
\begin{proof}
We claim that $\textnormal{Emb}_{\leq\Lambda}(\RR^{n+1})\setminus \cG$ closed. Consider $M_{j} \in \textnormal{Emb}_{\leq\Lambda}(\RR^{n+1})\setminus \cG$ with $M_{j}\to M\in \textnormal{Emb}_{\leq\Lambda}(\RR^{n+1})$. Let $\cM_{j}$ denote integral unit-regular Brakke flows starting at $M_{j}$ with non-round tangent flows at $X_{j}$. Passing to a subsequence, $\cM_{j}\rightharpoonup\cM$, a integral unit-regular Brakke flow starting from $M$. By Lemma \ref{lemm:stab-spherical-sing}, a further subsequence has $X_{j}\to X \in \sing\cM$ with $\cM$ having a non-round tangent flow at $X$. This shows $\cG$ is open. Finally, the density of $\cG$ follows from Theorem \ref{theo:low-ent-generic-flow-v2}. 
\end{proof}


\section{The first non-generic time for flows in $\RR^3$} \label{sec:generic.R3}

In this section, we will study the mean curvature flow of a generic initial surface in $\RR^{3}$. We will remove the low-entropy assumption considered in the previous section and study the possible singularities that generically arise. 

For $\cM$ an integral unit-regular Brakke flow, define $T_{\textrm{gen}}$ to be the supremum of times $T$ so that at any point $X \in \supp\cM$ with $\mathfrak{t}(X) < T$, all tangent flows at $X$ are multiplicity-one spheres, cylinders, or planes.

\begin{theorem}\label{theo:generic-R3}
Suppose that $M \subset \RR^{3}$ is a closed embedded surface of genus $g$. Then, there exist arbitrarily small $C^{\infty}$ graphs $M'$ over $M$ and corresponding cyclic integral unit-regular Brakke flows $\cM'$ with $\cM'(0) = \cH^2\lfloor M'$, so that either:
\begin{enumerate}
\item $T_\textnormal{gen}(\cM') = \infty$, or
\item there is $\mathbf{x}\in\RR^3$ so that some tangent flow to $\cM'$ at $(\mathbf{x},T_\textnormal{gen}(\cM'))$ is $k\cH^2\lfloor \sqrt{-t}\Sigma$ for $\Sigma$ a smooth shrinker of genus at most $g$ and either: $k\geq2$ or $\Sigma$ has a cylindrical end but $\Sigma$ is not a cylinder.
\end{enumerate}
\end{theorem}

We will prove this below. Note that Theorem \ref{theo:generic-R3} yields the following conditional result. Recall that the list of lowest entropy shrinkers is known to be the plane, the sphere, and then the cylinder by \cite{White:Brakke,ColdingMinicozziIlmanenWhite,BernsteinWang:TopologicalProperty}. Suppose that there is $\Lambda_g \in (\lambda_1,2]$ so that any smooth shrinker $\Sigma\subset \RR^3$ with $\genus(\Sigma)\leq g$ and $F(\Sigma) < \Lambda_g$ is either a plane, a sphere, a cylinder, or has no cylindrical ends.\footnote{Work of Brendle \cite{Brendle:genus0} implies that only possible genus zero self-shrinkers are the plane, sphere, and cylinder. This immediately implies that $\Lambda_{0}=2$. Ilmanen has conjectured that no non-cylindrical shrinker can have cylindrical ends \cite[\#12]{Ilmanen:problems}, which would mean we can take $\Lambda_g=2$ for all $g$. However, it could theoretically happen that the next lowest entropy shrinker is a counterexample to Ilmanen's conjecture, i.e., has a cylindrical end.} Then:

\begin{corollary}\label{coro:cond-generic-flow-R3-Lambdag}
If $M\subset \RR^3$ is a closed embedded surface with $\genus(M)\leq g$ and $\lambda(M) \leq \Lambda_g$, then there are arbitrarily small $C^\infty$ graphs $M'$ over $M$ and cyclic integral unit-regular Brakke flows $\cM'$ with $\cM'(0) = \cH^2\lfloor M'$ so that $\cM'$ has only multiplicity-one spherical or cylindrical tangent flows, i.e., $T_\textnormal{gen}(\cM') = \infty$.
\end{corollary}

We now establish certain preliminary results used in the proof of Theorem \ref{theo:generic-R3}. The proof of Theorem \ref{theo:generic-R3} can be found after the statement of Proposition \ref{prop:key-perturb-result-3d}. We define 
\[
\sing_\textrm{gen}(\cM) \subset \sing(\cM)
\]
as the set of singular points so that one tangent flow (and thus all of them by \cite{ColdingIlmanenMinicozzi}; alternatively, this follows from \cite{Schulze:Loj,ColdingMinicozzi:uniqueness-tangent-flow} or \cite{BernsteinWang:topology-small-ent}) is a multiplicity-one shrinking sphere or cylinder.

First, we note the following result establishing regularity of tangent flows at $T_\textrm{gen}$ (see also the proof of \cite[Theorem 1.2]{ChoiHaslhoferHershkovits}). 

\begin{proposition}\label{prop:3d-first-non-gen-sing-time-prop}
Consider a cyclic integral unit-regular Brakke flow $\cM$  in $\RR^3$, with $\cM(0) = \cH^2\lfloor M$ for a closed embedded surface $M$. Then, $\cM$ has the following properties:
\begin{itemize}
\item the level set flow of $M$ does not fatten before $T_\textrm{gen}(\cM)$,
\item for almost every $t \in [0,T_\textnormal{gen}(\cM)]$, the level set flow of $M$ is given by $M(t)$, a smooth embedded surface,
\item $\cM(t) = \cH^2\lfloor M(t)$ for almost every $t \in [0,T_\textnormal{gen}(\cM)]$, and
\item $t\mapsto \genus(M(t))$ is non-increasing for all smooth times $t \in [0,T_\textnormal{gen}(\cM)]$.
\end{itemize}
Furthermore, assuming that $T_{\textnormal{gen}}(\cM) < \infty$, then any tangent flow $\hat \cM$ at $(\mathbf{x},T_\textnormal{gen}(\cM))$ satisfies 
\[ \hat \cM(t) = k\cH^2\lfloor \sqrt{-t} \hat\Sigma \text{ for } t<0, \]
where $\hat\Sigma$ is a smooth embedded self-shrinker with 
\[ \genus(\hat\Sigma)\leq \lim_{\substack{t\nearrow T_\textnormal{gen}(\cM) \\ t \not \in \mathfrak{t}^{-1}(\sing \cM)}} \genus(M(t)) \]
 Moreover, $\hat\Sigma$ has finitely many ends, each of which is either asymptotically conical or cylindrical (with multiplicity one), and if $(\mathbf{x},T_\textnormal{gen}(\cM)) \in \sing(\cM) \setminus \sing_{\textnormal{gen}}(\cM)$ and $k=1$, then $\hat\Sigma$ has $\genus(\hat\Sigma) \geq 1$.
\end{proposition}
\begin{proof}
By \cite[Theorem 1.9]{ChoiHaslhoferHershkovits}, the level set flow of $M$ does not fatten for $t \in [0,T_\textrm{gen}(\cM))$. Hence, by \cite[Corollary 1.4]{ColdingMinicozzi:sing-generic} and Corollary \ref{cor:connected-reg-part}, for almost every $t\in[0,T_\textrm{gen}(\cM)]$, the level set flow of $M$ at time $t$ is a smooth surface $M(t)$ and we have that $\cM(t) = \cH^2\lfloor M(t)$. Now, by \cite[Theorem 1]{White:topology-weak} (cf.\ \cite{White:top-change-meanconvex}) $t\mapsto \genus(M(t))$ is non-increasing. 

These two facts suffice to repeat the proof in \cite{Ilmanen:singularities} with only superficial changes (to avoid the singular time-slices) Since this is a crucial point, we describe these modifications in some more detail here. The monotonicity formula gives that along the blow-up sequence corresponding to the given tangent flow, we have
\[
\int_{-1-\tau}^{-1} \int_{\lambda_j (M(T + \lambda_j^2 t)-\bx) \cap B_r(\bOh)} \left| \bH \right|^2 d\mu(t)  dt \leq C \tau r^2 R^2 + \delta_R(\lambda_j)
\]
as in \cite[\S 3]{Ilmanen:singularities}, where $C$ depends only on an entropy bound for $\cM$, $r<R$ are fixed, and $\delta_R(\lambda_j) \to 0$ as $j\to\infty$. Note that in \cite[\S 3]{Ilmanen:singularities}, this is proven for a smooth flow, but the same proof holds here since the flow is smooth for a.e.\ time and this inequality is valid at the level of a Brakke flow. Thus, since a.e.\ time $t$ is a regular time, we can follow \cite[Proof of Theorem 1]{Ilmanen:singularities} almost verbatim, except we can insist that $t_j\to-1$ is chosen so that $T+\lambda_j^2t_j$ is a smooth time for $M(t)$. By  \cite[Theorem 3]{Ilmanen:singularities}, $\lambda_j (M(T + \lambda_j^2 t_j)-\bx)$ has second fundamental form uniformly bounded in $L^2$ on compact sets. The remainder of the proof that $\hat\cM(t) = k\cH^2\lfloor \sqrt{-t}\hat \Sigma$ for $\hat\Sigma$ smooth embedded self-shrinker is then completed exactly as in \cite[Proof of Theorem 1]{Ilmanen:singularities}. 


This proves all but the last two claims. Finally, the statement about the ends of $\hat\Sigma$ is proven in \cite{Wang:ends-conical} (cf.\ \cite[Appendix A]{SunWang:compactness}), while genus zero shrinkers are classified in \cite{Brendle:genus0}. 
\end{proof}

We note that by Proposition \ref{prop:3d-first-non-gen-sing-time-prop}, we can unambiguously define:
\[
\genus_{T_\textnormal{gen}}(\cM) : = \lim_{\substack{t\nearrow T_\textnormal{gen}(\cM)\\t \not \in \mathfrak{t}^{-1}(\sing \cM)}} \genus(M(t)),
\]
the genus of $\cM$ right before the first non-generic singular time. This notion will be useful in the following proposition which will be the key mechanism used to perturb away asymptotically conical (and compact, non-spherical) singularities.

\begin{proposition}\label{prop:key-perturb-result-3d}
Suppose that $M \subset \RR^3$ is a closed embedded surface of genus $g$ and $\cM$ is a cyclic integral unit-regular Brakke flow with $\cM(0) = \cH^2\lfloor M$. Assume that $T_\textnormal{gen}(\cM) < \infty$ and that any tangent flow at time $T_\textnormal{gen}(\cM)$ has multiplicity one and that there is no non-cylindrical tangent flow at time $T_\textnormal{gen}(\cM)$ with a cylindrical end. 

Then, there exists arbitrarily small $C^\infty$ graphs $M'$ over $M$, and cyclic integral unit-regular Brakke flows $\cM'$ with $\cM'(0) = \cH^2\lfloor M'$, so that 
\[ T_\textnormal{gen}(\cM') > T_\textnormal{gen}(\cM) \text{ and } \genus_{T_\textnormal{gen}}(\cM') \leq \genus_{T_\textnormal{gen}}(\cM) - 1. \] 
\end{proposition}

Before proving this, we will show that it implies the full genericity result.
\begin{proof}[Proposition \ref{prop:key-perturb-result-3d} implies Theorem \ref{theo:generic-R3}] 
For $M$ a closed embedded surface of genus $g$, consider any cyclic integral unit-regular Brakke flow $\cM$ with $\cM(0) = \cH^{2}\lfloor M$. Such a flow $\cM$ exists by \cite[Theorem B.3]{HershkovtisWhite} (alternatively, one could perturb $M$ slightly at this step so that the level set flow of $M$ does not fatten, and apply \cite[Theorem 11.4]{Ilmanen:elliptic}). 

First, suppose that either $T_{\textrm{gen}}(\cM) = \infty$ or $T_{\textrm{gen}}(\cM) < \infty$ but at $T_{\textrm{gen}}(\cM)$ there is a tangent flow that either has multiplicity greater than one or is a non-cylindrical shrinker with a cylindrical end. In this case, we can take $M'=M$ and $\cM'=\cM$, completing the proof. In case this does not hold, Proposition \ref{prop:key-perturb-result-3d} yields a small $C^{\infty}$ perturbation $M_1$ of $M$, and a Brakke flow $\cM_{1}$ with $\cM_{1}(0) = \cH^{2}\lfloor M_{1}$. Moreover,
\[
\genus_{T_{\textrm{gen}}}(\cM_{1}) \leq \genus_{T_{\textrm{gen}}}(\cM) - 1 \leq \genus(M)  -1. 
\]
At this point, we can iterate. Either $\cM_{1}$ satisfies the desired conditions, or Proposition \ref{prop:key-perturb-result-3d} applies to $\cM_{1}$. In the former case, we can conclude the proof, and in the latter case we find a small $C^{\infty}$ perturbation $M_2$ of $M$ with a Brakke flow $\cM_{2}$ as above. Repeating this process $k$ times, we find that
\[
\genus_{T_{\textrm{gen}}}(\cM_{k}) \leq \genus(M) - k,
\] 
By Proposition \ref{prop:3d-first-non-gen-sing-time-prop}, it must eventually hold that $M_{k}$, $\cM_{k}$ satisfies one of the two desired conclusions (1) or (2) for some $k \leq \genus(M)$. Thus, after at most $\genus(M)$ perturbations, we find the desired $M'=M_{k}$ and $\cM'=\cM_{k}$. This completes the proof. 
\end{proof}

The proof Proposition \ref{prop:key-perturb-result-3d}, will depend on the following lemmata.

\begin{lemma}\label{lemm:R3-lower-bd-first-non-gen-sing}
	There is $\delta_0>0$ so that if $\cM$ is a cyclic integral unit-regular Brakke flow in $\RR^3$ with $\cM(0) = \cH^2\lfloor M$ for a smooth surface $M$, then for any
	\[
	X \in (\sing(\cM) \setminus \sing_\textnormal{gen}(\cM)) \cap \{ \mathfrak{t} = T_\textnormal{gen}(\cM)\},
	\]
	we have $\Theta_\cM(X) \geq \lambda(\SS^1) + \delta_0$. 
\end{lemma}
\begin{proof}
	This follows by combining Proposition \ref{prop:3d-first-non-gen-sing-time-prop} with \cite[Corollary 1.2]{BernsteinWang:TopologicalProperty}.
\end{proof}

\begin{lemma}[{cf.\ \cite[Theorem 4.3]{BernsteinWang:topology-small-ent}}]\label{lemm:neigborhood-result-good-gen-sing}
	Suppose that $\cM$ is a cyclic integral unit-regular Brakke flow in $\RR^3$ with $\cM(0) = \cH^2\lfloor M$ for some closed embedded surface $M$. Assume that $T_\textnormal{gen}(\cM) < \infty$ and that any tangent flow at time $T_\textnormal{gen}(\cM)$ has multiplicity one and that there is no non-cylindrical tangent flow with a cylindrical end.\footnote{That is, assume that we cannot simply take $\cM'=\cM$ in Theorem \ref{theo:generic-R3}.}
	
	Then for $(\mathbf{x}_0,T_\textnormal{gen}(\cM)) \in \sing (\cM)\setminus \sing_\textnormal{gen}(\cM)$, there are $r,\rho,\tau>0$ so that 
	\begin{align*}
	\cM_1 & : = \cM\lfloor (B_{4r}(\mathbf{x}_0) \times (T_\textnormal{gen}(\cM) - 2\tau,T_\textnormal{gen}(\cM)] \setminus \{(\mathbf{x_0},T_\textnormal{gen}(\cM))\},\\
	\cM_2 & : = \cM\lfloor ((B_{4r}(\mathbf{x}_0) \setminus B_{r/4}(\mathbf{x_0})) \times (T_\textnormal{gen}(\cM) - 2\tau,T_\textnormal{gen}(\cM)+2\tau)
	\end{align*}
	are smooth mean curvature flows. Moreover, any $(\mathbf{x},t) \in \supp \cM \cap (U_{1} \cup U_{2})$ satisfies
	\begin{equation}\label{eq:near-AC-sing-flow-is-weakly-AC-effective}
	|(\mathbf{x} - \mathbf{x}_0)^\perp| \leq \tfrac{1}{10} |\mathbf{x} - \mathbf{x}_0|,
	\end{equation}
	where
	\begin{align*}
	U_{1} : = \{(\mathbf{x}_0+\mathbf{x},T_\textnormal{gen}(\cM)-t) : 0 < \rho^2 t < |\mathbf{x}|^2 <  16r^2 , t < 2\tau \}\\
	U_{2} : = (B_{4r}(\mathbf{x}_0) \setminus B_{r/4}(\mathbf{x_0})) \times (T_\textnormal{gen}(\cM) - 2\tau,T_\textnormal{gen}(\cM)+2\tau). 
	\end{align*}
\end{lemma}

\begin{proof}
	Observe that by Proposition \ref{prop:3d-first-non-gen-sing-time-prop} and the given hypothesis, any\footnote{We emphasize that while we do not need to refer to uniqueness of the tangent flow in this proof, it does indeed hold in this setting by \cite{Schulze:Loj} for compact tangent flows and \cite{ChodoshSchulze} for asymptotically conical ones.}  tangent flow at $(\mathbf{x}_0,T_\textrm{gen}(\cM))$ is associated to a smooth multiplicity-one shrinker that is either compact or asymptotically conical. 
	
	We begin by proving that the smoothness assertion holds for $\cM_1$ for any $r,\tau>0$ small. Indeed, suppose there are singular points $X_i : = (\mathbf{x}_i,T_\textnormal{gen}(\cM) - t_i) \to (\mathbf{x}_0,T_\textnormal{gen}(\cM))$ with $t_i\geq 0$, rescaling around $ (\mathbf{x}_0,T_\textnormal{gen}(\cM))$ to ensure that $X_i$ are a unit distance from the space-time origin, we would find a singular point in a tangent flow to $\cM$ at $ (\mathbf{x}_0,T_\textnormal{gen}(\cM))$ lying in the parabolic hemisphere
	\[
	\{(\bx,t) : t\leq 0, |\bx|^2 + |t| = 1\}.
	\]
	However, no such point in the tangent flow can be singular (since such a flow would not be asymptotically conical). 
	
	We now consider \eqref{eq:near-AC-sing-flow-is-weakly-AC-effective} for points in $\supp\cM \cap U_{1}$. Note that by the smoothness of $\cM_{1}$, all such points are smooth points of $\cM$. We claim that there is $\rho>0$ sufficiently large so that \eqref{eq:near-AC-sing-flow-is-weakly-AC-effective} holds in $\supp\cM \cap U_{1}$, after shrinking $r,\tau>0$ if necessary. Choose $(\mathbf{x}_0 + \mathbf{x},T_\textrm{gen}(\cM) - t) \in \supp \cM$ with $(\mathbf{x},t) \to (0,0)$ and $0 < \rho^2 t < |\mathbf{x}|^2$ but so that 
	\[
	|\mathbf{x}^\perp| \geq \tfrac{1}{10} |\mathbf{x}|. 
	\]
	Rescaling around $(\mathbf{x}_0,T_\textrm{gen}(\cM))$ and passing to the limit, we find a tangent flow to $\cM$ at $(\mathbf{x}_0,T_\textrm{gen}(\cM))$ with associated shrinker $\Sigma_\rho$ so that for some $\mathbf{x}_\rho \in \Sigma$ with $|\mathbf{x}_\rho| \geq \rho$ 
	\[
	|\mathbf{x}_\rho^\perp| \geq \tfrac{1}{10} |\mathbf{x}_\rho|. 
	\]
	However, this will be in contradiction to \cite{ColdingMinicozzi:compactness-shrinkers}, Proposition \ref{prop:3d-first-non-gen-sing-time-prop}, and the fact that the set of tangent flows is compact.\footnote{Alternatively, one may argue as follows: by \cite{ChodoshSchulze}, $\Sigma_\rho$ is independent of $\rho$, 
	which immediately yields a contradiction since for any fixed asymptotically conical shrinker, $|\mathbf{x}^\perp| \leq o(1) |\mathbf{x}|$ as $\mathbf{x} \to\infty$.} Indeed, consider Brakke flows $\cM_\rho$ associated to $\Sigma_{\rho}$. We consider the point $(\rho^{-1}\mathbf{x}_\rho,-\rho^{-2})$ and take a subsequential limit of $\cM_\rho$ to find $\tilde\cM$ a shrinking flow associated to $\tilde\Sigma$ an asymptotically conical shrinker; however, the subsequential limit $(\tilde{\mathbf{x}},0)$ of the space-time points $(\rho^{-1}\mathbf{x}_\rho,-\rho^{-2})$ lies on the asymptotic cone of $\tilde\Sigma$ (and is not at the origin) and thus has $\tilde{\mathbf{x}}^\perp=0$. This is a contradiction, completing the proof. 
	
	Finally, we prove both the smoothness of $\cM_2$ and \eqref{eq:near-AC-sing-flow-is-weakly-AC-effective} for points in $\supp \cM \cap U_{2}$. If some tangent flow to $\cM$ at $(\mathbf{x}_0,T_\textnormal{gen}(\cM))$ is compact, then by considering shrinking spherical barriers, we can choose $r, \tau>0$ so that $\cM_2$ is empty. As such, we can assume that there is an asymptotically conical shrinker $\Sigma'$ associated to some tangent flow $\cM'$ of $\cM$ at $ (\mathbf{x}_0,T_\textnormal{gen}(\cM))$. Because $\Sigma'$ is asymptotically conical, $|\mathbf{x}||A_{\Sigma'}(\mathbf{x})| = O(1)$ and $|\mathbf{x}^\perp| \leq o(1) |\mathbf{x}|$ as $\mathbf{x}\to\infty$. Arguing as in \cite[Lemma 9.1]{ChodoshSchulze}, we can use pseudocality (e.g., \cite[Theorem 1.5]{IlmanenNevesSchulze}) on large balls along the end of $\Sigma$ to find $R>0$ sufficiently large so that
	\[
	\cM'\lfloor ((\RR^{n+1}\setminus B_R)\times [-1,1]) \text{ is smooth} 
	\]
	and satisfies $|\mathbf{x}^\perp| \leq \tfrac{1}{100} |\mathbf{x}|$. From this, we can choose $r,\tau >0$ so that the assertions about $\cM_2$ follow after choosing a blow-up sequence at $(\mathbf{x}_0,T_\textnormal{gen}(\cM))$ converging to $\cM'$. 
\end{proof}

\begin{proof}[Proof of Proposition \ref{prop:key-perturb-result-3d}]
Fix a closed set $K$ with $\partial K = M$. Choose smooth surfaces $M_{i}=\partial K_{i}$ with $M_{i}$ converging to $M$ in $C^{\infty}$, where $K_{i}$ are closed sets with $K_{i}\subset K_{i+1}$ and $M_{i}\cap M_{i+1}=\emptyset$ and $M_{i}\cap M = \emptyset$. We can moreover assume that the level set flow of $M_{i}$ does not fatten \cite[p.\ 63]{Ilmanen:elliptic}, so by \cite[Theorem 11.4]{Ilmanen:elliptic} there is a cyclic integral unit-regular Brakke flow $\cM_{i}$ with $\cM_{i}(0) = \cH^{2}\lfloor M_{i}$.

Passing to a subsequence, $\cM_{i}$ converges to a Brakke flow $\cM_\infty$ with $\cM_\infty(0) = \cH^{2}\lfloor M$. On the other hand, combining Proposition \ref{prop:3d-first-non-gen-sing-time-prop} with Corollary \ref{cor:connected-reg-part}, we find that $\cM_\infty = \cM$ for $t\in[0,T_\textnormal{gen}(\cM))$. In particular, $T_\textrm{gen}(\cM) = T_\textrm{gen}(\cM_\infty)$ and any tangent flow to $\cM_\infty$ at time $T_\textnormal{gen}(\cM_\infty)$ has multiplicity one and no such tangent flow is non-cylindrical but with a cylindrical end. 

We claim that for $i$ sufficiently large, $M'=M_i$ and $\cM'=\cM_i$ satisfy the assertion. Note that Lemma \ref{lemm:R3-lower-bd-first-non-gen-sing} and upper-semicontinuity of density imply that 
\[
\liminf_{i\to\infty}T_\textrm{gen}(\cM_i) \geq T_\textrm{gen}(\cM_\infty). 
\]
We claim that
\[
T_{\textrm{gen}}(\cM_i) > T_{\textrm{gen}}(\cM_\infty)
\]
for sufficiently large $i$. If not, we can pass to a subsequence so that
\begin{equation}\label{eq:3d-gen-Tgen-lim-eq}
T_\textrm{gen}(\cM_i) \leq T_\textrm{gen}(\cM_\infty).
\end{equation}
We claim that this leads to a contradiction using the strategy of proof from Proposition \ref{prop:schoenflies-blow-up-argument}. Choose 
\[
X_i \in (\sing(\cM_i)\setminus\sing_\textrm{gen}(\cM_i)) \cap \{\mathfrak{t}=T_\textrm{gen}(\cM_i)\},
\]
and let $X_i\to X_\infty$. By \eqref{eq:3d-gen-Tgen-lim-eq} and Proposition \ref{prop:3d-first-non-gen-sing-time-prop} any tangent flow to $\cM_\infty$ at $X_\infty$ is associated to a multiplicity one smooth shrinker with all ends (if any) asymptotically conical (note that $X_\infty$ cannot have a multiplicity-one cylindrical or spherical tangent flow by Lemma  \ref{lemm:R3-lower-bd-first-non-gen-sing}). In particular, Theorems \ref{theo:one.sided.construction} and \ref{theo:one.sided.uniqueness} apply to the shrinkers associated to any tangent flow to $\cM_\infty$ at $X_\infty$. We now use these results to obtain a contradiction to \eqref{eq:3d-gen-Tgen-lim-eq}. Briefly, the strategy will be as follows: rescaling $X_i$ around $X_\infty$ we obtain a flow that lies weakly to one-side of a self-shrinking tangent flow to $\cM_\infty$. If the flow lies strictly to one side, it has no non-spherical/cylindrical singularities so we obtain a contradiction. On the other hand, if it agrees with the tangent flow for $t<0$ then we use the observation that a conical or compact shrinking flow is smooth up to and including $t=0$ except at the origin. (Note that this would fail if the shrinker had a cylindrical end.) This will contradict \eqref{eq:3d-gen-Tgen-lim-eq}. 

We now give the full argument. After rescaling by $|X_i-X_\infty| \not = 0$, the flows $\cM_{i}$ converge either to a flow on one side of the tangent flow to $\cM_\infty$ at $X_\infty$ or a flow which agrees with a tangent flow to $\cM_\infty$ for $t<0$. In the first case, the limit has only multiplicity one cylindrical and spherical singularities by Theorems \ref{theo:one.sided.construction} and \ref{theo:one.sided.uniqueness}. This contradicts the choice of $X_i$ by Lemma  \ref{lemm:R3-lower-bd-first-non-gen-sing}. On the other hand, the second case cannot occur. Indeed, if the second case occured, then \eqref{eq:3d-gen-Tgen-lim-eq} would imply that some tangent flow to $\cM_\infty$ has a singularity at $(\bx,t)$ with $|(\bx,t)| = 1$ and $t\leq 0$, contradicting Proposition \ref{prop:3d-first-non-gen-sing-time-prop} and the assumption that no non-cylindrical tangent flow to $\cM_\infty$ at $T_\textrm{gen}(\cM_\infty)$ has cylindrical ends. 

As such, since the flows $\cM_i$ are converging to a flow on one-side of the tangent flow to $\cM_\infty$ at $X_\infty$, we see that 
\begin{equation}\label{eq:3d-gen-Tgen-lim-not-eq}
T_\textrm{gen}(\cM_i) > T_\textrm{gen}(\cM_\infty)
\end{equation} 
for $i$ sufficiently large by (8) in Theorem \ref{theo:one.sided.construction} combined with Theorem \ref{theo:one.sided.uniqueness}. 

It remains to prove the strict genus reduction.  As above, we first briefly sketch the idea for the reader's convenience. By the work of Brendle \cite{Brendle:genus0}, every non-generic singularity that occurs at time $T_\textnormal{gen}(\cM)$ has to have positive genus. Lemma \ref{lemm:neigborhood-result-good-gen-sing} will be used to show that this positive genus is captured in the tangent flow scale of our non-generic singularities. Our understanding of the long-time behavior of flows to one side of a non-generic shrinker (Theorems \ref{theo:one.sided.construction}, \ref{theo:one.sided.uniqueness}) and Lemma \ref{lemm:neigborhood-result-good-gen-sing} again will then imply that, near the non-generic singularities of $\cM_\infty$, the one-sided flows $\cM_i$ will experience strict genus reduction. The result will follow by a localization of the well-known genus monotonicity property of mean curvature flow, given in Appendix \ref{app:loc-top-monotonicity}.

We assume that 
\[
(\bOh,T_\textnormal{gen}(\cM_\infty)) \in \sing(\cM_\infty)\setminus \sing_\textrm{gen}(\cM_\infty)
\]
Fix the corresponding parameters $r,\rho,\tau$ as in Lemma  \ref{lemm:neigborhood-result-good-gen-sing}. 

Define\footnote{Note that this is \emph{spatial} (Euclidean) distance. }
\[
d_{i} : = d(\supp\cM_i(T_\textrm{gen}(\cM_\infty)),\bOh)>0. 
\]
Note that $\lim_{i\to\infty} d_{i} = 0$. Moreover, rescaling $\cM_\infty$ (resp.\ $\cM_i$) around $(\bOh,T_{\textrm{gen}}(\cM_{\infty}))$ by $d_{i}$ to $\cM_{\infty,i}$ (resp.\ $\tilde \cM_{i}$), we can pass to a subsequence so that as $i\to\infty$, $\cM_{\infty,i}$ converges to a tangent flow to $\cM_\infty$ at $(\bOh,T_\textrm{gen}(\cM_\infty))$ and $\tilde \cM_{i}$ converges to a parabolic dilation of the ancient one-sided flow described in Theorems \ref{theo:one.sided.construction} and \ref{theo:one.sided.uniqueness} associated to this tangent flow.

We begin by proving the following two claims that imply that perturbed flows $\cM_{i}$ lose genus locally around points $\mathbf{x}$. 

\begin{claim}[A]\label{claim:A}
There is $\bar \tau \in (0,\tau]$ so that for any $i$ sufficiently large  and $t \in [\bar \tau,2\bar \tau]$, 
\[
\cM_{i}(T_{\textrm{gen}}(\cM_{\infty})-t) \lfloor B_{3r}(\bOh)
\]
is smooth,\footnote{The restriction to a ball $B$ of a time-$t$ slice of a Brakke flow $\cM$, i.e., $\cM(t) \lfloor B$, is said to be smooth if $\mathfrak{t}^{-1}(t)\cap\sing\cM \cap B = \emptyset$. Note that this is stronger than simply asserting $\cM(t) \lfloor B = \cH^{2}\lfloor M$ for some smooth surface $M \subset B$. 
For example, the flow associated to a shrinking sphere disappearing at time $T$ satisfies $\cM(T) = \cH^{2}\lfloor \emptyset$, but $\cM(T)$ is not smooth in the sense above.} intersects $\partial B_{2r}(\bOh)$ transversely, and $\cM_{i}(T_{\textrm{gen}}(\cM_{\infty})-t) \lfloor B_{2r}(\bOh)$ has positive genus.\footnote{Recall: the genus of a surface (possibly with boundary) $\Gamma$ properly embedded in a ball $B\subset \RR^{3}$ is the genus of the surface obtained from  $\Gamma$ after capping off each boundary component with a disk.}
\end{claim}

\begin{claim}[B]
For $i$ sufficiently large, there is $\bar \eps=\bar \eps(i)>0$  so that for $t \in [0,\bar \eps)$
\[
\cM_{i}(T_{\textrm{gen}}(\cM_{\infty})-t) \lfloor B_{3r}(\bOh)
\] 
is a smooth genus zero surface.
\end{claim}
\begin{proof}[Proof Claim (A)]
By Proposition \ref{prop:3d-first-non-gen-sing-time-prop}, any tangent flow to $\cM_{\infty}$ at $(\bOh,T_{\textrm{gen}}(\cM_{\infty}))$ has multiplicity one and positive genus. Thus, by Lemma  \ref{lemm:neigborhood-result-good-gen-sing},\footnote{Specifically, the regularity of $\cM_{1}$.} we can take $\bar \tau$ sufficiently small so that $\cM_{\infty}(T_{\infty}(\cM_{\infty})-t)\lfloor B_{4r}(\bOh)$ is smooth and has positive genus for $t \in [\bar\tau/2,3\bar\tau]$. Combined with Brakke's theorem \cite{White:Brakke} and another application of Lemma \ref{lemm:neigborhood-result-good-gen-sing}\footnote{Specifically,  \eqref{eq:near-AC-sing-flow-is-weakly-AC-effective} on $\supp\cM \cap U_{1}$} the remaining assertions follow. 
\end{proof}

\begin{proof}[Proof of Claim (B)]
We have fixed a tangent flow to $\cM_{\infty}$ and associated one-sided flow from Theorem \ref{theo:one.sided.construction}. Let $\delta>0$ denote the interval of regularity around $t=0$ for the one-sided flow, as described in property (10) of Theorem \ref{theo:one.sided.construction}. We thus define 
\[
\bar \eps(i) = \frac{\delta d_i^2}{2}
\]
This will ensure that when rescaling by $d_{i}$, we are considering a short enough time interval to apply (10) in Theorem \ref{theo:one.sided.construction}.

We first show that for $i$ sufficiently large, $\cM_{i}(T_{\textrm{gen}}(\cM_{\infty})-t) \lfloor B_{3r}(\bOh)$ is smooth for all $t \in [0, \bar \eps(i))$. Suppose, instead, that there were some $\mathbf{y}_i$, $t_i$ such that
\begin{equation} \label{eq:potential.nearby.sing.points.yi}
	\mathbf{y}_i \in (\sing \cM_i (T_\textnormal{gen}(\cM_\infty) - t_i)) \cap B_{3r}(\bOh), \; t_i \in [0, \bar \eps(i)).
\end{equation}
Since
\[
\cM_{i} \lfloor (B_{4r}(\bOh) \times \{t\leq T_{\textrm{gen}}(\cM_{\infty})\}) \rightharpoonup \cM_{\infty} \lfloor (B_{4r}(\bOh) \times \{t\leq T_{\textrm{gen}}(\cM_{\infty})\})
\]
as Brakke flows (for $i \to \infty$), it follows by Lemma  \ref{lemm:neigborhood-result-good-gen-sing}\footnote{Specifically, the smoothness of $\cM_{1}$.} that $\mathbf{y}_{i}\to\bOh$ as $i\to\infty$.

On the other hand, by definition of $\bar\eps(i)$ and (10) in Theorem \ref{theo:one.sided.construction},
\begin{equation}\label{eq:R3-gen-tildedij-vs-dij}
\tilde d_{i} := d((\by_{i},T_{\textrm{gen}}(\cM_{\infty})-t_{i}) ,(\bOh,T_{\textrm{gen}}(\cM_{\infty})) \gg d_{i}.
\end{equation} 
In particular, rescaling $\cM_{i}$ by $\tilde d_{i}$ around $(\bOh,T_{\mathrm{gen}}(\cM_{\infty}))$, the flow converges to some flow $\tilde \cM_{\infty}$. By \eqref{eq:R3-gen-tildedij-vs-dij}, we have that $(\bOh,0) \in \supp \tilde \cM_{\infty}$. Thus, we have that for $t<0$, $ \tilde \cM_{\infty}$ agrees with a tangent flow to $\cM_{\infty}$ at $(\bOh,T_{\textrm{gen}}(\cM_{\infty}))$. This is a contradiction since Proposition \ref{prop:3d-first-non-gen-sing-time-prop} implies that $ \tilde \cM_{\infty} \lfloor ((\RR^{n+1}\times (-\infty,0]) \setminus\{(\bOh,0)\})$ is smooth. Thus, no points $\mathbf{y}_i$ as in \eqref{eq:potential.nearby.sing.points.yi} will exist. This completes the proof of the regularity assertion.

We finally prove that for $t_{i}\in[0,\bar\eps(i))$, the surface $\cM_{i}(T_{\textrm{gen}}(\cM_{\infty})-t_{i}) \lfloor B_{3r}(\bOh)$ has genus zero for $i$ large. We show below that that for some $R>0$ sufficiently large (independent of $i$), for any $i$ large and 
\begin{equation}\label{eq:claim-gen-flow-r3-claim-smooth-gen-zero}
\mathbf{x} \in \supp\cM_{i}(T_{\textrm{gen}}(\cM_{\infty})-t_{i}) \cap (B_{3r}(\bOh) \setminus B_{Rd_{i}}(\bOh)),
\end{equation}
we have $|\bx^{\perp}| \leq \frac{1}{5} |\bx|$. This follows from essentially the same scaling argument as above. Indeed, consider a sequence of a points $\mathbf{y}_{i}$ and times $t_i$ violating this bound while still satisfying \eqref{eq:claim-gen-flow-r3-claim-smooth-gen-zero} (we will choose $R>0$ large below). Rescaling $\cM_{i}$ around $(\bOh,T_\textrm{gen}(\cM_\infty))$ by
\[
\tilde d_{i} := d((\by_{i},T_{\textrm{gen}}(\cM_{\infty})-t_{i}) ,(\bOh,T_{\textrm{gen}}(\cM_{\infty})),
\]
we claim that it now must hold that 
\begin{equation}\label{eq:R3-gen-tildedij-vs-dij-TAKE2}
\limsup_{i\to\infty} \frac{\tilde d_{i}}{d_{i}} < \infty.
\end{equation}
Indeed, if this fails, we can argue precisely as in the previous paragraph to rescale by $\tilde d_i$ to find a tangent flow to $\cM_\infty$ at $(\bOh,T_\textrm{gen}(\cM_\infty))$; the points $(\by_i,t_i)$ converge--after rescaling--to a point on the tangent flow at $t=0$ (at a unit distance from $\bOh$). Clearly the cone satisfies the asserted bound, so this is a contradiction. 

Thus, \eqref{eq:R3-gen-tildedij-vs-dij-TAKE2} holds. In particular, the points $(\mathbf{y}_i,-t_i)$ remain a bounded distance from $(\bOh,0)$ when rescaling by $d_i$ (but lie outside of $B_R(\bOh) \times \RR$). It is easy to see\footnote{By the argument in Lemma \ref{lemm:dist-to-origin-eps-lambda}, the blow-down of the ancient one-sided flow agrees for $t<0$ with the shrinking Brakke flow associated to the fixed shrinker.} that we can take $R>0$ large so that the one-sided flow from Theorem \ref{theo:one.sided.construction} (scaled to have unit distance from $(\bOh,0)$) satisfies $|\mathbf{x}^{\perp}| \leq \frac{1}{10} |\mathbf{x}|$ for $(\mathbf{x},t)$ with $|\mathbf{x}|\geq R$ and $|t| < \delta$.

We now demonstrate that putting these facts together, we have proven the claim. After rescaling by $d_i$ the flows $\cM_i$ converge to the one sided flow to the tangent flow of $\cM_\infty$. By property (10) of Theorem \ref{theo:one.sided.construction} this one sided flow at time zero is smooth and strictly star-shaped and therefore has genus zero. Thus by smooth convergence and thus the transverse intersection at the boundary of the ball $B_{2Rd_{i}}(\bOh)$, together with the choice of $\bar \eps(i)$, $\cM_{i}(T_{\textrm{gen}}(\cM_{\infty})-t_{i}) \lfloor B_{2Rd_{i}}(\bOh)$ has genus zero for $i$ large. 
\end{proof}

Now, take $\bar \tau$ smaller if necessary and then fix $i$ large. We write $\cM_{i}=\cM'$ and assemble the following properties established above:
\begin{enumerate}
\item For $t \in (T_{\textrm{gen}}(\cM_{\infty})-2\bar \tau,T_{\textrm{gen}}(\cM_{\infty}) + 2\bar \tau)$ a smooth time for $\cM'$ we have that
\[ \cM'(t) = \cH^{2}\lfloor M'(t), \]
for $M'(t)$ smooth with
\[ \genus(M'(t))\leq \genus_{T_{\textrm{gen}}}(\cM_{\infty}); \]
this follows from the monotonicity of genus (cf.\ Proposition \ref{prop:3d-first-non-gen-sing-time-prop}) and the fact that $\cM_{i} \rightharpoonup \cM_{\infty}$ as Brakke flows. 
\item For $t\in[\bar \tau,2\bar\tau]$, $\cM'(T_{\infty}(\cM_{\infty})-t) \lfloor B_{3r}(\bOh)$ is smooth and has positive genus in $B_{2r}(\bOh$); this was proven in Claim (A) above. 
\item There is $0 < \bar \eps < \bar \tau$ so that for $t\in[0,\bar\eps)$, $\cM'(T_{\infty}(\cM_{\infty})-t) \lfloor B_{3r}(\bOh)$ is smooth and has zero genus; this was proven in Claim (B) above. 
\item We have that
\[ \cM'\lfloor ((B_{3r}(\bOh) \setminus B_{r/2}(\bOh)) \times (T_{\textrm{gen}}(\cM_{\infty}) - 2\bar\tau,T_{\textrm{gen}}(\cM_{\infty}) + 2\bar\tau)) \]
is a smooth flow of (a disjoint union of) topological annuli, intersecting $\partial B_{r'}(\bOh)$ transversely for $r \leq r' \leq 3r$; this follows from Lemma  \ref{lemm:neigborhood-result-good-gen-sing} and the fact that $\cM_{i} \rightharpoonup \cM_{\infty}$ as Brakke flows. 
\end{enumerate}
Choose
\[ \bar t_{1} \in (T_{\textrm{gen}}(\cM_{\infty}) - 2\bar \tau,T_{\textrm{gen}}(\cM_{\infty}) - \bar \tau], \]
\[ \bar t_{2} \in (T_{\textrm{gen}}(\cM_{\infty}) - \bar \eps, T_{\textrm{gen}}(\cM_{\infty})] \]
smooth times for $\cM'$. We claim that 
\begin{equation}\label{eq:R3-gen-goal-genus-drop}
g : = \genus( \cM'(\bar t_{2})) < \genus(\cM'(\bar t_{1})).
\end{equation}
By property (1) in the above list (and monotonicity of genus, cf.\  Proposition \ref{prop:3d-first-non-gen-sing-time-prop}), once we have established \eqref{eq:R3-gen-goal-genus-drop}, we will find
\[
\genus_{T_{\textrm{gen}}}(\cM') \leq  \genus(\cM'(\bar t_{2})) \leq \genus_{T_{\textrm{gen}}}(\cM_{\infty}) - 1,
\]
which will complete the proof.

It thus remains to establish \eqref{eq:R3-gen-goal-genus-drop}. 
We will show this by combining the properties above with a localization of White's  \cite{White:topology-weak} topological monotonicity, which we have included in Appendix \ref{app:loc-top-monotonicity}. Define
\[ B := B_{2r}(\bOh). \]
The key observation, which makes Appendix \ref{app:loc-top-monotonicity} applicable, is that, by property (4) above, the level set flow for times in $[\bar t_{1},\bar t_{2}]$ of $M'(\bar t_{1}) \times \{\bar t_{1}\}$ (which must agree with the restriction of $\cM'$) is a \emph{simple flow} (defined in Appendix \ref{app:loc-top-monotonicity}) in the tubular neighborhood 
\[
U:= B_{3r}(\bOh)\setminus \bar{B}_{r}(\bOh), 
\]
of $\partial B$ for $t \in [\bar t_1, \bar t_2]$. We can thus apply results of that appendix with $[\bar t_1, \bar t_2]$ in place of $[0,T]$, and $\RR^3 \setminus \bar B$ in place of $\Omega$. (Certainly, we can and will also apply White's global topological monotonicity results.) We invite the reader to recall the notation $W[\bar t_1, \bar t_2]$, $W[\bar t_1]$, $W[\bar t_2]$ from  \eqref{eq:localized.monotonicity.w.t}-\eqref{eq:localized.monotonicity.w.0T} in Appendix \ref{app:loc-top-monotonicity}, which we're going to make use of here. 

To quantify the genus drop, we'll use Lemmas \ref{lemm:R3-good-basis-comp-surf-ball} and \ref{lemm:R3-basis-comp-surf-ball-genus-in-ball-consequence} stated and proved below. Loosely speaking, Lemma \ref{lemm:R3-good-basis-comp-surf-ball} constructs a good choice of linearly independent set of loops in $H_1(W[\bar t_2])$ detecting the number $g$ and compatible with the geometry (namely, the smoothness of the flow in the annular region as established in (4) above). By the localized version of White's topological monotonicity established in Appendix \ref{app:loc-top-monotonicity}, we can homotop these loops back to time $\bar t_1$. The properties established in Lemma \ref{lemm:R3-good-basis-comp-surf-ball} are preserved under this process and then we can apply Lemma \ref{lemm:R3-basis-comp-surf-ball-genus-in-ball-consequence} to show that the genus at time $\bar t_1$ would have to be $\leq g$. This proves the desired genus monotonicity. 

Choose loops $\gamma_1^{\bar t_2},\dots,\gamma_{2g}^{\bar t_2}$ in $W[\bar t_2]$ as in Lemma \ref{lemm:R3-good-basis-comp-surf-ball}. That is, 
\[
\{[\gamma_1^{\bar t_2}],\dots,[\gamma_{2g}^{\bar t_2}]\} \subset H_1(W[\bar t_2])
\]
is linearly independent and each $\gamma_i^{\bar t_2}$ satisfies either:
\begin{itemize}
\item $\gamma_i^{\bar t_2}$ is contained in $\bar B^c$ (since $\genus(M'(\bar t_2)\cap B) = 0$ implies that no $\gamma_i^{\bar t_2}$ can be contained in $B$), or
\item there is some component $\cU_i[\bar t_2]$ of $W[\bar t_2] \cap \partial B$ that has non-zero signed intersection with $\gamma_i^{\bar t_2}$, and zero signed intersection with each previous $\gamma_j^{\bar t_2}$, $j < i$.
\end{itemize}
By the injectivity of $H_1(W[\bar t_2]) \to H_1(W[\bar t_1, \bar t_2])$ \cite[Theorem 6.2]{White:topology-weak}, the inclusion
\[ \{[\gamma_1^{\bar t_2}],\dots,[\gamma_{2g}^{\bar t_2}] \}\subset H_1(W[\bar t_1,\bar t_2]) \]
is linearly independent too. We now construct loops $\gamma_1^{\bar t_1}, \ldots, \gamma_{2g}^{\bar t_1}$ in $W[\bar t_1]$ so that:
\begin{itemize}
\item Each $\gamma_i^{\bar t_2}$ is homotopic to $ \gamma_i^{\bar t_1}$ in $W[\bar t_1,\bar t_2]$; see \cite[Theorem 5.4]{White:topology-weak}.
\item If $\gamma_i^{\bar t_2}$ is entirely contained in $\bar B^c$, then so is $\gamma_i^{\bar t_1}$ and the entire homotopy between them; see Theorem \ref{theo:localized-loops-to-0}.
\item If $\gamma_i^{\bar t_1}$ is not entirely contained in $\bar B^c$, there is some component $\cU_i[\bar t_1]$ of $W[\bar t_1] \cap \partial B$  that has non-zero signed intersection with $\gamma_i^{\bar t_1}$, and zero signed intersection with each previous $\gamma_j^{\bar t_1}$, $j<i$; this follows from the simplicity of the flow in $U\times[\bar t_1,\bar t_2]$ and the fact that  signed intersection is preserved under homotopy.
\end{itemize}
We can now easily complete the proof. If \eqref{eq:R3-gen-goal-genus-drop} were false, then $\genus(M'(\bar t_1)) = g$ by White's global topological monotonicity \cite{White:topology-weak}. Applying Lemma  \ref{lemm:R3-basis-comp-surf-ball-genus-in-ball-consequence} to $\gamma_1^{\bar t_1}, \ldots, \gamma_{2g}^{\bar t_1}$ now says that, because  $\genus(M'(\bar t_1)\cap B) > 0$ by property (2) above, at least one of the  $\gamma_i^{\bar t_1}$ must be contained in $B$, a contradiction.
\end{proof}

\begin{lemma}\label{lemm:R3-good-basis-comp-surf-ball}Suppose that $S \subset \RR^3$ is a closed and embedded genus-$g$ surface which is transverse to a sphere $\partial B \subset \RR^3$.  Denote $W := \RR^3\setminus S$. 
	
We can find loops  $\gamma_1,\ldots,\gamma_{2g}$ inside $W$ so that $\{[\gamma_1],\ldots,[\gamma_{2g}]\}\subset H_1(W) \approx\ZZ^{2g}$ is linearly independent and so that, for every $i = 1, \ldots, 2g$, either:
\begin{itemize}
\item $\gamma_i$ is contained in $B$ or in $\bar B^c$, or 
\item there is a component $\cU_i$ of $\partial B\setminus S$ that has non-zero signed intersection with $\gamma_{i}$, and zero signed intersection with each previous $\gamma_{j}$, $j < i$.
\end{itemize} 
Moreover, we can arrange that exactly $2\genus(S\cap B)$ of the $\gamma_{i}$ are contained entirely in $B$ and that if, in $H_1(W\cap \bar B)$,
\[
\sum_{\{i : \gamma_i \subset B\}} n_i[\gamma_i] = [\beta]
\]
for some cycle $\beta \subset W \cap \partial B$, then all of the coefficients $n_i$ vanish. 
\end{lemma}
\begin{proof}
We induct on the number of components $b$ of $S\cap \partial B$. 

First, consider $b=0$. In this case, $S$ decomposes into the disjoint union of two closed surfaces, $S_B := S \cap B$, $S_{\bar B^c} := S \setminus \bar B$, which do not meet $\partial B$. We have $\genus(S_{B}) + \genus(S_{\bar B^c}) = g$, so by applying Alexander duality we find a linearly independent set
\[
\{[\gamma_{1}],\dots,[\gamma_{2g}] \} \subset H_{1}(W)
\]
with $2\genus(S_{B})$ of the $\gamma_{i}$ contained  in $B$, and the remaining $2\genus(S_{\bar B^c})$ contained  in $\bar B^c$. Moreover, $W \cap \partial B = \partial B$ when $b=0$. Therefore, if a linear combination of $\gamma_i \subset B$ is homologous to a cycle in $W \cap \partial B$ then the combination must be $=0\in H_1(W)$ (since $H_1(\partial B)=0$). This completes the base case. 

Now, we consider the inductive step. Consider the $b$ components of $S\cap \partial B$. By the Jordan curve theorem, each component of $S\cap \partial B$ divides $\partial B$ into two regions. As such, we can find a component $\alpha$ of $S\cap \partial B$ so that there is a disk $D\subset \partial B$ with $\partial D =\alpha$ and $S\cap D^{\circ}=\emptyset$. Form the surface $S'$ by removing an annulus $A = U_{\eps/10}(\alpha) \subset S$ and then by gluing two disks that are small deformations of $D$, into and out of $B$ respectively, to cap off the boundary of $S\setminus A$. We can arrange that this all occurs in $U_{\eps}(D)\subset \RR^{3}$ (with $\eps>0$ small enough so that $U_{\eps}(D)$ is contractible). 

The surface $S'$ now satisfies the inductive hypothesis, since $S'\cap \partial B$ has $b-1$ components. Note that, by definition, 
\begin{equation} \label{eq:R3-good-basis-comp-ball-surf-genus}
\genus(S'\cap B) = \genus(S\cap B),
\end{equation}
although the genus of $S$ might be different from $S'$ as we will see below. 

There are two cases to consider: either $\alpha$ separates the component of $S$ that contains it, or it doesn't separate it.

\emph{Separating case.} Suppose that $\alpha$ separates the component of $S$ that contains it. It will be convenient to give a name to this component, so let us denote it $S_\alpha$. In this case, $S_\alpha \setminus A$ is a disconnected surface with boundary. Hence,\footnote{This can be seen by the inclusion-exclusion principle for Euler characteristic: if we can decompose a connected surface into two connected components $M=M_{1}\cup M_{2}$ where $M_{1}$ and $M_{2}$ intersect in a circle, then $\chi(M) + \chi(\SS^{1}) = \chi(M_{1}) + \chi(M_{2})$. We have that $\chi(\SS^1) = 0$, $\chi(M) = 2-2\genus(M)$, and $\chi(M_i) = 1-2\genus(M_i)$ (because they both have a single boundary component). Hence, $\genus(M) = \genus(M_{1}) + \genus(M_{2})$. }
	\[ \genus(S_\alpha) = \genus(S_\alpha \setminus A), \]
	so $\genus(S') = g$. Applying the inductive step to $S'$ (which has $b-1 < b$ boundary circles), we find a linearly independent set of loops $\gamma'_{1},\dots,\gamma'_{2g}$ in $\RR^{3}\setminus S'$ satisfying the conditions of the lemma with $S'$ in place of $S$. The curves $\gamma_i'$ that are not contained in $B$ or in $\bar B^c$ have associated components $\cU_i' \subset \partial B \setminus S'$ with the required  signed intersection properties, per the inductive step.
	
	Note that we can assume that the loops $\gamma_1', \ldots, \gamma_{2g}'$ are disjoint from $U_{\eps}(D)$. As such, they lie in $W$, so to prove the inductive step we can simply take
	\[ \gamma_1 := \gamma_1', \ldots, \gamma_{2g} := \gamma_{2g}'. \]
	For any $\gamma_i$ that is not contained in $B$ or in $\bar B^c$, we set $\cU_i := \cU_i' \setminus D$ or $\cU_i := \cU_i'$ depending on whether $D \subset \cU_i'$ or not (respectively). We claim this configuration of $\gamma_1, \ldots, \gamma_{2g}$ satisfies the properties we want. Note that the two bullet points are just a consequence of how our curves are disjoint from $U_\eps(D)$, and that $2\genus(S \cap B)$ of the $\gamma_i$ are contained in $B$ in view of  \eqref{eq:R3-good-basis-comp-ball-surf-genus} and the inductive step. It remains to check two required homological properties.
	
	Suppose there are $n_i$ so that 
	\[
	\sum_{\{i : \gamma_i \subset B\}} n_i[\gamma_i] = [\beta] \text{ in } H_1(W \cap \bar B),
	\]
	for some cycle $\beta \subset \partial B \setminus S$.  Note that the components $\beta''$ of $\beta$ that intersect $D$ must be fully contained inside $D$. We write $\beta = \beta' + \beta''$. Note further that we can assume that $\beta'$ consists of finitely many disjoint embedded circles.
	Thus, we can find a 2-chain $\sigma \subset B \setminus S$ such that
	\[
	\partial \sigma = \beta - \sum_{\{i : \gamma_i \subset B\}} n_i\gamma_i.
	\]
	 Using the structure of $\beta''$ we see that we can replace $\sigma$ by $ \sigma' + \sigma''$ such that $\sigma'$ is contained in $\bar B \setminus (S \cup U_{\eps}(D)) \subset \bar B \setminus S'$ and $\sigma''$ is contained in $\bar B \cap U_\eps(D)$. Since the latter region is contractible ($D$ was contractible), this implies that
		\[
	\sum_{\{i : \gamma_i \subset B\}} n_i[\gamma_i] =  [\beta'] \text{ in } H_1(\bar B\setminus S').
	\]
	By the inductive step, all of the coefficients vanish. 
	
	We finally show  $\{[\gamma_{1}],\dots,[\gamma_{2g}]\} \subset H_{1}(W)$ is linearly independent. Assume 
	\begin{equation}\label{eq:R3-comp-S-good-basis-sep-case-triv}
	n_{1}[\gamma_{1}]+\dots+n_{2g}[\gamma_{2g}] = 0 \text{ in } H_1(W).
	\end{equation}
	By construction and the inductive step, for any $\gamma_{i}$ not contained entirely in $B$ or $\bar B^c$, there is a component $\cU_i' \subset \partial B \setminus S'$ that has non-zero signed intersection with $\gamma_i$ and zero signed intersection with each previous $\gamma_j$, $j < i$. Proceeding from large indices to small this implies that any $\gamma_{i}$ not contained entirely in $B$ or in $\bar B^c$ has $n_{i}=0$ in \eqref{eq:R3-comp-S-good-basis-sep-case-triv}. The Mayer--Vietoris sequence for $(W\cap \bar B,W\cap B^c)$ yields the exact sequence
	\[
	\dots \to H_{1}(W \cap \partial B) \to H_{1}(W\cap \bar B) \oplus H_{1}(W\cap B^c) \to H_{1}(W) \to \dots
	\]
	Let $I_{B}$ denote the indices $i$ so that $\gamma_{i}\subset B$ and similarly for $I_{\bar B^c}$. Consider 
	\[
	\left( \sum_{i\in I_{B}} n_{i}[\gamma_{i}], - \sum_{i\in I_{\bar B^c}} n_{i}[\gamma_{i}] \right) \in H_{1}(W\cap \bar B) \oplus H_{1}(W\cap B^c) . 
	\]
	Seeing as we're assuming this is sent to $0\in H_{1}(W)$, exactness yields a $[\beta] \in H_1(W \cap \partial B)$ so that 
	\[
	[\beta] = \sum_{i\in I_{B}} n_{i}[\gamma_{i}] \text{ in } H_1(W\cap \bar B), \text{ and}
	\]
	\[
	[\beta] = - \sum_{i\in I_{\bar B^c}} n_{i}[\gamma_{i}] \text{ in } H_1(W\cap B^c).
	\]
	We have already seen above, though, that $n_i=0$ for all $i \in I_B$ since $\beta$ is a cycle in $\partial B \setminus S$. Thus $[\beta]=0$ in $H_1(W\cap \bar B)$.  Arguing as above we can replace $\beta$ by $\beta'$ (which has no component in $D$), such that
	\[  [\beta'] = 0 \text{ in } H_1(\bar B\setminus S') \]
	and
	\[
	 [\beta'] = - \sum_{i\in I_{\bar B^c}} n_{i}[\gamma_{i}] \text{ in } H_1(\RR^3\setminus (B\cup S')).
	\]
	Using Mayer-Vietoris as above with $S$ replaced by $S'$, we find that 
	\[
	\sum_{i\in I_{\bar B^c}} n_{i}[\gamma_{i}] = 0 \text{ in } H_1(\RR^3\setminus S').
	\]
	The inductive step implies that the $n_i$ all vanish. This completes the proof in the separating case. 

\emph{Nonseparating case}. We turn to case where $\alpha$ does not separate the component of $S$ that contains it. We continue to denote that component of $S$ by $S_\alpha$. Observe that\footnote{For a connected compact surface $M$ with $\partial M$ consisting of two circles, and the surface $M'$ formed by gluing these two boundary circles together, the inclusion-exclusion principle implies $\chi(M) = \chi(M')$.} 
	\[ \genus(S_\alpha \setminus A) + 1 = \genus(S_\alpha), \]
	so $\genus(S') = g-1$. We apply the inductive step to $S'$ (which has $b-1 < b$ boundary circles) to find a linearly independent set $\{[\gamma_1'],\dots,[\gamma_{2g-2}']\} \subset H_1(\RR^3\setminus S')$ satisfying the conditions of the lemma with $S'$ in place of $S$. For every $\gamma_i'$ that is not contained in $B$ or in $\bar B^c$, there exists a component $\cU_i' \subset \partial B \setminus S'$ with the signed intersection properties postulated by the inductive step.
	
	As in the previous case, we can assume that the cycles are disjoint from $U_{\eps}(D)$, and thus lie in $W$. So, we may take
	\[ \gamma_1 := \gamma_1, \; \ldots, \; \gamma_{2g-2} := \gamma_{2g-2}', \]
	and, as before, set $\cU_i := \cU_i' \setminus D$ or $\cU_i'$ depending on whether $D \subset \cU_i'$ or not (respectively). We further define $\gamma_{2g-1}\subset \bar B^c$ to be $\alpha$ shifted slightly into the non-compact component of $\RR^3\setminus (S_\alpha \cup \bar B)$. Finally, we define $\gamma_{2g}$ to be a loop in the compact component enclosed by $S_\alpha$ with the property that $\gamma_{2g}$ intersects the disk $D$ transversely and in precisely one point (it is easy to find such a curve thanks to the non-separating hypothesis); we take $\cU_{2g} := D^\circ$.
	
	We claim that the loops $\gamma_1,\dots,\gamma_{2g}$ satisfy the assertions of the lemma. The two bullet points are easily checked by the construction of $\gamma_{2g-1},\gamma_{2g}$ and the assumption that the curves obtained via the inductive step avoid $U_{\eps}(D)$. The other two claims in the assertion follow by essentially the same argument as in the separating case.

This completes the proof. 
\end{proof}
\begin{lemma}\label{lemm:R3-basis-comp-surf-ball-genus-in-ball-consequence} 
Suppose that $S \subset \RR^3$ is a closed and embedded genus-$g$ surface which is transverse to a sphere $\partial B \subset \RR^3$. Denote $W := \RR^3\setminus S$. 

Assume that we are given  $\{[\gamma_1],\dots,[\gamma_{2g}]\}\subset H_1(W)\approx \ZZ^{2g}$ which is linearly independent and where each $\gamma_i$ satisfies one of the following conditions:
\begin{itemize}
\item $\gamma_i$ is contained in $B$ or in $\bar B^c$, or 
\item there is a component $\cU_i$ of $\partial B\setminus S$ that has non-zero signed intersection with $\gamma_i$ and zero signed intersection with each previous $\gamma_j$, $j < i$.
\end{itemize} 
Then, at least one of the $\gamma_i$ is contained in $B$, provided $\genus(S\cap B)>0$.\footnote{We do not need this here, but with minor modifications one can show that at least $2\genus(S\cap B)$ curves $\gamma_i$ are contained entirely in $B$.}
\end{lemma}
\begin{proof}
Note that, since $\genus(S\cap B)>0$, Lemma \ref{lemm:R3-good-basis-comp-surf-ball} implies (among other things) that there is $\eta \subset B \setminus S$ so that $[\eta]\not = 0$ in $H_1(W)$ and so that for any $m\in\ZZ\setminus\{0\}$, $m\eta$ is not homologous in $\bar B\setminus S$ to a cycle in $\partial B \setminus S$. 

Now, assume that none of the $\gamma_i$ described above are contained in $B$. We claim that 
\[
\{[\gamma_1],\dots,[\gamma_{2g}],[\eta]\} \subset H_1(W)\approx \ZZ^{2g}
\]
is a linearly independent set. This is impossible, so we will have proven the lemma. To this end, assume that there are coefficients so that 
\[
m[\eta] = \sum_{i=1}^{2g} n_i[\gamma_i] \text{ in } H_1(W).
\]
As in Lemma \ref{lemm:R3-good-basis-comp-surf-ball}, by working downwards from $i=2g$ and considering the intersection of each $\gamma_i$ with appropriate components of $W \cap \partial B$, using the $\cU_i$'s, we can show that $n_i=0$ unless $\gamma_i$ is contained entirely in $\bar B^c$. As in the proof of Lemma \ref{lemm:R3-good-basis-comp-surf-ball}, applying Mayer--Vietoris to the pair $(W\cap \bar B,W\cap  B^c)$, we find that $m\eta$ must be homologous in $\bar B\setminus S$ to a cycle in $\partial B\setminus S$. This contradicts the above choice of $\eta$ unless $m=0$, but in this case this contradicts the linear independence of the $[\gamma_i]$. This completes the proof. 
\end{proof}

\appendix


\section{Geometry of asymptotically conical shrinkers} \label{sec:shrinker.geometry}

Consider a shrinker $\Sigma^{n}\subset \RR^{n+1}$ that is asymptotic to a smooth cone $\cC$. In \cite[Lemma 2.3]{ChodoshSchulze}, it was shown that the function $w : \cC \setminus B_{R}(\mathbf{0}) \to \RR$ parametrizing the end of $\Sigma$, i.e., such that
\[ \operatorname{graph}_{\cC} w := \{ \mathbf{x} + w(\mathbf{x}) \nu_{\cC}(\mathbf{x}) : \mathbf{x} \in \cC \setminus B_R(\mathbf{0}) \} \subset \Sigma, \]
must satisfy $w=O(r^{-1})$, $\nabla^{(k)}_{\partial_{r}}w = O(r^{-1-k})$, and $\nabla^{(k)} w = O(r^{-1-k+\eta})$ for any $\eta>0$. Here, $r = |\mathbf{x}|$ is the radial coordinate on the cone. The sharp asymptotics of $w$ (which we need in this paper) are, in fact:

\begin{lemma} \label{lemma:shrinker.geometry.decay}
	The function $w$ above satisfies $\nabla^{(k)}_{\cC} w = O(r^{-1-k})$ as $r \to \infty$.
\end{lemma}
\begin{proof}
	We prove this for $k=1$---higher derivatives follow by induction. The shrinker equation \eqref{eq:defn-shrinker} along $\Sigma$ implies  (using our curvature conventions from Section \ref{sec:prelim.curvature}) that 
	\[
		H_{\Sigma}(\mathbf{x}+w(\mathbf{x})\nu_{\cC}(\mathbf{x}))  + \tfrac 12 \bangle{\mathbf{x} + w(\mathbf{x})\nu_{\cC}(\mathbf{x}), \nu_{\Sigma}}= 0.
	\]
	Moreover, by \cite[(C.1)]{ChodoshSchulze} we have 
	\[
		\nu_{\Sigma}(\mathbf{x}) = (1+|(\textrm{Id}-w(\mathbf{x})A_{\cC}(\mathbf{x}))^{-1}\nabla w(\mathbf{x})|^{2})^{-\frac 12}(-(\textrm{Id}-w(\mathbf{x})A_{\cC}(\mathbf{x}))^{-1}\nabla w(\mathbf{x}) + \nu_{\cC}(\mathbf{x})).
	\]
	By combining these equations we find 
	\[
		r \nabla_{\partial_{r}} w(\mathbf{x}) - w(\mathbf{x})  = W(\mathbf{x}),
		\]
	where
	\[ 
		W(\mathbf{x}) := 2 (1+|(\textrm{Id}-w(\mathbf{x})A_{\cC}(\mathbf{x}))^{-1}\nabla w(\mathbf{x})|^{2})^{\frac 12}H_{\Sigma}(\mathbf{x}+w(\mathbf{x})\nu_{\cC}(\mathbf{x})).
	\]
	We have used the fact that $A_{\cC}(\partial_r, \cdot) \equiv 0$, as well as that $\textrm{Id} - wA_{\cC}$ is an endomorphism of $T\cC$ and $\nu_{\cC} \perp T\cC$. Observe that
	\begin{equation} \label{eq:shrinker.geometry.decay.W}
		\nabla^{(k)} W = O(r^{-1-k}).
	\end{equation}
	Indeed, $\nabla^{(k)} H_{\Sigma} = O(r^{-1-k})$, while the other terms decay at a faster rate. For $\mathbf{x} = rp$ for $p\in\Gamma$, the link of $\cC$, choose a vector $\vartheta \in T_{p} \Gamma$. Extend $\vartheta$ to be parallel along $\gamma:r\mapsto rp$. Note that $[r\vartheta, \partial_r] = 0$, so by \eqref{eq:shrinker.geometry.decay.W} we find:
	\[
		r \nabla_{\partial_{r}} (\nabla_{r\vartheta} w) - \nabla_{r\vartheta} w = \nabla_{r\vartheta} W = O(r^{-1})
	\]
	Integrating this from infinity (cf.\ \cite[Lemma 2.3]{ChodoshSchulze}), we find $\nabla_{r\vartheta} w = O(r^{-1})$. As $\vartheta = O(1)$, we find that $\nabla w = O(r^{-2})$ (decay of the radial component was  shown in \cite[Lemma 2.3]{ChodoshSchulze}). 
\end{proof}

Using this improved decay, one can set $\eta := 0$ in \cite[Corollary 2.4]{ChodoshSchulze}, \cite[Lemma 2.5]{ChodoshSchulze}, \cite[Lemma 2.7]{ChodoshSchulze}, \cite[Lemma 2.8]{ChodoshSchulze}. Thus, we have:

\begin{lemma} \label{lemma:shrinker.geometry.sff.diff}
	The second fundamental form of $\Sigma$ satisfies, for $k \geq 0$,
	\[
		|\nabla^{(k)}_{\cC}(A_{\Sigma}\circ F - A_{\cC})| = O(r^{-3-k})
	\]
	as $r \to \infty$. Here, $F: \mathbf{x} \mapsto \mathbf{x} + w(\mathbf{x}) \nu_{\cC}(\mathbf{x})$ parametrizes the end of $\Sigma$ over $\cC$. 
\end{lemma} 

\begin{corollary} \label{coro:shrinker.geometry.sff.radial}
	The second fundamental form of $\Sigma$ satisfies, for $k \geq 0$,
	\[ |\nabla^{(k)}_\Sigma A_\Sigma(\mathbf{x}^T, \cdot)| = O(r^{-2-k}) \]
	as $r \to \infty$. Here, $\mathbf{x}^T$ is the projection of the ambient position vector $\mathbf{x} \in \Sigma$ to $T_{\mathbf{x}} \Sigma$.
\end{corollary}


\section{Non-standard Schauder estimates} \label{sec:schauder}

We recall the following non-standard Schauder estimate due to Knerr:

\begin{theorem}[{\cite[Theorem 1]{Knerr:Schauder}}] \label{theo:schauder.estimate.knerr}
	Suppose that $B_2 \subset \RR^n$ and we are given coefficients $a_{ij}$, $b_i$, $c : B_2 \times [-2, 0] \to \RR$ and functions $u$, $h : B_2 \times [-2, 0] \to \RR$ so that $u$ is a classical solution of
	\[ \tfrac{\partial}{\partial t} u - a_{ij} \tfrac{\partial^2}{\partial x^i \partial x^j} u - b_i \tfrac{\partial}{\partial x^i} u - cu = h. \]
	Assume
	\[ \sup_{t \in [-2, 0]} \Big[ \Vert a_{ij}(\cdot, t) \Vert_{0,\alpha;B_2} + \Vert b_i(\cdot, t) \Vert_{0,\alpha;B_2} + \Vert c(\cdot, t) \Vert_{0,\alpha;B_2} \Big] \leq \Lambda \]
	and
	\[ a_{ij}(x, t) \xi_i \xi_j \geq \lambda |\xi|^2, \; \forall \xi \in \RR^n, \]
	with $\lambda$, $\Lambda \in (0, \infty)$. Then, for $T \in [-1, 0]$,
	\begin{multline} \label{eq:schauder.estimate.knerr}
		\sum_{j=0}^2 \Vert D_{\mathbf{x}}^j u \Vert_{0,\alpha,\alpha/2;B_1 \times [-1, T]} + \sup_{t \in [-1, T]} \Vert \tfrac{\partial}{\partial t} u(\cdot, t) \Vert_{0,\alpha;B_1} \\
			\leq C \sup_{t \in [-4, T]} \Big[ \Vert u(\cdot, t) \Vert_{0;B_2} + \Vert h(\cdot, t) \Vert_{0,\alpha;B_2} \Big],
	\end{multline}
	where $C = C(n, \alpha, \lambda, \Lambda)$. Here, $\Vert \cdot \Vert_{0,\alpha,\alpha/2}$ denotes the standard parabolic spacetime H\"older norm and $D^j_{\mathbf{x}} u$ denotes the matrix of $j$ partial derivatives in spatial directions.
\end{theorem}

Note that this differs from the standard Schauder estimates because we're only assuming H\"older continuity on $a_{ij}$, $b_i$, $c$, $h$ in the space directions. As a result, we only get a spatial H\"older bound on $\tfrac{\partial}{\partial t} u$. The other H\"older bound remain as in the standard Schauder theory.

We also have the following variant:

\begin{corollary} \label{coro:schauder.estimate.L1}
	Assume the setup of Theorem \ref{theo:schauder.estimate.knerr}. Then, for $T \in [-1, 0]$,
	\begin{multline} \label{eq:schauder.estimate.L1}
		\sum_{j=0}^2 \Vert D_{\mathbf{x}}^j u \Vert_{0,\alpha,\alpha/2;B_1 \times [-1, T]} + \sup_{t \in [-1, T]} \Vert \tfrac{\partial}{\partial t} u(\cdot, t) \Vert_{0,\alpha;B_1} \\
			\leq C \Big[ \Vert u \Vert_{L^1(B_2 \times [-2, T])} + \sup_{t \in [-4, T]} \Vert h(\cdot, t) \Vert_{0,\alpha;B_2} \Big],
	\end{multline}
	where $C = C(n, \alpha, \lambda, \Lambda)$. 
\end{corollary}
\begin{proof}
	For simplicity, let's prove this for $T = 0$. Let us consider the seminorm
	\[ [u]_{B_r \times [-r^2, 0]}^* := [D_{\mathbf{x}}^2 u]_{\alpha,\alpha/2;B_r \times [-r^2, 0]}. \]
	By interpolation and integrating along line segments, we can show that for each $\eps > 0$ there exists $C = C(n, \alpha, \eps)$ such that
	\[ \sup_{t \in [-4, 0]} \Vert u \Vert_{0;B_2} \leq \eps [u]^*_{B_2 \times [-4, 0]} + C \Vert u \Vert_{L^1(B_2 \times [-4, 0])}. \]
	Thus, \eqref{eq:schauder.estimate.knerr} implies
	\begin{equation} \label{eq:schauder.estimate.knerr.interpolated}
		[u]^*_{B_1 \times [-1, 0]} \leq \eps [u]^*_{B_2 \times [-4, 0]} + C \Big[ \Vert u \Vert_{L^1(B_2 \times [-4, 0])} + \sup_{t \in [-4, T]} \Vert h(\cdot, t) \Vert_{0,\alpha;B_2} \Big],
	\end{equation}
	where $C = C(n, \alpha, \eps, \lambda, \Lambda)$. By scaling down to parabolic balls $B_r \times [-r^2, 0]$ and also recentering in space and time, we obtain
	\[ r^{2+\alpha} [u]^*_{B_r(y_0) \times [t_0-r^2, t_0]} \leq \eps r^{2+\alpha} [u]^*_{B_{2r}(y_0) \times [t_0-4r^2, t_0]} + \gamma, \]
	where
	\[ \gamma := C \Big[ \Vert u \Vert_{L^1(B_2 \times [-4, 0])} + \sup_{t \in [-4, T]} \Vert h(\cdot, t) \Vert_{0,\alpha;B_2} \Big] \]
	is just the second term of the right hand side of \eqref{eq:schauder.estimate.knerr.interpolated}. We now apply the  absorption lemma due to L. Simon, \cite[Lemma, p. 398]{Simon:Schauder} on the monotone subadditive function $S(B_{r}(y_0) \times [t_0-r^2, t_0]) := [u]^*_{B_r(y_0) \times [t_0^2-r^2, t_0]}$, with scaling exponent $2+\alpha$. (Note that this monotone subadditive function extends trivially to convex sets.) By L. Simon's absorption lemma, we can choose $\eps$ small enough depending on $n, \alpha$, such that
	\[ [u]^*_{B_1 \times [-1, 0]} \leq C' \Big[ \Vert u \Vert_{L^1(B_2 \times [-4, 0])} + \sup_{t \in [-4, 0]} \Vert h(\cdot, t) \Vert_{0,\alpha;B_2} \Big], \]
	where $C' = C'(n, \alpha, \lambda, \Lambda)$. This yields \eqref{eq:schauder.estimate.L1}: the first summand of the left hand side is obtained by interpolation, and the second by reusing the parabolic PDE.
\end{proof}


\section{Brakke flow uniqueness of regular mean curvature flows}\label{app:uniqueness-BF}
We include here the following uniqueness result for Brakke flows. 
\begin{proposition}\label{prop:unique-cont-brakke}
Suppose that $\cM$ is an integral unit-regular Brakke flow in $\RR^{n+1}\times[t_0,t_2]$ and $[t_0,t_2] \ni t\mapsto M(t)$ is a smooth mean curvature flow with
\[
\lim_{r\to 0} \sup_{\bx \in M(t)} \Theta_{M(t)}((\bx,t),r) = 1
\]
for all $t \in (t_0,t_2]$ and 
\[
\sup_{(\bx,t)\in\RR^{n+1}\times[t_0,t_2]}|A_{M(t)}(\bx)|<\infty. 
\]
If $\cM(t_1) = \cH^n\lfloor M(t_1)$ for some $t_1 \in (t_0,t_2]$ then $\cM(t) = \cH^n\lfloor M(t)$ for all $t \in [t_1,t_2]$. 
\end{proposition}
\begin{proof}
The monotonicity formula and unit regularity property (cf.\ \cite{White:Brakke}) of $\cM$ implies that $\cM$ is the multiplicity one Brakke flow associated to a smooth flow for some interval $t \in [t_1,t_1+\eta]$. As in \cite[Proposition 4.4]{BernsteinWang:TopologicalProperty} (following \cite{EckerHuisken:interior,Ecker:book}), this flow is a smooth graph over $M(t)$ and has bounded curvature; thus $\cM(t) = \cH^n\lfloor M(t)$ by e.g. \cite{ChenYin} or \cite{EckerHuisken:graphs}. 

We can thus conclude via a continuity argument. Let $T \in (t_1,t_2]$ denote the first time the assertion fails. Suppose that $T < t_2$. Unit regularity and the assumptions about $M(t)$ imply that $\cM(T) = \cH^n\lfloor M(T)$. Thus, we can repeat the previous argument to conclude that $\cM(t) = \cH^n\lfloor M(t)$ for $t \in [T,T+\eta']$ for some $\eta'>0$. This is a contradiction, completing the proof. 
\end{proof}

\section{Ilmanen's localized avoidance principle} \label{sec:Ilmanen.avoidance}

In this section we will give a proof of Ilmanen's localized avoidance principle for mean curvature flow. The proof is a parabolic version of the barrier principle and moving around barriers in \cite{Ilmanen:maximum}. 

Let $\Omega$ be an open subset of $\RR^{n+1}\times \RR$, and let $\Gamma \subset  \RR^{n+1}\times \RR$ be relatively closed in $\Omega$. We call $\Gamma$ a \emph{barrier} (resp. \emph{strict barrier}) for mean curvature flow in $\Omega$ provided that, for every smooth open set $E \subset \Omega \setminus \Gamma$ and for every $(\mathbf{x},t) \in \partial E \cap \Gamma \cap \Omega$ with $\nabla_{\partial E}\mathfrak{t}(\mathbf{x},t) \neq 0$, we have
\begin{equation}\label{def:barrier}
f(\mathbf{x},t) \leq \mathbf{H}_{\partial E(t)} \cdot \nu(\mathbf{x},t) 
\end{equation}
(\text{resp.} $f(\mathbf{x},t) < \mathbf{H}_{\partial E(t)} \cdot \nu(\mathbf{x},t)$), where $\mathbf{H}_{\partial E(t)}$ the mean curvature vector of $\partial E(t)$, $\nu(\mathbf{x},t)$ is the inward normal of $\partial E(t)$ at $\mathbf{x}$, and $f \nu$ is the normal speed of the evolution $t\mapsto \partial E(t)$ in a neighborhood of $(\mathbf{x},t)$.

Let $W \subset \RR^{n+1}\times \RR$ be open and let $u:W \rightarrow \RR$ be smooth, positive, bounded and such that $u$ vanishes on $\partial W(t)$ for all $t \in \mathfrak{t}(W)$. For $\mathbf{p}, \mathbf{q} \in W(t)$, define the distance
\begin{equation}\label{eq:ilmanen-distance} d_t(\mathbf{p}, \mathbf{q}) := \inf \Big\{ \int_\gamma u(\gamma(s), t)^{-1} \, ds : \gamma \text{ is a curve joining } \mathbf{p}, \mathbf{q} \text{ in } W(t) \Big\}. 
\end{equation}
We assume that, for each $t \in \mathfrak{t}(W)$, the distance $d_t$ is complete. We use the standard convention that $\inf \emptyset = \infty$. Note that $d_t$ is just the distance in the (complete) conformally Euclidean metric $g_t := u(\cdot,t)^{-2} g_{\RR^{n+1}}$. More generally, we can consider the distance between two closed sets in $W_t$ defined in the usual way. 
For $U \subset W$, $U$ open, define 
\begin{equation*}
U^r = \big\{ (\mathbf{x},t) \in U : d_t(\mathbf{x}, \partial U(t)) > r \big\}\, .
\end{equation*}
Define the degenerate second order elliptic operator
\[ Ku(\mathbf{x},t) = \inf_S \tr_S D^2 u(\mathbf{x},t)\]
where $S$ ranges over all $n$-dimensional subspaces of $\RR^{n+1}$. 

\begin{lemma}\label{lem: barrier} Suppose that $W \setminus U$ is a barrier in $W$ and $u:W \rightarrow \RR$ is as above, with
$$ u_t  - Ku \leq 0\ (\text{resp.} < 0)\, .$$
Then $W\setminus U^r$ is a barrier (resp.~a strict barrier) in $W$. 
\end{lemma}

\begin{proof} Let $E \Subset W$ be a smooth open set with
$$ E \subset U^r \text{ and } (\mathbf{x},t) \in \partial E\cap (W\setminus U^r)\, .$$
We have to show that \eqref{def:barrier} holds. Define 
$$ E^s := \{ (\mathbf{x},t) \in W : d_t(\mathbf{x}, \bar{E}) < s\}, \quad F := E^r\, .$$
Then $\bar{F}$ is compact, $F\subset U$, and $\partial F$ meets $\partial U$. 

We fix $\bm{\gamma}(t)$ to be a shortest $g_t$-geodesic from $\partial E(t)$ to $\partial U(t)$ with endpoints $\mathbf{x} \in E(t)$ and $\mathbf{y} \in \partial F(t) \cap \partial U(t)$. Thus the normal exponential map of $\partial E(t)$ with respect to $g_t$ has no focal points along $\bm{\gamma}(t) \setminus \{\mathbf{y}\}$. Note that this also holds for the normal exponential map of $\partial E(\tau)$ with respect to $g_\tau$ in a spacetime neighborhood of $\bm{\gamma} \setminus \{\mathbf{y}\}$. Therefore in a spacetime neighborhood of $\bm{\gamma} \setminus \{\mathbf{y}\}$, $(\tau,s) \mapsto \partial E^s(\tau)$ is smooth and smoothly varying. For $\tau$ close to $t$, let $\mathbf{x}(\tau)$ be the normal evolution of $\mathbf{x}$ along $\tau \mapsto \partial E(\tau)$ such that $\mathbf{x}(t) = \mathbf{x}$. For $\tau$ close to $t$ we define $\gamma(\tau)$ to be the normal $g(\tau)$-geodesic starting at $x(\tau)$, i.e.~the normal evolution of $\mathbf{x}(\tau)$ along $s \mapsto \partial E^s(\tau)$.  
 Note that
\begin{equation*} 
		1 = g_\tau(\bm{\gamma}'(\tau), \bm{\gamma}'(\tau)) = u^{-2} g_{\RR^{n+1}}(\bm{\gamma}'(\tau), \bm{\gamma}'(\tau)) \implies |\bm{\gamma}'(\tau)|_{g_{\RR^{n+1}}} = u.
\end{equation*}
We denote $\mathbf{x}(\tau,s) = \bm{\gamma}(\tau, s)$ and $f_\tau \nu_{\partial E^s(\tau)}$ to be the normal velocity of the evolution $\tau \mapsto \partial E^s(\tau)$ in $\RR^{n+1}$. Furthermore, note that the $g_\tau$-length of $\bm{\gamma}(\tau, \cdot)$ satisfies 
\[ \ell_{g_\tau}\big(\bm{\gamma}(\tau, [0,s]) \big) = s \]
and thus

\begin{align*} 
0	& = \tfrac{d}{d\tau} \ \ell_{g_\tau}\big(\bm{\gamma}(\tau, [0,s])\big) =  \tfrac{d}{d\tau} \int_{\bm{\gamma}(\tau, [0,s])} d\ell_{g_\tau}
			 = \tfrac{d}{d\tau} \int_{\bm{\gamma}(\tau, [0,s])} u^{-1} \, d\ell_{g_{\RR^{n+1}}} \nonumber \\
			& = - u(\mathbf{x}(\tau,s))^{-1} f(\mathbf{x}(\tau,s),\tau) + u(\mathbf{x}(\tau,0))^{-1} f(\mathbf{x}(\tau,0),\tau) - \int_{\bm{\gamma}(\tau, [0,s])}u^{-2} u_\tau \, d\ell_{g_{\RR^{n+1}}} \\
			&= - u(\mathbf{x}(\tau,s))^{-1} f(\mathbf{x}(\tau,s),\tau) + u(\mathbf{x}(\tau,0))^{-1} f(\mathbf{x}(\tau,0),\tau) - \int_0^su^{-1} u_\tau \, d\ell_{g_\tau}
 \end{align*}
Differentiating the last equation in $s$ yields
 \begin{equation} \label{eq:1} 
 \tfrac{\partial}{\partial s} f = - u_\tau - f Du \cdot  \nu_{\partial E^s(t)}
 \end{equation}
 Similarly, looking at the evolution of $ s \mapsto \partial E^s(t)$ in $\mathbb{R}^{n+1}$ we have
 \begin{equation} \label{eq:2} 
 \begin{split}
  \tfrac{\partial}{\partial s} H_{\partial E^s(t)} & = - \Delta_{\partial E^s(t)} u - |A_{\partial E^s(t)}|^2 u\\
  & = - \tr_S D^2 u  - H_{\partial E^s(t)}  Du \cdot  \nu_{\partial E_t^s} - |A|^2 u\\
  & \leq - K u  - H_{\partial E^s(t)} Du \cdot  \nu_{\partial E^s(t)}
  \end{split}
 \end{equation}
 Combining 
 \eqref{eq:1}, \eqref{eq:2} we see that $\psi := f-H$ satisfies, along $\bm{\gamma}(t)$,
  \begin{equation} \label{eq:3} 
 \tfrac{\partial}{\partial s} \psi \geq - C \psi \, .
  \end{equation}
We first assume that $\mathbf{y}(t)$ is not a focal point of the exponential map of $\partial E(t)$. This implies that $F$ is locally smooth around $y$ and $\nabla_{\partial F}\mathfrak{t}(\mathbf{y},t) \neq 0$. If $\psi(0) >0$ then \eqref{eq:3} implies that $\psi(r) >0$ which gives a contradiction to the assumption that $W \setminus U$ is a barrier. If $\psi(0) \geq 0$ and $ u_t  - Ku < 0 $ then likewise $\psi(r)>0$ which again yields a contradiction, proving that $P \setminus U_r$ is a strict barrier.

If the normal exponential map of $\partial E(t)$ focuses at $\mathbf{y}(t)$, then we may approximate $E$ by $E' \subset E$ such that $E' \cap \partial U_r = \{\mathbf{x}\}$, $\mathbf{y}$ is not a focal point and such that in the above argument we can replace $E$ by $E'$.
\end{proof}

\begin{lemma}\label{weak set flows barriers} If $\Omega$ is an open subset of spacetime and $\cM$ is a closed weak set flow in $\Omega$, then $\cM$ is a barrier in $\Omega$.
\end{lemma}

\begin{proof}
Assume  $E \subset \Omega \setminus \cM$ is open and smooth, and at $(\mathbf{x}_0,t_0) \in \partial E \cap \cM \cap \Omega$ we have
\begin{equation}\label{eq:non.barrier}
\nabla_{\partial E}\mathfrak{t}(\mathbf{x}_0,t_0) \neq 0, \; f(\mathbf{x}_0,t_0) > H_{\partial E(t)}(\mathbf{x}_0,t_0). 
\end{equation}
where $H_{\partial E(t)}(\mathbf{x},t) =  \mathbf{H}_{\partial E(t)} \cdot \nu_{\partial E(t)}$ and $ \nu_{\partial E(t)}$ is the inward pointing unit normal of $\partial E(t)$. We can furthermore assume that $\partial E \cap \cM = \{(\mathbf{x}_0,t_0)\}$. For small $r>0$, \eqref{eq:non.barrier} implies that 
\begin{equation} \label{eq:non.barrier.nearby}
f > H_{\partial E(t)} \text{ on } B_r(\mathbf{x}_0) \times [t_0-r^2,t_0],
\end{equation} and that $\partial E(t)$ is $C^2$-close to an $n$-dimensional plane for all $t \in [t_0-r^2,t_0]$. We can thus solve mean curvature flow $\mathcal{S}=(S(t))_{t \in [t_0-r^2,t_0]}$ with the induced parabolic boundary conditions on $\partial E \cap B_r(\mathbf{x}_0) \times [t_0-r^2,t_0]$. Note that $t \mapsto \partial E(t) \cap B_r(\mathbf{x}_0)$ is a barrier for $\mathcal{S}$ from one side, in view of \eqref{eq:non.barrier.nearby}. Thus $\mathcal{S}$ has to run into $\cM$, contradicting that $\cM$ is a weak set flow. Thus, \eqref{eq:non.barrier} fails, and the result follows.
\end{proof}

We can now state and prove Ilmanen's localized avoidance principle. For $R$, $\alpha\geq0$, and $(\mathbf{x}_{0},t_{0})\in\RR^{n+1}\times\RR$, we set
\begin{equation}\label{eq:defn-ilmanen-avoidance-u}
u_\alpha(\mathbf{x}, t) := (R^{2} - |\mathbf{x}-\mathbf{x}_{0}|^2 - (2n+\alpha)(t-t_{0}))_+ 
\end{equation}
on $\RR \times \RR^{n+1} $. Note that for $\alpha>0$:
\begin{equation} \label{eq:u:differential.inequality}
	\tfrac{\partial}{\partial t} u_\alpha(\mathbf{x},t) < K u_\alpha(\mathbf{x},t).
\end{equation}
for all $(\mathbf{x},t)$ with $u_\alpha(\mathbf{x},t)>0$.

\begin{theorem}[Ilmanen]\label{theo:ilmanen-avoidance}
Consider two closed weak set flows $\cM$, $\cM'$ in $\mathbb{R}^{n+1}$ and constants satisfying $R>0,\gamma>0$, $a<b<a + \frac{R^{2}-\gamma}{2n}$. Assume that
\[
\cM(t) \cap B_{\sqrt{\gamma + R^{2} - 2n(t-a)}}(\mathbf{x}_{0}) \quad \textrm{and} \quad \cM'(t) \cap B_{\sqrt{\gamma + R^{2} - 2n(t-a)}}(\mathbf{x}_{0})
\]
are disjoint for $t \in [a,b)$. Then, using this choice of $R$ and $\mathbf{x}_{0}$ along with $t_{0}=a$ and $\alpha=0$ in \eqref{eq:defn-ilmanen-avoidance-u}, we have that $t\mapsto d_{t}(\cM(t),\cM'(t))$ is non-decreasing for $t \in [a,b)$ and
\[
\cM(b) \cap \cM'(b) \cap B_{\sqrt{R^{2} - \gamma - 2n(b-a)}}(\mathbf{x}_{0}) = \emptyset.
\]
\end{theorem}

Before proving Theorem \ref{theo:ilmanen-avoidance}, let us indicate how we plan to apply it. If $\cM(a)$, $\cM'(a)$ are disjoint and one knows \emph{a priori} that 
\[
\cM(t) \cap \cM'(t) \cap (B_{\sqrt{R^{2} + \gamma - 2n(t-a)}}(\mathbf{x}_{0}) \setminus B_{\sqrt{R^{2} - \gamma - 2n(t-a)}}(\mathbf{x}_{0})) = \emptyset.
\]
for $t \in [a,b]$, then Theorem \ref{theo:ilmanen-avoidance} and a straightforward continuity argument imply that
\[
\cM(t) \cap \cM'(t) \cap B_{\sqrt{R^{2} - \gamma - 2n(t-a)}}(\mathbf{x}_{0}) = \emptyset
\]
for $t \in [a,b]$. In other words, if the two weak set flows are disjoint near the boundary of the comparison region, then they remain disjoint. 

\begin{proof}[Proof of Theorem \ref{theo:ilmanen-avoidance}] We first note that the assumptions imply that for suffciently small $\alpha>0$, the distance $d^\alpha_t$ with respect to $(u_\alpha(\cdot,t))^{-2} g_{\RR^{n+1}}$ between $\cM$ and $\cM'$ is attained away from the boundary of the set $W:=\{u_\alpha(\mathbf{x},t) > 0\}$ for all $t \in [a,b)$. Assume that $d^\alpha_a(\cM,\cM')= r>0$. We can thus argue as in \cite[$C^{1,1}$ Interposition Lemma]{Ilmanen:levelset} find a $C^{1,1}$ hypersurface $\Gamma$ in $(\{u_\alpha(\mathbf{x},a)>0\}, g_a)$ separating $\cM$ and $\cM'$, such that $d^\alpha_a(\cM,\Gamma) = d^\alpha_a(\cM',\Gamma) = r/2$, with both distances attained away from the boundary of  $W$. Consider $\Gamma' = \Gamma \cap B_{R - \eta}(\mathbf{x}_0)$ for suitable small $\eta >0$ such that $\Gamma'$ has smooth boundary. Solve smooth mean curvature flow $\Gamma'_t$ starting at $\Gamma'$ with fixed Dirichlet boundary conditions for $a\leq t < a+ \varepsilon$ for small $\varepsilon >0$. We can assume that $d^\alpha_t(\cM,\Gamma'_t)$ and $d^\alpha_t(\cM',\Gamma'_t)$ are attained away from the boundary $\partial \Gamma'$ for all  $a\leq t < a+ \varepsilon$.  Choose $U = W \setminus \cM$. Note that for $0<s<r/2$ we have $\partial U^s(a) \cap \Gamma' = \emptyset$. So by Lemma \ref{lem: barrier} and Lemma \ref{weak set flows barriers}, together with \eqref{eq:u:differential.inequality} we have that $\Gamma'(t) \subset U^s(t)$ for all $0<s<r/2$ and $a\leq t < a+ \varepsilon$. This implies that the distance $d^\alpha_t$ between $\cM$ and $\Gamma'_t$ is non-decreasing on $[a,a+\varepsilon)$. We can argue similarly to see that the distance $d^\alpha_t$ between $\cM'$ and $\Gamma'_t$ is non-decreasing on $[a,a+\varepsilon)$. Thus $d^\alpha_t(\cM,\cM')$ is non-decreasing 
for $t \in [a,a+\varepsilon)$. But the monotonicity formula implies that $d^\alpha_{a+\varepsilon}(\cM,\cM')\geq \limsup_{t\to (a+\varepsilon)^+} d^\alpha_{t}(\cM,\cM')$. Thus a direct continuity argument implies that $d^\alpha_t(\cM,\cM')$ is non-decreasing 
for $t \in [a,b)$. Letting $\alpha \rightarrow 0$ gives the result.
\end{proof}

We note that this implies a well-known Frankel property for self-shrinkers. For completeness, we state our result in full generality, in the context of $F$-stationary varifolds, i.e., varifolds in $\RR^{n+1}$ that are stationary for the conformally  Euclidean metric in Section \ref{sec:prelim.shrinkers} whose stationary points coincide with self-shrinkers.

\begin{corollary}[Frankel property for shrinkers]\label{coro:Frankel}
If $V,V'$ are $F$-stationary varifolds, then $\supp V \cap \supp V' \not = \emptyset$.
\end{corollary}
\begin{proof}
If $\supp V \cap \supp V' = \emptyset$, then the associated self-similarly shrinking Brakke flows $\cM$, $\cM'$ satisfy 
\[ \supp\cM(t) \cap \supp\cM'(t) = \emptyset, \; t < 0. \]
Applying Theorem \ref{theo:ilmanen-avoidance} with $a=-1$, $b=0$, $R > \sqrt{2n}$, and recalling that the support of the spacetime track of a Brakke flow is a weak set flow \cite[10.5]{Ilmanen:elliptic} we arrive at a contradiction; indeed, $\mathbf{0} \in \supp \cM(0) \cap \supp \cM'(0)$. 
\end{proof}


\section{The Ecker--Huisken maximum principle}\label{sec:Ecker-Huisken maximum principle}

For the reader's convenience, we recall here a special case of the variant of the Ecker--Huisken maximum principle (see \cite{EckerHuisken:graphs}) proven in \cite{BernsteinWang:TopologicalProperty}, which we're going to make use of:
\begin{theorem}[{\cite[Theorem A.1]{BernsteinWang:TopologicalProperty}}]\label{theo:ecker-huisken}
Suppose that $\{\Sigma_t\}_{t \in [a,b)}$ is a smooth mean curvature flow in $\RR^{n+1} \setminus B_R$, with $\partial\Sigma_t\subset \partial B_R$. Assume that $u$ is a $C^2$ function on $\Sigma_t$ so that 
\begin{enumerate}
\item it satisfies
\[
\left(  \tfrac{\partial}{\partial t} - \Delta_{\Sigma_t}\right) u \geq \mathbf{a} \cdot \nabla_{\Sigma_t} u + b u
\]
with $\sup_{t\in[a,b)}\sup_{\Sigma_t} |\mathbf{a}| + |b| < \infty$,
\item $u > 0$ on the parabolic boundary $\Sigma_a \cup (\cup_{t\in [a,b)} \partial\Sigma_t)$,
\item for all $t \in [a,b)$, and
\[
\int_{\Sigma_t} \left( |u|^2 + |\tfrac{\partial u}{\partial t}|^2 + |\nabla_{\Sigma_t} u|^2 + |\nabla^2_{\Sigma_t}u|^2 \right) \rho_{(0,b)}d\cH^n < \infty
\]
where
\[
\rho_{(0,b)}(\mathbf{x},t) = (4\pi(b-t))^{-\frac n2} e^{-\frac{|\mathbf{x}|^2}{4(b-t)}}.
\]
\end{enumerate}
Then, for all $t \in [a,b)$, $\inf_{\Sigma_t} u \geq 0$. 
\end{theorem}


\section{Weak set flows of cones} \label{app:weak-set-flows} 

For this appendix, the reader might find it useful to recall the notions set forth in Section \ref{sec:prelim}. We collect results of \cite{HershkovtisWhite} on weak set flows and outermost flows and show that they are also applicable (with minor modifications) to the flow of hypercones. 

\begin{proposition}[{\cite[Proposition A.3]{HershkovtisWhite}}]\label{thm:app.1}
Suppose that $F$ is any closed subset of $\RR^{n+1}$, and let $\cM \subset \RR^{n+1} \times \RR^+$ be its level set flow. Set:
$$ M(t) := \{\mathbf{x} \in \RR^{n+1} : (\mathbf{x},t) \in \partial \cM \}\, .$$
Then $t \mapsto M(t)$ is a weak set flow.
\end{proposition}

In what follows, we consider $F$ to be the closure of its interior in $\RR^{n+1}$ and satisfy
\[
\partial F = \partial F^c.
\]
We call such a set $F$  \emph{admissible}.\footnote{Note that this slightly extends  the definition in \cite{HershkovtisWhite}, where $\partial F$ ($\partial U$ in their notation) would be a compact, smooth hypersurface. This extension allows us to flow from non-compact and non-smooth initial surfaces. This does not change anything in the analysis of \cite{HershkovtisWhite}.} 
Let $F':= \overline{F^c}$, denote the level set flows of $F$, $F'$ by $\cM$, $\cM'$, and set $F(t):= \cM(t)$, $F'(t):= \cM'(t)$. In line with Proposition \ref{thm:app.1}, we set:
\begin{align*}
M(t) &:= \{ (\mathbf{x},t) \subset \RR^{n+1} : \mathbf{x} \in \partial \cM\},\\
M'(t) &:= \{ (\mathbf{x},t) \subset \RR^{n+1} : \mathbf{x} \in \partial \cM'\}.
\end{align*}
(Here $\partial \cM$, $\partial \cM'$ are the relative boundaries of $\cM$, $\cM'$ as subsets of $\RR^{n+1}\times \RR^+$). We call
\[ t \mapsto M(t), \; t \mapsto M'(t) \]
the \emph{outer} and \emph{inner} flows of $M:= \partial F$. By Proposition \ref{thm:app.1}, $M(t)$, $M'(t)$ are contained in the level set flow generated by $M$. Furthermore,
\[
M(t) = \lim_{\tau \nearrow t} \partial F(\tau) 
\]
for all $t > 0$, and $M(t) = \partial F(t)$ for all but countably many $t$. See \cite[Theorems B.2,  C.1]{HershkovtisWhite}. Note that \cite[Theorems B.2]{HershkovtisWhite} directly carries over to $M = \partial F$ where $F$ is admissible. 

Let $\Gamma \subset \SS^{n}$ denote a fixed smooth, embedded, closed hypersurface. Consider the equidistant deformations  $(\Gamma_s)_{-\varepsilon<s<\varepsilon}$ of $\Gamma \subset \SS^{n}$ for some consistent choice of normal orientation. We further consider the regular hypercone $C=C(\Gamma)$ and the smooth perturbations $C_s= C(\Gamma_s)$. Note that $C_s$ divides $\RR^{n+1}$ into two open sets $\Omega_s^\pm$ such that $C_s = \partial \Omega_s^\pm$ as well as $C(\Gamma) \cap \Omega_s^+ = \emptyset$ for $s>0$ and $C(\Gamma) \cap \Omega_s^- = \emptyset$ for $s<0$. We now consider, 
\[
\Sigma_{s,r} := \partial \big( \Omega^+_s \setminus B_r(0)),
\] 
for $0<r<1$ and $s > 0$. We denote with $\tilde{\Sigma}_{s,r}$ a smoothing of ${\Sigma}_{s,r}$ that rounds off the corners near $\partial B_r(0)$.  Similarly we set:
\[
\Sigma'_{s,r} := \partial \big( (\Omega^+_s \cap B_{1/r}(0)) \setminus B_r(0)\Big)
\]
for $0<r<1$, $s>0$, and $\tilde{\Sigma}_{s,r}$ to be a smoothing of ${\Sigma}'_{s,r}$ that rounds off its corners. Note that by using the smoothings $\tilde{\Sigma}'_{s,r}$ we can construct compact regions $F_i \subset \Omega^+$ with smooth boundaries such that 
\begin{itemize}
\item[(1)] For each $i$, $F_{i}$ is contained in the interior of $F_{i+1}$.
\item[(2)] $\cup F_i = \Omega^+$.
\item[(3)] $\cH^{n}\lfloor \partial F_i \rightarrow \cH^{n}\lfloor C(\Gamma)$. 
\end{itemize}
By perturbing $F_i$ slightly, we can also assume that
\begin{itemize}
\item[(4)] the level set flow of $\partial F_i$ never fattens.
\end{itemize}
We then directly generalize \cite[Theorems B.3,  B.5]{HershkovtisWhite}. The proof extends verbatim.

\begin{theorem}\label{thm:app.2} There is an integral unit-regular Brakke flow $t \in [0,\infty) \mapsto \mu(t)$ such that $\mu(0) = \cH^{n} \lfloor C(\Gamma)$ and such that the spacetime support of the flow is the spacetime set swept out by $t \in [0,\infty) \mapsto M(t)$, where $t  \mapsto M(t)$ is the outer flow of $C(\Gamma)$. That is, for $t>0$, the Gaussian density of the flow $\mu(\cdot)$ at $(\mathbf{x},t)$ is $>0$ if and only if $\mathbf{x} \in M(t)$.
\end{theorem}

\begin{remark} \label{rem:app.3}
(i) Note that by uniqueness of the level set flow, the outer flow satisfies $M(t) = \sqrt{t} M(1)$. Together with unit regularity and White's local regularity theory this implies that there is a $R_0=R_0(\Gamma)$ such that the Brakke flow constructed in  Theorem \ref{thm:app.2} is a smooth expanding solution, agreeing with $M(t)$, outside of $B_{\sqrt{t} R_0}(0)$ for all $t>0$. \\[1ex]
(ii) Theorem \ref{thm:app.2}, and all of the above, applies also to the inner flow.
\end{remark}


\section{Brakke flows with small singular set} \label{app:reg.set.conn}

In this section we show that if a Brakke flow has small singular set, then the regular set is connected, provided it is connected in a neighborhood of the initial time. To prove this, we show that for a closed set $S \subset \RR^{n+k}\times \RR$, a Brakke flow (with bounded area ratios) on $(\RR^{n+k}\times \RR)\setminus S$ extends across $S$ provided $S$ has vanishing $n$-dimensional parabolic Hausdorff measure.\footnote{See also \cite[Appendix D]{CESY} where it is shown that an integral $2$-dimensional Brakke flow in $\RR^3\setminus\{0\}$ with bounded area ratios extends across the origin.} 

\begin{remark}
In \cite[Claim 8.4]{ChoiHaslhoferHershkovits} it was observed that the classification of low entropy ancient flows implies connectivity of the regular part of a flow in $\RR^3$ with only (multiplicity-one) spherical and cylindrical singularities, by an argument similar to Kleiner--Lott's proof \cite[Theorem 7.1]{KleinerLott} that a singular Ricci flow of $3$-manifolds has only finitely many bad world lines.  We show here that one can prove connectivity under considerably weaker hypothesis. We note that our approach has no hope of estimating the number of bad world lines. It would be interesting to study the Hausdorff dimension of bad world lines in a $k$-convex mean curvature flow in $\RR^{n+1}$. 
\end{remark}

We first recall a well known extension theorem for varifolds, originally considered by de~Giorgi--Stampacchia \cite{DeGiorgiStampacchia}.

\begin{lemma}\label{lem:varifold-extension}
Let $V$ be a rectifiable $n$-varifold in $\RR^{n+k}$ with bounded area ratios, i.e.,  $\|V\|(B_r(\bx)) \leq C r^n$. If $S\subset \RR^{n+k}$ is closed, $\cH^{n-1}(S) =0$, and the restricted varifold $V' : = V\lfloor (\RR^{n+k}\setminus S)$ has absolutely continuous first variation $\bH' \in L^1_\textnormal{loc}(\RR^{n+k};\mu_{V'})$, then $V$ has absolutely continuous first variation equal to $\bH'$, too.
 \end{lemma}
 
\begin{proof}
Without loss of generality, we assume that $S$ is compact. For $\delta>0$, we can find balls $\{B_{r_i}(\bx_i)\}_{i=1}^N$ covering $S$ with
\[
\sum_{i=1}^N r_i^{n-1} < \delta. 
\]
Choose cut-off functions $0\leq\xi_i\leq1$ with $\xi_i \equiv 1$ outside of $B_{2r_i}(\bx_i)$, $\xi_i \equiv 0$ on $B_{r_i}(\bx_i)$, and $|D \xi_i| \leq \tfrac{2}{r_i}$. Then, set $\xi_\delta = \Pi_{i=1}^N \xi_i$ and note that 
\[
|D\xi_\delta| \leq \sum_{i=1}^N \frac{2}{r_i} \chi_{B_{r_i}(\bx_i)} 
\]
For a vector field $\Xi \in C^1_c(\RR^{n+1})$, we have
\[
\int \xi_\delta \Div_M \Xi \, d\mu_V + \int \xi_\delta \Xi \cdot \bH' \, d\mu_{V'}  = - \int D^T \xi_\delta \cdot \Xi \, d\mu_{V'}\]
Note that
\[ \left| \int D^T \xi_\delta \cdot \Xi \, d\mu_{V'}\right|  \leq \sum_{i=1}^N \frac{2}{r_i} \Vert V\Vert(B_{r_i}(\bx_i)) \Vert \Xi\Vert_{L^\infty} \leq C \delta. \]
Sending $\delta\to 0$, the dominated convergence theorem implies 
\[ \int \Div_M \Xi \, d\mu_V = - \int \Xi \cdot \bH' \, d\mu_{V'} .\]
Thus, $\delta V$ is absolutely continuous with respect to $d\mu_V$ and, since $\mu_V(S) = 0$, the generalized mean curvature of $V$ also equals $\bH'$. This completes the proof.
\end{proof}

We now extend this to Brakke flows (recall our conventions in Section \ref{sec:prelim.brakke.flow}).

\begin{theorem}\label{thm:Brakke-flow-extension}
Consider $(\mu(t))_{t\in I}$ be a $1$-parameter family of Radon measures on $\RR^{n+k}$ and $S \subset \RR^{n+k}\times \RR$ a closed set with $\cH^n_P(S) = 0$. Assume that 
\begin{enumerate}
\item The measures $\mu(t)$ have uniformly bounded area ratios, i.e., $\mu(t)(B_r(\bx)) \leq C r^n$. 
\item For almost every $t \in I$, there exists an integral $n$-dimensional varifold $V(t)$ with $\mu(t) = \mu_{V(t)}$ so that $V'(t) = V(t) \lfloor (\RR^{n+k}\setminus S(t))$ has absolutely continuous first variation in $L^1_\textnormal{loc}(\RR^{n+k};d\mu_{V'(t)})$ and has mean curvature $\bH$ orthogonal to $\textrm{Tan}(V'(t),\cdot)$ almost everywhere. 
\item For any compact set $K \subset (\RR^{n+k}\times \RR)\setminus S$, we have
\[
\int_K  (1+|\bH|^2) d\mu(t) dt < \infty
\]
\item If $[t_1,t_2]\subset I$ and $f \in C^1_c((\RR^{n+k}\times [t_1,t_2]) \setminus S)$ has $f\geq 0$, then 
\[
\int f(\cdot,t_{2}) \, d\mu(t_{2}) - \int f(\cdot,t_{1}) \, d\mu(t_{1}) \leq \int_{t_{1}}^{t_{2}} \int\left( - |\bH|^{2} f + \bH \cdot \nabla f + \tfrac{\partial }{\partial t} f \right) \, d\mu(t) \, dt.
\]
\end{enumerate}
Then $(\mu(t))_{t\in I}$ is a Brakke flow on $\RR^{n+k}$. 
\end{theorem}
\begin{proof}
It suffices to prove this for $S$ compact. We begin by defining the relevant cutoff function. Choose a family of parabolic balls
\[
P_{r_i}(\bx_i,t_i) = B_{r_i}(\bx_i) \times (t_i-r_i^2,t_i+r_i^2)
\]
where $i=1,\dots,N$, so that $S \subset \cup_{i=1}^N P_{r_i}(\bx_i,t_i)$ and 
\[
\sum_{i=1}^N r_i^n < \delta. 
\]
For each parabolic ball, choose a cutoff function $0\leq\zeta_i\leq1$ so that $\zeta_i \equiv 1$ on $P_{2r_i}(\bx_i,t_i)$ and $\zeta_i \equiv 0$ on $P_{r_i}(\bx_i,t_i)$. We can assume that $|D\zeta_i| \leq C/r_i$ and $|\tfrac{\partial}{\partial t} \zeta_i|\leq C/r_i^2$. Set 
\[
\zeta_\delta = \min_i \zeta_i
\]
and define a mollified function $\zeta_{\delta,\eps}$ as follows. Choose $0\leq \varphi_1,\varphi_{n+k}\leq 1$ standard mollifiers on $\RR,\RR^{n+k}$ and set 
\[
\zeta_{\delta,\eps}(\bx,t) = \int_{\RR^{n+k}\times \RR}  \eps^{-n-k-2} \varphi_{n+k}(\eps^{-1}(\bx-\by)) \varphi_1(\eps^{-2}(t-s)) \zeta_\delta(\by,s) \, d\by ds.
\]
We now estimate the derivatives of $\zeta_{\delta,\eps}$. 
\begin{claim}
\begin{equation}\label{eq:cutoff-time-der}
\limsup_{\eps\to 0} |\tfrac{\partial}{\partial t} \zeta_{\delta,\eps}(\bx,t)| \leq  C \sum_{i=1}^N \frac{1}{r_i^2} \chi_{P_{2r_i}(\bx_i,t_i)} 
\end{equation}
\end{claim}
\begin{proof}[Proof of \eqref{eq:cutoff-time-der}]
We have 
\begin{align*}
& \left| \frac{\partial}{\partial t} \zeta_{\delta,\eps}(\bx,t) \right| \\
& = \left| \int_{\RR^{n+k}\times \RR}  \eps^{-n-k-4} \varphi_{n+k}(\eps^{-1}(\bx-\by)) \varphi_1'(\eps^{-2}(t-s)) \zeta_\delta(\by,s) \, d\by ds\right|\\
& = \left| \int_{\RR^{n+k}\times \RR}  \eps^{-n-k-4} \varphi_{n+k}(\eps^{-1}(\bx-\by)) \varphi_1'(\eps^{-2}(t-s)) (\zeta_\delta(\by,s) - \zeta_\delta(\by,t))\, d\by ds \right|\\
& \leq C \max_i \sup_{(\by,s) \in P_\eps(\bx,t)} |\tfrac{\partial}{\partial t} \zeta_i(\by,s) | \\
& \leq C \sum_{i=1}^N \sup_{(\by,s) \in P_\eps(\bx,t)} |\tfrac{\partial}{\partial t} \zeta_i(\by,s) |,
\end{align*}
which implies the inequality follows after sending $\eps\to 0$. 
\end{proof}
\begin{claim}
\begin{equation}\label{eq:cutoff-spatial-der}
\limsup_{\eps\to 0} |D \zeta_{\delta,\eps}(\bx,t)|^2 \leq C \sum_{i=1}^N \frac{1}{r_i^2} \chi_{P_{2r_i}(\bx_i,t_i)}.
\end{equation}
\end{claim}
\begin{proof}[Proof of \eqref{eq:cutoff-spatial-der}]
As in the proof of \eqref{eq:cutoff-time-der}, we find 
\begin{align*}
& \left| D \zeta_{\delta,\eps}(\bx,t) \right|^2 \leq C \max_i \sup_{(\by,s) \in P_\eps(\bx,t)} |D \zeta_i(\by,s) |^2  \leq C \sum_{i=1}^N \sup_{(\by,s) \in P_\eps(\bx,t)} |D \zeta_i(\by,s) |^2 
\end{align*}
This implies the claim, as before. 
\end{proof}

Now, for $0\leq f \in C^2_c(\RR^{n+k}\times[t_1,t_2])$ we consider $\zeta_{\delta,\eps}^2 f$ in (4) above. We find
\begin{align*}
& \int \zeta_{\delta,\eps}(\cdot,t_2)^2 f(\cdot,t_2) d\mu(t_2) - \int \zeta_{\delta,\eps}(\cdot,t_1)^2 f(\cdot,t_1) d\mu(t_1)\\
& \leq \int_{t_1}^{t_2} \int \left( - |\bH|^2 \zeta_{\delta,\eps}^2 f + \zeta_{\eps,\delta}^2 \bH \cdot \nabla f + \zeta_{\eps,\delta}^2 \tfrac{\partial}{\partial t} f\right) d\mu(t) dt\\
& + \int_{t_1}^{t_2} \int \left(2 \zeta_{\eps,\delta} f \bH \cdot \nabla \zeta_{\eps,\delta} + f \tfrac{\partial}{\partial t} \zeta_{\eps,\delta}^2\right) d\mu(t) dt\\
& \leq \int_{t_1}^{t_2} \int \left( - (1-\gamma) |\bH|^2 \zeta_{\delta,\eps}^2 f + \zeta_{\eps,\delta}^2 \bH \cdot \nabla f + \zeta_{\eps,\delta}^2 \tfrac{\partial}{\partial t} f\right) d\mu(t) dt\\
& + \int_{t_1}^{t_2} \int \left( \gamma^{-1}|\nabla \zeta_{\eps,\delta}|^2 + f \tfrac{\partial}{\partial t} \zeta_{\eps,\delta}^2\right) d\mu(t) dt\\
& \leq \int_{t_1}^{t_2} \int \left( - (1-\gamma) |\bH|^2 \zeta_{\delta,\eps}^2 f + \zeta_{\eps,\delta}^2 \bH \cdot \nabla f + \zeta_{\eps,\delta}^2 \tfrac{\partial}{\partial t} f\right) d\mu(t) dt\\
& + C \gamma^{-1} \Vert f \Vert_{C^1} \int_{t_1}^{t_2} \sum_{i=1}^N \frac{1}{r_i^2} \mu(t)(B_{r_i}(\bx_i)) \chi_{(t_i-r_i^2,t_i+r_i^2)}(t) dt\\
& \leq \int_{t_1}^{t_2} \int \left( - (1-\gamma) |\bH|^2 \zeta_{\delta,\eps}^2 f + \zeta_{\eps,\delta}^2 \bH \cdot \nabla f + \zeta_{\eps,\delta}^2 \tfrac{\partial}{\partial t} f\right) d\mu(t) dt\\
& + C\gamma^{-1}  \delta \Vert f \Vert_{C^1} \\
& \leq \int_{t_1}^{t_2} \int \left( - (1-2\gamma) |\bH|^2 \zeta_{\delta,\eps}^2 f + \zeta_{\eps,\delta}^2 \tfrac{\partial}{\partial t} f\right) d\mu(t) dt\\
& + C\gamma^{-1}  \delta \Vert f \Vert_{C^1} + C \gamma^{-1} \Vert D^2 f\Vert_{L^\infty}.
\end{align*}
In the final inequality, we have used \cite[Lemma 6.6]{Ilmanen:elliptic}.

Sending $\delta\to 0$, we can use Lemma \ref{lem:varifold-extension} (and Lemma \ref{lemm:slicing-par-haus} below) to conclude that for almost every $t$, the varifold $V(t)$ has absolutely continuous first variation in $L^1_\textrm{loc}(\RR^{n+k},d\mu(t))$ and that
\[
\int_{t_1}^{t_2} \int_K (1+|\bH|^2) d\mu(t) dt < \infty
\]
for any compact set $K$ and $[t_1,t_2]\subset I$. Then, dominated convergence and the above inequality guarantees 
\begin{align*}
& \int   f(\cdot,t_2) d\mu(t_2) - \int   f(\cdot,t_1) d\mu(t_1)\\
& \leq \int_{t_1}^{t_2} \int \left( - (1-\gamma) |\bH|^2  f +  \bH \cdot \nabla f + \tfrac{\partial}{\partial t} f\right) d\mu(t) dt,
\end{align*}
which implies (4) after sending $\gamma \to 0$. This completes the proof, after observing that $0\leq f\in C^1_c$ can be approximated by $0\leq f\in C^2_c$.
\end{proof}
\begin{lemma}\label{lemm:slicing-par-haus}
 Suppose that $S\subset \RR^{n+k}\times \RR$ is a closed set with $\cH^n_P(S) = 0$. Then for almost every $t$, 
 \[
 \cH^{n-2}(S(t)) =0,
 \]
 where $S(t) = S \cap \mathfrak{t}^{-1}(t)$. 
\end{lemma}
\begin{proof}
As usual, we can assume that $S$ is compact. Choose parabolic balls $P_{r_i}(\bx_i,t_i)$ covering $S$ with $r_i < \delta$ and $\sum_i r_i^n < \delta$. Set $\cI(t) : = \{i : t \in (t_i-r_i^2,t_i+r_i^2)\}$ and note that
\[
S(t) \subset \bigcup_{i\in \cI(t)} B_{r_i}(\bx_i). 
\]
Note that
\[
\int_{t_1}^{t_2} \sum_{i\in \cI(t)} r_i^{n-2} dt = \int_{t_1}^{t_2} \sum_{i} r_i^{n-2}\chi_{(t_i-r_i^2,t_i+r_i^2)}(t) dt = 2\sum_i r_i^n < 2\delta.
\]
This proves that 
\[
|\{ t \in [t_1,t_2] : \cH^{n-2}_\delta(S(t)) > \eps\}| < C \frac{\delta}{\eps}. 
\]
Because $\cH_\delta^{n-2}(S(t))$ is non-decreasing as $\delta\searrow0$, we thus see that 
\[
|\{ t \in [t_1,t_2] : \cH^{n-2}(S(t)) > \eps\}| =0.
\]
Sending $\eps\to 0$ completes the proof. 
\end{proof}

For a Brakke flow $\cM$, define $\widehat{\reg} \cM$ to be the set of points $(\bx,t)$ so that there is $\eps>0$ with
\[
\cM\lfloor (B_\eps(\bx) \times (t-\eps^2,t]) = k \cH^n\lfloor M(t),
\]
where $k$ is a positive integer and $M(t)$ is a smooth mean curvature flow. Note that points in $\reg \cM$ are defined similarly, but with $k=1$; thus, $\reg \cM \subset \widehat{\reg}\, \cM$.

\begin{corollary}\label{cor:regset-connectedness}
Consider $\cM = (\mu(t))_{t\in I}$ a unit-regular integral $n$-dimensional Brakke flow in $\RR^{n+k}$ with $\mu(t) = \cH^n\lfloor M(t)$ for $t \in [0,\delta)$, where $M(t)$ is a mean curvature flow of connected, properly embedded submanifolds of $\RR^{n+k}$ and $\delta>0$. If 
\[
\cH^n_P(\supp \cM\setminus \widehat{\reg} \cM) = 0,
\]
then $\widehat{\reg}\, \cM = \reg\cM$ and $\reg\cM$ is connected. 
\end{corollary}
\begin{proof}
We claim that $\mathfrak{M}(0) \not = \emptyset$ for any  a connected component, $\mathfrak{M}$, of $\widehat{\reg}\, \cM$. From this, we immediately have that the multiplicity on this component is $k=1$, so $\widehat{\reg}\, \cM = \reg\cM$. Moreover, since $M(t)$ is connected for $t \in [0,\delta)$, we also will have $\reg\cM$ is connected. 

Now, consider $\mathfrak{M}$ as above. Set $\widehat{\mathfrak{M}} : = \mathfrak{M} \cup (\supp \cM\setminus \widehat{\reg}\, \cM)$. Theorem \ref{thm:Brakke-flow-extension} implies that $\mu(t) \lfloor \widehat{\mathfrak{M}}(t)$ is a Brakke flow. However, if $\mathfrak{M}(0) =\emptyset$, then we can apply Huisken's monotonicity formula to conclude that $\mu(t) \lfloor \widehat{\mathfrak{M}}(t) = 0$ for all $t$. This is a contradiction, completing the proof. 
\end{proof}

Combining White's parabolic stratification \cite[Theorem 9]{White:stratification} with the previous corollary this implies:

\begin{corollary}\label{cor:connected-reg-part}
Suppose that $\cM$ is a unit-regular integral $n$-dimensional Brakke flow in $\RR^{n+k}$ with $\mu(t) = \cH^n\lfloor M(t)$ for $t \in [0,\delta)$, where $M(t)$ is a mean curvature flow of connected, properly embedded submanifolds of $\RR^{n+k}$ and $\delta>0$. Assume that $\cM$ has the following properties:
\begin{enumerate}
\item If there is a static or quasi-static planar tangent flow at $X$, then $X \in \widehat{\reg}\, \cM$. 
\item There are no static or quasi-static tangent-flows supported on a union of half-planes or polyhedral cones. 
\end{enumerate}
Then $ \widehat{\reg}\, \cM=\reg\cM$ is connected. 
\end{corollary}


\section{Localized topological monotonicity} \label{app:loc-top-monotonicity}

In this appendix we localize some of the results from \cite{White:topology-weak}. We say a closed subset $\cM$ of a spacetime $\RR^{n+1} \times \RR$ is a \emph{simple flow} in an open set $U \subset \RR^{n+1}$ with smooth boundary and over a time interval $I \subset \RR$, or a simple flow in $U \times I$ for short, 
if there is a compact $n$-manifold $M$, with or without boundary, and a continuous map 
\[
f : M \times I \to \RR^{n+1}
\]
so that:
\begin{enumerate}
\item $\cM(t) \cap \overline{U} = f(M, t)$, where $\cM(t) := \{ \mathbf{x} \in \RR^{n+1} : (\mathbf{x}, t) \in \cM \}$,
\item $f$ is smooth on $M^\circ \times I$, where  $M^\circ := M \setminus \partial M$,
\item $f(\cdot,t)$, $t \in I$, is an embedding of $M^\circ$ into $U$, 
\item $t \mapsto f( M^\circ, t)$, $t \in I$, is a smooth mean curvature flow: $(\frac{\partial }{\partial t}f(\cdot,t))^\perp = \mathbf{H}(\cdot,t)$, and 
\item $f|_{\partial M \times I}$ is a smooth family of embeddings of $\partial M$ into $\partial U$. 
\end{enumerate}

The following lemma is easily proven but we will use it repeatedly in the sequel.
\begin{lemma}\label{lemm:top-restr-straighten-simple-flow}
If $\cM \subset \RR^{n+1} \times \RR$ is a simple flow in $U\times [0,T]$ then 
 we have a diffeomorphism 
\[ (U \times [0,T]) \setminus \cM \approx (U \setminus \cM(0)) \times [0,T] \]
that restricts to diffeomorphisms $U \setminus \cM(t) \approx U \setminus \cM(0)$ along each fibre.
\end{lemma}

We recall some definitions from \cite{White:topology-weak}. For $\cM \subset \RR^{n+1} \times [0,T]$, $t \in [0, T]$, we set:
\begin{align}
W[t] & := \cM^c \cap \mathfrak{t}^{-1}(\{t\}), \label{eq:localized.monotonicity.w.t} \\
W[0,T] & := \cM^c \cap \mathfrak{t}^{-1}([0,T]). \label{eq:localized.monotonicity.w.0T}
\end{align} 
The results of \cite{White:topology-weak} apply precisely to these $W[t]$, $W[0,T]$. Since we wish to localize some of these results to open subsets $\Omega \subset \RR^{n+1}$ with smooth boundary, we introduce the following localized objects. 
\begin{align}
W_\Omega[t]  &:= \cM^c \cap \Omega \cap \mathfrak{t}^{-1}(\{t\}), \label{eq:localized.monotonicity.w.omega.t} \\
W_\Omega[0,T] & := \cM^c \cap \Omega \cap \mathfrak{t}^{-1}([0,T]). \label{eq:localized.monotonicity.w.omega.0T}
\end{align}
Note that, in this notation, $W[t] = W_{\RR^{n+1}}[t]$ and $W[0,T] = W_{\RR^{n+1}}[0,T]$.

The following is a localization of \cite[Theorem 5.2]{White:topology-weak}.

\begin{theorem}\label{theo:connected-components}
Let $\cM$ be a level set flow and $\Omega\subset \RR^{n+1}$ be an open set with smooth boundary, so that $\cM$ is a simple flow in $U \times [0,T]$ for some tubular neighborhood $U$ of $\partial \Omega$. Then:
\begin{enumerate}
\item For every point $X$ in $W_\Omega[0,T]$, there is a time-like path in $W_\Omega[0,T]$ joining $X$ to a point $Y=(\mathbf{y},0)$ at time $0$.
\item If $X,Y$ are in different connected components of $W_\Omega[0]$, then they are in different connected components of $W_\Omega[0,T]$. 
\end{enumerate}
\end{theorem}
\begin{proof}
To prove (1), note that for $X \in U\times [0,T]$, it is not hard to construct such a path (by the simplicity assumption). In general, by \cite[Theorem 5.2(i)]{White:topology-weak}, we can find a time-like path in $\cM^c$ connecting $X$ to time $0$. If this path remains in $\Omega\times [0,T]$, the claim follows. On the other hand, if the path does not remain in $\Omega\times [0,T]$, then it must enter $U\times [0,T]$ at some point. In this case, we can stop and concatenate with the path in $U\times [0,T]$ that exists by the fact that the flow is smooth in that region. 

For (2), consider $X,Y \in W_{\Omega}[0]$ that are in distinct connected components of $W_{\Omega}[0]$, but in the same connected component of $W_{\Omega}[0,T]$. First, consider the case when at least one of the points, say $X$ is in a connected component $V$ of $W_{\Omega}[0]$ that does not intersect the tubular neighborhood $U$. Because $\cM$ is simple in $U\times [0,T]$, the component $\cV$ of $\cM^{c} \cap \mathfrak{t}^{-1}([0,T])$ containing $V$ does not intersect $U\times [0,T]$; thus, it is contained in $\Omega\times [0,T]$. As such, $X,Y$ are in distinct components of $W[0] : = \cM^{c} \cap \mathfrak{t}^{-1}(\{0\})$ but in the same component of $W[0,T] : = \cM^{c}\cap \mathfrak{t}^{-1}([0,T])$. This contradicts \cite[Theorem 5.2(ii)]{White:topology-weak}. Thus, both $X$ and $Y$ must be connected in $W_{\Omega}[0]$ to $U$. As such we can assume below, without loss of generality, that $X,Y \in U$.

Let us set up some notation. For each connected component $V$ of $W_\Omega[0]$, we write $V_U : = V\cap U$ (note that $V_U$ may be disconnected). Write $\partial V_U = \partial_{-}V_U \cup \partial_{+}V_U \cup \partial_{\cM}V_U$, where $\partial_{-}V_U = (\partial V\cap\partial\Omega)\setminus \cM(0)$, $\partial_{+}V_U = (V \cap \partial U \cap \Omega)\setminus \cM(0)$ and $\partial_{\cM}V = \partial V\cap \cM(0)$ are distinct and $\partial_{-}V_U$ (resp.\  $\partial_{+}V_U$) is relatively open in $\partial \Omega$ (resp.\ $\partial U$). Let $V(X) \not = V(Y)$ denote the components of $W_\Omega[0]$ containing $X,Y$. 

Because $X$ and $Y$ are assumed to be in the same connected component of $W_{\Omega}[0,T]$, they are in the same connected component of $W[0]$ by \cite[Theorem 5.2(ii)]{White:topology-weak}. Choose a path $\gamma\subset W[0]$ between $X$ and $Y$ so that $\gamma$ is transverse to $\partial U \cup \partial \Omega$. For $* \in \{X,Y\}$, we can assume that $\gamma$ does not intersect $\partial_+ V(*)_U$ (we might have to exchange the points $*\in\{X,Y\}$ for some other point in $V(*)_U$). Indeed, we can simply consider the last time that $\gamma$ intersects $\partial_+V(X)_U$ and the earliest time that $\gamma$ intersects $\partial_+V(Y)_U$ and truncate $\gamma$ near these times (to still have endpoints in $U$). 

Choose a curve $\eta \subset W_\Omega[0,T]$ from $Y$ to $X$ so that $\eta \cap (U \times [0,T]) \subset U \times \{0\}$ and consists of two arcs exiting $U$ through $\partial_+ V(Y)_U \cup \partial_+ V(X)_U$ (with a single transverse intersection with each). Concatenating $\gamma$ with $\eta$, we can find a loop $\sigma_1$ in $W[0,T]$. By \cite[Theorem 5.4]{White:topology-weak}, there is a homotopy of loops in W[0,T] between $\sigma_1$ and a loop $\sigma_0$ in $W[0]$. Perturb $\sigma_{0}$ slightly so it is transverse to $\partial U$. By construction and the simplicity of $\cM$ in $U\times [0,T]$, the loop $\sigma_0$ has the property that for $* \in \{X,Y\}$, the mod $2$ intersection number of $\sigma_{0}$ with $\partial_{+} V(*)_U$ is $1$. This is a contradiction. 
\end{proof}
The following is a localized version of \cite[Theorem 5.4]{White:topology-weak}. 
\begin{theorem}\label{theo:localized-loops-to-0}
Let $\cM$ be a level set flow and $\Omega\subset \RR^{n+1}$ be an open set with smooth boundary, so that $\cM$ is a simple flow in $U \times [0,T]$ for some tubular neighborhood $U$ of $\partial \Omega$ in $\Omega$. Then, any loop in $W_\Omega[0,T]$ is homotopic to one in $W_\Omega[0]$. In particular 
\[
H_1(W_\Omega[0]) \to H_1(W_\Omega[0,T])
\]
is surjective. 
\end{theorem}
\begin{proof}
Fix a cover $\Pi: \tilde W_\Omega[0,T]\to W_\Omega[0,T]$ associated to $\iota_0: W_\Omega[0]\to W_\Omega[0,T]$. Set $\hat W : = \cM^c \cap (U\times[0,T])$. Note that $\hat W \subset W_\Omega[0,T]$ deformation retracts onto $W_U[0]\subset \hat W$, by the assumption that $\cM$ is simple in $U\times [0,T]$. Set $W_{\Omega,k}[0,T] : = (W_\Omega[0,T] \cap W_k[0,T]) \cup \hat W$, where $W_k[0,T] = W_k \cap \mathfrak{t}^{-1}([0,T])$ (see Section \ref{sec:prelim.weak.flow.level.set.flow} for the definition of $W_k$). Because $\hat W$ deformation retracts onto $W_U[0]$, we can find a lift
\[
\tilde \iota_0 :  W_{\Omega,0}[0,T] \to \tilde W_\Omega[0,T]. 
\]
In the remainder of the proof we inductively define lifts of $\iota_k : W_{\Omega,k}[0,T] \to W_\Omega[0,T]$,
\[
\tilde\iota_k : W_{\Omega,k}[0,T] \to \tilde W_\Omega[0,T],
\]
so that $\tilde \iota_k |_{W_{\Omega,k-1}[0,T]} = \tilde\iota_{k-1}$. Having done so, we can fit these lifts together to produce a lift $\tilde\iota : W_\Omega[0,T]\to \tilde W_\Omega[0,T]$; thus, the covering $\Pi$ was trivial, completing the proof.  

Let $\cM'$ be a classical flow corresponding to $F: M'\times [a,b] \to \RR^{n+1}$ in $\RR^{n+1} \times [0,T]$ disjoint from $\cM(0)$ so that $\partial\cM' \subset W_{k-1}$. Set $\cM'_\Omega : = \cM'\cap (\Omega\times [0,T])$. (Note that $\cM'$ might not intersect $\partial\Omega \times [0,T]$ transversely and there is no guarantee that points in $\cM'_\Omega$ can be connected to a part of the heat boundary of $\cM'$.) 

\begin{claim}
There is a unique lift $\phi : \cM'_\Omega \to \tilde{W}_\Omega[0,T]$ so that $\phi(X) = \tilde\iota_{k-1}(X)$ for all points $X \in \cM'\cap W_{\Omega,k-1}[0,T]$. 
\end{claim} 
\begin{proof}
Fix $X = F({p},t)\in \cM'_\Omega$. Choose an open set $\cO\Subset \Omega\times[0,T]$ so that $(\Omega \setminus U) \times [0,T]  \subset \cO$, $X \in \cO$, and $\partial\cO$ is a small $C^\infty$ perturbation of $\partial\Omega\times[0,T]$ intersecting $\cM$ transversely. Define
\[
t_0 = \inf\{\tau \in [a,t] : F({p}\times (\tau,t))\subset \cO\}. 
\]
It is clear that $F({p},t_0) \in W_{\Omega,k-1}[0,T]$, so we can consider $\tilde \gamma$ the unique lift of the curve $\gamma : [t_0,t] \ni \tau \mapsto F({p},\tau)$ with $\tilde\gamma(t_0) = \tilde\iota_{k-1}(F({p},t_0))$. We then define $\phi(X) = \tilde\gamma(t)$. 

It is clear that $\phi$ is continuous and does not depend on the choice of $\cO$. It remains to check that $\phi(X) = \tilde\iota_{k-1}(X)$ for $X \in \cM'\cap W_{\Omega,k-1}[0,T]$. Choose $\cO$ as above and let $V$ denote the connected component of $\cM'\cap W_{\Omega,k-1}[0,T]\cap\cO$ containing $X$. The argument in \cite[Lemma 5.3]{White:topology-weak} can be easily adapted to show that $V$ contains a point $Y\in\partial \cM'\cup \partial\cO\subset W_{\Omega,k-1}[0,T]$. 
 Since $\phi(Y) = \tilde\iota_{k-1}(Y)$, the maps agree on all of $V$. This completes the proof of the claim. 
\end{proof}

\begin{claim}
If $\cM_1,\cM_2$ are two classical flows with heat boundaries in $W_{k-1}$ and $X \in \cM_1\cap\cM_2 \cap(\Omega\times [0,1])$ then $\phi_1(X) = \phi_2(X)$. 
\end{claim}
\begin{proof}
Given $X$, we can choose $\cO$ as above but with $\partial \cO$ transverse to $\cM_1$ and $\cM_2$. Now, as in \cite[p.\ 328]{White:topology-weak}, the maximum principle guarantees that there is a connected subset $K$ of $\cM_1\cap \cM_2$ containing $X$ and some point in $ \partial\cM_1\cup\partial\cM_2$. Either $K\cap\partial\cO = \emptyset$, in which case there is $Y \in (\partial\cM_1\cup \partial\cM_2) \cap K \cap \cO$ or $K\cap\partial\cO \not= \emptyset$, in which case there is $Y \in K \cap \partial\cO$. Either way, $Y \in W_{\Omega,k-1}[0,T]$. By the previous claim, $\phi_1(Y) = \tilde\iota_{k-1}(Y) = \phi_2(Y)$. Because $\phi_1|_K,\phi_2|_K$ agree at $Y$, they must also agree at $X$. 
\end{proof}

This completes the proof. 
\end{proof}

The following is a localized version of \cite[Theorem 6]{White:topology-weak}.

\begin{theorem}\label{theo:local-top-monotone-kernel-timeT}
Consider  $\Omega\subset \RR^{n+1}$ an open set with smooth compact boundary and $U$ a fixed tubular neighborhood of $\partial \Omega$. Choose $T_{0}$ so that the mean curvature flow of $\partial\Omega$, $t\mapsto \partial\Omega(t)$ remains smooth and inside some open set $\tilde U\Subset U$ for $t \in [0,T_{0}]$. Then, for any $0<T\leq T_{0}$, let $\cM$ be a weak set flow in $\RR^{n+1}$ that is simple in $U \times [0,T]$. Then,
\[
H_{n-1}(W_{\Omega}[T]) \to H_{n-1}(W_{\Omega}[0,T]) 
\]
is injective. 
\end{theorem}
\begin{proof}
For $0<T\leq T_{0}$ fixed, suppose that $[C]\in H_{n-1}(W_{\Omega}[T])$ is a polyhedral $(n-1)$-chain so that there is $P$ a polyhedral $n$-chain in $W_{\Omega}[0,T]$ with $\partial P = C$. We can assume that the support $\Gamma$ of $P$ is disjoint from $\tilde U\cup\{t=0\}$. Consider the projection $\pi(\bx,t) = (\bx,T)$. Set $\pi_{\#}P = P'$ and note that $\partial P'=C$. We aim to show that $P'$ is homologous (relative to its boundary) to a chain disjoint from $\cM(T)$.

Let $\cM'$ be the level set flow generated by $\Gamma$. By the avoidance principle for weak set flows (cf.\ \cite[Theorem 4.1]{White:topology-weak}), $\cM'(t)$ remains a positive distance from $\cM(t)$ as well as a positive distance from $\partial\Omega(t)$. In particular, we can enlarge $\Omega$ slightly to $\Omega'$ to ensure that $\cM'$ avoids some tubular neighborhood $U'$ of $\partial\Omega'$ (so in particular, it is a simple flow in $U'\times [0,T]$).

Fatten $\cM'(T)$ slightly to get a closed set $K$ in $\RR^{n+1}\times \{T\}$ that is disjoint from $\tilde U \cup \cM(T)$ and has smooth boundary. If $\gamma$ is a loop in $(\Omega\times\{T\})\setminus K$, then by Theorem \ref{theo:localized-loops-to-0} applied to $\cM'$, $\gamma$ is homologous in $(\Omega' \times [0,T])\setminus \cM'$ to a loop at $t=0$. In particular, this means that the oriented intersection number of $\gamma$ with $P$ (and thus $P'$) is zero. 

Now, assign each component of 
\[
(\Omega'\times \{T\}) \setminus (K\cup P')
\]
a multiplicity so that the multiplicity changes by $n$ when crossing a face of $P'$ with multiplicity $n$; we can do this consistently, since the intersection of any loop avoiding $K$ with $P'$ is zero (this is only well defined up to a global additive constant, but this will not matter). This yields a $(n+1)$ chain $Q$ in $\Omega'\times \{T\}$ whose boundary is a chain in $K$ along with the part of $P'$ that is disjoint from $K$. Now $P'-\partial Q$ has $\partial(P'-\partial Q) = C$ and is supported in $K$. As such, $P'-\partial Q$ is disjoint from $\cM(t)$. The result follows. 
\end{proof}

\bibliographystyle{alpha}
\bibliography{ac-avoidance}

\end{document}